\documentclass[10pt,reqno]{amsart}

\usepackage{graphicx}
\usepackage{times}
\usepackage[colorlinks=true,linkcolor=blue,citecolor=blue]{hyperref}%
\usepackage{dsfont}

\usepackage{xcolor}
\usepackage{stmaryrd}
\usepackage{pdfrender,xcolor}
\usepackage{tikz}
\usepackage{pgfplots}
\usetikzlibrary{decorations.pathmorphing}

\newtheorem{theorem}{Theorem}[section]
\newtheorem{lemma}[theorem]{Lemma}
\newtheorem{assump}[theorem]{Assumption}
\newtheorem{corollary}[theorem]{Corollary}
\newtheorem{prop}[theorem]{Proposition}
\theoremstyle{definition}
\newtheorem{definition}[theorem]{Definition}
\newtheorem{notation}[theorem]{Notation}
\newtheorem{example}[theorem]{Example}
\theoremstyle{remark}
\newtheorem{remark}[theorem]{Remark}
\numberwithin{equation}{section}

\usepackage[top=1.2in, bottom=1.2in, left=1.2in, right=1.2in]{geometry}

\usepackage[scr]{rsfso}
\usepackage[english]{babel}
\usepackage[utf8x]{inputenc}
\usepackage{fancyhdr}
\usepackage{amsmath}
\usepackage{amsthm}
\usepackage{amssymb}
\usepackage{comment}
\usepackage{enumitem}
\setlist{leftmargin=*}
\usepackage[integrals]{wasysym}
\usepackage{bm}

\newcommand\nc{\newcommand}
\nc{\on}{\operatorname}
\nc{\E}{\mathbb{E}}
\nc{\R}{\mathbb R}
\nc{\C}{\mathbb C}
\nc{\Q}{\mathbb Q}
\nc{\Z}{\mathbb Z}
\nc{\N}{\mathbb N}
\nc{\F}{\mathbb F}
\nc{\T}{\mathrm{T}}
\nc{\wt}{\widetilde}
\nc{\ol}{\overline}
\nc{\rnc}{\renewcommand}
\nc{\e}{\varepsilon}
\nc{\DMO}{\DeclareMathOperator}
\nc{\grad}{\nabla}
\nc{\fsp}{\fontdimen2\font=2.3pt}
\nc{\fspp}{\fontdimen2\font=2.17pt}
\nc{\abbr}[1]{{\sc\lowercase{#1}}}

\nc{\bphi}{{\boldsymbol{\phi}}}

\nc{\ocolor}{\color}

\nc{\X}{{X}}
\nc{\Y}{{Y}}
\rnc{\t}{{t}}
\nc{\x}{{x}}
\nc{\y}{{y}}
\nc{\s}{{s}}
\nc{\z}{{z}}
\nc{\w}{{w}}
\rnc{\r}{{r}}
\rnc{\a}{{a}}
\rnc{\b}{{b}}
\rnc{\k}{{k}}
\rnc{\u}{{u}}
\nc{\n}{{n}}
\newcommand{\m}{{m}}
\rnc{\L}{{L}}
\rnc{\leq}{\leqslant}
\rnc{\geq}{\geqslant}
\rnc{\d}{\mathrm{d}}
\rnc{\O}{\mathrm{O}}
\newenvironment{nouppercase}{%
  \renewcommand{\uppercasenonmath}[1]{}}{}
\title{\fsp\Large {{KPZ} equation from a class of nonlinear {SPDE}s in infinite volume}}

\author{Kevin Yang}\thanks{Department of Mathematics and Department of Statistics, University of California, Berkeley; yangkev@berkeley.edu}

\usepackage{setspace}
\begin{document}
\setstretch{1.0}
\fsp
\raggedbottom
\begin{nouppercase}
\maketitle
\end{nouppercase}
\begin{center}
\today
\end{center}

\begin{abstract}
\fspp We study a general class of nonlinear Ginzburg-Landau \abbr{SPDE}s in infinite volume under weak nonlinearity scaling and with non-equilibrium initial data. We derive the \abbr{KPZ} equation as a continuum limit of these equations. This makes rigorous the original derivation of the \abbr{KPZ} equation from physics \cite{KPZ} in the full-space setting, which was a problem posed in \cite{HQ}. Our analysis is based on a stochastic heat kernel for a linearization of said \abbr{SPDE}s.
\end{abstract}

{\hypersetup{linkcolor=blue}
\setcounter{tocdepth}{1}
\tableofcontents}

\allowdisplaybreaks
\section{Introduction}\label{section:intro}
The \emph{Kardar-Parisi-Zhang} (\abbr{KPZ}) equation is a stochastic \abbr{PDE} (\abbr{SPDE}) that was proposed in \cite{KPZ} as a \emph{universal} model for fluctuations of random interfaces whose growth has a local slope dependence \cite{C11,Q}. This equation is given below, where $\alpha>0$ and $\beta\in\R$ are fixed:
\begin{align}
\partial_{\t}\mathbf{h}_{\t,\X}&=\alpha\partial_{\X}^{2}\mathbf{h}_{\t,\X}+\beta|\partial_{\X}\mathbf{h}_{\t,\X}|^{2}+\xi_{\t,\X}, \quad(\t,\X)\in[0,\infty)\times\R.\label{eq:kpz}
\end{align}
Above, $\xi$ is a space-time white noise, i.e. a Gaussian process on $[0,\infty)\times\R$ with covariance kernel {\small$\E\xi_{\t,\X}\xi_{\s,\Y}=\delta_{\s=\t}\delta_{\X=\Y}$}. In this paper, we take the Cole-Hopf solution given by defining $\mathbf{h}:=\alpha\beta^{-1}\log\mathbf{Z}$, where {\small$\mathbf{Z}$} solves
\begin{align}
\partial_{\t}\mathbf{Z}_{\t,\X}&=\alpha\partial_{\X}^{2}\mathbf{Z}_{\t,\X}+\tfrac{\beta}{\alpha}\mathbf{Z}_{\t,\X}\xi_{\t,\X}.\label{eq:she}
\end{align}
To be precise, if we let {\small$\mathbf{H}$} be the Gaussian heat kernel associated to {\small$\partial_{\t}-\alpha\partial_{\X}^{2}$}, then we say that a (mild or Duhamel) solution to \eqref{eq:she} is a function on $[0,\infty)\times\R$ which solves the following stochastic integral equation:
\begin{align}
\mathbf{Z}_{\t,\X}&={\int_{\R}}\mathbf{H}_{0,\t,\X,\Y}\mathbf{Z}_{0,\Y}\d\Y+\tfrac{\beta}{\alpha}{\int_{0}^{\t}\int_{\R}}\mathbf{H}_{\s,\t,\X,\Y}\mathbf{Z}_{\s,\Y}\xi_{\s,\Y}\d\Y\d\s, \quad (\t,\X)\in[0,\infty)\times\R.
\end{align}
To see \eqref{eq:kpz} as a model for interface fluctuations, take a Gaussian free field on $\R$. Its natural Langevin dynamics is the \emph{additive} stochastic heat equation given by \eqref{eq:kpz} with {\small$\beta=0$}. Finally, to give the dynamics a slope-dependence, one introduces the quadratic first-order term. We emphasize, however, that this reasoning makes the choice of a Gaussian free field and the choice of a quadratic function of {\small$\partial_{\X}\mathbf{h}$}.

A popular generalization of this to non-Gaussian field theories \cite{HH,Spohn} and non-quadratic nonlinearities \cite{HQ} gives rise to another \abbr{SPDE} that is often known as a \emph{Ginzburg-Landau} \abbr{SPDE} \cite{HH,Spohn}. It is given by
\begin{align}
\partial_{\t}\mathbf{J}_{\t,\X}&=\partial_{\X}\mathscr{U}'(\partial_{\X}\mathbf{J}_{\t,\X})+{\color{black}\mathbf{F}[\partial_{\X}\mathbf{J}_{\t,\X}]}+\xi_{\t,\X},\quad(\t,\X)\in[0,\infty)\times\R.\label{eq:glcont}
\end{align}
Above, $\mathscr{U}$ is a smooth ``potential". The term {\small$\partial_{\X}\mathscr{U}'(\partial_{\X}\mathbf{J}_{\t,\X})$} is a second-order operator, and {\color{black}we take $\mathbf{F}$ to be a polynomial in the corresponding first-order operator {\small$\mathscr{U}'(\partial_{\X}\mathbf{J}_{\t,\X})$}}. Like \eqref{eq:kpz}, this \abbr{SPDE} is analytically ill-posed. There does not seem to be a Cole-Hopf map that linearizes \eqref{eq:glcont} in general. So, one must first regularize the \abbr{SPDE} \eqref{eq:glcont} in order to study it. A popular regularization \cite{DGP,CYau,Spohn} is the following space-discretization.
\begin{itemize}
\item First, discretize space; this makes \eqref{eq:glcont} into an \abbr{SDE}, which can be solved rigorously.
\item Take the ``continuum limit" of solutions to the discretized equations as the scale of the discretization vanishes.
\end{itemize}
Since the \abbr{RHS} of \eqref{eq:glcont} consists of a (nonlinear) smoothing mechanism {\small$\partial_{\X}\mathscr{U}'(\partial_{\X}\mathbf{J}_{\t,\X})$}, a slope-dependent flux {\small$\mathbf{F}[\mathscr{U}'(\partial_{\X}\mathbf{J}_{\t,\X})]$}, and a space-time white noise {\small$\xi_{\t,\X}$}, the {weak \abbr{KPZ} universality conjecture} (for \eqref{eq:glcont}) says that if we apply \emph{weak nonlinearity} scaling for the polynomial {\small$\mathbf{F}$} (relative to the discretization scale), then such a continuum limit converges to the solution of \eqref{eq:kpz} itself (with {\small$\alpha,\beta$} depending on {\small$\mathscr{U},\mathbf{F}$}) for quadratic {\small$\mathscr{U}$}. This conjecture dates back to the original physics paper \cite{KPZAIM}. It is the focus of open problems {\small$3$} and {\small$6$} in Section 1.3 of \cite{HQ} (the former concerns general smoothing mechanisms, and the latter concerns the non-compact nature of the domain {\small$\R$}).

\emph{In this paper, we establish this convergence} for a choice of discretization that has explicit invariant measures; this is natural from the field-theory perspective  \cite{Spohn}, though our discretization also happens to prohibit comparison principles even if {\small$\mathscr{U}$} is convex. We consider a large class of {\small$\mathscr{U}$} and polynomials {\small$\mathbf{F}$}, and we are able to derive the \abbr{KPZ} equation with essentially initial data by analyzing \eqref{eq:glcont} for a large class of \emph{non-equilibrium} initial data for \eqref{eq:glcont}. Combining our results with \cite{ACQ} gives a first derivation of random matrix eigenvalue behavior, a hallmark signature of strongly interacting systems, from \eqref{eq:glcont}. See Theorems \ref{theorem:main} and \ref{theorem:mainwedge} for our results. 

Our results extend the works \cite{HQ,HS,HX,KWX,KZ,YEJP,YEJP25} from the torus {\small$\R/\Z$} to the full-line {\small$\R$}. (As we noted earlier, this extension is the focus of open problem {\small$6$} in Section 1.3 of \cite{HQ}.) Our results also improve \cite{DGP} from Brownian initial data to essentially arbitrary initial data at the level of the limit \abbr{KPZ} equation; they also improve on the regularity assumptions on the initial data in \cite{HQ,HS,HX,KWX,KZ,YEJP,YEJP25}. In particular, this paper gives the first derivation of \eqref{eq:kpz} from \eqref{eq:glcont} for non-equilibrium data. See Section \ref{subsection:background} for more discussion.

{At a technical level, we show that the exponential Cole-Hopf transform of \eqref{eq:glcont} converges to the \abbr{SHE}, despite the fact that said Cole-Hopf map does \emph{not} linearize \eqref{eq:glcont} exactly. In particular, the Cole-Hopf transform of \eqref{eq:glcont} satisfies a perturbation of \eqref{eq:she} by additional, complicated stochastic error terms. Thus, many of the problems in the analysis of general infinite-volume singular \abbr{SPDE}s appear in the Cole-Hopf \abbr{SPDE} as well.

Our analysis of the Cole-Hopf \abbr{SPDE} (and proof of the main theorems) is based on constructing a \emph{(stochastic) heat kernel} for said \abbr{SPDE}. This lets us carefully estimate the propagation speed of \eqref{eq:glcont} in space and effectively reduce its analysis to the compact setting. The main difficulty behind this is in combining fairly weak stochastic estimates with the analytic theory of heat kernels. This method is in contrast to the weak solution and comparison principle methods in energy solutions (used in \cite{DGP,YEJP,GP17}) and the fixed-point methods in regularity structures (used in \cite{HQ,HS,HX,KWX,KZ} for quadratic {\small$\mathscr{U}$}). See Section \ref{subsection:intro-method} for an intuitive discussion of this method.
\subsection{Background}\label{subsection:background}
Essentially all of the earlier progress on universality for \eqref{eq:glcont} is restricted to the torus {\small$\R/\Z$} instead of the line {\small$\R$}. The first such result, which takes quadratic {\small$\mathscr{U}(\a)=\a^{2}$} and polynomial {\small$\mathbf{F}$}, is of \cite{HQ}, whose methods are based in regularity structures and are thus significantly different than those of this paper. Extensions to more general {\small$\mathbf{F}$} followed in \cite{HS,HX,KWX,KZ,YEJP}. For non-quadratic (but still uniformly convex) potentials and linear {\small$\mathbf{F}$}, see \cite{YEJP25}, which also restricts to the torus, since it is also based on fixed-point methods.

For progress on universality on the line, see \cite{DGP}; this considers general {\small$\mathscr{U}$}, but it is limited to a highly special invariant measure initial data; it is based on the method of ``energy solutions", which avoids the Cole-Hopf transform and is thus substantially different than the methods of this paper. In particular, the use of ``energy solutions" depends crucially on important tools from the theory of stochastic homogenization that are available exclusively for invariant measure initial data; finding an alternative to these tools remains an outstanding problem. We also emphasize that the methods of \cite{YEJP}, which use comparison principles to extend beyond invariant measure initial data, do not apply at all to the \abbr{SPDE}s in this paper. This is due to the possible non-convexity of {\small$\mathscr{U}$} in \eqref{eq:glcont} and the nature of the discretization needed to have an explicit invariant measure.

The derivation of \abbr{KPZ} from other interfaces in statistical mechanics has, in general, seen a wealth of attention (see the ``Big Picture Questions" in \cite{KPZAIM}). Prior results in this more general direction include a derivation of \abbr{KPZ} from interacting particle systems \cite{ACQ,BG,CGST,CS,CST,CT,CTIn,DT,GJ15,GJS15,GJ17,YCMP,YFOM}, directed polymers and related models \cite{AC,AKQ,CG,DDP1,DDP2,P2}, and other \abbr{SPDE} interface models \cite{DY_23,GHS}. Some of these hold in infinite-volume. However, in the ones that do, either the Cole-Hopf transform exactly linearizes the associated ``height function" (in which case this height function is ``integrable"), or said height function is a perturbation of an ``integrable" one that is so weak that it does not even contribute to {\small$\alpha,\beta$} in the limit \abbr{KPZ} equation.

{As for well-posedness of singular \abbr{SPDE}s in infinite-volume, we refer to \cite{HL18} for the \abbr{SHE} via regularity structures in weighted function spaces. We also refer to \cite{PR19,ZZZ} for a (very recent!) construction of solutions to \abbr{KPZ} in infinite-volume by means of paracontrolled distributions, also in weighted function spaces. These two works, however, are crucially dependent on the exact structure of \abbr{KPZ}, in that the potential {\small$\mathscr{U}$} in \eqref{eq:glcont} must be Gaussian, and the nonlinearity {\small$\mathbf{F}$} is quadratic.

{We also note the work \cite{BC24}, which studies a geometric singular \abbr{SPDE} by means of perturbing around a (deterministic) \emph{geometric} (i.e. covariant) heat kernel to access the global (in time) behavior.}

{Finally, let us also mention \cite{CYau}, which considers \eqref{eq:glcont} with {\small$\mathbf{F}\equiv0$} and derives \eqref{eq:kpz} with {\small$\beta=0$}.}
\subsection{Acknowledgements}
The author was supported by \abbr{NSF} \abbr{DMS}-2203075. We thank Bjoern Bringmann for bringing \cite{BC24} to our attention, {\color{black}as well as Cyril Labb\'{e} and Ran Tao for interesting conversations. We also thank two anonymous referees for their valuable comments, which led to important improvements in the paper} {\color{black}(including the inclusion of Remark \ref{remark:ws})}.
%
%
%
\section{Main results}\label{section:main}
We start by introducing a space-discretization of \eqref{eq:glcont}. The following yields infinite-dimensional \abbr{SDE}s (with dimensions indexed by {\small$\Z$}); their solution theory is discussed in Appendix A of \cite{DGP}. First, some notation.
\begin{itemize}
\item In this paper, elements in $\R^{\Z}$ will be denoted by {\small${\boldsymbol{\phi}}=({\boldsymbol{\phi}}_{\x})_{\x\in\Z}$}. For any $\x\in\Z$, we define {\small$\tau_{\x}:\R^{\Z}\to\R^{\Z}$} to be a shift operator, so that {\small$(\tau_{\x}{\boldsymbol{\phi}})_{\w}:={\boldsymbol{\phi}}_{\x+\w}$} for all $\w\in\Z$ and {\small${\boldsymbol{\phi}}\in\R^{\Z}$}.
\item The parameter $N$ is our scaling parameter; it will always be an integer, and we will take it to $\infty$.
\end{itemize}
With more notation to be introduced shortly, the \abbr{SPDE} of interest is given as follows, in which $\t\geq0$ and $\x\in\Z$:
\begin{align}
\d\mathbf{j}^{N}_{\t,\x}&=N^{\frac32}\grad^{\mathbf{X}}_{1}\mathscr{U}'[{\boldsymbol{\phi}}_{\t,\x}]\d\t+N\mathbf{F}[\tau_{\x}{\boldsymbol{\phi}}_{\t}]{\color{black}\d\t}+\sqrt{2}N^{\frac12}\d\mathbf{b}_{\t,\x}.\label{eq:curr}
\end{align}
%
\begin{itemize}
\item Above, $\mathscr{U}\in\mathscr{C}^{\infty}(\R)$ is a fixed ``potential" with derivative {\small$\mathscr{U}'$}. We make assumptions on it in Assumption \ref{assump:potential}, and we comment on the well-posedness of \eqref{eq:curr} after Assumption \ref{assump:potential} in Remark \ref{remark:potential}.
\item For any {\small$\mathfrak{l}\in\Z$}, the operator {\small$\grad^{\mathbf{X}}_{\mathfrak{l}}$} is the length-$\mathfrak{l}$ space-gradient, i.e. {\small$\grad^{\mathbf{X}}_{\mathfrak{l}}\mathsf{f}_{\x}:=\mathsf{f}_{\x+\mathfrak{l}}-\mathsf{f}_{\x}$} for any $\mathsf{f}:\Z\to\R$.
\item The processes {\small$\t\mapsto\mathbf{b}_{\t,\x}$}, for {\small$\x\in\Z$}, are jointly independent standard Brownian motions.
\item We define {\small${\boldsymbol{\phi}}_{\t,\x}:=N^{1/2}(\mathbf{j}^{N}_{\t,\x}-\mathbf{j}^{N}_{\t,\x-1})$} to be a rescaled space-gradient of {\small$\mathbf{j}^{N}$}, and we set {\small${\boldsymbol{\phi}}_{\t}:=({\boldsymbol{\phi}}_{\t,\x})_{\x\in\Z}\in\R^{\Z}$}.
\item The function $\mathbf{F}$ is defined as follows for any input {\small${\boldsymbol{\phi}}\in\R^{\Z}$} and for a fixed, deterministic constant {\small$\beta_{2}\neq0$}:
\begin{align}
\mathbf{F}[{\boldsymbol{\phi}}]:=\mathbf{F}_{2}[\bphi]+\mathbf{F}_{>2}[\bphi]&:=\tfrac13\beta_{2}\left(\mathscr{U}'[{\boldsymbol{\phi}}_{1}]\mathscr{U}'[{\boldsymbol{\phi}}_{2}]+\mathscr{U}'[{\boldsymbol{\phi}}_{0}]\mathscr{U}'[{\boldsymbol{\phi}}_{1}]+\mathscr{U}'[{\boldsymbol{\phi}}_{-1}]\mathscr{U}'[{\boldsymbol{\phi}}_{0}]\right)\nonumber\\
&+\sum_{\d=3}^{\deg}\beta_{\d}\Big(\prod_{\ell=1}^{\d}\mathscr{U}'[\bphi_{\ell}]+\prod_{\ell=1}^{\d}\mathscr{U}'[\bphi_{\ell-1}]+\ldots+\prod_{\ell=1}^{\d}\mathscr{U}'[\bphi_{-\ell+1}]\Big).\label{eq:nonlinearity}
\end{align}
Above, {\small$\deg>0$} is a fixed positive integer (indicating the degree of {\small$\mathbf{F}$}), and {\small$\beta_{3},\ldots,\beta_{\deg}\in\R$} are also fixed. We emphasize the assumption that $\mathbf{F}[\bphi]$ has no degree-$1$ part. This is for convenience, since including a degree-$1$ part would yield a \abbr{KPZ} limit for $\mathbf{j}^{N}$ only after following a diverging-speed characteristic in space (see \cite{DGP}). Finally, to be completely clear, we note that {\small$\mathbf{F}[\tau_{\x}{\boldsymbol{\phi}}_{\t}]$} in \eqref{eq:curr} is given by replacing {\small${\boldsymbol{\phi}}_{\cdot}$} with {\small${\boldsymbol{\phi}}_{\t,\x+\cdot}$} in \eqref{eq:nonlinearity}.
\end{itemize}
{\color{black}Let us now turn to {\small${\boldsymbol{\phi}}_{\t,\x}=N^{1/2}(\mathbf{j}^{N}_{\t,\x}-\mathbf{j}^{N}_{\t,\x-1})=-N^{1/2}\grad^{\mathbf{X}}_{-1}\mathbf{j}^{N}_{\t,\x}$}. First, we note that the scaling {\small$N^{1/2}$} differs from the natural scaling of {\small$N$} for a discrete gradient. Indeed, the scaling factor {\small$N^{1/2}$} is chosen so that {\small$\bphi_{\t,\x}$} is typically of order {\small$1$} since, a posteriori, if {\small$\mathbf{j}^{\N}$} is close to \eqref{eq:kpz}, then it should have H\"{o}lder regularity in space of roughly {\small$1/2$}. We now record the \abbr{SDE} satisfied by {\small$\t\mapsto\bphi_{\t,\x}$} using \eqref{eq:curr}, which we explain after:}
\begin{align}
\d{\boldsymbol{\phi}}_{\t,\x}&=N^{2}\Delta\mathscr{U}'[{\boldsymbol{\phi}}_{\t,\x}]\d\t+N^{\frac32}\wt{\mathbf{F}}[\tau_{\x}{\boldsymbol{\phi}}_{\t}]-\sqrt{2}N\grad^{\mathbf{X}}_{-1}\d\mathbf{b}_{\t,\x}.\label{eq:phi}
\end{align}
Above, {\small$\Delta=-\grad^{\mathbf{X}}_{1}\grad^{\mathbf{X}}_{-1}$} is a discrete Laplacian. The function {\small$\wt{\mathbf{F}}[\bphi]:=\mathbf{F}[\bphi]-\mathbf{F}[\tau_{-1}\bphi]$} is given by
\begin{align}
\wt{\mathbf{F}}[\tau_{\x}{\boldsymbol{\phi}}_{\t}]:=\wt{\mathbf{F}}_{2}[\tau_{\x}{\boldsymbol{\phi}}_{\t}]+\wt{\mathbf{F}}_{>2}[\tau_{\x}{\boldsymbol{\phi}}_{\t}]&:=\tfrac13\beta_{2}(\mathscr{U}'[{\boldsymbol{\phi}}_{\x+1}]\mathscr{U}'[{\boldsymbol{\phi}}_{\x+2}]-\mathscr{U}'[{\boldsymbol{\phi}}_{\x-2}]\mathscr{U}'[{\boldsymbol{\phi}}_{\x-1}])\nonumber\\
&+\sum_{\d=3}^{\deg}\beta_{\d}\Big(\prod_{\ell=1}^{\d}\mathscr{U}'[\bphi_{\x+\ell}]-\prod_{\ell=1}^{\d}\mathscr{U}'[\bphi_{\x-\ell}]\Big)\label{eq:phi-nl}
\end{align}
Our choice of discretization \eqref{eq:curr} ensures that \eqref{eq:phi} has an explicit one-parameter family of invariant measures. Let us now introduce these invariant measures, since they will be important for the statement of our main result.
\begin{definition}\label{definition:gcmeasure}
Fix any $\sigma\in\R$. Consider the product probability measure below on {\small$\R^{\Z}$}:
\begin{align}
\mathbb{P}^{\sigma}[\d{\boldsymbol{\phi}}]:=\prod_{\x\in\Z}\mathcal{Z}_{\upsilon_{\sigma},\mathscr{U}}^{-1}\exp\left[-\mathscr{U}[{\boldsymbol{\phi}}_{\x}]+\upsilon_{\sigma}{\boldsymbol{\phi}}_{\x}\right]\d{\boldsymbol{\phi}}_{\x}.\label{eq:gcmeasure}
\end{align}
Here, {\small$\mathcal{Z}_{\upsilon_{\sigma},\mathscr{U}}$} is a normalizing partition function, and {\small$\upsilon_{\sigma}\in\R$} is chosen so that {\small$\E^{\sigma}{\boldsymbol{\phi}}_{\x}=\sigma$} for all $\x\in\Z$, where {\small$\E^{\sigma}$} is expectation with respect to {\small$\mathbb{P}^{\sigma}$}. (Well-defined-ness of {\small$\mathbb{P}^{\sigma}$} will follow from Assumption \ref{assump:potential}.)
\end{definition}
The probability measure {\small$\mathbb{P}^{\sigma}$} is often called the \emph{grand-canonical ensemble} or \emph{grand-canonical measure}. The proof that these are invariant measures for \eqref{eq:phi} is a standard generator computation; {see Appendix \ref{subsection:generator}}.

The final object we must introduce is the microscopic Cole-Hopf map
\begin{align}
\mathbf{Z}^{N}_{\t,\x}&:=\exp\left(\lambda\mathbf{j}^{N}_{\t,\x}-\lambda\mathscr{R}_{\lambda}\t\right),\label{eq:ch}
\end{align}
where the constants {\small$\lambda,\mathscr{R}_{\lambda}$} are defined as follows. First, we define {\small$\alpha:=\partial_{\sigma}\E^{\sigma}\mathscr{U}'[{\boldsymbol{\phi}}_{0}]|_{\sigma=0}$} and
\begin{align}
\beta:=\tfrac12\partial_{\sigma}^{2}\E^{\sigma}(\mathbf{F}_{2}[{\boldsymbol{\phi}}])|_{\sigma=0}.\label{eq:beta}
\end{align}
Set {\small$\lambda:=\alpha^{-1}\beta$}. {It turns out that {\small$\alpha>0$} and {\small$\beta\neq0$}; see {Appendix \ref{subsection:alpha-beta}}.} Let us clarify that {\small$\beta$}, and thus the limiting \abbr{SPDE} for {\small$\mathbf{Z}^{N}$}, does \emph{not} depend on (the coefficients in) {\small$\mathbf{F}_{>2}$}; indeed, only the ``quadratic piece" of {\small$\mathbf{F}$} should persist in the large-{\small$N$} limit under the weakly asymmetric scaling that leads to \eqref{eq:kpz}. Next, the renormalization {\small$\mathscr{R}_{\lambda}$} is
\begin{align}
\mathscr{R}_{\lambda}:=\tfrac{1}{12}\lambda^{3}\E^{0}\left(\mathscr{U}'[{\boldsymbol{\phi}}_{0}]{\boldsymbol{\phi}}_{0}^{3}\right){-\tfrac16\beta_{2}\lambda^{2}}.\label{eq:renorm}
\end{align}
%
\subsection{Precise statement of main results}
We first introduce assumptions on the potential $\mathscr{U}:\R\to\R$. 
\begin{assump}\label{assump:potential}
\fsp Suppose {\small$\mathscr{U}\in\mathscr{C}^{\infty}(\R)$} has the form {\small$\mathscr{U}=\mathscr{U}_{1}+\mathscr{U}_{2}$}, where {\small$\Lambda^{-1}\leq\mathscr{U}''_{1}[\cdot]\leq\Lambda$} for a fixed {\small$\Lambda\geq1$}, and where {\small$\|\mathscr{U}_{2}\|_{\mathrm{L}^{\infty}(\R)}+\|\mathscr{U}_{2}''\|_{\mathrm{L}^{\infty}(\R)}<\infty$}. We also assume that {\small$\upsilon_{0}=0$} (see Definition \ref{definition:gcmeasure}).
\end{assump}
\begin{remark}\label{remark:potential}
\fsp Following Appendix A of \cite{DGP}, we solve \eqref{eq:curr} by first solving \eqref{eq:phi}, which we do by putting it on a large discrete torus {\small$\mathbb{T}_{\ell}$} of length {\small$\ell$}, establishing product invariant measures as in Appendix \ref{subsection:generator}, and extending the length {\small$\ell\to\infty$} using the sub-Gaussian one-dimensional marginals of said invariant measures. (One can also restrict \eqref{eq:phi} to {\small$\mathbb{T}_{\ell}$} and take {\small$\ell\to\infty$} with {\small$N$} at \emph{any} rate without losing any of the main points of this paper.)

To solve \eqref{eq:phi} after restricting to {\small$\mathbb{T}_{\ell}$}, we note that although the \abbr{RHS} is not uniformly Lipschitz in the solution, one can construct local solutions and extend them to global solutions via said invariant measures. 
\end{remark}
The second assumption presents an a priori estimate on randomness of the initial data. We will shortly clarify its statement, purpose, and to what extent it is perhaps necessary for \eqref{eq:curr} to converge to \abbr{KPZ}.
\begin{assump}\label{assump:noneq}
\fsp Let {\small$\mathbb{P}^{\mathrm{init}}$} be the probability measure on {\small$\R^{\Z}$} for the law of the initial data to \eqref{eq:phi}. Let {\small$\mathfrak{P}^{\mathrm{init}}$} be its Radon-Nikodym derivative with respect to {\small$\mathbb{P}^{0}$}. We assume that there exists a fixed constant  {\small$\gamma_{\mathrm{data}}>0$} such that 
\begin{align}
\|\mathfrak{P}^{\mathrm{init}}\|_{\mathrm{L}^{\infty}(\R^{\Z})}\lesssim N^{\frac13-\gamma_{\mathrm{data}}}.\label{eq:noneq}
\end{align}
Above, and throughout the paper, the meaning of {\small$\a\lesssim\b$} is that $|\a|\leq\mathrm{C}|\b|$ for a fixed constant $\mathrm{C}>0$. 
\end{assump}
We now briefly explain Assumption \ref{assump:noneq}. First, having invariant measure initial data essentially corresponds to \eqref{eq:noneq} but with {\small$1$} on the \abbr{RHS}. In particular, \eqref{eq:noneq} gives a substantial improvement on the initial data assumption in \cite{DGP}. Indeed, as explained in Example \ref{example:niceinit}, allowing for a divergent \abbr{RHS} in \eqref{eq:noneq} is what lets us treat essentially any initial data for the limit \abbr{KPZ} equation (instead of just Brownian initial data as in \cite{DGP} or sufficiently regular initial data as in \cite{HQ,HS,HX,KWX,KZ,YEJP,YEJP25} ). Thus, the assumption \eqref{eq:noneq} is far from being a perturbative one.

Second, it is possible to let the \abbr{RHS} of \eqref{eq:noneq} grow exponentially in {\small$N$}. Indeed, \eqref{eq:noneq} will only be used for the spatially-local theory of \eqref{eq:curr}. This analysis is done in \cite{YEJP25} with an exponentially large \abbr{RHS} of \eqref{eq:noneq}. However, the details of \cite{YEJP25} are complicated, and the assumption \eqref{eq:noneq} substantially simplifies them. Thus, in order to focus on the main \emph{new} ideas needed to go to infinite volume while keeping this paper self-contained (and providing a much simpler non-equilibrium analysis than \cite{YEJP25}), we adopt \eqref{eq:noneq}. See Remark \ref{remark:localidea} for a further discussion.

Finally, it is not even clear if one can remove the assumption \eqref{eq:noneq} entirely for general potentials {\small$\mathscr{U}$}. Indeed, a key use of \eqref{eq:noneq} (or an exponential relaxation as in the above paragraph) is in establishing regularity estimates for \eqref{eq:curr}. It is unclear if one could show said estimates otherwise, i.e. via analytic means, given the quasilinear and singular nature of \eqref{eq:noneq}. (For quadratic {\small$\mathscr{U}$}, the equation \eqref{eq:curr} is semilinear instead of quasilinear, so it is perhaps possible to remove \eqref{eq:noneq} in this specific case. Because our interest is in general {\small$\mathscr{U}$} for the sake of universality, we leave this to future work.)
\begin{example}\label{example:niceinit}
With respect to the measure {\small$\mathbb{P}^{0}$} (from Definition \ref{definition:gcmeasure}), the {\small${\boldsymbol{\phi}}_{0,\x}$} random variables are independent, mean-zero, and sub-Gaussian random walk increments. Thus, when we sum these random variables and rescale by {\small$N^{-1/2}$}, which is how one gets {\small$\mathbf{j}^{N}_{0,\cdot}$} from {\small${\boldsymbol{\phi}}_{0,\cdot}$}, the process {\small$\X\mapsto\mathbf{j}^{N}_{0,N\X}$} is roughly Brownian; this is by Donsker's principle. (Although {\small$\X\mapsto\mathbf{j}^{N}_{0,N\X}$} is only defined for {\small$\X\in N^{-1}\Z$}, we extend it to {\small$\X\in\R$} by linear interpolation.)

By a classical support theorem for Brownian motion, for any continuous {\small$\mathsf{f}:\R\to\R$}, any compact set {\small$\mathbb{K}\subseteq\R$}, and any $\e>0$, the probability that a standard Brownian motion is within $\e$ of $\mathsf{f}$ uniformly on $\mathbb{K}$ is bounded away from zero by a constant depending on $\e,\mathbb{K}$, and the size of {\small$\mathsf{f}$}. So, an allowable initial data is given by conditioning {\small$\mathbb{P}^{0}$} so that {\small$\X\mapsto\mathbf{j}^{N}_{0,N\X}$} is within $\e$ of a given continuous $\mathsf{f}:\R\to\R$ uniformly on $\mathbb{K}\subseteq\R$. Since the \abbr{RHS} of the constraint \eqref{eq:noneq} diverges in {\small$N$}, we can allow for $\mathbb{K}$ and $\mathsf{f}$ to grow in $N$, and for the error $\e$ to vanish in $N$.
\end{example}
\subsubsection{Universality for continuous initial data}
Our main result has two halves to it. The first focuses on initial data for {\small$\mathbf{Z}^{N}$} that is continuous with sufficiently controlled growth at $\infty$. Before we state it, let us (as above) first declare that any function which is defined on the lattice {\small$N^{-1}\Z$} extends to {\small$\R$} by linear interpolation. 
\begin{theorem}\label{theorem:main}
\fsp Suppose that in addition to Assumptions \ref{assump:potential} and \ref{assump:noneq}, the following assumptions are satisfied.
\begin{enumerate}
\item For any $p\geq1$ and $\zeta>0$, we have the following moment estimates for some $\kappa_{p}>0$ and any $|\mathfrak{l}|\lesssim N$:
\begin{align}
\sup_{\x\in\Z}\E|\mathbf{Z}^{N}_{0,\x}|^{2p}\lesssim_{p}\exp(\tfrac{\kappa_{p}|x|}{N})\quad\mathrm{and}\quad\sup_{\x\in\Z}\E|\grad^{\mathbf{X}}_{\mathfrak{l}}\mathbf{Z}^{N}_{0,\x}|^{2p}\lesssim_{p,\zeta}N^{-p+p\zeta}|\mathfrak{l}|^{p-p\zeta}\exp(\tfrac{\kappa_{p}|x|}{N}).\label{eq:mainI}
\end{align}
\item The initial data {\small$\X\mapsto\mathbf{Z}^{N}_{0,N\X}$} converges locally uniformly in {\small$\X\in\R$} to some function {\small$\X\mapsto\mathbf{Z}_{0,\X}$} in probability.
\end{enumerate}
Then, there exists a coupling between {\small$\mathbf{Z}^{N}$} and the solution $\mathbf{Z}$ to \eqref{eq:she} with initial data {\small$\mathbf{Z}_{0,\cdot}$} from bullet point (2) above such that {\small$\mathbf{Z}^{N}_{\t,N\X}-\mathbf{Z}_{\t,\X}\to0$} locally uniformly on $[0,1]\times\R$ in probability in the large-$N$ limit.
\end{theorem}
The restriction to a time-horizon of $1$ in Theorem \ref{theorem:main} is for convenience; any finite time-horizon is okay.
\subsubsection{Universality for wedge-type initial data}\label{subsubsection:mainwedge}
We will now illustrate an application of our methods to a special example of singular initial data that is necessarily far from equilibrium. For this setting, the relevant solution to \eqref{eq:she} is the narrow-wedge solution with initial data given by the Dirac point mass at $0$. Precisely, this solution {\small$\mathbf{Z}^{\mathrm{nw}}$} is the unique (adapted) solution to
\begin{align}
\mathbf{Z}^{\mathrm{nw}}_{\t,\X}&=\mathbf{H}_{0,\t,\X,0}+{\int_{0}^{\t}\int_{\R}}\mathbf{H}_{\s,\t,\X,\Y}\mathbf{Z}^{\mathrm{nw}}_{\s,\Y}\xi_{\s,\Y}\d\Y\d\s,
\end{align}
where $\mathbf{H}$ is the heat kernel for {\small$\partial_{\t}-\alpha\partial_{\X}^{2}$}. For the existence and uniqueness of solutions, see Section 1 of \cite{ACQ}.
\begin{theorem}\label{theorem:mainwedge}
\fsp There exists a choice of initial data for \eqref{eq:curr}-\eqref{eq:phi} such that Assumption \ref{assump:noneq} for small {\small$\gamma_{\mathrm{data}}>0$} is satisfied, such that the estimates in \eqref{eq:mainI} hold, and such that the following holds for some deterministic {\small$\mathcal{T}_{N}\in\R$}:
\begin{itemize}
\item There exists a coupling between {\small$\mathbf{Z}^{N}$} and the narrow-wedge solution {\small$\mathbf{Z}^{\mathrm{nw}}$} to \eqref{eq:she} such that for any fixed $\tau>0$, we have {\small$\mathcal{T}_{N}\mathbf{Z}^{N}_{\t,N\X}-\mathbf{Z}^{\mathrm{nw}}_{\t,\X}\to0$} locally uniformly on $[\tau,1]\times\R$ in probability. 
\end{itemize}
\end{theorem}
The initial data in Theorem \ref{theorem:mainwedge} will (eventually) be constructed by using Example \ref{example:niceinit} with {\small$\mathsf{f}$} therein becoming more and more singular in the large-{\small$N$} limit, ultimately converging to a Dirac point mass at {\small$0$}. (In particular, the {\small$\e>0$} parameter in Example \ref{example:niceinit} vanishes as {\small$N\to\infty$}.) This is explained in detail in Section \ref{subsection:mainproofs}.

By Corollary 1.3 in \cite{ACQ}, we know that {\small$\log\mathbf{Z}^{\mathrm{nw}}_{\t,\X}-\log\mathbf{H}_{0,\t,\X,0}$} converges weakly as {\small$\t\to\infty$} to the Tracy-Widom distribution for fluctuations of extreme eigenvalues of random Hermitian matrices. {Theorem \ref{theorem:mainwedge} shows that the same is true for {\small$\mathbf{Z}^{N}_{\t,N\X}$} in place of {\small$\mathbf{Z}^{\mathrm{nw}}_{\t,\X}$} under a double limit given by first taking {\small$N\to\infty$} then {\small$\t\to\infty$}}. (Because our coupling estimates in Theorems \ref{theorem:main} and \ref{theorem:mainwedge} are actually quantitative in {\small$N$}, one can actually estimate how the two scaling parameters in this double limit are allowed to relate to each other. We omit this estimate, however, since it both requires tracking optimal exponents throughout this paper for all the analysis and is likely far from the optimal ``{\small$1$:$2$:$3$}" \abbr{KPZ} scaling for \eqref{eq:curr}.)
\subsection{Plan for the paper}
In Section \ref{subsection:intro-method}, we give heuristics of the key technical ideas before we delve into details. In Section \ref{section:proofoutlinemain}, we break the proof of Theorems \ref{theorem:main} and \ref{theorem:mainwedge} into several propositions and lemmas, which we spend the rest of the paper proving. A more detailed outline for the proof of these ingredients is given in Section \ref{subsection:paper-outline}.
\subsection{Notation}\label{subsection:notation}
The key pieces of notation used throughout this paper are introduced in the current section and Section \ref{section:proofoutlinemain}, the latter of which contains the main steps in the proofs of Theorems \ref{theorem:main} and \ref{theorem:mainwedge}. Most other notation is introduced in the section where is used. In any case, to aid the reading of this paper, we briefly summarize the notation (and related conventions) used in this paper as follows, focusing on notation used in multiple sections.

\begin{enumerate}
\item We use big-Oh notation, i.e. we write {\small$\a=\mathrm{O}(\b)$} if {\small$|\a|\leq\mathrm{C}|\b|$} for some constant $\mathrm{C}>0$. The parameters on which $\mathrm{C}$ depends will be placed as subscripts in the $\mathrm{O}$-notation. We also write {\small$\a\lesssim\b$} to mean {\small$\a=\mathrm{O}(\b)$}, and we will also write {\small$\a=\mathrm{o}(\b)$} if $|\a|/|\b|$ vanishes in the large-$N$ limit.
\item For any {\small$\a\leq\b$}, we sometimes write {\small$\llbracket\a,\b\rrbracket:=[\a,\b]\cap\Z$}.
\item We will say that an event holds with \emph{high probability} if its probability is {\small$1-\mathrm{o}(1)$}.
\item The object {\small$\mathbf{Z}^{N}$} is the Cole-Hopf map \eqref{eq:ch}, in which {\small$\mathbf{j}^{N}$} is a microscopic height function \eqref{eq:curr} whose discrete derivative {\small$\bphi$} solves \eqref{eq:phi}. The object {\small$\mathbf{S}^{N}$} {\color{black}in \eqref{eq:zsmooth}} is a convolution of {\small$\mathbf{Z}^{N}$} with a smooth kernel {\small$\mathscr{S}^{N}$} on a length-scale {\small$N^{1-\delta_{\mathbf{S}}}$} just below the macroscopic length-scale {\small$N$}; here, {\small$\delta_{\mathbf{S}}>0$} is small, and these objects are introduced in Definition \ref{definition:zsmooth}. Also introduced there is the ratio {\small$\mathbf{R}^{N}:=\mathbf{Z}^{N}/\mathbf{S}^{N}$}; spatial regularity of {\small$\mathbf{Z}^{N}$} will help us show that {\small$\mathbf{R}^{N}$} is close to {\small$1$}, and a cutoff-version {\small$\mathbf{R}^{N,\wedge}$} of {\small$\mathbf{R}^{N}$} to reflect this is given in \eqref{eq:rwedge}.
\begin{itemize}
\item Any bold English letter with one space and one time subscript is a modification of {\small$\mathbf{Z}^{N}$} or {\small$\mathbf{S}^{N}$}. For example, {\small$\mathbf{S}^{N,\zeta}$} is a modification of {\small$\mathbf{S}^{N}$} given by implementing space-cutoffs in the dynamics of {\small$\mathbf{S}^{N}$}; see \eqref{eq:szeta-sde}. Also, {\small$\mathbf{Q}^{N}$} is the solution to the lattice \abbr{SHE} \eqref{eq:qsde}.
\item Bold and capitalized \emph{Greek} letters are often used to follow existing conventions in the literature (see {\small$\boldsymbol{\Psi}$} in Notation \ref{notation:localav}) or to conveniently package important terms in our analysis (see {\small$\boldsymbol{\Phi}$} in \eqref{eq:lkzetaIa}, for example).
\end{itemize}
\item The function {\small$\boldsymbol{\chi}^{(\zeta)}:\Z\to\R$} is a smoothened version of the indicator of a neighborhood of {\small$0$} of radius {\small$N^{1+\zeta}$}, i.e. a factor of {\small$N^{\zeta}$}-larger than the macroscopic length-scale.
\begin{itemize}
\item We will consider one large parameter {\small$\zeta_{\mathrm{large}}$}, as well as a separate sequence of slightly decreasing parameters {\small$\zeta_{\i}$} in Notation \ref{notation:zeta} that gets close to {\small$0$}.
\end{itemize}
\item Any bold letter with two space subscripts and two time subscripts is a space-time kernel. For example, {\small$\mathbf{H}^{N}$} is a classical (diffusively-scaled) heat kernel for a symmetric simple random walk on {\small$\Z$}; see \eqref{eq:heatkernel}, and {\small$\mathbf{K}^{N,\zeta}$} is the \emph{stochastic} kernel (i.e. fundamental solution) for the {\small$\mathbf{S}^{N,\zeta}$} evolution. There are other kernels, like {\small$\mathbf{K}^{N,\zeta,\sim}$} and {\small$\mathbf{L}^{N,\zeta,\sim}$} from Section \ref{section:stochheat} that are technical modifications of {\small$\mathbf{K}^{N,\zeta}$}, but these will be defined in context.
\item Notation such as {\small$\mathfrak{q}$} (possibly with subscripts) often denote functions of \eqref{eq:phi} that are error terms (in the {\small$\mathbf{Z}^{N}$} and {\small$\mathbf{S}^{N}$} dynamics); these errors are called \emph{admissible} and are classified in Definition \ref{definition:admissible}.
\item The space-time averaging operator {\small$\mathbf{Av}^{\mathbf{T},\mathbf{X}}$}, the space-averaging operator {\small$\mathbf{Av}^{\mathbf{X}}$}, and their averaging scales {\small$\mathfrak{n}_{\mathbf{Av}}:=N^{1-3\delta_{\mathbf{S}}/2}$} and {\small$\mathfrak{t}_{\mathbf{Av}}:=N^{-2/3-10\delta_{\mathbf{S}}}$} are introduced in Definition \ref{definition:eq-operators}. In the same place, we also introduce relevant space and time gradients {\small$\grad^{\mathbf{X}}$} and {\small$\grad^{\mathbf{T},\mathrm{av}}$}, respectively.
\end{enumerate}
%
%
%
\section{Heuristics behind Theorems \ref{theorem:main} and \ref{theorem:mainwedge}}\label{subsection:intro-method}
We now give a discussion of the main steps and ideas taken towards Theorems \ref{theorem:main} and \ref{theorem:mainwedge} without including all the rigorous details. The first (standard) step is to compute the evolution equation for {\small$\mathbf{Z}^{N}$} using \eqref{eq:curr}, \eqref{eq:ch}, and the It\^{o} formula. It is essentially given by 
\begin{align}
\d\mathbf{Z}^{N}_{\t,\x}&\approx N^{2}\alpha\Delta\mathbf{Z}^{N}_{\t,\x}+\sqrt{2}\lambda N^{\frac12}\mathbf{Z}^{N}_{\t,\x}\d\mathbf{b}_{\t,\x}+N\mathfrak{q}[\tau_{\x}\bphi_{\t}]\mathbf{Z}^{N}_{\t,\x}\d\t,\label{eq:heuristic1}
\end{align}
in which {\small$\mathfrak{q}$} is ``stochastically small" in the following sense. If we let {\small$\mathsf{Av}[\mathfrak{q}]_{\t,\x}$} be the space-time average of {\small$\mathfrak{q}[\tau_{\s}\bphi_{\z}]$} over {\small$(\s,\z)$} in a block of time-scale {\small$\mathfrak{t}$} and length-scale {\small$\mathfrak{l}$} around {\small$(\t,\x)$}, then we have the weak-type estimate
\begin{align}
\int_{0}^{1}\sum_{\x\in\mathbb{I}}\E|\mathsf{Av}[\mathfrak{q}]_{\t,\x}|^{2}\d\t\lesssim |\mathbb{I}|\cdot(N^{-2}\mathfrak{t}^{-1}|\mathfrak{l}|^{-1}+\ldots),\label{eq:heuristic2}
\end{align}
where {\small$\mathbb{I}$} denotes any interval in {\small$\Z$}, and the $\ldots$ on the \abbr{RHS} are lower-order error terms. (The time-horizon of {\small$1$} on the \abbr{LHS} of \eqref{eq:heuristic2} is not important, but it is convenient to fix.) We note that \eqref{eq:heuristic2} is essentially a \abbr{CLT}-type estimate, because it says that time-averaging and space-average both introduce square-root cancellations. (The extra {\small$N^{-2}$} factor appears because the speed of \eqref{eq:phi} is order {\small$N^{2}$}, and faster fluctuations means more cancellation.) We also clarify that the estimate \eqref{eq:heuristic2} deteriorates as the size {\small$|\mathbb{I}|$} increases. Indeed, we eventually take {\small$\mathfrak{t}\ll1$} and {\small$|\mathfrak{l}|\ll N$}, so \eqref{eq:heuristic2} is a \emph{local} stochastic estimate that we then sum over a possibly large set {\small$\mathbb{I}$}.

The quantity {\small$\mathsf{Av}[\mathfrak{q}]$} is of interest because, by regularity of {\small$\mathbf{Z}^{N}$}, it is natural to expect that if the averaging scales are sub-macroscopic (i.e. {\small$\mathfrak{t}\ll1$} and {\small$|\mathfrak{l}|\ll N$}), then we have 
\begin{align}
\d\mathbf{Z}^{N}_{\t,\x}&\approx N^{2}\alpha\Delta\mathbf{Z}^{N}_{\t,\x}+\sqrt{2}\lambda N^{\frac12}\mathbf{Z}^{N}_{\t,\x}\d\mathbf{b}_{\t,\x}+N\mathsf{Av}[\mathfrak{q}]_{\t,\x}\mathbf{Z}^{N}_{\t,\x}\d\t.\label{eq:heuristic3}
\end{align}
\emph{We warn the reader that this is possibly {not} quite true}, because {\small$\mathbf{Z}^{N}$} does not have ``great" regularity. In particular, when making these ideas rigorous, we will have to average both {\small$\mathfrak{q}$} and {\small$\mathbf{Z}^{N}$} together. We then use \eqref{eq:ch} to extract two factors from {\small$\mathbf{Z}^{N}$}. One is ``sufficiently smooth"; the other is some explicit functional of \eqref{eq:curr}-\eqref{eq:phi} that does not affect the \abbr{CLT}-type cancellation mechanism leading to \eqref{eq:heuristic2}. (In addition, for entirely technical reasons, we will also not look at {\small$\mathbf{Z}^{N}$} but rather its spatial-convolution with a smooth approximation to a delta function living on a length-scale of the type {\small$N^{1-\delta}$}, just below the macroscopic scale. We also have to consider the ratio between {\small$\mathbf{Z}^{N}$} and this small-scale smoothing, but these are rather technical points that we skip over for now.)
\begin{remark}\label{remark:localidea}
\fsp The a priori estimate \eqref{eq:noneq} will be used only to prove an appropriate, more detailed version of \eqref{eq:heuristic2} (as well as some more technical a priori estimates on spatial-averages of the {\small$\t\mapsto\bphi_{\t}$} process on scales {\small$\ll N$}). We also note that the methods in \cite{YEJP25} are designed to establish these local-type estimates under a much more relaxed assumption on the relative entropy (compared to \eqref{eq:noneq}). However, as noted after Assumption \ref{assump:noneq}, adopting \eqref{eq:noneq} {\color{black}instead of} a weaker relative entropy estimate allows for a much cleaner argument. Thus, we adopt \eqref{eq:noneq} to highlight and focus on the new key ideas for the infinite-volume problem.
\end{remark}
We now describe the key technical problem. The estimate \eqref{eq:heuristic2} is perhaps optimal in {\color{black}the sense} that current tools do not allow for a stronger version of it. This suggests that {\small$\mathsf{Av}[\mathfrak{q}]$} is not small uniformly in space-time, {\color{black}but} only after adding weights that deteriorate in space at infinity. This eventually leads to a non-close-ability issue when analyzing the Duhamel expansion of \eqref{eq:heuristic3} that often appears in the analysis of singular \abbr{SPDE}s in infinite-volume.

Thus, we avoid using the Duhamel expansion for \eqref{eq:heuristic3}, therefore departing from standard methods. Instead, we will consider the following \emph{hierarchy} of equations, parameterized by {\small$\zeta>0$}:
\begin{align}
\d\mathbf{Z}^{N,\zeta}_{\t,\x}&\approx N^{2}\alpha\Delta\mathbf{Z}^{N,\zeta}_{\t,\x}+\sqrt{2}\lambda N^{\frac12}\mathbf{Z}^{N,\zeta}_{\t,\x}{\color{black}\d\mathbf{b}_{\t,\x}}+N{\ocolor{blue}\mathbf{1}_{|\x|\leq N^{1+\zeta}}}\mathsf{Av}[\mathfrak{q}]_{\t,\x}\mathbf{Z}^{N,\zeta}_{\t,\x}\d\t.\label{eq:heuristic4}
\end{align}
We write {\small$\approx$} because of the error terms hidden in \eqref{eq:heuristic3}, but also because the exact defining equation for {\small$\mathbf{Z}^{N,\zeta}$} that we consider is slightly (for purely technical reasons) different. A short argument will let us compare \eqref{eq:heuristic3} and \eqref{eq:heuristic4} for {\small$\zeta$} large enough, but we want {\small$\zeta$} to be small in order to effectively reduce to the compact setting. To do this, we will let {\small$\mathbf{K}^{N,\zeta}_{\s,\t,\x,\y}$}, for {\small$(\s,\t,\x,\y)\in[0,\infty)^{2}\times\Z^{2}$} with {\small$\s\leq\t$}, denote the fundamental solution (or heat kernel) associated to \eqref{eq:heuristic4} (when viewed as an equation that is linear in its solution). \emph{Our key estimate is}
\begin{align}
\sup_{0\leq\s\leq\t\leq1}\sup_{\x\in\Z}\sum_{\y\in\Z}\exp(\tfrac{\kappa|\x-\y|}{N})|\mathbf{K}^{N,\zeta}_{\s,\t,\x,\y}|\lesssim \exp(N^{-\beta}N^{\zeta})N^{\delta},\label{eq:heuristic5}
\end{align}
where {\small$\kappa,\delta$} are arbitrary (but independent of {\small$N$}), and where {\small$\beta>0$} is independent of all other parameters in \eqref{eq:heuristic5}. (We clarify that the factor {\small$N^{\zeta}$} in \eqref{eq:heuristic5} comes from the support length of the blue indicator in \eqref{eq:heuristic4} and the factor {\small$|\mathbb{I}|$} in the estimate \eqref{eq:heuristic2}. The factor of {\small$N^{-\beta}$} comes from the \abbr{CLT}-cancellations as explained after \eqref{eq:heuristic2}. Moreover, the exponential nature of the \abbr{RHS} of \eqref{eq:heuristic5} comes from a Gronwall-type reason. The factor of {\small$N^{\delta}$} is technical.)

The estimate \eqref{eq:heuristic5} implies that because we only care about {\small$\mathbf{Z}^{N}_{\t,\x}$} for {\small$|\x|\lesssim1$} in Theorems \ref{theorem:main} and \ref{theorem:mainwedge}, we can modify \eqref{eq:heuristic4} at points outside the interval {\small$\llbracket-N^{1+\zeta-\e\beta},N^{1+\zeta-\e\beta}\rrbracket$} for any {\small$\e>0$} fixed up to an exponentially small (in {\small$N$}) error. In particular, we can slightly lower {\small$\zeta$} in \eqref{eq:heuristic4}, and doing this inductively lets us assume that {\small$\zeta>0$} is small in \eqref{eq:heuristic4}, thereby (effectively) reducing to the compact setting. \emph{We clarify that proving the estimate \eqref{eq:heuristic5} only uses \eqref{eq:heuristic2}} (and the other technical estimates for \eqref{eq:phi} as discussed in Remark \ref{remark:localidea}) \emph{as stochastic inputs}. Thus, if we condition on \eqref{eq:heuristic2} without the expectation as well as said estimates for \eqref{eq:phi}, then \eqref{eq:heuristic5} holds essentially deterministically (i.e. outside an event of probability {\small$\mathrm{O}(N^{-\mathrm{D}})$} for any {\small$\mathrm{D}=\mathrm{O}(1)$}). 

To establish \eqref{eq:heuristic5} rigorously, we will expand {\small$\mathbf{K}^{N,\zeta}$} around the fundamental solution to the \abbr{SHE} part of \eqref{eq:heuristic4}. However, the stochastic error term in \eqref{eq:heuristic4} is only small in a very weak sense, so classical expansion methods for the heat kernels do not seem to apply. A large portion of this paper is dedicated to devising an alternative scheme to analyze this expansion that is compatible with the weak-type estimate \eqref{eq:heuristic2} for the stochastic term.
\begin{remark}\label{remark:ws}
\fsp {\color{black}The estimation in \eqref{eq:heuristic5} of the {\small$\mathbf{K}^{N,\zeta}$}-kernel for different {\small$\zeta$}-parameters can be thought of as analogous to the use of weighted spaces (as in \cite{HL15,HL18}, for instance). It would be interesting to see a parallel between these two approaches in more detail. To this end, we reemphasize the key challenge that the equation \eqref{eq:heuristic3} of interest has a coefficient {\small$\mathsf{Av}[\mathfrak{q}]$} for which we currently have only the weak-type bound \eqref{eq:heuristic2} (and a few other inputs that are on the more technical side). In particular, one of the goals behind introducing the {\small$\mathbf{K}^{N,\zeta}$}-kernels is to analyze equations of this type, which do not appear to be addressed in \cite{HL15,HL18}, for example.}
\end{remark}
\begin{remark}\label{remark:sdeclarify}
\fsp {Let us briefly comment on \eqref{eq:heuristic3}. The main error term obtained by replacing {\small$\mathfrak{q}$} with {\small$\mathsf{Av}[\mathfrak{q}]$} (in going from \eqref{eq:heuristic1} to \eqref{eq:heuristic3}), which is suppressed in the {\small$\approx$} notation, is given by gradients of {\small$\mathfrak{q}$}. Ultimately, \eqref{eq:heuristic3} reads as
\begin{align}
\d\mathbf{Z}^{N}_{\t,\x}&\approx N^{2}\alpha\Delta\mathbf{Z}^{N}_{\t,\x}+\sqrt{2}\lambda N^{\frac12}\mathbf{Z}^{N}_{\t,\x}\d\mathbf{b}_{\t,\x}+N\mathsf{Av}[\mathfrak{q}]_{\t,\x}\mathbf{Z}^{N}_{\t,\x}\d\t+N\grad^{\star}(\mathfrak{q}[\tau_{\x}\bphi_{\t}]\mathbf{Z}^{N}_{\t,\x})\d\t,\label{eq:sdeclarify1}
\end{align}
where {\small$\grad^{\star}$} is a space-time average of space-time gradients on time-scales {\small$\lesssim\mathfrak{t}$} and length-scales {\small$\lesssim\mathfrak{l}$}. Here, {\small$\mathfrak{t},\mathfrak{l}$} are the averaging scales in {\small$\mathsf{Av}[\mathfrak{q}]$}; see \eqref{eq:heuristic2}. By the same token, we can further replace {\small$N\grad^{\star}(\mathfrak{q}[\tau_{\x}\bphi_{\t}]\mathbf{Z}^{N}_{\t,\x})\d\t$} with {\small$N\grad^{\star}(\mathsf{Av}[\mathfrak{q}]_{\t,\x}\mathbf{Z}^{N}_{\t,\x})\d\t+N\grad^{\star}\grad^{\star}(\mathfrak{q}[\tau_{\x}\bphi_{\t}]\mathbf{Z}^{N}_{\t,\x})\d\t$}, at which we point we iterate a finite number of times. Therefore, we are left to study a \emph{finite} number of terms of the form below (for all {\small$\ell=0,\ldots,\ell_{\star}$} for some {\small$\ell_{\star}=\mathrm{O}(1)$}):
\begin{align}
N(\grad^{\star})^{\ell}(\mathsf{Av}[\mathfrak{q}]_{\t,\x}\mathbf{Z}^{N}_{\t,\x})\d\t.\label{eq:sdeclarify2}
\end{align}
It turns out that only one time-averaged time-gradient will be necessary, not {\small$\ell$}-many. We will also need to study {\small$N\grad^{\mathbf{X}}_{\mathfrak{l}_{0}}$} applied to \eqref{eq:sdeclarify2} for all {\small$|\mathfrak{l}_{0}|\leq1$}, though this should be thought of as a bounded operator. Lastly, for entirely technical reasons, we will need to study not {\small$\mathbf{Z}^{N}$} but its convolution with a smooth approximation to the identity. In particular, all terms in \eqref{eq:sdeclarify1}-\eqref{eq:sdeclarify2} will be convolved with a mollifier (which absorbs all space-gradients).}
\end{remark}
%
%
\section{Proof outlines for the main results (Theorems \ref{theorem:main} and \ref{theorem:mainwedge})}\label{section:proofoutlinemain}
The goal of this section is to record the main ingredients needed to prove Theorems \ref{theorem:main} and \ref{theorem:mainwedge}. After stating these ingredients, we use them to prove Theorems \ref{theorem:main} and \ref{theorem:mainwedge}. The rest of this paper is dedicated to establishing each of the ingredients.
\subsection{Step 1 -- the basic \abbr{SDE}}
For various technical reasons, it will be convenient to not only consider the Cole-Hopf map \eqref{eq:ch}, but also a suitable space-regularization that we now introduce.
\begin{definition}\label{definition:zsmooth}
\fsp Consider a function {\small$\mathscr{S}\in\mathscr{C}^{\infty}_{\mathrm{c}}(\R)$} with non-negative values such that {\small$\int_{\R}\mathscr{S}_{\x}\d\x=1$}. We now define the following version of {\small$\mathscr{S}$} that is rescaled in space to length-scales of {\small$N^{1-\delta_{\mathbf{S}}}$}, where {\small$\delta_{\mathbf{S}}>0$} is a small parameter:
\begin{align}
\mathscr{S}^{N}_{\x}:=\mathrm{c}_{N}N^{-1+\delta_{\mathbf{S}}}\mathscr{S}_{N^{-1+\delta_{\mathbf{S}}}\cdot\x}
\end{align}
The constant {\small$\mathrm{c}_{N}=1+\mathrm{o}(1)$} is deterministic and chosen so that {\small$\|\mathscr{S}^{N}\|_{\ell^{1}(\Z)}=1$}. We will now define the following smoothing of {\small$\mathbf{Z}^{N}$} and the ratio between {\small$\mathbf{Z}^{N}$} and said smoothing:
\begin{align}
\mathbf{S}^{N}_{\t,\x}:=[\mathscr{S}^{N}\star\mathbf{Z}^{N}_{\t,\cdot}]_{\x}\quad\text{and}\quad \mathbf{R}^{N}_{\t,\x}:=\mathbf{Z}^{N}_{\t,\x}\Big/\mathbf{S}^{N}_{\t,\x}.\label{eq:zsmooth}
\end{align}
We clarify that {\small$\star$} denotes convolution of functions on {\small$\Z$}.
\end{definition}
We note that convolution against {\small$\mathscr{S}^{N}$} for sufficiently regular functions (like {\small$\mathbf{Z}^{N}$}) is approximately the identity map, because {\small$\mathscr{S}^{N}$} is supported on the length-scale {\small$N^{1-\delta_{\mathbf{S}}}$}, and functions like {\small$\mathbf{Z}^{N}$} have regularity on much larger scales of order {\small$N$}. In particular, replacing {\small$\mathbf{Z}^{N}$} with {\small$\mathbf{S}^{N}$} is harmless. On the other hand, the regularization in \eqref{eq:zsmooth} is enormously helpful (essentially) because we can now take as many derivatives as we want while only giving up small powers of {\small$N$} for each derivative; since {\small$\mathbf{Z}^{N}$} itself is not smooth, this upgrade is significant.

Next, we will introduce a convenient way to organize the main error terms in the {\small$\mathbf{S}^{N}$} equation.
\begin{definition}\label{definition:jets}
\fsp We say that {\small$\mathsf{F}:\R^{\Z}\to\R$} is a \emph{local function} if there exists a subset $\mathbb{I}\subseteq\Z$ independent of $N$ such that for any {\small${\boldsymbol{\phi}}\in\R^{\Z}$}, the value {\small$\mathsf{F}[{\boldsymbol{\phi}}]$} depends only on {\small${\boldsymbol{\phi}}_{\x}$} for $\x\in\mathbb{I}$; we often refer to {\small$\mathbb{I}$} as the \emph{support} of {\small$\mathsf{F}$}. 

Fix any {\small$\mathfrak{k}=0,1,2$}. We let {\small$\mathrm{Jet}_{\mathfrak{k}}^{\perp}$} tbe the set of local functions whose {\small$\mathfrak{k}$}-jet is zero, i.e. {\small$\mathsf{F}\in\mathrm{Jet}_{\mathfrak{k}}^{\perp}$} if and only if
\begin{align*}
\partial_{\sigma}^{\ell}\E^{\sigma}\mathsf{F}[\bphi]|_{\sigma=0}=0 \quad\text{for all } \ell=0,\ldots,\mathfrak{k}.
\end{align*}
\end{definition}
Intuitively, functions in {\small$\mathrm{Jet}_{\mathfrak{k}}^{\perp}$} should be thought of as having an additional factor of {\small$N^{-(\mathfrak{k}+1)/2}$}. Indeed, local equilibrium suggests that {\small$\mathsf{F}[\tau_{\x}\bphi_{\t}]$} averages out into {\small$\E^{\sigma[\tau_{\x}\bphi_{\t}]}\mathsf{F}$}, where {\small$\sigma[\tau_{\x}\bphi_{\t}]$} is a scale-{\small$\mathrm{o}(N)$} mollification of {\small$\bphi_{\t}$} near the point {\small$\x$}. The \abbr{CLT} suggests further that we essentially have {\small$|\sigma[\tau_{\x}\bphi_{\t}]|\lesssim N^{-1/2}$}, at which point the factor of {\small$N^{-(\mathfrak{k}+1)/2}$} comes from the Taylor expanding {\small$\E^{\sigma[\tau_{\x}\bphi_{\t}]}\mathsf{F}$} about {\small$0$}. In particular, any local function of the form {\small$N^{-1+\k/2}\mathfrak{f}$} with {\small$\mathfrak{f}\in\mathrm{Jet}_{\k}^{\perp}$} is thought of as {\small$\ll N^{-1}$}. Such a property, together with a technical support assumption and a priori estimate, will define the class of local functions which we must analyze.
\begin{definition}\label{definition:admissible}
\fsp We say that a local function {\small$\mathfrak{q}:\R^{\Z}\to\R$} is \emph{admissible} if:
\begin{enumerate}
\item it has the form {\small$\mathfrak{q}=N^{-1+\k/2}\mathfrak{f}$} with {\small$\mathfrak{f}\in\mathrm{Jet}_{\k}^{\perp}$} for some {\small$\k\in\{0,1,2\}$};
\item {\color{black}{\small$\mathfrak{f}_{}[\bphi]$}} depends only on {\small$\bphi_{\w}$} for either {\small$\w\in\{1,\ldots,\mathfrak{l}\}$} or {\small$\w\in\{-\mathfrak{l},\ldots,0\}$} for some {\small$\mathfrak{l}=\mathrm{O}(1)$};
\item we have the deterministic local polynomial estimate below for some positive {\small$\mathrm{C}\lesssim1$}:
\begin{align}
|{\color{black}\mathfrak{f}}[{\boldsymbol{\phi}}]|&\lesssim 1+ \sum_{|\w|\lesssim1}|\bphi_{\w}|^{\mathrm{C}}.\label{eq:f-estimate}
\end{align}
\end{enumerate}
\end{definition}
Lastly, before we record the dynamics for {\small$\mathbf{S}^{N}$}, we will introduce an efficient and intuitive way to package the terms therein. (Otherwise, the resulting equation is perhaps quite complicated, and its structure is easily lost.)
\begin{definition}\label{definition:eq-operators}
\fsp We will introduce the following operators for functions {\small$\mathsf{A}:[0,\infty)\times\R\to\R$}.
\begin{enumerate}
\item Let us introduce notation that will capture the main error terms. Suppose {\small$\mathfrak{q}\in\mathscr{C}^{\infty}(\R^{\Z})$} is a local function so that {\small$\mathfrak{q}[\bphi]$} depends only on {\small$\bphi_{\w}$} for {\small$\w\in\{1,\ldots,\mathfrak{l}_{\mathfrak{q}}\}$} for some {\small$1\leq\mathfrak{l}_{\mathfrak{q}}\lesssim1$}. We define the following space-time average, in which {\small$\mathfrak{t}_{\mathbf{Av}}=N^{-2/3-10\delta_{\mathbf{S}}}$} and {\small$\mathfrak{n}_{\mathbf{Av}}=N^{1-3\delta_{\mathbf{S}}/2}$}:
\begin{align}
\mathbf{Av}^{\mathbf{T},\mathbf{X},\mathfrak{q}}_{\t,\x}:=\mathfrak{t}_{\mathbf{Av}}^{-1}{\int_{0}^{\mathfrak{t}_{\mathbf{Av}}}}\mathfrak{n}_{\mathbf{Av}}^{-1}\sum_{\j=1,\ldots,\mathfrak{n}_{\mathbf{Av}}}\Big(\mathfrak{q}[\tau_{\x+\j}\bphi_{\t-\r}]\cdot\mathbf{Z}^{N}_{\t-\r,\x+\j}(\mathbf{Z}^{N}_{\t,\x})^{-1}\Big)\d\r.\label{eq:mult-op-1}
\end{align}
Note that {\small$\mathfrak{t}_{\mathbf{Av}}\ll1$} and {\small$\mathfrak{n}_{\mathrm{Av}}\ll N$}; thus, these scales are mesoscopic. {\color{black}We also note that the exact exponents of {\small$\approx-2/3$} and {\small$\approx1$} are not so meaningful; they are just sufficiently large to yield enough cancellations in \eqref{eq:mult-op-1}.} On the other hand, suppose {\small$\mathfrak{q}[\bphi]$} depends only on {\small$\bphi_{\w}$} for {\small$\w\in\{-\mathfrak{l}_{\mathfrak{q}},\ldots,0\}$} for some {\small$1\leq\mathfrak{l}_{\mathfrak{q}}\lesssim1$}. In this case, let us define the space-time average in the same way but with the ``left-wards" orientation:
\begin{align}
\mathbf{Av}^{\mathbf{T},\mathbf{X},\mathfrak{q}}_{\t,\x}:=\mathfrak{t}_{\mathbf{Av}}^{-1}{\int_{0}^{\mathfrak{t}_{\mathbf{Av}}}}\mathfrak{n}_{\mathbf{Av}}^{-1}\sum_{\j=1,\ldots,\mathfrak{n}_{\mathbf{Av}}}\Big(\mathfrak{q}[\tau_{\x-\j}\bphi_{\t-\r}]\cdot\mathbf{Z}^{N}_{\t-\r,\x-\j}(\mathbf{Z}^{N}_{\t,\x})^{-1}\Big)\d\r.\label{eq:mult-op-2}
\end{align}
We remark that \eqref{eq:mult-op-1} and \eqref{eq:mult-op-2} make sense for any other non-negative {\small$\mathfrak{t}_{\mathbf{Av}}$}. In particular, if {\small$\mathfrak{t}_{\mathbf{Av}}=0$}, then the meaning of \eqref{eq:mult-op-1}-\eqref{eq:mult-op-2} is obtained by setting {\small$\r=0$} and dropping the {\small$\d\r$} integration; for this, we write
\begin{align}
\mathbf{Av}^{\mathbf{X},\mathfrak{q}}_{\t,\x}:=\mathbf{Av}^{\mathbf{T},\mathbf{X},\mathfrak{q}}_{\t,\x}|_{\mathfrak{t}_{\mathbf{Av}}\mapsto0}.
\end{align}
In particular, by \eqref{eq:ch}, if {\small$\mathfrak{q}[\bphi]$} is supported on {\small$\bphi_{\w}$} for {\small$\w\geq1$}, then
\begin{align}
\mathbf{Av}^{\mathbf{X},\mathfrak{q}}_{\t,\x}&:=\mathfrak{n}_{\mathbf{Av}}^{-1}\sum_{\j=1,\ldots,\mathfrak{n}_{\mathbf{Av}}}\mathfrak{q}[\tau_{\x+\j}\bphi_{\t}]\cdot\exp\Big\{\lambda N^{-\frac12}(\bphi_{\t,\x+1}+\ldots+\bphi_{\t,\x+\j})\Big\}.\nonumber
\end{align}
On the other hand, if {\small$\mathfrak{q}[\bphi]$} is supported on {\small$\bphi_{\w}$} for {\small$\w\leq0$}, then
\begin{align}
\mathbf{Av}^{\mathbf{X},\mathfrak{q}}_{\t,\x}&:=\mathfrak{n}_{\mathbf{Av}}^{-1}\sum_{\j=1,\ldots,\mathfrak{n}_{\mathbf{Av}}}\mathfrak{q}[\tau_{\x-\j}\bphi_{\t}]\cdot\exp\Big\{\lambda N^{-\frac12}(\bphi_{\t,\x}+\ldots+\bphi_{\t,\x-\j+1})\Big\}.\nonumber
\end{align}
\item For any {\small$\mathfrak{t}_{0}\geq0$}, we define the following averaged time-gradient operator acting on functions {\small$\mathsf{f}:[0,1]\to\R$}:
\begin{align}
\grad^{\mathbf{T},\mathrm{av}}_{\mathfrak{t}_{0}}\mathsf{f}_{\t}:=\mathbf{1}_{\mathfrak{t}_{0}>0}\mathfrak{t}_{0}^{-1}{\int_{0}^{\mathfrak{t}_{0}}}\Big(\mathsf{f}_{\t}-\mathsf{f}_{\t-\r}\Big)\d\r.\label{eq:timegrad}
\end{align}
In \eqref{eq:timegrad}, we implicitly extend {\small$\mathsf{f}_{\t}$} to {\small$\t\not\in[0,1]$} via {\small$\mathsf{f}_{\t}=\mathsf{f}_{1}$} if {\small$\t\geq1$} and {\small$\mathsf{f}_{\t}=\mathsf{f}_{0}$} if {\small$\t\leq0$}.
\end{enumerate}
\end{definition}
Let us now record the evolution equation for {\small$\mathbf{S}^{N}$}; it is a lattice approximation to \abbr{SHE} with some error terms. These error terms more or less resemble convolutions with {\small$\mathscr{S}^{N}$} of \eqref{eq:sdeclarify2} (and its image under {\small$N\grad^{\mathbf{X}}_{\mathfrak{l}_{0}}$} for {\small$|\mathfrak{l}_{0}|\leq1$} as noted in Remark \ref{remark:sdeclarify}), in which the {\small$\mathfrak{q}$} are admissible local functions. However, it is convenient to use a notation that essentially unfolds each {\small$\grad^{\star}$} as a (weighted) average of gradients whose coefficients we write explicitly below (see the {\small$|\mathfrak{l}_{1}|^{-1}\ldots|\mathfrak{l}_{\m}|^{-1}\mathrm{c}_{N,\n,\mathfrak{l}_{1},\ldots,\mathfrak{l}_{\m}}$} in \eqref{eq:s-sdeI}-\eqref{eq:s-sdeII}). We clarify this notation immediately after Proposition \ref{prop:s-sde}, though let us again refer the reader to Remark \ref{remark:sdeclarify} for an intuitive version of the following result.
\begin{prop}\label{prop:s-sde}
\fsp There exist constants {\small$\mathrm{K},\mathrm{M}=\mathrm{O}(1)$} such that the following holds. We have the following \abbr{SDE} for {\small$\mathbf{S}^{N}$}, in which we use notation to be explained afterwards:
\begin{align}
\d\mathbf{S}^{N}_{\t,\x}&=\mathscr{T}_{N}\mathbf{S}^{N}_{\t,\x}\d\t+[\mathscr{S}^{N}\star(\sqrt{2}\lambda N^{\frac12}\mathbf{R}^{N}_{\t,\cdot}\mathbf{S}^{N}_{\t,\cdot}\d\mathbf{b}_{\t,\cdot})]_{\x}\label{eq:s-sde}\\
&+N\sum_{\substack{\n=1,\ldots,\mathrm{K}\\\m=0,\ldots,\mathrm{M}\\0<|\mathfrak{l}_{1}|,\ldots,|\mathfrak{l}_{\m}|\lesssim \mathfrak{n}_{\mathbf{Av}}}}\tfrac{1}{|\mathfrak{l}_{1}|\ldots|\mathfrak{l}_{\m}|}\mathrm{c}_{N,\n,\mathfrak{l}_{1},\ldots,\mathfrak{l}_{\m}}\grad^{\mathbf{X}}_{\mathfrak{l}_{1}}\ldots\grad^{\mathbf{X}}_{\mathfrak{l}_{\m}}[\mathscr{S}^{N}\star(\mathbf{Av}^{\mathbf{T},\mathbf{X},\mathfrak{q}_{\n}}_{\t,\cdot}\cdot\mathbf{R}^{N}_{\t,\cdot}\mathbf{S}^{N}_{\t,\cdot})]_{\x}\d\t\label{eq:s-sdeI}\\
&+N\sum_{\substack{\n=1,\ldots,\mathrm{K}\\\m=0,\ldots,\mathrm{M}\\0<|\mathfrak{l}_{1}|,\ldots,|\mathfrak{l}_{\m}|\lesssim \mathfrak{n}_{\mathbf{Av}}}}\tfrac{1}{|\mathfrak{l}_{1}|\ldots|\mathfrak{l}_{\m}|}\mathrm{c}_{N,\n,\mathfrak{l}_{1},\ldots,\mathfrak{l}_{\m}}\grad^{\mathbf{T},\mathrm{av}}_{\mathfrak{t}_{\mathbf{Av}}}\grad^{\mathbf{X}}_{\mathfrak{l}_{1}}\ldots\grad^{\mathbf{X}}_{\mathfrak{l}_{\m}}[\mathscr{S}^{N}\star(\mathbf{Av}^{\mathbf{X},\mathfrak{q}_{\n}}_{\t,\cdot}\cdot\mathbf{R}^{N}_{\t,\cdot}\mathbf{S}^{N}_{\t,\cdot})]_{\x}\d\t\label{eq:s-sdeII}\\
&+\mathrm{Err}[\mathbf{R}^{N}_{\t,\cdot}\mathbf{S}^{N}_{\t,\cdot}]_{\x}\d\t.\label{eq:s-sdeIII}
\end{align}
%
\begin{itemize}
\item We set {\small$\mathscr{T}_{N}:=(N^{2}\alpha+\frac12N\lambda^{2})\Delta$}, where {$\Delta$} is the discrete Laplacian as in \eqref{eq:phi}, and $\alpha,\lambda$ are defined after \eqref{eq:ch}. 
\item For each {\small$\n$} in the summations in \eqref{eq:s-sdeI}-\eqref{eq:s-sdeII}, the local function {\small$\mathfrak{q}_{\n}$} is admissible (see Definition \ref{definition:admissible}).
\item The parameters {\small$\mathfrak{t}_{\mathbf{Av}},\mathfrak{n}_{\mathbf{Av}}$} are from Definition \ref{definition:eq-operators}.
\item The {\small$\mathrm{c}_{N,\n,\mathfrak{l}_{1},\ldots,\mathfrak{l}_{\m}}\in\R$} are deterministic coefficients satisfying the following estimates.
\begin{itemize}
\item If {\small$\min\{|\mathfrak{l}_{1}|,\ldots,|\mathfrak{l}_{\m}|\}\geq2$}, then we have {\small$|\mathrm{c}_{N,\n,\mathfrak{l}_{1},\ldots,\mathfrak{l}_{\m}}|\lesssim 1$}. In general, we have {\small$|\mathrm{c}_{N,\n,\mathfrak{l}_{1},\ldots,\mathfrak{l}_{\m}}|\lesssim N$}.
\end{itemize}
\item The quantity {\small$\mathrm{Err}[\mathbf{R}^{N}\mathbf{S}^{N}]$} is a linear operator in {\small$\mathbf{S}^{N}$} which satisfies the following deterministic estimate for some positive {\small$\mathrm{C}\lesssim1$}, in which {\small$\mathsf{f}:\Z\to\R$} is any test function:
\begin{align}
|\mathrm{Err}[\mathbf{R}^{N}_{\t,\cdot}\mathsf{f}]_{\x}|\lesssim N^{-\frac12+\mathrm{C}\delta_{\mathbf{S}}}\cdot N^{-1+\delta_{\mathbf{S}}}\sum_{|\w|\lesssim N^{1-\delta_{\mathbf{S}}}}\Big(1+\sum_{|\z|\lesssim1}|\bphi_{\t,\x+\w+\z}|^{\mathrm{C}}\Big)\cdot|\mathbf{R}^{N}_{\t,\x+\w}|\cdot|\mathsf{f}_{\x+\w}|.\label{eq:err-estimate}
\end{align}
\end{itemize}
\end{prop}
The parameter {\small$\mathrm{K}$} controls the number of error terms appearing in the {\small$\mathbf{S}^{N}$} \abbr{SDE}. The parameter {\small$\mathrm{M}$} controls the number of replacements by space-time block averages that we alluded to in Remark \ref{remark:sdeclarify}. 

We will now explain the structure of these last three terms in \eqref{eq:s-sdeI}-\eqref{eq:s-sdeIII} in some more detail.
\begin{itemize}
\item The weight {\small$|\mathfrak{l}_{1}|^{-1}\ldots|\mathfrak{l}_{\m}|^{-1}$} in \eqref{eq:s-sdeI}-\eqref{eq:s-sdeII} should be thought of as inducing an average over {\small$\mathfrak{l}_{1},\ldots,\mathfrak{l}_{\m}$}. Indeed, the sum of this weight over such {\small$\mathfrak{l}_{1},\ldots,\mathfrak{l}_{\m}$} is {\small$\lesssim(\log N)^{\m}$}; since our estimates will be quantitative in powers of {\small$N$}, all {\small$\log$} factors will be unimportant. However, using this factor instead of {\small$\mathfrak{n}_{\mathbf{Av}}^{-\m}$} turns out to be convenient for various technical reasons.
\item The bound {\small$|\mathfrak{c}_{\n,\mathfrak{l}_{1},\ldots,\mathfrak{l}_{\m}}|\lesssim N$} for {\small$\min\{|\mathfrak{l}_{1}|,\ldots,|\mathfrak{l}_{\m}|\}\leq1$} is harmless. Indeed, if {\small$\min\{|\mathfrak{l}_{1}|,\ldots,|\mathfrak{l}_{\m}|\}\leq1$}, one of the gradients in \eqref{eq:s-sdeI} has length-scale {\small$\mathrm{O}(1)$} and thus has a natural scaling of {\small$N^{-1}$}.
\item The term \eqref{eq:s-sdeII} is the error in replacing the space-average {\small$\mathbf{Av}^{\mathbf{X},\mathfrak{q}_{\n}}$} by its scale-{\small$\mathfrak{t}_{\mathbf{Av}}$}-average in time. 
\item The term \eqref{eq:s-sdeIII} should be thought of as small because of \eqref{eq:err-estimate}. The exact form of {\small$\mathrm{Err}$} is not so important.
\end{itemize}
The immediate problem we run into when we study \eqref{eq:s-sde}-\eqref{eq:s-sdeIII} is that the error terms (i.e. everything but the first line) are supported everywhere in space. Thus, the next phase of this step is to give an a priori compactification of \eqref{eq:s-sde}-\eqref{eq:s-sdeIII}. Let us make this precise. First, we introduce the following notation for a cutoff version of {\small$\mathbf{R}^{N}$} from \eqref{eq:zsmooth} (which will equal {\small$\mathbf{R}^{N}$} itself with high probability by spatial regularity of {\small$\mathbf{Z}^{N}$}):
\begin{align}
\mathbf{R}^{N,\wedge}_{\t,\x}:=\mathbf{R}^{N}_{\t,\x}\cdot\mathbf{1}[N^{\frac13\delta_{\mathbf{S}}}|\mathbf{R}^{N}_{\t,\x}-1|\leq1].\label{eq:rwedge}
\end{align}
(Recall from Section \ref{subsection:notation} that an event is said to hold with high probability if its probability is {\small$1-\mathrm{o}(1)$}.) For any {\small$\zeta>0$}, we also let {\small$\boldsymbol{\chi}^{(\zeta)}_{\x}:=[\mathscr{S}^{N}\star\mathbf{1}_{|\cdot|\leq N^{1+\zeta}}]_{\x}$} be a regularization of the indicator function of the radius-{\small$N^{1+\zeta}$} neighborhood of the origin; the smoothness will be technically important. We also let {\small$\mathbf{S}^{N,\zeta}$} be the solution to the following \abbr{SDE} with the same initial data as {\small$\mathbf{S}^{N}$}, where {\small$\zeta_{\mathrm{large}}>0$} is a (large) parameter independent of {\small$\zeta$}:
\begin{align}
&\d\mathbf{S}^{N,\zeta}_{\t,\x}=\mathscr{T}_{N}\mathbf{S}^{N,\zeta}_{\t,\x}\d\t+{\ocolor{blue}\boldsymbol{\chi}^{(\zeta_{\mathrm{large}})}_{\x}}[\mathscr{S}^{N}\star(\sqrt{2}\lambda N^{\frac12}{\ocolor{blue}\mathbf{R}^{N,\wedge}_{\t,\cdot}}\mathbf{S}^{N,\zeta}_{\t,\cdot}\d\mathbf{b}_{\t,\cdot})]_{\x}\nonumber\\
&+N{\ocolor{blue}\boldsymbol{\chi}^{(\zeta)}_{\x}}\sum_{\substack{\n=1,\ldots,\mathrm{K}\\\m=0,\ldots,\mathrm{M}\\0<|\mathfrak{l}_{1}|,\ldots,|\mathfrak{l}_{\m}|\lesssim \mathfrak{n}_{\mathbf{Av}}}}\tfrac{1}{|\mathfrak{l}_{1}|\ldots|\mathfrak{l}_{\m}|}\mathrm{c}_{N,\n,\mathfrak{l}_{1},\ldots,\mathfrak{l}_{\m}}\grad^{\mathbf{X}}_{\mathfrak{l}_{1}}\ldots\grad^{\mathbf{X}}_{\mathfrak{l}_{\m}}[\mathscr{S}^{N}\star(\mathbf{Av}^{\mathbf{T},\mathbf{X},\mathfrak{q}_{\n}}_{\t,\cdot}\cdot\mathbf{R}^{N}_{\t,\cdot}\mathbf{S}^{N,\zeta}_{\t,\cdot})]_{\x}\d\t\nonumber\\
&+N{\ocolor{blue}\boldsymbol{\chi}^{(\zeta)}_{\x}}\sum_{\substack{\n=1,\ldots,\mathrm{K}\\\m=0,\ldots,\mathrm{M}\\0<|\mathfrak{l}_{1}|,\ldots,|\mathfrak{l}_{\m}|\lesssim \mathfrak{n}_{\mathbf{Av}}}}\tfrac{1}{|\mathfrak{l}_{1}|\ldots|\mathfrak{l}_{\m}|}\mathrm{c}_{N,\n,\mathfrak{l}_{1},\ldots,\mathfrak{l}_{\m}}\grad^{\mathbf{T},\mathrm{av}}_{\mathfrak{t}_{\mathbf{Av}}}\grad^{\mathbf{X}}_{\mathfrak{l}_{1}}\ldots\grad^{\mathbf{X}}_{\mathfrak{l}_{\m}}[\mathscr{S}^{N}\star(\mathbf{Av}^{\mathbf{X},\mathfrak{q}_{\n}}_{\t,\cdot}\cdot\mathbf{R}^{N}_{\t,\cdot}\mathbf{S}^{N,\zeta}_{\t,\cdot})]_{\x}\d\t\nonumber\\
&+{\ocolor{blue}\boldsymbol{\chi}^{(\zeta)}_{\x}}\cdot\mathrm{Err}[\mathbf{R}^{N}_{\t,\cdot}\mathbf{S}^{N,\zeta}_{\t,\cdot}]_{\x}\d\t.\label{eq:szeta-sde}
\end{align}
The difference between \eqref{eq:szeta-sde} and \eqref{eq:s-sde}-\eqref{eq:s-sdeIII} are the space-cutoffs and the {\small$\mathbf{R}^{N,\wedge}$} in the noise term (which we write in blue for emphasis). Note that {\small$\mathbf{S}^{N,\zeta}$} is adapted to the Brownian filtration since the time-average {\small$\mathbf{Av}^{\mathbf{T},\mathbf{X},\mathfrak{q}_{\n}}$} is backwards in time (see Definition \ref{definition:eq-operators}). Well-posedness of continuous path-wise solutions to \eqref{eq:szeta-sde} with initial data {\small$\mathbf{S}^{N}_{0,\cdot}$} is standard; it is a (linear) perturbation of the heat equation associated to {\small$\mathscr{T}_{N}$} with bounded coefficients (by well-posedness of \eqref{eq:phi}). For example, solutions can be generated via Picard iteration; {by using said Picard iteration, we obtain the following, which is essentially an elementary polynomial-in-{\small$N$} bound on the propagation speed of \eqref{eq:curr}-\eqref{eq:phi} in space. (The proof is elementary and thus omitted.)}
\begin{lemma}\label{lemma:aprioricompact}
\fsp Fix {\small$\mathrm{L}>0$}. If {\small$\zeta,\zeta_{\mathrm{large}}=\mathrm{O}(1)$} are large enough, then with probability $1-\mathrm{o}(1)$, we have
\begin{align}
\sup_{\t\in[0,1]}\sup_{|\x|\leq\mathrm{L}N}|\mathbf{S}^{N,\zeta}_{\t,\x}-\mathbf{S}^{N}_{\t,\x}|\lesssim_{\mathrm{L}}N^{-\gamma_{\mathrm{L}}}, \quad\text{where } \gamma_{\mathrm{L}}>0.
\end{align}
\end{lemma}
\subsection{Step 2 -- stochastic heat kernel}
Our strategy now focuses on reducing the parameter {\small$\zeta$} in Lemma \ref{lemma:aprioricompact} until it is small, thereby comparing {\small$\mathbf{S}^{N}$} to a ``compactified version". (We will keep {\small$\zeta_{\mathrm{large}}$} in Lemma \ref{lemma:aprioricompact} fixed.)

Note that \eqref{eq:szeta-sde} is linear in the solution {\small$\mathbf{S}^{N,\zeta}$}. So, it can be solved via the fundamental solution, or (stochastic) heat kernel, for \eqref{eq:szeta-sde}. Precisely, let {\small$\mathbf{K}^{N,\zeta}_{\s,\t,\x,\y}$} be a function of {\small$(\s,\t,\x,\y)\in[0,\infty)^{2}\times\Z^{2}$} with {\small$\s\leq\t$} such that
\begin{align}
&\d\mathbf{K}^{N,\zeta}_{\s,\t,\x,\y}=\mathscr{T}_{N}\mathbf{K}^{N,\zeta}_{\s,\t,\x,\y}\d\t+{\ocolor{blue}\boldsymbol{\chi}^{(\zeta_{\mathrm{large}})}_{\x}}[\mathscr{S}^{N}\star(\sqrt{2}\lambda N^{\frac12}{\ocolor{blue}\mathbf{R}^{N,\wedge}_{\t,\cdot}}\mathbf{K}^{N,\zeta}_{\s,\t,\cdot,\y}\d\mathbf{b}_{\t,\cdot})]_{\x}\nonumber\\
&+N{\ocolor{blue}\boldsymbol{\chi}^{(\zeta)}_{\x}}\sum_{\substack{\n=1,\ldots,\mathrm{K}\\\m=0,\ldots,\mathrm{M}\\0<|\mathfrak{l}_{1}|,\ldots,|\mathfrak{l}_{\m}|\lesssim \mathfrak{n}_{\mathbf{Av}}}}\tfrac{1}{|\mathfrak{l}_{1}|\ldots|\mathfrak{l}_{\m}|}\mathrm{c}_{N,\n,\mathfrak{l}_{1},\ldots,\mathfrak{l}_{\m}}\grad^{\mathbf{X}}_{\mathfrak{l}_{1}}\ldots\grad^{\mathbf{X}}_{\mathfrak{l}_{\m}}[\mathscr{S}^{N}\star(\mathbf{Av}^{\mathbf{T},\mathbf{X},\mathfrak{q}_{\n}}_{\t,\cdot}\cdot\mathbf{R}^{N}_{\t,\cdot}\mathbf{K}^{N,\zeta}_{\s,\t,\cdot,\y})]_{\x}\d\t\nonumber\\
&+N{\ocolor{blue}\boldsymbol{\chi}^{(\zeta)}_{\x}}\sum_{\substack{\n=1,\ldots,\mathrm{K}\\\m=0,\ldots,\mathrm{M}\\0<|\mathfrak{l}_{1}|,\ldots,|\mathfrak{l}_{\m}|\lesssim \mathfrak{n}_{\mathbf{Av}}}}\tfrac{1}{|\mathfrak{l}_{1}|\ldots|\mathfrak{l}_{\m}|}\mathrm{c}_{N,\n,\mathfrak{l}_{1},\ldots,\mathfrak{l}_{\m}}\grad^{\mathbf{T},\mathrm{av}}_{\mathfrak{t}_{\mathbf{Av}}}\grad^{\mathbf{X}}_{\mathfrak{l}_{1}}\ldots\grad^{\mathbf{X}}_{\mathfrak{l}_{\m}}[\mathscr{S}^{N}\star(\mathbf{Av}^{\mathbf{X},\mathfrak{q}_{\n}}_{\t,\cdot}\cdot\mathbf{R}^{N}_{\t,\cdot}\mathbf{K}^{N,\zeta}_{\s,\t,\cdot,\y})]_{\x}\d\t\nonumber\\
&+{\ocolor{blue}\boldsymbol{\chi}^{(\zeta)}_{\x}}\cdot\mathrm{Err}[\mathbf{R}^{N}_{\t,\cdot}\mathbf{K}^{N,\zeta}_{\s,\t,\cdot,\y}]_{\x}\d\t.\label{eq:kzeta-sde}
\end{align}
We clarify that {\small$\mathscr{T}_{N}$} in the first line acts on the {\small$\x$}-variable. In particular, the function {\small$(\t,\x)\mapsto\mathbf{K}^{N,\zeta}_{\s,\t,\x,\y}$} satisfies the same equation as {\small$\mathbf{S}^{N,\zeta}$}, but we give it the initial data {\small$\mathbf{K}^{N,\zeta}_{\s,\s,\x,\y}=\mathbf{1}_{\x=\y}$}. Thus, we have the representation 
\begin{align}
\mathbf{S}^{N,\zeta}_{\t,\x}=\sum_{\y\in\Z}\mathbf{K}^{N,\zeta}_{0,\t,\x,\y}\mathbf{S}^{N}_{0,\y}.\label{eq:duhamelszeta}
\end{align}
Solutions to \eqref{eq:kzeta-sde} can be generated by a standard Picard iteration (which converges for the same reason that was given after \eqref{eq:szeta-sde}). Moreover, {\small$(\s,\t)\mapsto\mathbf{K}^{N,\zeta}_{\s,\t,\x,\y}$} is continuous for {\small$0\leq\s\leq\t\leq1$} with probability {\small$1$}.

The key estimate for stochastic heat kernels {\small$\mathbf{K}^{N,\zeta}$} is given below. We explain its utility afterwards. (Throughout this paper, we will say that an event holds with \emph{high probability} if its probability is {\small$1-\mathrm{o}(1)$}.)
\begin{prop}\label{prop:stochheat}
\fsp There exists {\color{black}{\small$\beta_{0}>0$}} so that for any finite {\small$\zeta,\kappa,\delta>0$}, the following holds with high probability:
\begin{align}
\sup_{0\leq\s\leq\t\leq1}\sup_{\x\in\Z}\sum_{\y\in\Z}\exp(\tfrac{\kappa|\x-\y|}{N})|\mathbf{K}^{N,\zeta}_{\s,\t,\x,\y}|\lesssim_{\kappa,\delta}\exp(N^{-{\color{black}\beta_{0}}}N^{\zeta})N^{\delta}.\label{eq:stochheatI}
\end{align}
\end{prop}
{\color{black}We note that Proposition \ref{prop:stochheat} holds for all $\zeta>0$ fixed, not necessarily large.}

The quantity {\small$|\mathbf{K}^{N,\zeta}_{\s,\t,\x,\y}|$} on the \abbr{LHS} of \eqref{eq:stochheatI} is therefore exponentially small in {\small$N$} if {\small$|\x-\y|\gtrsim N^{1+\zeta-\nu}$} for any {\small$\nu>0$} independent of {\small$\zeta$}. This will allow us to reduce the length-scale parameter {\small$\zeta$} by a small amount. In order to neatly present this successive decreasing of {\small$\zeta$}, we introduce the following sequence of ``cutoff exponents".
\begin{notation}\label{notation:zeta}
\fsp Fix a small parameter {\small$\e_{\mathrm{cut}}>0$}, and let {\small$\delta_{\mathrm{cut}}>0$} be another small parameter depending only on {\small$\e_{\mathrm{cut}}>0$}; we will clarify the dependence shortly, and we require that {\small$\e_{\mathrm{cut}},\delta_{\mathrm{cut}}>0$} are independent of {\small$N$}.
\begin{itemize}
\item Let {\small$\zeta_{1}$} be any {\small$\mathrm{O}(1)$} choice of {\small$\zeta$} for which the estimate in Lemma \ref{lemma:aprioricompact} holds.
\item For any {\small$\i\geq1$}, we define {\small$\zeta_{\i+1}=\zeta_{\i}-\delta_{\mathrm{cut}}$}.
\item We will stop this sequence at the index {\small$\i_{\e_{\mathrm{cut}}}$}; it is the first index for which {\small$\zeta_{\i_{\e_{\mathrm{cut}}}+1}\in[\e_{\mathrm{cut}},2\e_{\mathrm{cut}}]$}, and this can be achieved if we take {\small$\delta_{\mathrm{cut}}$} small enough depending only on {\small$\e_{\mathrm{cut}}$} (this is the aforementioned dependence). For convenience, we will also use the notation {\small$\zeta_{\mathrm{final}}:=\zeta_{\i_{\e_{\mathrm{cut}}}}$}.
\end{itemize}
\end{notation}
We now state the aforementioned successive decreasing of {\small$\zeta$}. The following result essentially follows from Proposition \ref{prop:stochheat} and some elementary manipulations; its proof is given in Section \ref{section:compact}.
\begin{lemma}\label{lemma:stochheat}
\fsp Fix any index {\small$\i=1,\ldots,\i_{\e_{\mathrm{cut}}}$} and any {\small$\mathrm{D}>0$}. We have the following estimate with high probability, simultaneously over {\small$\t\in[0,1]$} and {\small$\x\in\Z$} so that {\small$|\x|\lesssim N$}:
\begin{align}
|\mathbf{S}^{N,\zeta_{\i}}_{\t,\x}-\mathbf{S}^{N,\zeta_{\i+1}}_{\t,\x}|\lesssim_{\mathrm{D}}N^{-\mathrm{D}}.\label{eq:stochheatII}
\end{align}
\end{lemma}
\subsection{Step 3 -- analysis in the ``almost compact" setting}
Lemma \ref{lemma:stochheat} essentially reduces the error term analysis to the compact case, in which we can compare {\small$\mathbf{S}^{N,\zeta_{\mathrm{final}}}$} to a lattice \abbr{SHE}. Let us make this precise. We let {\small$\mathbf{Q}^{N}$} be the solution to the following \abbr{SDE} with initial data {\small$\mathbf{Q}^{N}_{0,\cdot}=\mathbf{Z}^{N}_{0,\cdot}$}.
\begin{align}
\d\mathbf{Q}^{N}_{\t,\x}&=\mathscr{T}_{N}\mathbf{Q}^{N}_{\t,\x}\d\t+\sqrt{2}\lambda N^{\frac12}\mathbf{Q}^{N}_{\t,\x}\d\mathbf{b}_{\t,\x}.\label{eq:qsde}
\end{align}
The following result compares {\small$\mathbf{Q}^{N}$} to {\small$\mathbf{Z}^{N}$}; it will eventually be shown by comparing each to {\small$\mathbf{S}^{N,\zeta_{\mathrm{final}}}$}.
\begin{prop}\label{prop:comparison}
\fsp Fix any $\mathrm{L}>0$. With high probability, we have
\begin{align}
\sup_{\t\in[0,1]}\sup_{|\x|\leq\mathrm{L}N}|\mathbf{Z}^{N}_{\t,\x}-\mathbf{Q}^{N}_{\t,\x}|\lesssim_{\mathrm{L}}N^{-\gamma_{\mathrm{L}}}, \quad\text{where } \gamma_{\mathrm{L}}>0.
\end{align}
\end{prop}
\subsection{Proofs of the main theorems}\label{subsection:mainproofs}
We will now use the results of this section to prove the main results.
\begin{proof}[Proof of Theorem \ref{theorem:main}]
By Proposition \ref{prop:comparison}, it suffices to obtain that Theorem \ref{theorem:main} holds if we replace {\small$\mathbf{Z}^{N}$} by {\small$\mathbf{Q}^{N}$}. This is a standard stability of the \abbr{SHE} (see \eqref{eq:she}) under spatial-regularizations, in this case a spatial-discretization; said argument is identical to the proof of Theorem 2.1 in \cite{BG}, for example, so we do not produce it here.
\end{proof}
\begin{proof}[Proof of Theorem \ref{theorem:mainwedge}]
We claim that Theorem \ref{theorem:mainwedge} holds if we replace {\small$\mathbf{Z}^{N}$} with {\small$\mathbf{Q}^{N}$} for some constant {\small$\mathscr{T}_{N}$} which satisfies {\small$|\mathscr{T}_{N}|\lesssim_{\e}N^{\e}$} for any {\small$\e>0$}. By Proposition \ref{prop:comparison}, this would complete the proof of Theorem \ref{theorem:mainwedge} as written. To prove the claim for {\small$\mathbf{Q}^{N}$}, we start by constructing the appropriate initial data {\small$\mathbf{Q}^{N}_{0,\cdot}=\mathbf{Z}^{N}_{0,\cdot}$}.
\begin{enumerate}
\item Recall {\small$\lambda:=\alpha^{-1}\beta$} from \eqref{eq:ch}; as shown in Appendix \ref{subsection:alpha-beta}, we have {\small$\lambda\neq0$}. Thus, we can consider a function {\small$\mathsf{f}:\Z\to\R$} satisfying {\small$\lambda\mathsf{f}_{\x}=-(\log N)^{1/9}N^{-1}|\x|$} for all {\small$\x\in\Z$}.
\item We choose the following initial data {\small$\mathbf{j}^{N}_{0,\cdot}$}. First, set {\small$\mathbf{j}^{N}_{0,0}=0$}, and recall the relation {\small$\bphi_{0,\x}=N^{1/2}(\mathbf{j}^{N}_{0,\x}-\mathbf{j}^{N}_{0,\x-1})$}. In particular, {\small$\mathbf{j}^{N}_{0,\cdot}$} is now a function of the initial data {\small$\bphi_{0,\cdot}$} only. With this in mind, we will let the initial data for {\small$\mathbf{j}^{N}_{0,\cdot}$} be given by sampling {\small$\bphi_{0,\cdot}\sim\mathbb{P}^{0}$}, but we condition on the following event:
\begin{align}
\mathscr{E}_{\mathsf{f}}:=\Big\{\sup_{|\x|\leq (\log N)^{2/9}N}|\mathbf{j}^{N}_{0,\x}-\mathsf{f}_{\x}|\leq(\log N)^{-\frac19}\Big\}.\label{eq:mainwedge1}
\end{align}
\end{enumerate}
We now show that this choice of initial data satisfies Assumption \ref{assump:noneq} by arguing as follows. (Below, all suprema are taken over {\small$\x\in\R$}, with the understanding that functions of {\small$\x\in\Z$} extend to {\small$\x\in\R$} by linear interpolation.)
\begin{enumerate}
\item If we let {\small$\mathbb{P}^{\mathrm{init}}$} denote the law of the initial data chosen in the previous bullet points (as a probability measure on {\small$\R^{\Z}$}), then by definition of conditional probability, we have 
\begin{align}
\|\tfrac{\d\mathbb{P}^{\mathrm{init}}}{\d\mathbb{P}^{0}}\|_{\mathrm{L}^{\infty}(\R^{\Z})}\leq\mathbb{P}(\mathscr{E}_{\mathsf{f}})^{-1}.\label{eq:mainwedge2}
\end{align}
\item Recall that the initial data {\small$\x\mapsto\mathbf{j}^{N}_{0,\x}$} is a process with {\small$\mathbf{j}^{N}_{0,0}=0$} and whose increments are given by {\small$N^{-1/2}$} times i.i.d. sub-Gaussian random variables; see Assumption \ref{assump:potential}. By appealing to the \abbr{KMT} coupling (see \cite{KMT}), we can construct a two-sided Brownian motion {\small$\w\mapsto\mathsf{b}_{\w}$} for {\small$\w\in\Z$} such that 
\begin{align}
\mathbb{P}\Big(\sup_{|\x|\leq(\log N)^{1/9}N}|\mathbf{j}^{N}_{0,\x}-\mathsf{b}_{\upsilon N^{-1}\x}|\geq N^{-2\delta}\Big)\lesssim\exp(-N^{\frac12-2\delta}),
\end{align}
where {\small$\delta>0$} is fixed, and {\small$\upsilon:=\E^{0}|\bphi_{0}|^{2}$}. (We clarify that the {\small$N^{-1}$}-scaling in the Brownian motion parameter is there because the increments of {\small$\mathbf{j}^{N}_{0,\cdot}$} are sub-Gaussian with variance {\small$N^{-1}$}, so that the increments of the two processes on the \abbr{LHS} above have the same variance; this is also the reason why we need to introduce {\small$\upsilon$}.) By the previous display and by the definition of {\small$\mathscr{E}_{\mathsf{f}}$} from \eqref{eq:mainwedge1}, we have
\begin{align}
\mathbb{P}(\mathscr{E}_{\mathsf{f}})\geq\mathbb{P}(\mathscr{E}_{\mathsf{f},\mathsf{b}})-\mathrm{O}_{\delta}(\exp(-N^{\frac12-2\delta})),\label{eq:mainwedge3a}
\end{align}
where {\small$\mathscr{E}_{\mathsf{f},\mathsf{b}}$} is the following event (in which we compare {\small$\mathsf{f},\mathsf{b}$}):
\begin{align}
\mathscr{E}_{\mathsf{f},\mathsf{b}}:=\Big\{\sup_{|\x|\leq(\log N)^{2/9}N}|\mathsf{f}_{\x}-\mathsf{b}_{\upsilon N^{-1}\x}|\lesssim (\log N)^{-\frac19}\Big\}=\Big\{\sup_{|\x|\leq(\log N)^{2/9}}|\mathsf{f}_{N\x}-\mathsf{b}_{\upsilon\x}|\lesssim (\log N)^{-\frac19}\Big\}.\nonumber
\end{align}
\item Observe that the process {\small$\x\mapsto\mathsf{f}_{N\x}-\mathbf{b}_{\upsilon\x}$}, if we allow for all {\small$\x\in\R$}, is a perturbation of a Brownian motion {\small$\mathbf{b}_{\upsilon\x}$} by a function {\small$\mathsf{f}_{N\x}$} whose Cameron-Martin norm on the interval {\small$[-(\log N)^{2/9}N,(\log N)^{2/9}N]$} is {\small$\lesssim(\log N)^{1/2}$}. Thus, using the classical Cameron-Martin theorem for Brownian motion, we can compute the density of the law of the Brownian motion {\small$\x\mapsto-\mathbf{b}_{\upsilon\x}$} with respect to the law of the process {\small$\x\mapsto\mathsf{f}_{N\x}-\mathbf{b}_{\upsilon\x}$} , and this density has a second moment that is {\small$\lesssim\exp(\mathrm{O}(|\log N|^{1/2}))$}. So, by Cauchy-Schwarz (to change measure), we have 
\begin{align*}
\mathbb{P}\Big(\sup_{|\x|\leq(\log N)^{2/9}}|\mathsf{b}_{\upsilon\x}|\lesssim(\log N)^{-\frac19}\Big)\lesssim\exp(\mathrm{O}(|\log N|^{\frac12}))\mathbb{P}(\mathscr{E}_{\mathsf{f},\mathsf{b}})^{\frac12}.
\end{align*}
We can bound the \abbr{LHS} of the previous display using scaling properties of Brownian motion and the standard reflection principle for Brownian motion (which implies that {\small$\sup_{|\x|\leq1}|\mathsf{b}_{\x}|$} has a smooth, non-vanishing density in a neighborhood of {\small$0$}). In particular, we obtain the following calculation:
\begin{align*}
\mathbb{P}\Big(\sup_{|\x|\leq(\log N)^{2/9}}|\mathsf{b}_{\upsilon\x}|\leq\mathrm{C}(\log N)^{-\frac19}\Big)&=1-\mathbb{P}\Big(\sup_{|\x|\leq(\log N)^{2/9}}|\mathsf{b}_{\upsilon\x}|>\mathrm{C}(\log N)^{-\frac19}\Big)\\
&=1-\mathbb{P}\Big(\sup_{|\x|\leq1}|\mathsf{b}_{\x}|>\mathrm{C}\upsilon^{\frac12}(\log N)^{-\frac{2}{9}}\Big)\gtrsim (\log N)^{-\frac{2}{9}}.
\end{align*}
We combine the last two displays with \eqref{eq:mainwedge3a} to get {\small$\mathbb{P}(\mathscr{E}_{\mathsf{f},\mathsf{b}})\gtrsim\exp(-\mathrm{c}|\log N|^{1/2})\cdot(\log N)^{-4/9}\gtrsim N^{-1/10}$}. Putting this together with \eqref{eq:mainwedge2} shows that our choice of {\small$\mathbb{P}^{\mathrm{init}}$} satisfies the constraint \eqref{eq:noneq}.
\end{enumerate}
Now, we show that this choice of initial data for {\small$\mathbf{Z}^{N}$} satisfies the first estimate in \eqref{eq:mainI}. To this end, we first note the following formula for the initial data constructed earlier:
\begin{align}
\mathbf{Z}^{N}_{0,\x}&=\exp\Big(-\tfrac{(\log N)^{1/9}}{N}|\x|+\mathrm{O}(\tfrac{1}{(\log N)^{1/9}})\Big)\cdot\exp(\mathsf{w}_{\x}).\label{eq:mainwedge4}
\end{align}
Above, the process {\small$\x\mapsto\mathsf{w}_{\x}$} is supported for {\small$|\x|>(\log N)^{2/9}N$}, and it is given by {\small$\lambda N^{-1/2}$} times sum of i.i.d. {\small$\bphi_{\w}$} random variables for {\small$\w$} between the boundary of {\small$\llbracket-(\log N)^{2/9}N,(\log N)^{2/9}N\rrbracket$} and {\small$\x$}, where {\small$\llbracket\a,\b\rrbracket:=[\a,\b]\cap\Z$}. It comes from the lack of conditioning {\small$\mathbb{P}^{0}$} outside of the interval {\small$\llbracket-(\log N)^{2/9}N,(\log N)^{2/9}N\rrbracket$} when constructing the initial data, and more precisely, it has the following form:
\begin{align}
\mathsf{w}_{\x}&:=\mathbf{1}_{|\x|\geq(\log N)^{\frac29}N}\cdot\begin{cases}\lambda N^{-\frac12}\sum_{\w\in\llbracket(\log N)^{2/9}N,\x\rrbracket}\bphi_{\w}&\x\geq(\log N)^{\frac29}N\\ \lambda N^{-\frac12}\sum_{\w\in\llbracket\x,-(\log N)^{2/9}N\rrbracket}\bphi_{\w}&\x\leq-(\log N)^{\frac29}N\end{cases}\label{eq:mainwedge4b}
\end{align}
The first exponential factor in \eqref{eq:mainwedge4} is deterministically {\small$\lesssim1$}. On the other hand, since the potential {\small$\mathscr{U}$} is bounded below by a quadratic polynomial (see Assumption \ref{assump:potential}), if {\small$\bphi\sim\mathbb{P}^{0}$}, we have {\small$\exp(2p\mathsf{w}_{\x})\lesssim\exp(\kappa_{p}|\x|/N)$} just by comparison with the case of Gaussian {\small$\bphi$} and joint independence of {\small$\bphi_{\w}$} over {\small$\w\in\Z$}. Therefore, our choice \eqref{eq:mainwedge4} of initial data satisfies the first estimate in \eqref{eq:mainI}.

We now show convergence of {\small$\mathbf{Q}^{N}$} with initial data \eqref{eq:mainwedge4} to the narrow-wedge solution of \eqref{eq:she}. To this end, we recall that {\small$\mathsf{w}_{\x}$} in \eqref{eq:mainwedge4}-\eqref{eq:mainwedge4b} is supported on {\small$|\x|\geq(\log N)^{1/9}N$}. Thus, we can rewrite \eqref{eq:mainwedge4} as 
\begin{align*}
\mathbf{Q}^{N}_{0,\x}=\mathbf{Q}^{N,1}_{0,\x}+\mathbf{Q}^{N,2}_{0,\x}&:=\exp\Big(-\tfrac{(\log N)^{1/9}}{N}|\x|+\mathrm{O}(\tfrac{1}{(\log N)^{1/9}})\Big)\mathbf{1}_{|\x|<(\log N)^{2/9}N}\\
&+\exp\Big(-\tfrac{(\log N)^{1/9}}{N}|\x|+\mathrm{O}(\tfrac{1}{(\log N)^{1/9}})\Big)\exp(\mathsf{w}_{\x})\mathbf{1}_{|\x|\geq(\log N)^{2/9}N}.
\end{align*}
Let {\small$\mathbf{Q}^{N,\i}$} be the solution to \eqref{eq:qsde} with initial data {\small$\mathbf{Q}^{N,\i}_{0,\cdot}$} for {\small$\i=1,2$}. We now claim the following.
\begin{enumerate}
\item There exists a deterministic constant {\small$\mathscr{T}_{N}\lesssim(\log N)^{1/9}$} such that {\small$\mathscr{T}_{N}\mathbf{Q}^{N,1}_{\t,N\X}$} converges to the narrow-wedge solution of \eqref{eq:she} locally uniformly in {\small$(\t,\X)$} in probability.
\item For the same {\small$\mathscr{T}_{N}\lesssim(\log N)^{1/9}$}, we have {\small$\mathscr{T}_{N}\mathbf{Q}^{N,2}_{\t,N\X}\to0$} locally uniformly in probability.
\end{enumerate}
Because the \abbr{SDE} \eqref{eq:qsde} is linear in the solution, we have {\small$\mathbf{Q}^{N}=\mathbf{Q}^{N,1}+\mathbf{Q}^{N,2}$}, so the previous two claims would yield Theorem \ref{theorem:mainwedge}. To establish the second claim, we can use the moment bounds on {\small$\exp(\mathsf{w}_{\x})$} from after \eqref{eq:mainwedge4b} and an elementary calculation to show that {\small$\mathbf{Q}^{N,2}_{0,\cdot}$} satisfies the estimates in \eqref{eq:mainI}. Moreover, since we restrict to {\small$|\x|\geq(\log N)^{2/9}N$} in {\small$\mathbf{Q}^{N,2}_{0,\cdot}$}, the exponential factor in {\small$\mathbf{Q}^{N,2}_{0,\cdot}$} vanishes uniformly in {\small$\x$}. Ultimately, we can apply the same stability of \abbr{SHE} under space-discretization as in the proof of Theorem \ref{theorem:main} to deduce claim (2) above.

It remains to prove (1). A geometric series shows that for some {\small$\mathscr{T}_{N}>0$} of order {\small$(\log N)^{1/9}$}, we have
\begin{align*}
\tfrac{1}{N}\sum_{\x\in\Z}\mathscr{T}_{N}\mathbf{Q}^{N,0}_{0,\x}&=1+\mathrm{o}(1)\quad\text{and}\quad\tfrac{1}{N}\sum_{\x\in\Z}\mathbf{1}_{|\x|\geq(\log N)^{-1/10}N}\mathscr{T}_{N}\mathbf{Q}^{N,0}_{0,\x}=\mathrm{o}(1).
\end{align*}
Thus, {\small$\mathbf{Q}^{N,0}_{0,N\X}$} converges as a Borel measure in {\small$\X\in\R$} (after linear interpolation of its values from {\small$N^{-1}\Z$} to {\small$\R$}) to a Dirac point mass at the origin. Now, let {\small$\mathbf{H}^{N}_{\s,\t,\x,\y}$} be a function of {\small$(\s,\t,\x,\y)\in[0,\infty)^{2}\times\Z^{2}$} with {\small$\s\leq\t$} that solves
\begin{align}
\partial_{\t}\mathbf{H}^{N}_{\s,\t,\x,\y}&=\mathscr{T}_{N}\mathbf{H}^{N}_{\s,\t,\x,\y}\quad\text{and}\quad \mathbf{H}^{N}_{\s,\s,\x,\y}=\mathbf{1}_{\x=\y}.\label{eq:heatkernel}
\end{align}
(A.8) in \cite{DT} gives {\small$N\mathbf{H}^{N}_{\s,\t,N\X,N\Y}\to\mathbf{H}_{\s,\t,\X,\Y}$} uniformly in {\small$\X,\Y\in\R$}, where {\small$\mathbf{H}^{N}_{\s,\t,N\X,N\Y}$} extends from {\small$N^{-1}\Z$} to {\small$\R$} via linear interpolation; this holds locally uniformly on the domain {\small$\{0\leq\s<t\leq\T\}$} for any fixed {\small$\T>0$}. (We clarify that (A.8) in \cite{DT} states pointwise convergence, but this upgrades to uniform convergence because of the space-time regularity and sub-exponential decay-in-space estimates for {\small$\mathbf{H}^{N}$} in Proposition A.1 in \cite{DT}, followed by an application of the Arzela-Ascoli lemma). Moreover, by Proposition \ref{prop:hke}, we obtain the following estimate for the Riemann sum approximation for the integration of the heat kernel {\small$N\mathbf{H}^{N}$} against the initial data {\small$\mathbf{Q}^{N,0}$}:
\begin{align*}
\tfrac{1}{N}\sum_{\y\in\Z}N\mathbf{H}^{N}_{\s,\t,\x,\y}\cdot\mathscr{T}_{N}\mathbf{Q}^{N,0}_{0,\y}\lesssim|\t-\s|^{-\frac12}\cdot\tfrac{1}{N}\sum_{\y\in\Z}\mathscr{T}_{N}\mathbf{Q}^{N,0}_{0,\y}\lesssim|\t-\s|^{-\frac12}.
\end{align*}
To complete the proof, we can now use Lemma \ref{lemma:qstabilitywedge}, a standard stability for \abbr{SHE} with Dirac initial data.
\end{proof}
\subsection{Outline of the rest of the paper}\label{subsection:paper-outline}
We now explain the goal of each remaining (non-appendix) section.
\begin{itemize}
\item In Section \ref{section:s-sde}, we prove Proposition \ref{prop:s-sde}, the formula for the evolution equation for {\small$\mathbf{Z}^{N}$}.
\item In Section \ref{section:technical}, we prove the necessary \emph{stochastic} estimates that will be used in later sections.
\item In Section \ref{section:stochheat}, we use an \emph{analytic} argument, assuming the stochastic estimates in Section \ref{section:technical}, to show Proposition \ref{prop:stochheat}. (In particular, Sections \ref{section:technical} and \ref{section:stochheat} separate the probabilistic and analytic methodologies in this paper.)
\item In Section \ref{section:compact}, we prove {\color{black}Lemma \ref{lemma:stochheat}}. In Section \ref{section:comparison}, we prove Proposition \ref{prop:comparison}. Both follow almost immediately from Proposition \ref{prop:stochheat} and some elementary manipulations.
\end{itemize}
%
%
%
\section{Proof of Proposition \ref{prop:s-sde}}\label{section:s-sde}
The goal of this section is to derive the \abbr{SDE} \eqref{eq:s-sde}-\eqref{eq:s-sdeIII}. Before we start, let us recall the relevant notation in section from Definition \ref{definition:zsmooth}, \ref{definition:jets}, \ref{definition:eq-operators}, and Proposition \ref{prop:s-sde}.
\subsection{The \abbr{SDE} for {\small$\mathbf{Z}^{N}$}}
We start by computing the evolution equation for {\small$\mathbf{Z}^{N}$} from \eqref{eq:ch}. As explained in Section \ref{subsection:intro-method}, it is a lattice \abbr{SHE} with an error term denoted by {\small$\mathfrak{e}$} below. In recording the \abbr{SDE} for {\small$\mathbf{Z}^{N}$} below, we will decompose the error term {\small$\mathfrak{e}$} into groups of terms with the same power of {\small$N$}. We emphasize that said error term essentially consists of {\small$\mathbf{Z}^{N}$} multiplied by local functions which are either supported to the right of {\small$0$} or to the left of {\small$1$}. In view of Definition \ref{definition:admissible}, this property will be important as it will lead to {\small$\mathfrak{e}$} consisting only of admissible local functions. However, {\color{black}ensuring} this property requires shifting various local functions in space and gathering the resulting error terms, which leads to the error term {\small$\mathfrak{e}$} involving a lot of different terms.

Before we state the result for the \abbr{SDE} for {\small$\mathbf{Z}^{N}$}, let us now carefully introduce the components to the error term {\small$\mathfrak{e}$} that will appear in the \abbr{SDE} for {\small$\mathbf{Z}^{N}$}. It is given by 
\begin{align}
\mathfrak{e}[\tau_{\x}{\boldsymbol{\phi}}_{\t}]\mathbf{Z}^{N}_{\t,\x}:=\mathfrak{e}_{2}[\tau_{\x}{\boldsymbol{\phi}}_{\t}]\mathbf{Z}^{N}_{\t,\x}+\mathfrak{e}_{1}[\tau_{\x}{\boldsymbol{\phi}}_{\t}]\mathbf{Z}^{N}_{\t,\x}+\mathfrak{e}_{0}[\tau_{\x}{\boldsymbol{\phi}}_{\t}]\mathbf{Z}^{N}_{\t,\x}+\mathfrak{f}_{\mathrm{err}}[\tau_{\x}{\boldsymbol{\phi}}_{\t}]\mathbf{Z}^{N}_{\t,\x},\label{eq:msheII}
\end{align}
where the terms on the \abbr{RHS} of \eqref{eq:msheII} are given as follows. 
\begin{enumerate}
\item First, recall {\small$\beta_{2}$} from \eqref{eq:nonlinearity}, and define the ``renormalized quadratic polynomial" as
\begin{align}
\mathscr{Q}[{\boldsymbol{\phi}}]:=\tfrac12\beta_{2}\mathscr{U}'[{\boldsymbol{\phi}}_{0}]\mathscr{U}'[{\boldsymbol{\phi}}_{1}]-\tfrac12\lambda(\mathscr{U}'[{\boldsymbol{\phi}}_{0}]{\boldsymbol{\phi}}_{0}-1).\label{eq:quadratic_renorm}
\end{align}
We have the following, where {\small$\deg\geq3$} is from \eqref{eq:nonlinearity}:
\begin{align}
\mathfrak{e}_{2}[\tau_{\x}{\boldsymbol{\phi}}_{\t}]\mathbf{Z}^{N}_{\t,\x}&:=\lambda N\mathscr{Q}[\bphi_{\t,\x+1}]\mathbf{Z}^{N}_{\t,\x}+\lambda N\mathscr{Q}[\bphi_{\t,\x-1}]\mathbf{Z}^{N}_{\t,\x}\label{eq:msheII(2)a}\\
&+\lambda N\mathbf{F}_{>2}[\tau_{\x+2\deg}\bphi_{\t}]\mathbf{Z}^{N}_{\t,\x}\d\t\label{eq:msheII(2)aa}\\
&-\tfrac16\beta_{2}\lambda N(\Delta\{\mathscr{U}'[\bphi_{\t,\x+2}]\mathscr{U}'[\bphi_{\t,\x+3}]\}\mathbf{Z}^{N}_{\t,\x})\d\t.\label{eq:msheII(2)b}
\end{align}
The \abbr{RHS} of \eqref{eq:msheII(2)a} will come from the contribution of {\small$\mathbf{F}_{2}$} (see \eqref{eq:nonlinearity}) as well as elementary manipulations for the discrete Laplacian acting on {\small$\mathbf{Z}^{N}$}. The term \eqref{eq:msheII(2)aa} is the contribution of {\small$\mathbf{F}_{>2}$} (after shifting in space by {\small$2\deg$}). The last line \eqref{eq:msheII(2)b} is an error term that appears from shifting the contribution of {\small$\mathbf{F}_{2}$} in space.
\item Second, define {\small$\mathscr{V}[{\boldsymbol{\phi}}_{\x}]:=\mathscr{U}[{\boldsymbol{\phi}}_{\x}]-\frac12\alpha{\boldsymbol{\phi}}_{\x}^{2}$}, so that {\small$\mathscr{V}'[\bphi_{\x}]=\mathscr{U}'[\bphi_{\x}]-\alpha\bphi_{\x}$}. Let us also define
\begin{align}
\mathscr{G}[\bphi]:=\bphi_{2}\mathscr{U}'[\bphi_{2}]\cdot(\mathscr{U}'[\bphi_{1}]-\mathscr{U}'[\bphi_{3}])+\bphi_{2}\cdot(\mathscr{U}'[\bphi_{4}]\mathscr{U}'[\bphi_{5}]-\mathscr{U}'[\bphi_{3}]\mathscr{U}'[\bphi_{4}]).\label{eq:funnygradient}
\end{align}
We have the following, where {\small$c_{\ell}\in\R$} are fixed constants:
\begin{align}
\mathfrak{e}_{1}[\tau_{\x}{\boldsymbol{\phi}}_{\t}]\mathbf{Z}^{N}_{\t,\x}&:=\tfrac12\lambda N^{\frac32}\grad^{\mathbf{X}}_{1}(\mathscr{V}'[{\boldsymbol{\phi}}_{\t,\x}]\mathbf{Z}^{N}_{\t,\x})-\tfrac12\lambda N^{\frac32}\grad^{\mathbf{X}}_{-1}(\mathscr{V}'[{\boldsymbol{\phi}}_{\t,\x+1}]\mathbf{Z}^{N}_{\t,\x})\label{eq:msheII(1)a}\\
&-\tfrac16\beta_{2}\lambda^{2} N^{\frac12}\mathscr{G}[\tau_{\x}\bphi_{\t}]\mathbf{Z}^{N}_{\t,\x}\d\t\label{eq:msheII(1)b}\\
&+N^{\frac12}\Big(\mathbf{F}_{>2}[\tau_{\x+2\deg}\bphi_{\t}]\cdot\sum_{\ell=1,\ldots,2\deg}c_{\ell}\bphi_{\t,\x+\ell}\Big)\mathbf{Z}^{N}_{\t,\x}\d\t\label{eq:msheII(1)c}
\end{align}
The \abbr{RHS} of \eqref{eq:msheII(1)a} will come from replacing the contribution of {\small$N^{3/2}\grad^{\mathbf{X}}_{1}\mathscr{U}'[\bphi_{\t,\x}]$} by {\small$N^{3/2}\grad^{\mathbf{X}}_{1}\bphi_{\t,\x}$}, the latter of which is related to the Laplacian of {\small$\mathbf{Z}^{N}$} by means of Taylor expanding the exponential in \eqref{eq:ch} (as in \cite{BG}). The last two lines \eqref{eq:msheII(1)b}-\eqref{eq:msheII(1)c} will come as error terms obtained by shifting {\small$\mathbf{F}_{2},\mathbf{F}_{>2}$} in space, respectively.
\item Third, we have the following, where {\small$c_{\ell_{1},\ell_{2}}\in\R$} are fixed constants:
\begin{align}
\mathfrak{e}_{0}[\tau_{\x}{\boldsymbol{\phi}}_{\t}]\mathbf{Z}^{N}_{\t,\x}&:=\tfrac{1}{12}\lambda^{4}\left(\mathscr{U}'[{\boldsymbol{\phi}}_{\t,\x+1}]{\boldsymbol{\phi}}_{\t,\x+1}^{3}-\E^{0}\left\{\mathscr{U}'[{\boldsymbol{\phi}}_{0}]{\boldsymbol{\phi}}_{0}^{3}\right\}\right)\mathbf{Z}^{N}_{\t,\x}\label{eq:msheII(0)a}\\
&-\tfrac16\beta_{2}\lambda^{3}(\mathscr{G}[\tau_{\x}\bphi_{\t}]\bphi_{\t,\x+1}-1)\mathbf{Z}^{N}_{\t,\x}\d\t\label{eq:msheII(0)a1}\\
&+\Big(\mathbf{F}_{>2}[\tau_{\x+2\deg}\bphi_{\t}]\cdot\sum_{\ell_{1},\ell_{2}=1,\ldots,2\deg}c_{\ell_{1},\ell_{2}}\bphi_{\t,\x+\ell_{1}}\bphi_{\t,\x+\ell_{2}}\Big)\cdot\mathbf{Z}^{N}_{\t,\x}\d\t\label{eq:msheII(0)a2}\\
&-\lambda N\grad^{\mathbf{X}}_{1}\Big(\mathbf{F}_{>2}[\tau_{\x+2\deg}\bphi_{\t}]\mathbf{Z}^{N}_{\t,\x}\Big)\d\t\label{eq:msheII(0)a3}\\
&+\tfrac12\lambda^{2}N\grad^{\mathbf{X}}_{1}\left(\left[\alpha{\boldsymbol{\phi}}_{\t,\x}^{2}-1\right]\mathbf{Z}^{N}_{\t,\x}\right)\label{eq:msheII(0)b}\\
&+\tfrac12\lambda^{2}N\grad^{\mathbf{X}}_{-1}\left(\left[\alpha{\boldsymbol{\phi}}_{\t,\x+1}^{2}-1\right]\mathbf{Z}^{N}_{\t,\x}\right)\label{eq:msheII(0)c}\\
&+\tfrac13\beta_{2}\lambda N\grad^{\mathbf{X}}_{1}(\Delta\{\mathscr{U}'[\bphi_{\t,\x+2}]\mathscr{U}'[\bphi_{\t,\x+3}]\}\mathbf{Z}^{N}_{\t,\x})\label{eq:msheII(0)d}\\
&-\tfrac{1}{12}\beta_{2}\lambda^{3}(\Delta\{\mathscr{U}'[\bphi_{\t,\x+2}]\mathscr{U}'[\bphi_{\t,\x+3}]\})\bphi_{\t,\x+1}^{2}\mathbf{Z}^{N}_{\t,\x}\d\t.\label{eq:msheII(0)e}\\
&-\tfrac{1}{12}\beta_{2}\lambda^{3}(\Delta\{\mathscr{U}'[\bphi_{\t,\x+2}]\mathscr{U}'[\bphi_{\t,\x+3}]\})\bphi_{\t,\x+2}^{2}\mathbf{Z}^{N}_{\t,\x}\d\t.\label{eq:msheII(0)e1}\\
&+\tfrac{1}{12}\lambda^{4}\left(\grad^{\mathbf{X}}_{1}\{\mathscr{U}'[{\boldsymbol{\phi}}_{\t,\x+1}]{\boldsymbol{\phi}}_{\t,\x+1}^{3}\}\right)\mathbf{Z}^{N}_{\t,\x}\d\t.\label{eq:msheII(0)f}
\end{align}
The terms in the previous display are all error terms obtained essentially by Taylor expanding the exponential factor in {\small$\mathbf{Z}^{N}$} as in \cite{BG}, as well as from shifting {\small$\mathbf{F}_{2},\mathbf{F}_{>2}$} in space. 
\item Finally, the last error term, which ultimately contributes to \eqref{eq:s-sdeIII}, is given by
\begin{align}
\mathfrak{f}_{\mathrm{err}}[\tau_{\x}{\boldsymbol{\phi}}_{\t}]\mathbf{Z}^{N}_{\t,\x}&=N^{\frac12}\grad^{\mathbf{X}}_{1}\left(\mathfrak{f}_{1}[\tau_{\x}{\boldsymbol{\phi}}_{\t}]\mathbf{Z}^{N}_{\t,\x}\right)\label{eq:msheIIfa}\\
&+N^{\frac12}\grad^{\mathbf{X}}_{-1}\left(\mathfrak{f}_{2}[\tau_{\x}{\boldsymbol{\phi}}_{\t}]\mathbf{Z}^{N}_{\t,\x}\right)\label{eq:msheIIfb}\\
&+N^{-\frac12}\mathfrak{f}_{3}[\tau_{\x}{\boldsymbol{\phi}}_{\t}]\mathbf{Z}^{N}_{\t,\x}\label{eq:msheIIfc}\\
&+N^{\frac32}\grad^{\mathbf{X}}_{1}\grad^{\mathbf{X}}_{1}(\mathfrak{f}_{4}[\tau_{\x}\bphi_{\t}]\mathbf{Z}^{N}_{\t,\x})\label{eq:msheIIfd}\\
&+N\Delta(\mathfrak{f}_{5}[\tau_{\x}\bphi_{\t}]\mathbf{Z}^{N}_{\t,\x}),\label{eq:msheIIfe}
\end{align}
where the {\small$\mathfrak{f}_{i}$} satisfy the following for some $\mathrm{c},\mathrm{C}=\mathrm{O}(1)$:
\begin{align}
|\mathfrak{f}_{i}[\bphi]|\lesssim1+\sum_{|\w|\leq\mathrm{c}}|\bphi_{\w}|^{\mathrm{C}}.\label{eq:msheIIff}
\end{align}
\end{enumerate}
\begin{lemma}\label{lemma:mshe}
\fsp With notation to be explained afterwards, we have the following \abbr{SPDE}:
\begin{align}
\d\mathbf{Z}^{N}_{\t,\x}&=\mathscr{T}_{N}\mathbf{Z}^{N}_{\t,\x}\d\t+\sqrt{2}\lambda N^{\frac12}\mathbf{Z}^{N}_{\t,\x}\d\mathbf{b}_{\t,\x}+\mathfrak{e}[\tau_{\x}{\boldsymbol{\phi}}_{\t}]\mathbf{Z}^{N}_{\t,\x}\d\t.\label{eq:msheI}
\end{align}
Above, the Brownian motions $\mathbf{b}$ are exactly the ones from \eqref{eq:curr}. Also, {\small$\mathscr{T}_{N}:=N^{2}\alpha\Delta+\frac14N\lambda^{2}\Delta$} as in Proposition \ref{prop:s-sde}, and $\alpha,\lambda$ are defined after \eqref{eq:ch}. Finally, the error term $\mathfrak{e}$ is given in \eqref{eq:msheII}.
\end{lemma}
\begin{proof}
\emph{In this argument only}, we will write $\a\approx\b$ if {\small$\a,\b$} are equal modulo terms of the form \eqref{eq:msheIIfa}-\eqref{eq:msheIIfe}. This will clean up many identities. (Functions that satisfy \eqref{eq:msheIIff} to appear in the argument below will be polynomials in {\small$\bphi_{\w},\mathscr{U}'[\bphi_{\w}]$} for {\small$|\w|\lesssim1$}.) First, by \eqref{eq:ch}, the formula {\small$\mathbf{j}^{N}_{\t,\x}-\mathbf{j}^{N}_{\t,\x-1}=N^{-1/2}{\boldsymbol{\phi}}_{\t,\x}$}, and Taylor expansion, we have
\begin{align}
&\grad^{\mathbf{X}}_{1}\mathbf{Z}^{N}_{\t,\x}=\mathbf{Z}^{N}_{\t,\x+1}-\mathbf{Z}^{N}_{\t,\x}=\left[\exp(\lambda\grad^{\mathbf{X}}_{1}\mathbf{j}^{N}_{\t,\x})-1\right]\mathbf{Z}^{N}_{\t,\x}\\
&\approx(\lambda N^{-\frac12}{\boldsymbol{\phi}}_{\t,\x+1}+\tfrac12\lambda^{2}N^{-1}{\boldsymbol{\phi}}_{\t,\x+1}^{2}+\tfrac16\lambda^{3}N^{-\frac32}{\boldsymbol{\phi}}_{\t,\x+1}^{3}+\tfrac{1}{24}\lambda^{4}N^{-2}{\boldsymbol{\phi}}_{\t,\x+1}^{4})\mathbf{Z}^{N}_{\t,\x}.\label{eq:grad-z-plus}
\end{align}
By the same token, we have
\begin{align}
&\grad^{\mathbf{X}}_{-1}\mathbf{Z}^{N}_{\t,\x}=\mathbf{Z}^{N}_{\t,\x-1}-\mathbf{Z}^{N}_{\t,\x}=\left[\exp(\lambda\grad^{\mathbf{X}}_{-1}\mathbf{j}^{N}_{\t,\x})-1\right]\mathbf{Z}^{N}_{\t,\x}\\
&\approx(-\lambda N^{-\frac12}{\boldsymbol{\phi}}_{\t,\x}+\tfrac12\lambda^{2}N^{-1}{\boldsymbol{\phi}}_{\t,\x}^{2}-\tfrac16\lambda^{3}N^{-\frac32}{\boldsymbol{\phi}}_{\t,\x}^{3}+\tfrac{1}{24}\lambda^{4}N^{-2}{\boldsymbol{\phi}}_{\t,\x}^{4})\mathbf{Z}^{N}_{\t,\x}.\label{eq:grad-z-minus}
\end{align}
Now, by \eqref{eq:grad-z-minus}, we have the following formula for the Laplacian {\small$\Delta\mathbf{Z}^{N}:=-\grad^{\mathbf{X}}_{1}\grad^{\mathbf{X}}_{-1}\mathbf{Z}^{N}$}:
\begin{align}
N^{2}\Delta\mathbf{Z}^{N}_{\t,\x}&\approx\lambda N^{\frac32}\grad^{\mathbf{X}}_{1}\left({\boldsymbol{\phi}}_{\t,\x}\mathbf{Z}^{N}_{\t,\x}\right)-\tfrac12\lambda^{2}N\grad^{\mathbf{X}}_{1}\left({\boldsymbol{\phi}}_{\t,\x}^{2}\mathbf{Z}^{N}_{\t,\x}\right).\label{eq:lap-z-1}
\end{align}
(Indeed, when we take gradients of the third and fourth terms in \eqref{eq:grad-z-minus}, we obtain terms of the form given by the \abbr{RHS} of \eqref{eq:msheIIfa} since the corresponding power of {\small$N$} is {\small$\mathrm{O}(N^{1/2})$}.) Since {\small$\grad^{\mathbf{X}}_{\ell}$} operators all commute with each other, we also have {\small$\Delta\mathbf{Z}^{N}:=-\grad^{\mathbf{X}}_{-1}\grad^{\mathbf{X}}_{1}\mathbf{Z}^{N}$}; thus, by \eqref{eq:grad-z-plus}, we get
\begin{align}
N^{}\Delta\mathbf{Z}^{N}_{\t,\x}&\approx-\lambda N^{\frac32}\grad^{\mathbf{X}}_{-1}\left({\boldsymbol{\phi}}_{\t,\x+1}\mathbf{Z}^{N}_{\t,\x}\right)-\tfrac12\lambda^{2} N\grad^{\mathbf{X}}_{-1}\left({\boldsymbol{\phi}}_{\t,\x+1}^{2}\mathbf{Z}^{N}_{\t,\x}\right).\label{eq:lap-z-2}
\end{align}
By averaging the two expressions \eqref{eq:lap-z-1} and \eqref{eq:lap-z-2}, we obtain
\begin{align*}
\alpha N^{2}\Delta\mathbf{Z}^{N}_{\t,\x}&\approx\tfrac12\lambda \alpha N^{\frac32}\grad^{\mathbf{X}}_{1}\left({\boldsymbol{\phi}}_{\t,\x}\mathbf{Z}^{N}_{\t,\x}\right)-\tfrac12\lambda \alpha N^{\frac32}\grad^{\mathbf{X}}_{-1}\left({\boldsymbol{\phi}}_{\t,\x+1}\mathbf{Z}^{N}_{\t,\x}\right)\\
&-\tfrac14\lambda^{2}N\alpha \grad^{\mathbf{X}}_{1}\left({\boldsymbol{\phi}}^{2}_{\t,\x}\mathbf{Z}^{N}_{\t,\x}\right)-\tfrac14\lambda^{2}\alpha N\grad^{\mathbf{X}}_{-1}\left({\boldsymbol{\phi}}^{2}_{\t,\x+1}\mathbf{Z}^{N}_{\t,\x}\right).
\end{align*}
We now write {\small$\alpha\bphi_{\t,\cdot}^{2}=1+\alpha\bphi_{\t,\cdot}^{2}-1$} in the second line. Since {\small$\Delta=\grad^{\mathbf{X}}_{1}+\grad^{\mathbf{X}}_{-1}$}, this gives
\begin{align}
\mathscr{T}_{N}\mathbf{Z}^{N}_{\t,\x}&:=\alpha N^{2}\Delta\mathbf{Z}^{N}_{\t,\x}+\tfrac14\lambda^{2}N\Delta\mathbf{Z}^{N}_{\t,\x}\nonumber\\
&\approx\tfrac12\lambda \alpha N^{\frac32}\grad^{\mathbf{X}}_{1}\left({\boldsymbol{\phi}}_{\t,\x}\mathbf{Z}^{N}_{\t,\x}\right)-\tfrac12\lambda \alpha N^{\frac32}\grad^{\mathbf{X}}_{-1}\left({\boldsymbol{\phi}}_{\t,\x+1}\mathbf{Z}^{N}_{\t,\x}\right)\nonumber\\
&-\tfrac14\lambda^{2}N\grad^{\mathbf{X}}_{1}\left([\alpha\bphi_{\t,\x}^{2}-1]\mathbf{Z}^{N}_{\t,\x}\right)-\tfrac14\lambda^{2} N\grad^{\mathbf{X}}_{-1}\left([\alpha\bphi_{\t,\x+1}^{2}-1]\mathbf{Z}^{N}_{\t,\x}\right).\label{eq:lap-z}
\end{align}
We now compute the dynamics of {\small$\mathbf{Z}^{N}$}. By \eqref{eq:ch}, \eqref{eq:curr}, and the It\^{o} formula, we have 
\begin{align}
&\d\mathbf{Z}^{N}_{\t,\x}=\lambda\mathbf{Z}^{N}_{\t,\x}\d\mathbf{j}^{N}_{\t,\x}+\tfrac12\lambda^{2}\mathbf{Z}^{N}_{\t,\x}\d[\mathbf{j}^{N}_{\t,\x},\mathbf{j}^{N}_{\t,\x}]-\lambda\mathscr{R}_{\lambda}\mathbf{Z}^{N}_{\t,\x}\d\t\nonumber\\
&=\lambda N^{\frac32}\left(\grad^{\mathbf{X}}_{1}\mathscr{U}'[{\boldsymbol{\phi}}_{\t,\x}]\right)\mathbf{Z}^{N}_{\t,\x}\d\t+(\lambda N\mathbf{F}[\tau_{\x}{\boldsymbol{\phi}}_{\t}]+\lambda^{2}N-\lambda\mathscr{R}_{\lambda})\mathbf{Z}^{N}_{\t,\x}\d\t+\sqrt{2}\lambda N^{\frac12}\mathbf{Z}^{N}_{\t,\x}\d\mathbf{b}_{\t,\x}.\label{eq:mshe-ito}
\end{align}
Let us unfold the first term in the second line \eqref{eq:mshe-ito}. For this, we use a discrete Leibniz rule: for any {\small$\mathsf{f},\mathsf{g}:\Z\to\R$}, we have {\small$\mathsf{f}_{\x}\grad^{\mathbf{X}}_{\mathfrak{l}}\mathsf{g}_{\x}=\grad^{\mathbf{X}}_{1}(\mathsf{f}\mathsf{g})_{\x}-\mathsf{g}_{\x+\mathfrak{l}}\grad^{\mathbf{X}}_{\mathfrak{l}}\mathsf{f}_{\x}$}. This gives
\begin{align*}
\lambda N^{\frac32}\left(\grad^{\mathbf{X}}_{1}\mathscr{U}'[{\boldsymbol{\phi}}_{\t,\x}]\right)\mathbf{Z}^{N}_{\t,\x}&=\lambda N^{\frac32}\grad^{\mathbf{X}}_{1}\left(\mathscr{U}'[{\boldsymbol{\phi}}_{\t,\x}]\mathbf{Z}^{N}_{\t,\x}\right)-\lambda N^{\frac32}\mathscr{U}'[{\boldsymbol{\phi}}_{\t,\x+1}]\grad^{\mathbf{X}}_{1}\mathbf{Z}^{N}_{\t,\x}\\
&\approx\lambda N^{\frac32}\grad^{\mathbf{X}}_{1}\left(\mathscr{U}'[{\boldsymbol{\phi}}_{\t,\x}]\mathbf{Z}^{N}_{\t,\x}\right)-\lambda^{2}N\mathscr{U}'[{\boldsymbol{\phi}}_{\t,\x+1}]{\boldsymbol{\phi}}_{\t,\x+1}\mathbf{Z}^{N}_{\t,\x}\\
&-\tfrac12\lambda^{3}N^{\frac12}\mathscr{U}'[{\boldsymbol{\phi}}_{\t,\x+1}]{\boldsymbol{\phi}}_{\t,\x+1}^{2}\mathbf{Z}^{N}_{\t,\x}-\tfrac16\lambda^{4}\mathscr{U}'[{\boldsymbol{\phi}}_{\t,\x+1}]{\boldsymbol{\phi}}_{\t,\x+1}^{3}\mathbf{Z}^{N}_{\t,\x},
\end{align*}
where the second step follows by \eqref{eq:grad-z-plus}. Next, we can also write {\small$\grad^{\mathbf{X}}_{1}(\mathscr{U}'[{\boldsymbol{\phi}}_{\t,\x}])=-\grad^{\mathbf{X}}_{-1}(\mathscr{U}'[{\boldsymbol{\phi}}_{\t,\x+1}])$}. Thus, by the same reasoning applied to give us the previous display (but using \eqref{eq:grad-z-minus} instead of \eqref{eq:grad-z-plus}), we get
\begin{align*}
\lambda N^{\frac32}\left(\grad^{\mathbf{X}}_{1}\mathscr{U}'[{\boldsymbol{\phi}}_{\t,\x}]\right)\mathbf{Z}^{N}_{\t,\x}&=-\lambda N^{\frac32}\grad^{\mathbf{X}}_{-1}\left(\mathscr{U}'[{\boldsymbol{\phi}}_{\t,\x+1}]\mathbf{Z}^{N}_{\t,\x}\right)+\lambda N^{\frac32}\mathscr{U}'[{\boldsymbol{\phi}}_{\t,\x}]\grad^{\mathbf{X}}_{-1}\mathbf{Z}^{N}_{\t,\x}\\
&\approx-\lambda N^{\frac32}\grad^{\mathbf{X}}_{-1}\left(\mathscr{U}'[{\boldsymbol{\phi}}_{\t,\x+1}]\mathbf{Z}^{N}_{\t,\x}\right)-\lambda^{2}N\mathscr{U}'[{\boldsymbol{\phi}}_{\t,\x}]{\boldsymbol{\phi}}_{\t,\x}\mathbf{Z}^{N}_{\t,\x}\\
&+\tfrac12\lambda^{3}N^{\frac12}\mathscr{U}'[{\boldsymbol{\phi}}_{\t,\x}]{\boldsymbol{\phi}}_{\t,\x}^{2}\mathbf{Z}^{N}_{\t,\x}-\tfrac16\lambda^{4}\mathscr{U}'[{\boldsymbol{\phi}}_{\t,\x}]{\boldsymbol{\phi}}_{\t,\x}^{3}\mathbf{Z}^{N}_{\t,\x}.
\end{align*}
By averaging the previous two displays, we obtain
\begin{align}
\lambda N^{\frac32}\left(\grad^{\mathbf{X}}_{1}\mathscr{U}'[{\boldsymbol{\phi}}_{\t,\x}]\right)\mathbf{Z}^{N}_{\t,\x}&\approx\tfrac12\lambda N^{\frac32}\grad^{\mathbf{X}}_{1}\left(\mathscr{U}'[{\boldsymbol{\phi}}_{\t,\x}]\mathbf{Z}^{N}_{\t,\x}\right)-\tfrac12\lambda N^{\frac32}\grad^{\mathbf{X}}_{-1}\left(\mathscr{U}'[{\boldsymbol{\phi}}_{\t,\x+1}]\mathbf{Z}^{N}_{\t,\x}\right)\label{eq:u-grad-1}\\
&-\tfrac12\lambda^{2}N\mathscr{U}'[{\boldsymbol{\phi}}_{\t,\x+1}]{\boldsymbol{\phi}}_{\t,\x+1}\mathbf{Z}^{N}_{\t,\x}-\tfrac12\lambda^{2}N\mathscr{U}'[{\boldsymbol{\phi}}_{\t,\x}]{\boldsymbol{\phi}}_{\t,\x}\mathbf{Z}^{N}_{\t,\x}\label{eq:u-grad-2}\\
&-\tfrac14\lambda^{3}N^{\frac12}\left(\grad^{\mathbf{X}}_{1}\{\mathscr{U}'[{\boldsymbol{\phi}}_{\t,\x}]{\boldsymbol{\phi}}_{\t,\x}^{2}\}\right)\mathbf{Z}^{N}_{\t,\x}-\tfrac16\lambda^{4}\mathscr{U}'[{\boldsymbol{\phi}}_{\t,\x+1}]{\boldsymbol{\phi}}_{\t,\x+1}^{3}\mathbf{Z}^{N}_{\t,\x}\label{eq:u-grad-3}\\
&+\tfrac{1}{12}\lambda^{4}\left(\grad^{\mathbf{X}}_{1}\{\mathscr{U}'[{\boldsymbol{\phi}}_{\t,\x}]{\boldsymbol{\phi}}_{\t,\x}^{3}\}\right)\mathbf{Z}^{N}_{\t,\x}.\label{eq:u-grad-4}
\end{align}
We must now further expand the \abbr{RHS} of the {\color{black}display above}. We start with the \abbr{RHS} of \eqref{eq:u-grad-1}. We subtract {\small$\alpha\bphi_{\t,\x}$} and {\small$\alpha\bphi_{\t,\x+1}$} from {\small$\mathscr{U}'[\bphi_{\t,\x}]$} and {\small$\mathscr{U}'[\bphi_{\t,\x+1}]$} inside the gradients, respectively. This gives
\begin{align*}
{\textstyle\mathsf{RHS}}\eqref{eq:u-grad-1}&=\tfrac12\lambda N^{\frac32}\alpha\grad^{\mathbf{X}}_{1}(\bphi_{\t,\x}\mathbf{Z}^{N}_{\t,\x})-\tfrac12\lambda N^{\frac32}\alpha\grad^{\mathbf{X}}_{-1}(\bphi_{\t,\x+1}\mathbf{Z}^{N}_{\t,\x})\\
&+\tfrac12\lambda N^{\frac32}\grad^{\mathbf{X}}_{1}(\mathscr{V}'[\bphi_{\t,\x}]\mathbf{Z}^{N}_{\t,\x})-\tfrac12\lambda N^{\frac32}\grad^{\mathbf{X}}_{-1}(\mathscr{V}'[\bphi_{\t,\x+1}]\mathbf{Z}^{N}_{\t,\x})\\
&\approx\mathscr{T}_{N}\mathbf{Z}^{N}_{\t,\x}+\tfrac12\lambda N^{\frac32}\grad^{\mathbf{X}}_{1}(\mathscr{V}'[\bphi_{\t,\x}]\mathbf{Z}^{N}_{\t,\x})-\tfrac12\lambda N^{\frac32}\grad^{\mathbf{X}}_{-1}(\mathscr{V}'[\bphi_{\t,\x+1}]\mathbf{Z}^{N}_{\t,\x})\\
&+\tfrac12\lambda^{2}N\grad^{\mathbf{X}}_{1}\left([\alpha\bphi_{\t,\x}^{2}-1]\mathbf{Z}^{N}_{\t,\x}\right)+\tfrac12\lambda^{2} N\grad^{\mathbf{X}}_{-1}\left([\alpha\bphi_{\t,\x+1}^{2}-1]\mathbf{Z}^{N}_{\t,\x}\right),
\end{align*}
where we recall {\small$\mathscr{V}[\bphi_{0}]:=\mathscr{U}[\bphi_{0}]-\alpha\bphi_{0}^{2}$}, and where the second identity follows by \eqref{eq:lap-z} applied to the \abbr{RHS} of the first line. We leave \eqref{eq:u-grad-2} alone. We move to the first term in \eqref{eq:u-grad-3}. By another discrete Leibniz rule,
\begin{align}
-\tfrac14\lambda^{3}N^{\frac12}\left(\grad^{\mathbf{X}}_{1}\{\mathscr{U}'[{\boldsymbol{\phi}}_{\t,\x}]{\boldsymbol{\phi}}_{\t,\x}^{2}\}\right)\mathbf{Z}^{N}_{\t,\x}\d\t&=-\tfrac14\lambda^{3}N^{\frac12}\grad^{\mathbf{X}}_{1}\left(\{\mathscr{U}'[{\boldsymbol{\phi}}_{\t,\x}]{\boldsymbol{\phi}}_{\t,\x}^{2}\}\mathbf{Z}^{N}_{\t,\x}\right)\d\t\label{eq:z-merge-4-1a}\\
&\quad+\tfrac14\lambda^{3}N^{\frac12}\mathscr{U}'[{\boldsymbol{\phi}}_{\t,\x+1}]{\boldsymbol{\phi}}_{\t,\x+1}^{2}\grad^{\mathbf{X}}_{1}\mathbf{Z}^{N}_{\t,\x}\d\t\label{eq:z-merge-4-1b}\\
&\approx\tfrac14\lambda^{4}\mathscr{U}'[{\boldsymbol{\phi}}_{\t,\x+1}]{\boldsymbol{\phi}}_{\t,\x+1}^{3}\mathbf{Z}^{N}_{\t,\x}\d\t,\label{eq:z-merge-4-1d}
\end{align}
where the second line follows by \eqref{eq:grad-z-plus} and by noting that the \abbr{RHS} of the first line above has the form \eqref{eq:msheIIfa}. If we combine the previous three displays, then we arrive at 
\begin{align}
\lambda  N^{\frac32}\Big(\grad^{\mathbf{X}}_{1}\mathscr{U}'[\bphi_{\t,\x}]\Big)\mathbf{Z}^{N}_{\t,\x}&\approx\mathscr{T}_{N}\mathbf{Z}^{N}_{\t,\x}+\tfrac12\lambda N^{\frac32}\grad^{\mathbf{X}}_{1}(\mathscr{V}'[\bphi_{\t,\x}]\mathbf{Z}^{N}_{\t,\x})-\tfrac12\lambda N^{\frac32}\grad^{\mathbf{X}}_{-1}(\mathscr{V}'[\bphi_{\t,\x+1}]\mathbf{Z}^{N}_{\t,\x})\\
&+\tfrac12\lambda^{2}N\grad^{\mathbf{X}}_{1}\left([\alpha\bphi_{\t,\x}^{2}-1]\mathbf{Z}^{N}_{\t,\x}\right)+\tfrac12\lambda^{2} N\grad^{\mathbf{X}}_{-1}\left([\alpha\bphi_{\t,\x+1}^{2}-1]\mathbf{Z}^{N}_{\t,\x}\right)\\
&-\tfrac12\lambda^{2}N\mathscr{U}'[{\boldsymbol{\phi}}_{\t,\x+1}]{\boldsymbol{\phi}}_{\t,\x+1}\mathbf{Z}^{N}_{\t,\x}-\tfrac12\lambda^{2}N\mathscr{U}'[{\boldsymbol{\phi}}_{\t,\x}]{\boldsymbol{\phi}}_{\t,\x}\mathbf{Z}^{N}_{\t,\x}\\
&+\tfrac{1}{12}\lambda^{4}\mathscr{U}'[\bphi_{\t,\x+1}]\bphi_{\t,\x+1}^{3}\mathbf{Z}^{N}_{\t,\x}+\tfrac{1}{12}\lambda^{4}\left(\grad^{\mathbf{X}}_{1}\{\mathscr{U}'[{\boldsymbol{\phi}}_{\t,\x}]{\boldsymbol{\phi}}_{\t,\x}^{3}\}\right)\mathbf{Z}^{N}_{\t,\x}.
\end{align}
We now plug the previous display into the first term in \eqref{eq:mshe-ito}. When we do so, we split {\small$\mathbf{F}=\mathbf{F}_{2}+\mathbf{F}_{>2}$}, and we combine the third line above with {\small$\lambda N\mathbf{F}_{2}[\tau_{\x}\bphi_{\t}]\mathbf{Z}^{N}_{\t,\x}$} and {\small$\lambda^{2}N\mathbf{Z}^{N}_{\t,\x}$}. Recall {\small$\lambda\mathscr{R}_{\lambda}:=(1/12)\lambda^{4}\E^{0}(\mathscr{U}'[\bphi_{0}]\bphi_{0}^{3}){-\tfrac16\beta_{2}\lambda^{3}}$} from \eqref{eq:renorm}; we combine the first term in this with the first term on the fourth line above. Ultimately, we get
\begin{align}
\d\mathbf{Z}^{N}_{\t,\x}&\approx\mathscr{T}_{N}\mathbf{Z}^{N}_{\t,\x}\d\t+\sqrt{2}\lambda N^{\frac12}\mathbf{Z}^{N}_{\t,\x}\d\mathbf{b}_{\t,\x}{+\tfrac16\beta_{2}\lambda^{3}}\label{eq:zbasic1}\\
&+\tfrac12\lambda N^{\frac32}\grad^{\mathbf{X}}_{1}(\mathscr{V}'[\bphi_{\t,\x}]\mathbf{Z}^{N}_{\t,\x})\d\t-\tfrac12\lambda N^{\frac32}\grad^{\mathbf{X}}_{-1}(\mathscr{V}'[\bphi_{\t,\x+1}]\mathbf{Z}^{N}_{\t,\x})\d\t\label{eq:zbasic2}\\
&+\lambda N(\mathbf{F}_{2}[\tau_{\x}\bphi_{\t}]-\tfrac12\lambda\{\mathscr{U}'[\bphi_{\t,\x+1}]\bphi_{\t,\x+1}-1\}-\tfrac12\lambda\{\mathscr{U}'[\bphi_{\t,\x}]\bphi_{\t,\x}-1\})\mathbf{Z}^{N}_{\t,\x}\d\t\label{eq:zbasic3}\\
&+\lambda N\mathbf{F}_{>2}[\tau_{\x}\bphi_{\t}]\mathbf{Z}^{N}_{\t,\x}\d\t\label{eq:zbasic3b}\\
&+\tfrac14\lambda^{2}N\grad^{\mathbf{X}}_{1}\left([\alpha\bphi_{\t,\x}^{2}-1]\mathbf{Z}^{N}_{\t,\x}\right)\d\t+\tfrac14\lambda^{2} N\grad^{\mathbf{X}}_{-1}\left([\alpha\bphi_{\t,\x+1}^{2}-1]\mathbf{Z}^{N}_{\t,\x}\right)\d\t\label{eq:zbasic4}\\
&+\tfrac{1}{12}\lambda^{4}(\mathscr{U}'[\bphi_{\t,\x+1}]\bphi_{\t,\x+1}^{3}-\E^{0}\{\mathscr{U}'[\bphi_{0}]\bphi_{0}^{3}\})\mathbf{Z}^{N}_{\t,\x}\d\t\label{eq:zbasic5}\\
&+\tfrac{1}{12}\lambda^{4}\left(\grad^{\mathbf{X}}_{1}\{\mathscr{U}'[{\boldsymbol{\phi}}_{\t,\x}]{\boldsymbol{\phi}}_{\t,\x}^{3}\}\right)\mathbf{Z}^{N}_{\t,\x}\d\t.\label{eq:zbasic6}
\end{align}
Now, recall {\small$\mathscr{Q}$} from \eqref{eq:quadratic_renorm} and {\small$\mathbf{F}$} from \eqref{eq:nonlinearity}; we have
\begin{align}
\eqref{eq:zbasic3}&=\lambda N\mathscr{Q}[\tau_{\x+1}\bphi_{\t}]\mathbf{Z}^{N}_{\t,\x}\d\t+\lambda N\mathscr{Q}[\tau_{\x-1}\bphi_{\t}]\mathbf{Z}^{N}_{\t,\x}\d\t\label{eq:zshift31a}\\
&-\tfrac16\beta_{2}\lambda N\Big(\Delta\{\mathscr{U}'[\bphi_{\t,\x}]\mathscr{U}'[\bphi_{\t,\x+1}]\}\Big)\cdot\mathbf{Z}^{N}_{\t,\x}\d\t.\label{eq:zshift31b}
\end{align}
For the rest of this proof, the following calculation, whose last line uses \eqref{eq:grad-z-plus}, will be important:
\begin{align}
\mathsf{f}_{\x}\mathbf{Z}^{N}_{\t,\x}&=\mathsf{f}_{\x+1}\mathbf{Z}^{N}_{\t,\x}-(\grad^{\mathbf{X}}_{1}\mathsf{f}_{\x})\mathbf{Z}^{N}_{\t,\x}\nonumber\\
&=\mathsf{f}_{\x+1}\mathbf{Z}^{N}_{\t,\x}-\grad^{\mathbf{X}}_{1}(\mathsf{f}_{\x}\mathbf{Z}^{N}_{\t,\x})+\mathsf{f}_{\x+1}\grad^{\mathbf{X}}_{1}\mathbf{Z}^{N}_{\t,\x}\nonumber\\
&\approx\mathsf{f}_{\x+1}\mathbf{Z}^{N}_{\t,\x}-\grad^{\mathbf{X}}_{1}(\mathsf{f}_{\x}\mathbf{Z}^{N}_{\t,\x})+\mathsf{f}_{\x+1}(\lambda N^{-\frac12}\bphi_{\t,\x+1}+\tfrac12\lambda^{2}N^{-1}\bphi_{\t,\x+1}^{2})\mathbf{Z}^{N}_{\t,\x}.\label{eq:zshift}
\end{align}
(In \eqref{eq:zshift}, we will always take {\small$\mathsf{f}$} to be a polynomial in {\small$\mathscr{U}'[\bphi_{\w}],\bphi_{\w}$} for {\small$|\w|\lesssim1$} times a factor {\small$\mathrm{O}(N)$}. So, the order {\small$N^{-3/2}$} term that is missing from the last line of \eqref{eq:zshift} is actually absorbed by the {\small$\approx$} notation.) We will now apply \eqref{eq:zshift} to the term \eqref{eq:zshift31b}. This gives us
\begin{align*}
\eqref{eq:zshift31b}&\approx-\tfrac16\beta_{2}\lambda N\Big(\Delta\{\mathscr{U}'[\bphi_{\t,\x+1}]\mathscr{U}'[\bphi_{\t,\x+2}]\}\Big)\mathbf{Z}^{N}_{\t,\x}\d\t\\
&+\tfrac16\beta_{2}\lambda N\grad^{\mathbf{X}}_{1}\Big(\Delta\{\mathscr{U}'[\bphi_{\t,\x}]\mathscr{U}'[\bphi_{\t,\x+1}]\}\mathbf{Z}^{N}_{\t,\x}\Big)\d\t\\
&-\tfrac16\beta_{2}\lambda^{2} N^{\frac12}\Big(\Delta\{\mathscr{U}'[\bphi_{\t,\x+1}]\mathscr{U}'[\bphi_{\t,\x+2}]\}\Big)\bphi_{\t,\x+1}\mathbf{Z}^{N}_{\t,\x}\d\t\\
&-\tfrac{1}{12}\beta_{2}\lambda^{3}\Big(\Delta\{\mathscr{U}'[\bphi_{\t,\x+1}]\mathscr{U}'[\bphi_{\t,\x+2}]\}\Big)\bphi_{\t,\x+1}^{2}\mathbf{Z}^{N}_{\t,\x}\d\t.
\end{align*}
We again apply \eqref{eq:zshift} to the \abbr{RHS} of the first line above:
\begin{align*}
-\tfrac16\beta_{2}\lambda N\Big(\Delta\{\mathscr{U}'[\bphi_{\t,\x+1}]\mathscr{U}'[\bphi_{\t,\x+2}]\}\Big)\mathbf{Z}^{N}_{\t,\x}\d\t&\approx-\tfrac16\beta_{2}\lambda N\Big(\Delta\{\mathscr{U}'[\bphi_{\t,\x+2}]\mathscr{U}'[\bphi_{\t,\x+3}]\}\Big)\mathbf{Z}^{N}_{\t,\x}\d\t\\
&+\tfrac16\beta_{2}\lambda N\grad^{\mathbf{X}}_{1}\Big(\Delta\{\mathscr{U}'[\bphi_{\t,\x+1}]\mathscr{U}'[\bphi_{\t,\x+2}]\}\mathbf{Z}^{N}_{\t,\x}\Big)\d\t\\
&-\tfrac16\beta_{2}\lambda^{2} N^{\frac12}\Big(\Delta\{\mathscr{U}'[\bphi_{\t,\x+2}]\mathscr{U}'[\bphi_{\t,\x+3}]\}\Big)\bphi_{\t,\x+1}\mathbf{Z}^{N}_{\t,\x}\d\t\\
&-\tfrac{1}{12}\beta_{2}\lambda^{3}\Big(\Delta\{\mathscr{U}'[\bphi_{\t,\x+2}]\mathscr{U}'[\bphi_{\t,\x+3}]\}\Big)\bphi_{\t,\x+1}^{2}\mathbf{Z}^{N}_{\t,\x}\d\t.
\end{align*}
We now combine the previous two displays. When doing so, the order {\small$N^{1/2}$} terms combine to {\small$-\frac16\beta_{2}\lambda^{2}\mathscr{G}[\tau_{\x-1}\bphi_{\t}]$}; see \eqref{eq:funnygradient}. Moreover, by using \eqref{eq:zshift}, we can shift the order {\small$1$} terms in space modulo terms of the form \eqref{eq:msheIIfa}-\eqref{eq:msheIIfe}. The same reasoning applies to terms with {\small$N\grad^{\mathbf{X}}_{1}$}. In particular, we ultimately arrive at 
\begin{align*}
\eqref{eq:zshift31b}&\approx-\tfrac16\beta_{2}\lambda N\Big(\Delta\{\mathscr{U}'[\bphi_{\t,\x+2}]\mathscr{U}'[\bphi_{\t,\x+3}]\}\Big)\mathbf{Z}^{N}_{\t,\x}\d\t\\
&+\tfrac13\beta_{2}\lambda N\grad^{\mathbf{X}}_{1}\Big(\Delta\{\mathscr{U}'[\bphi_{\t,\x+2}]\mathscr{U}'[\bphi_{\t,\x+3}]\}\mathbf{Z}^{N}_{\t,\x}\Big)\d\t\\
&-\tfrac16\beta_{2}\lambda^{2}N^{\frac12}\mathscr{G}[\tau_{\x-1}\bphi_{\t}]\mathbf{Z}^{N}_{\t,\x}\d\t\\
&-\tfrac{1}{12}\beta_{2}\lambda^{3}\Big(\Delta\{\mathscr{U}'[\bphi_{\t,\x+2}]\mathscr{U}'[\bphi_{\t,\x+3}]\}\Big)\bphi_{\t,\x+2}^{2}\mathbf{Z}^{N}_{\t,\x}\d\t\\
&-\tfrac{1}{12}\beta_{2}\lambda^{3}\Big(\Delta\{\mathscr{U}'[\bphi_{\t,\x+2}]\mathscr{U}'[\bphi_{\t,\x+3}]\}\Big)\bphi_{\t,\x+1}^{2}\mathbf{Z}^{N}_{\t,\x}\d\t.
\end{align*}
We now apply \eqref{eq:zshift} to third line of the previous display. We get
\begin{align*}
-\tfrac16\beta_{2}\lambda^{2}N^{\frac12}\mathscr{G}[\tau_{\x-1}\bphi_{\t}]\mathbf{Z}^{N}_{\t,\x}\d\t&\approx-\tfrac16\beta_{2}\lambda^{2}N^{\frac12}\mathscr{G}[\tau_{\x}\bphi_{\t}]\mathbf{Z}^{N}_{\t,\x}\d\t-\tfrac16\beta_{2}\lambda^{3}\mathscr{G}[\tau_{\x}\bphi_{\t}]\bphi_{\t,\x+1}\mathbf{Z}^{N}_{\t,\x}\d\t.
\end{align*}
In particular, if we combine the previous two displays, then we have
\begin{align}
\eqref{eq:zshift31b}&\approx-\tfrac16\beta_{2}\lambda N\Big(\Delta\{\mathscr{U}'[\bphi_{\t,\x+2}]\mathscr{U}'[\bphi_{\t,\x+3}]\}\Big)\mathbf{Z}^{N}_{\t,\x}\d\t\label{eq:zshiftapply1}\\
&+\tfrac13\beta_{2}\lambda N\grad^{\mathbf{X}}_{1}\Big(\Delta\{\mathscr{U}'[\bphi_{\t,\x+2}]\mathscr{U}'[\bphi_{\t,\x+3}]\}\mathbf{Z}^{N}_{\t,\x}\Big)\d\t\label{eq:zshiftapply2}\\
&-\tfrac16\beta_{2}\lambda^{2}N^{\frac12}\mathscr{G}[\tau_{\x}\bphi_{\t}]\mathbf{Z}^{N}_{\t,\x}\d\t\label{eq:zshiftapply3}\\
&-\tfrac16\beta_{2}\lambda^{3}\mathscr{G}[\tau_{\x}\bphi_{\t}]\bphi_{\t,\x+1}\mathbf{Z}^{N}_{\t,\x}\d\t\label{eq:zshiftapply4}\\
&-\tfrac{1}{12}\beta_{2}\lambda^{3}\Big(\Delta\{\mathscr{U}'[\bphi_{\t,\x+2}]\mathscr{U}'[\bphi_{\t,\x+3}]\}\Big)\bphi_{\t,\x+2}^{2}\mathbf{Z}^{N}_{\t,\x}\d\t\label{eq:zshiftapply5}\\
&-\tfrac{1}{12}\beta_{2}\lambda^{3}\Big(\Delta\{\mathscr{U}'[\bphi_{\t,\x+2}]\mathscr{U}'[\bphi_{\t,\x+3}]\}\Big)\bphi_{\t,\x+1}^{2}\mathbf{Z}^{N}_{\t,\x}\d\t.\label{eq:zshiftapply6}
\end{align}
We now plug \eqref{eq:zshiftapply1}-\eqref{eq:zshiftapply6} into \eqref{eq:zshift31b}, which we then plug into \eqref{eq:zbasic1}-\eqref{eq:zbasic6}. The final result is the following display, which we shortly match to the desired equation \eqref{eq:msheI}-\eqref{eq:msheII} (especially given how complicated it is):
\begin{align}
\d\mathbf{Z}^{N}_{\t,\x}&\approx\mathscr{T}_{N}\mathbf{Z}^{N}_{\t,\x}\d\t+\sqrt{2}\lambda N^{\frac12}\mathbf{Z}^{N}_{\t,\x}\d\mathbf{b}_{\t,\x}{+\tfrac16\beta_{2}\lambda^{3}}\label{eq:zfinal1}\\
&+\tfrac12\lambda N^{\frac32}\grad^{\mathbf{X}}_{1}(\mathscr{V}'[\bphi_{\t,\x}]\mathbf{Z}^{N}_{\t,\x})\d\t-\tfrac12\lambda N^{\frac32}\grad^{\mathbf{X}}_{-1}(\mathscr{V}'[\bphi_{\t,\x+1}]\mathbf{Z}^{N}_{\t,\x})\d\t\label{eq:zfinal2}\\
&+\lambda N\mathscr{Q}[\tau_{\x+1}\bphi_{\t}]\mathbf{Z}^{N}_{\t,\x}\d\t+\lambda N\mathscr{Q}[\tau_{\x-1}\bphi_{\t}]\mathbf{Z}^{N}_{\t,\x}\d\t\label{eq:zfinal3}\\
&+\lambda N\mathbf{F}_{>2}[\tau_{\x}\bphi_{\t}]\mathbf{Z}^{N}_{\t,\x}\d\t\label{eq:zfinal3a}\\
&-\tfrac16\beta_{2}\lambda N\Big(\Delta\{\mathscr{U}'[\bphi_{\t,\x+2}]\mathscr{U}'[\bphi_{\t,\x+3}]\}\Big)\mathbf{Z}^{N}_{\t,\x}\d\t\label{eq:zfinal3b}\\
&+\tfrac13\beta_{2}\lambda N\grad^{\mathbf{X}}_{1}\Big(\Delta\{\mathscr{U}'[\bphi_{\t,\x+2}]\mathscr{U}'[\bphi_{\t,\x+3}]\}\mathbf{Z}^{N}_{\t,\x}\Big)\d\t\label{eq:zfinal3c}\\
&-\tfrac16\beta_{2}\lambda^{2} N^{\frac12}\mathscr{G}[\tau_{\x}\bphi_{\t}]\mathbf{Z}^{N}_{\t,\x}\d\t\label{eq:zfinal3d}\\
&-\tfrac16\beta_{2}\lambda^{3}\mathscr{G}[\tau_{\x}\bphi_{\t}]\bphi_{\t,\x+1}\mathbf{Z}^{N}_{\t,\x}\d\t\label{eq:zfinal3d1}\\
&-\tfrac{1}{12}\beta_{2}\lambda^{3}\Big(\Delta\{\mathscr{U}'[\bphi_{\t,\x+2}]\mathscr{U}'[\bphi_{\t,\x+3}]\}\Big)\bphi_{\t,\x+2}^{2}\mathbf{Z}^{N}_{\t,\x}\d\t\label{eq:zfinal3e}\\
&-\tfrac{1}{12}\beta_{2}\lambda^{3}\Big(\Delta\{\mathscr{U}'[\bphi_{\t,\x+2}]\mathscr{U}'[\bphi_{\t,\x+3}]\}\Big)\bphi_{\t,\x+1}^{2}\mathbf{Z}^{N}_{\t,\x}\d\t\label{eq:zfinal3e1}\\
&+\tfrac14\lambda^{2}N\grad^{\mathbf{X}}_{1}\left([\alpha\bphi_{\t,\x}^{2}-1]\mathbf{Z}^{N}_{\t,\x}\right)\d\t+\tfrac14\lambda^{2} N\grad^{\mathbf{X}}_{-1}\left([\alpha\bphi_{\t,\x+1}^{2}-1]\mathbf{Z}^{N}_{\t,\x}\right)\d\t\label{eq:zfinal4}\\
&+\tfrac{1}{12}\lambda^{4}(\mathscr{U}'[\bphi_{\t,\x+1}]\bphi_{\t,\x+1}^{3}-\E^{0}\{\mathscr{U}'[\bphi_{0}]\bphi_{0}^{3}\})\mathbf{Z}^{N}_{\t,\x}\d\t\label{eq:zfinal5}\\
&+\tfrac{1}{12}\lambda^{4}\left(\grad^{\mathbf{X}}_{1}\{\mathscr{U}'[{\boldsymbol{\phi}}_{\t,\x}]{\boldsymbol{\phi}}_{\t,\x}^{3}\}\right)\mathbf{Z}^{N}_{\t,\x}\d\t.\label{eq:zfinal6}
\end{align}
We conclude by matching \eqref{eq:zfinal1}-\eqref{eq:zfinal6} to \eqref{eq:msheI}-\eqref{eq:msheII}. The first two terms on the \abbr{RHS} of \eqref{eq:zfinal1} matches the first two terms from the \abbr{RHS} of \eqref{eq:msheI}. The terms \eqref{eq:zfinal3}-\eqref{eq:zfinal3b} match \eqref{eq:msheII(2)a}-\eqref{eq:msheII(2)b}. The terms \eqref{eq:zfinal2} and \eqref{eq:zfinal3d} match \eqref{eq:msheII(1)a}-\eqref{eq:msheII(1)b}. The last term in \eqref{eq:zfinal1} and \eqref{eq:zfinal3d1} combine to match \eqref{eq:msheII(0)a1}. The term \eqref{eq:zfinal4} matches \eqref{eq:msheII(0)b}-\eqref{eq:msheII(0)c}. The term \eqref{eq:zfinal5} matches the \abbr{RHS} of \eqref{eq:msheII(0)a}. The terms \eqref{eq:zfinal3e}-\eqref{eq:zfinal3e1} match \eqref{eq:msheII(0)e}-\eqref{eq:msheII(0)e1}. The term \eqref{eq:zfinal6} matches \eqref{eq:msheII(0)f} up to a spatial shift (that is made up for by an error term of the form in \eqref{eq:msheIIfa}-\eqref{eq:msheIIfe} due to \eqref{eq:zshift}). Finally, the term \eqref{eq:zfinal3a} matches (in the {\small$\approx$} sense) the sum of \eqref{eq:msheII(2)aa}, \eqref{eq:msheII(1)c}, and \eqref{eq:msheII(0)a2}-\eqref{eq:msheII(0)a3}; this follows by a similar repeated application of \eqref{eq:zshift} a total of {\small$2\deg$}-many times.
\end{proof}
\begin{corollary}\label{corollary:mshe}
\fsp Retain the setting of Lemma \ref{lemma:mshe}, and recall the function {\small$\mathscr{S}^{N}$} from Definition \ref{definition:zsmooth}. We have 
\begin{align}
\d\mathbf{S}^{N}_{\t,\x}&=\mathscr{T}_{N}\mathbf{S}^{N}_{\t,\x}\d\t+[\mathscr{S}^{N}\star(\sqrt{2}\lambda N^{\frac12}\mathbf{Z}^{N}_{\t,\cdot}\d\mathbf{b}_{\t,\cdot})]_{\x}+[\mathscr{S}^{N}\star(\mathfrak{e}[\tau_{\cdot}\bphi_{\t}]\mathbf{Z}^{N}_{\t,\cdot})]_{\x}\d\t\label{eq:smsheI}\\
&=\mathscr{T}_{N}\mathbf{S}^{N}_{\t,\x}\d\t+[\mathscr{S}^{N}\star(\sqrt{2}\lambda N^{\frac12}\mathbf{Z}^{N}_{\t,\cdot}\d\mathbf{b}_{\t,\cdot})]_{\x}+[\mathscr{S}^{N}\star(\mathfrak{e}[\tau_{\cdot}\bphi_{\t}]\mathbf{R}^{N}_{\t,\cdot}\mathbf{S}^{N}_{\t,\cdot})]_{\x}\d\t\label{eq:smsheII}
\end{align}
\end{corollary}
\begin{proof}
\eqref{eq:smsheI} follows by convolving \eqref{eq:msheI} with {\small$\mathscr{S}^{N}$} and noting that this convolution operator commutes with the (re-scaled) discrete Laplacian {\small$\mathscr{T}_{N}$}. \eqref{eq:smsheII} follows by {\small$\mathbf{Z}^{N}=\mathbf{R}^{N}\cdot\mathbf{S}^{N}$}; see \eqref{eq:zsmooth}.
\end{proof}
We now make another fairly quick observation; it gives an estimate for the product {\small$\mathfrak{f}[\tau_{\x}\bphi_{\t}]\mathbf{Z}^{N}_{\t,\x}$} in \eqref{eq:msheII} which resembles the estimate \eqref{eq:err-estimate} for the error operator in Proposition \ref{prop:s-sde}. Before we state this result, we recall the smoothing kernel {\small$\mathscr{S}^{N}$} and the object {\small$\mathbf{R}^{N}$} from Definition \ref{definition:zsmooth}.
\begin{lemma}\label{lemma:ferrestimate}
\fsp Consider the function {\small$\mathfrak{f}_{\mathrm{err}}$} from \eqref{eq:msheIIfa}-\eqref{eq:msheIIfe}. We have the identity 
\begin{align}
[\mathscr{S}^{N}\star(\mathfrak{f}_{\mathrm{err}}[\tau_{\cdot}\bphi_{\t}]\mathbf{Z}^{N}_{\t,\cdot})]_{\x}&=[\mathscr{S}^{N}\star(\mathfrak{f}_{\mathrm{err}}[\tau_{\cdot}\bphi_{\t}]\mathbf{R}^{N}_{\t,\cdot}\mathbf{S}^{N}_{\t,\cdot})]_{\x}.\label{eq:ferrestimateI}
\end{align}
Also, the operator {\small$\mathsf{f}_{\cdot}\mapsto\mathscr{S}^{N}\star(\mathfrak{f}_{\mathrm{err}}[\tau_{\cdot}\bphi_{\t}]\mathbf{R}^{N}_{\t,\cdot}\mathsf{f}_{\cdot})$}, which acts on functions {\small$\mathsf{f}:\Z\to\R$}, satisfies the following pointwise estimate (in which {\small$\delta_{\mathbf{S}}$} is a small parameter from Definition \ref{definition:zsmooth} and {\small$\mathrm{C}=\mathrm{O}(1)$}):
\begin{align}
|[\mathscr{S}^{N}\star(\mathfrak{f}_{\mathrm{err}}[\tau_{\cdot}\bphi_{\t}]\mathbf{R}^{N}_{\t,\cdot}\mathsf{f}_{\cdot})]_{\x}|\lesssim N^{-\frac12+\mathrm{C}\delta_{\mathbf{S}}}\cdot N^{-1+\delta_{\mathbf{S}}}\sum_{|\w|\lesssim N^{1-\delta_{\mathbf{S}}}}\Big(1+\sum_{|\z|\lesssim1}|\bphi_{\t,\x+\w+\z}|^{\mathrm{C}}\Big)\cdot|\mathbf{R}^{N}_{\t,\x+\w}|\cdot|\mathsf{f}_{\x+\w}|.\label{eq:ferrestimateII}
\end{align}
\end{lemma}
\begin{proof}
The identity \eqref{eq:ferrestimateI} is immediate by {\small$\mathbf{Z}^{N}=\mathbf{R}^{N}\mathbf{S}^{N}$}. We now prove \eqref{eq:ferrestimateII} by giving the general reasoning and then one concrete illustration. This estimate \eqref{eq:ferrestimateII} follows because each term on the \abbr{RHS} of \eqref{eq:msheIIfa}-\eqref{eq:msheIIfe} has a scaling factor of {\small$N^{-1/2}$} (each gradient introduces a scaling factor of {\small$N^{-1}$}, modulo powers of {\small$N^{\delta_{\mathbf{S}}}$}, after we use summation-by-parts to move the gradients to the kernel {\small$\mathscr{S}^{N}$}). Moreover, each {\small$\mathfrak{f}_{i}$} on the \abbr{RHS} of \eqref{eq:msheIIfa}-\eqref{eq:msheIIfe} admits the polynomial estimate \eqref{eq:msheIIff}, which, after we average against {\small$\mathscr{S}^{N}$}, turns into the averaged polynomial estimate on the \abbr{RHS} of \eqref{eq:ferrestimateII}.

Let us illustrate this reasoning with an example. We focus on \eqref{eq:msheIIfd}, as it has the highest number of gradients and is scaled larger than \eqref{eq:msheIIfe}. Because convolution commutes with discrete gradients, we have 
\begin{align*}
[\mathscr{S}^{N}\star(N^{\frac32}\grad^{\mathbf{X}}_{1}\grad^{\mathbf{X}}_{1}\mathfrak{f}_{4}[\tau_{\cdot}\bphi_{\t}]\mathbf{R}^{N}_{\t,\cdot}\mathsf{f}_{\cdot})]_{\x}&=[N^{2}\grad^{\mathbf{X}}_{1}\grad^{\mathbf{X}}_{1}\mathscr{S}^{N}\star(N^{-\frac12}\mathfrak{f}_{4}[\tau_{\cdot}\bphi_{\t}]\mathbf{R}^{N}_{\t,\cdot}\mathsf{f}_{\cdot})]_{\x}.
\end{align*}
Next, note that {\small$|N^{2}\grad^{\mathbf{X}}_{1}\grad^{\mathbf{X}}_{1}\mathscr{S}^{N}_{\x-\w}|\lesssim N^{2\delta_{\mathbf{S}}}N^{-1+\delta_{\mathbf{S}}}\mathbf{1}_{|\w|\lesssim N^{1-\delta_{\mathbf{S}}}}$}, since {\small$\mathscr{S}^{N}$} is smooth and supported on the length-scale {\small$N^{1-\delta_{\mathbf{S}}}$} (see Definition \ref{definition:zsmooth}). Thus, we have 
\begin{align*}
|[N^{2}\grad^{\mathbf{X}}_{1}\grad^{\mathbf{X}}_{1}\mathscr{S}^{N}\star(N^{-\frac12}\mathfrak{f}_{4}[\tau_{\cdot}\bphi_{\t}]\mathbf{R}^{N}_{\t,\cdot}\mathsf{f}_{\cdot})]_{\x}|&\lesssim N^{-\frac12+2\delta_{\mathbf{S}}}\cdot N^{-1+\delta_{\mathbf{S}}}\sum_{|\w|\lesssim N^{1-\delta_{\mathbf{S}}}}|\mathfrak{f}_{4}[\tau_{\x+\w}\bphi_{\t}]|\cdot|\mathbf{R}^{N}_{\t,\x+\w}|\cdot|\mathsf{f}_{\x+\w}|.
\end{align*}
If we now plug the estimate \eqref{eq:msheIIff} for {\small$i=4$} into the \abbr{RHS} above, we complete the proof.
\end{proof}
\subsection{Analyzing the $\mathfrak{e}$-coefficient in \eqref{eq:msheII}}
We now examine {\small$\mathfrak{e}_{\k}$} further for {\small$\k=0,1,2$}. In particular, the following result shows that the coefficients in {\small$\mathfrak{e}_{\k}$} belong to {\small$\mathrm{Jet}_{\k}^{\perp}$}. 
\begin{lemma}\label{lemma:classify}
\fsp Recall the notation in Lemma \ref{lemma:mshe} and Definition \ref{definition:jets}.
\begin{enumerate}
\item The local function {\small$\mathscr{Q}[\bphi]$} from \eqref{eq:quadratic_renorm} belongs to {\small$\mathrm{Jet}_{2}^{\perp}$}.
\item The local function {\small$\mathscr{V}'[{\boldsymbol{\phi}}_{0}]=\mathscr{U}'[\bphi_{0}]-\alpha\bphi_{0}$} belongs to {\small$\mathrm{Jet}_{1}^{\perp}$}.
\item The local function {\small$\mathscr{U}'[\bphi_{0}]\bphi_{0}-1$} belongs to {\small$\mathrm{Jet}_{1}^{\perp}$}.
\item The local function {\small$\alpha{\boldsymbol{\phi}}_{0}^{2}-1$} belongs to {\small$\mathrm{Jet}_{0}^{\perp}$}. 
\item The local function {\small$\mathscr{U}'[\bphi_{0}]$} belongs to {\small$\mathrm{Jet}_{0}^{\perp}$}
\item Any local function of the form {\small$\mathsf{F}[{\boldsymbol{\phi}}]-\E^{0}\mathsf{F}[{\boldsymbol{\phi}}]$} belongs to {\small$\mathrm{Jet}_{0}^{\perp}$}.
\item Any local function of the form {\small$\mathsf{F}[\boldsymbol{\phi}]-\mathsf{F}[\tau_{\ell}\boldsymbol{\phi}]$} for some {\small$\ell$} belongs to {\small$\mathrm{Jet}_{2}^{\perp}$}.
\item The local function {\small$\mathscr{G}[\bphi]$} from \eqref{eq:funnygradient} belongs to {\small$\mathrm{Jet}_{1}^{\perp}$}.
\item The local function {\small$\mathscr{G}[\bphi]\bphi_{1}-1$} belongs to {\small$\mathrm{Jet}_{0}^{\perp}$}.
\item The local function {\small$\mathbf{F}_{>2}[\bphi]$} belongs to {\small$\mathrm{Jet}_{2}^{\perp}$}.
\item For any {\small$\w,\w'\in\Z$}, the local function {\small$\mathbf{F}_{>2}[\bphi]\bphi_{\w}$} belongs to {\small$\mathrm{Jet}_{1}^{\perp}$}, and {\small$\mathbf{F}_{>2}[\bphi]\bphi_{\w}\bphi_{\w'}$} belongs to {\small$\mathrm{Jet}_{0}^{\perp}$}.
\end{enumerate}
\end{lemma}

\begin{proof}
We first record the following for any smooth {\small$\mathsf{F}:\R\to\R$} for which the expectations below are well-defined; it follows by integration-by-parts, and it generalizes Gaussian-integration-by-parts:
\begin{align}
\E^{\sigma}\left(\mathsf{F}[{\boldsymbol{\phi}}_{0}]\mathscr{U}'[{\boldsymbol{\phi}}_{0}]\right)&=\upsilon_{\sigma}\E^{\sigma}\mathsf{F}[{\boldsymbol{\phi}}_{0}]+\E^{\sigma}\mathsf{F}'[{\boldsymbol{\phi}}_{0}].\label{eq:classify1}
\end{align}
We now prove point (1) in Lemma \ref{lemma:classify}. We first claim that {\small$\E^{\sigma}(\mathscr{U}'[\bphi_{0}]\mathscr{U}'[\bphi_{1}])=(\E^{\sigma}\mathscr{U}'[\bphi_{0}])^{2}=\upsilon_{\sigma}^{2}$}. This follows because {\small$\bphi_{\x}$} are i.i.d. over {\small$\x\in\Z$} and by \eqref{eq:classify1} for {\small$\mathsf{F}\equiv1$}. On the other hand, we use \eqref{eq:classify1} for {\small$\mathsf{F}[{\boldsymbol{\phi}}_{0}]={\boldsymbol{\phi}}_{0}$} to get
\begin{align}
\E^{\sigma}\left(\mathscr{U}'[{\boldsymbol{\phi}}_{\x}]{\boldsymbol{\phi}}_{\x}-1\right)&=\upsilon_{\sigma}\E^{\sigma}{\boldsymbol{\phi}}_{\x}+1-1=\upsilon_{\sigma}\sigma.\label{eq:classify2}
\end{align}
If we combine the previous two displays, we deduce
\begin{align*}
\E^{\sigma}\mathscr{Q}[\bphi]&=\tfrac12\beta_{2}\upsilon_{\sigma}^{2}-\tfrac12\lambda\upsilon_{\sigma}\sigma.
\end{align*}
Since {\small$\upsilon_{\sigma}|_{\sigma=0}=0$}, we have {\small$\partial_{\sigma}^{\j}\E^{\sigma}\mathscr{Q}[\bphi]|_{\sigma=0}=0$} for {\small$\j=0,1$}. On the other hand, we know that {\small$\lambda=\beta\alpha^{-1}$}, where {\small$2\beta=\partial_{\sigma}^{2}\E^{\sigma}\mathbf{F}_{2}[\bphi]|_{\sigma=0}=\beta_{2}\partial_{\sigma}^{2}\E^{\sigma}\mathscr{U}'[\bphi_{0}]\mathscr{U}'[\bphi_{1}]|_{\sigma=0}=\beta_{2}\partial_{\sigma}^{2}\upsilon_{\sigma}^{2}|_{\sigma=0}$} by \eqref{eq:nonlinearity} and the i.i.d. property of {\small$\phi_{\x}$} variables. Moreover, since {\small$\upsilon_{0}=0$} by Assumption \ref{assump:noneq}, when we evaluate the second-$\sigma$ derivative of {\small$\upsilon_{\sigma}\sigma$} at zero, only the cross-term {\small$2\partial_{\sigma}\upsilon_{\sigma}|_{\sigma=0}\cdot\partial_{\sigma}\sigma|_{\sigma=0}=2\partial_{\sigma}\upsilon_{\sigma}|_{\sigma=0}$} contributes. In particular, we have 
\begin{align*}
\partial_{\sigma}^{2}\E^{\sigma}\mathscr{Q}[\bphi]|_{\sigma=0}&=\tfrac12\beta_{2}(\partial_{\sigma}^{2}\upsilon_{\sigma}^{2})|_{\sigma=0}-\lambda\partial_{\sigma}\upsilon_{\sigma}|_{\sigma=0}.
\end{align*}
Now, recall that {\small$\alpha=\partial_{\sigma}\E^{\sigma}\mathscr{U}'[\bphi_{0}]|_{\sigma=0}$}. In particular, by \eqref{eq:classify1} with {\small$\mathsf{F}\equiv1$}, {\color{black}the last term} in the previous display is equal to $-\lambda\alpha=-\beta$.  We argued in the previous paragraph that this is equal to {\small$-\beta_{2}(\partial_{\sigma}^{2}\upsilon_{\sigma}^{2})|_{\sigma=0}/2$}. Hence, the previous display vanishes, and point (1) follows. We now prove point (2) in Lemma \ref{lemma:classify}. If we use \eqref{eq:classify1} with $\mathsf{F}\equiv1$, and if we recall that {\small$\E^{\sigma}{\boldsymbol{\phi}}_{0}=\sigma$}, then we deduce
\begin{align*}
\E^{\sigma}\mathscr{V}'[{\boldsymbol{\phi}}_{0}]&=\upsilon_{\sigma}-\alpha\sigma.
\end{align*}
When we set {\small$\sigma=0$}, then the \abbr{RHS} vanishes. When we differentiate in $\sigma$ and set {\small$\sigma=0$}, then we get {\small$\partial_{\sigma}\upsilon_{\sigma}|_{\sigma=0}-\alpha=0$}, where the last identity follows by \eqref{eq:classify1} with {\small$\mathsf{F}\equiv1$} and {\small$\alpha=\partial_{\sigma}\E^{\sigma}\mathscr{U}'[\bphi_{0}]|_{\sigma=0}$}. This shows point (2) in Lemma \ref{lemma:classify}. We now show point (3). This follows immediately by \eqref{eq:classify2} and {\small$\upsilon_{0}=0$}. We now show point (4). Note that
\begin{align*}
1&=\partial_{\sigma}\E^{\sigma}{\boldsymbol{\phi}}_{0}={\int_{\R}}{\boldsymbol{\phi}}_{0}\partial_{\sigma}\left(\mathcal{Z}_{\upsilon_{\sigma},\mathscr{U}}^{-1}\exp\left\{-\mathscr{U}[{\boldsymbol{\phi}}_{0}]+\upsilon_{\sigma}{\boldsymbol{\phi}}_{0}\right\}\right)\d{\boldsymbol{\phi}}_{0}=-\tfrac{\partial_{\sigma}\mathcal{Z}_{\upsilon_{\sigma},\mathscr{U}}}{\mathcal{Z}_{\upsilon_{\sigma},\mathscr{U}}}\E^{\sigma}{\boldsymbol{\phi}}_{0}+\partial_{\sigma}\upsilon_{\sigma}\E^{\sigma}{\boldsymbol{\phi}}_{0}^{2}.
\end{align*}
{\color{black}If} {\small$\sigma=0$}, then the first term on the far \abbr{RHS} is equal to $0$, since {\small$\E^{\sigma}{\boldsymbol{\phi}}_{0}=\sigma$} for any $\sigma$. We also recall from earlier in this proof that {\small$\partial_{\sigma}\upsilon_{\sigma}|_{\sigma=0}=\alpha$}. Thus, we deduce that {\small$\E^{0}(\alpha{\boldsymbol{\phi}}_{0}^{2}-1)=\alpha\E^{0}{\boldsymbol{\phi}}_{0}^{2}-1=0$}, so that {\small$\alpha{\boldsymbol{\phi}}_{0}^{2}-1$} belongs to {\small$\mathrm{Jet}_{0}^{\perp}$}. This completes the proof of point (4). We now show point (5). This follows from \eqref{eq:classify1} for {\small$\mathsf{F}\equiv1$} (recall that {\small$\upsilon_{0}=0$}). Point (6) is clear by construction. Point (7) follows because {\small$\bphi_{\w}$} are exchangeable. 

We now show point (8). This follows from the observation that for any {\small$\sigma$}, we have {\small$\E^{\sigma}\mathscr{G}[\bphi]=0$} by using \eqref{eq:funnygradient} and the independence and exchange-ability of {\small$\bphi_{\w}$} under the product measure {\small$\mathbb{P}^{\sigma}$}. We now show point (9). By independence of {\small$\bphi_{\w}$} under {\small$\mathbb{P}^{0}$}, and by {\small$\E^{0}\mathscr{U}'[\bphi_{0}]=0$}, we can use \eqref{eq:funnygradient} to get
\begin{align*}
\E[\mathscr{G}[\bphi]\bphi_{1}-1]=\E\Big(\mathscr{U}'[\bphi_{2}]\bphi_{2}\mathscr{U}'[\bphi_{1}]\bphi_{1}-1\Big)=(\E[\mathscr{U}'[\bphi_{0}]\bphi_{0}])^{2}-1.
\end{align*}
We use \eqref{eq:classify1} with {\small$\mathsf{F}[\bphi_{0}]=\bphi_{0}$} and {\small$\upsilon_{0}=0$} (see Assumption \ref{assump:potential}) to get that {\color{black}the rightmost term} is zero, so (9) follows. To prove (10), note that {\small$\E^{\sigma}\mathbf{F}_{>2}[\bphi]$} is a polynomial in {\small$\E^{\sigma}\mathscr{U}'[\bphi_{0}]=\upsilon_{\sigma}$} of degree {\small$\geq3$}. However, {\small$\upsilon_{0}=0$}, so (10) follows by calculus. To prove (11), {\small$\mathbf{F}_{>2}[\bphi]\bphi_{\w}$} has a factor of the form {\small$\mathscr{U}'[\bphi_{\z}]\mathscr{U}'[\bphi_{\z'}]$} such that {\small$\mathbf{F}_{>2}[\bphi]\bphi_{\w}\cdot(\mathscr{U}'[\bphi_{\z}]\mathscr{U}'[\bphi_{\z'}])^{-1}$} is independent of {\small$\bphi_{\z},\bphi_{\z'}$}. In particular, {\small$\E^{\sigma}(\mathbf{F}_{>2}[\bphi]\bphi_{\w})$} has a factor of {\small$\upsilon_{\sigma}^{2}$}, and thus {\small$\mathbf{F}_{>2}[\bphi]\bphi_{\w}$} is in {\small$\mathrm{Jet}_{1}^{\perp}$}. By the same token, {\small$\E^{\sigma}\mathbf{F}_{>2}[\bphi]\bphi_{\w}\bphi_{\w'}$} has a factor of {\small$\upsilon_{\sigma}$}, and thus {\small$\mathbf{F}_{>2}[\bphi]\bphi_{\w}\bphi_{\w'}$} belongs to {\small$\mathrm{Jet}_{0}^{\perp}$}.   
\end{proof}
We now use Lemma \ref{lemma:classify} and re-organize the terms in \eqref{eq:msheII}-\eqref{eq:msheII(0)f} to realize them as {\small$\mathfrak{q}_{\n}$} terms in \eqref{eq:s-sde}-\eqref{eq:s-sdeIII}. In particular, we argue below that by Lemma \ref{lemma:classify}, the coefficients in {\small$\mathfrak{e}[\tau_{\x}\bphi_{\t}]$} from \eqref{eq:msheI} are admissible.
\begin{corollary}\label{corollary:qnform}
\fsp Retain the setting of Lemma \ref{lemma:mshe} and Corollary \ref{corollary:mshe}. We have the following for any index {\small$\i=0,1,2$}, which uses notation to be explained afterwards:
\begin{align}
[\mathscr{S}^{N}\star(\mathfrak{e}_{\i}[\tau_{\cdot}\bphi_{\t}]\mathbf{Z}^{N}_{\t,\cdot})]_{\x}\d\t&=N\sum_{\n_{1}=1,\ldots,\mathrm{K}_{1}}\mathrm{c}_{\n_{1}}[\mathscr{S}^{N}\star(\mathfrak{q}_{\n_{1}}[\tau_{\cdot}\bphi_{\t}]\mathbf{Z}^{N}_{\t,\cdot})]_{\x}\d\t\label{eq:qnformI1}\\
&+N\sum_{\n_{2}=\mathrm{K}_{1}+1,\ldots,\mathrm{K}_{2}}\sum_{|\mathfrak{l}_{0}|\leq1}\mathrm{c}_{\n_{2}}\cdot N\grad^{\mathbf{X}}_{\mathfrak{l}_{0}}[\mathscr{S}^{N}\star(\mathfrak{q}_{\n_{2}}[\tau_{\cdot}\bphi_{\t}]\mathbf{Z}^{N}_{\t,\cdot})]_{\x}\d\t.\label{eq:qnformI2}
\end{align}
%
\begin{itemize}
\item Above, {\small$\mathrm{K}_{1}\leq\mathrm{K}_{2}$} are uniformly bounded positive integers, and {\small$\mathrm{c}_{\n_{1}},\mathrm{c}_{\n_{2}}=\mathrm{O}(1)$} are deterministic.
\item The functions {\small$\mathfrak{q}_{\n}$} are admissible in the sense of Definition \ref{definition:admissible}.
\end{itemize}
\end{corollary}
\begin{proof}
Recall the notation \eqref{eq:msheII}-\eqref{eq:msheII(0)f}. We will show that each term appearing there has the form of one of the summands in \eqref{eq:qnformI1}-\eqref{eq:qnformI2} (after we take convolution with {\small$\mathscr{S}^{N}$}). Since there are {\small$\mathrm{O}(1)$}-many terms in \eqref{eq:msheII}-\eqref{eq:msheII(0)f}, this would complete the proof.
\begin{enumerate}
\item Consider the function {\small$\mathfrak{q}[\bphi]:=\lambda \mathscr{Q}[\tau_{1}\bphi]$}, where {\small$\mathscr{Q}$} is defined in \eqref{eq:quadratic_renorm}. By point (1) in Lemma \ref{lemma:classify}, the function {\small$\bphi\mapsto\mathscr{Q}[\tau_{1}\bphi]$} is in {\small$\mathrm{Jet}_{2}^{\perp}$}; because {\small$\E^{\sigma}$} measures are translation-invariant (in space), the shift {\small$\tau_{1}$} does not change this. Moreover, by inspection in \eqref{eq:quadratic_renorm}, the function {\small$\bphi\mapsto\mathscr{Q}[\tau_{1}\bphi]$} depends only on {\small$\bphi_{\w}$} for {\small$\w$} in a uniformly bounded interval to the right of {\small$0\in\Z$}. Thus, {\small$\mathfrak{q}[\bphi]:=\lambda \mathscr{Q}[\tau_{1}\bphi]$} is admissible. The same reasoning also holds for the other local functions in \eqref{eq:msheII(2)a}-\eqref{eq:msheII(2)b} (except that we also use (7) and (10) in Lemma \ref{lemma:classify}). So, we know that {\small$[\mathscr{S}^{N}\star(\mathfrak{e}_{2}[\tau_{\cdot}\bphi_{\t}]\mathbf{Z}^{N}_{\t,\cdot})]_{\x}\d\t$} (see \eqref{eq:msheII(2)a}-\eqref{eq:msheII(2)b}) has the form of the \abbr{RHS} of \eqref{eq:qnformI1}.
\item Take the function {\small$\mathfrak{q}[\bphi]:=N^{-1/2}\mathscr{V}'[\bphi_{0}]$}. This depends on {\small$\bphi_{\w}$} for {\small$\w\in\{-\mathfrak{l},\ldots,0\}$} for some {\small$\mathfrak{l}=\mathrm{O}(1)$}. Moreover, by point (2) in Lemma \ref{lemma:classify}, this function is in {\small$\mathrm{Jet}_{-1}^{\perp}$} and thus admissible. Now, because convolution commutes with the discrete gradient, we have
\begin{align*}
[\mathscr{S}^{N}\star\{N^{\frac32}\grad^{\mathbf{X}}_{1}(\mathscr{V}'[\bphi_{\t,\cdot}]\mathbf{Z}^{N}_{\t,\cdot})\}]_{\x}&=N\cdot N\grad^{\mathbf{X}}_{1}[\mathscr{S}^{N}\star\{N^{-\frac12}\mathscr{V}'[\tau_{\cdot}\bphi_{\t}]\mathbf{Z}^{N}_{\t,\cdot}\}]_{\x}\d\t.
\end{align*}
In particular, the \abbr{LHS} above has the form of \eqref{eq:qnformI2}. The same reasoning holds for the other local functions in \eqref{eq:msheII(1)a}-\eqref{eq:msheII(1)c}) (except we use (8) and (11) in Lemma \ref{lemma:classify}). Ultimately, we deduce that  {\small$[\mathscr{S}^{N}\star(\mathfrak{e}_{1}[\tau_{\cdot}\bphi_{\t}]\mathbf{Z}^{N}_{\t,\cdot})]_{\x}\d\t$} (see \eqref{eq:msheII(1)a}-\eqref{eq:msheII(1)c}) has the form on the \abbr{RHS} of \eqref{eq:qnformI1}-\eqref{eq:qnformI2}.
\item Take \eqref{eq:msheII(0)a}-\eqref{eq:msheII(0)f}; by Lemma \ref{lemma:classify}, all of the local functions therein that are being multiplied by {\small$\mathbf{Z}^{N}$} belong to {\small$\mathrm{Jet}_{0}^{\perp}$}. This is a lower-degree {\small$\mathrm{Jet}_{\k}^{\perp}$}-space than in point (2) above by one degree. But, we now have an extra factor of {\small$N^{-1/2}$} compared to point (2). Thus, the same reasoning in the previous bullet point also shows that {\small$[\mathscr{S}^{N}\star(\mathfrak{e}_{0}[\tau_{\cdot}\bphi_{\t}]\mathbf{Z}^{N}_{\t,\cdot})]_{\x}\d\t$} has the form of the \abbr{RHS} of \eqref{eq:qnformI1}-\eqref{eq:qnformI2}.
\end{enumerate}
This completes the proof.
\end{proof}
\subsection{Introducing space-time averages}
The goal is to now replace {\small$\mathfrak{q}_{\n_{1}},\mathfrak{q}_{\n_{2}}$} by their respective space-time averages. Recall the notation from Definition \ref{definition:eq-operators} and \eqref{eq:zsmooth}.
\begin{lemma}\label{lemma:spacetimeaverage}
\fsp Fix any {\small$\mathfrak{q}_{\n_{1}}$} on the \abbr{RHS} of \eqref{eq:qnformI1}. We have the following (with notation to be explained afterwards):
\begin{align}
[\mathscr{S}^{N}\star(\mathfrak{q}_{\n_{1}}[\tau_{\cdot}\bphi_{\t}]\mathbf{Z}^{N}_{\t,\cdot})]_{\x}&=\sum_{\substack{\m=0,\ldots,\mathrm{M}\\0<|\mathfrak{l}_{1}|,\ldots,|\mathfrak{l}_{\m}|\leq\mathfrak{n}_{\mathbf{Av}}}}\tfrac{1}{|\mathfrak{l}_{1}|\ldots|\mathfrak{l}_{\m}|}\mathrm{b}_{\m,\mathfrak{l}_{1},\ldots,\mathfrak{l}_{\m}}\grad^{\mathbf{X}}_{\mathfrak{l}_{1}}\ldots\grad^{\mathbf{X}}_{\mathfrak{l}_{\m}}[\mathscr{S}^{N}\star(\mathbf{Av}^{\mathbf{T},\mathbf{X},\mathfrak{q}_{\n_{1}}}_{\t,\cdot}\cdot\mathbf{R}^{N}_{\t,\cdot}\mathbf{S}^{N}_{\t,\cdot})]_{\x}\label{eq:spacetimeaverageIa}\\
&+\sum_{\substack{\m=0,\ldots,\mathrm{M}\\0<|\mathfrak{l}_{1}|,\ldots,|\mathfrak{l}_{\m}|\leq\mathfrak{n}_{\mathbf{Av}}}}\tfrac{1}{|\mathfrak{l}_{1}|\ldots|\mathfrak{l}_{\m}|}\mathrm{b}_{\m,\mathfrak{l}_{1},\ldots,\mathfrak{l}_{\m}}\grad^{\mathbf{T},\mathrm{av}}_{\mathfrak{t}_{\mathbf{Av}}}\grad^{\mathbf{X}}_{\mathfrak{l}_{1}}\ldots\grad^{\mathbf{X}}_{\mathfrak{l}_{\m}}[\mathscr{S}^{N}\star(\mathbf{Av}^{\mathbf{X},\mathfrak{q}_{\n_{1}}}_{\t,\cdot}\cdot\mathbf{R}^{N}_{\t,\cdot}\mathbf{S}^{N}_{\t,\cdot})]_{\x}\label{eq:spacetimeaverageIb}\\
&+\mathrm{Err}_{1}[\mathbf{R}^{N}_{\t,\cdot}\mathbf{S}^{N}_{\t,\cdot}]_{\x}.\label{eq:spacetimeaverageIc}
\end{align}
%
\begin{itemize}
\item The parameter {\small$\mathrm{M}=\mathrm{O}(1)$} is a positive integer.
\item The coefficients {\small$\mathrm{b}_{\m,\mathfrak{l}_{1},\ldots,\mathfrak{l}_{\m}}$} are deterministic and {\small$\mathrm{O}(1)$}.
\item The term {\small$\mathrm{Err}_{1}[\mathbf{R}^{N}\mathbf{S}^{N}]_{\t,\x}$} is a linear operator in {\small$\mathbf{S}^{N}$} evaluated at {\small$(\t,\x)$}, and for {\small$\mathrm{C}=\mathrm{O}(1)$} is large, we have
\begin{align}
|\mathrm{Err}_{1}[\mathbf{R}^{N}_{\t,\cdot}\mathbf{S}^{N}_{\t,\cdot}]_{\x}|\lesssim N^{-\frac32}\cdot N^{-1+\delta_{\mathbf{S}}}\sum_{|\w|\lesssim N^{1-\delta_{\mathbf{S}}}}\Big(1+\sum_{|\z|\lesssim1}|\bphi_{\t,\x+\w+\z}|^{\mathrm{C}}\Big)\cdot|\mathbf{R}^{N}_{\t,\x+\w}|\cdot|\mathbf{S}^{N}_{\t,\x+\w}|.\label{eq:spacetimeaverageII}
\end{align}
\end{itemize}
\end{lemma}
\begin{proof}
Let us assume that {\small$\mathfrak{q}_{\n_{1}}[\bphi]$} depends only on {\small$\bphi_{\w}$} for {\small$\w\in\{1,\ldots,\mathfrak{l}\}$} for some {\small$\mathfrak{l}=\mathrm{O}(1)$}. The argument in the case where {\small$\w\in\{-\mathfrak{l},\ldots,0\}$} is analogous. We start with the following elementary computation:
\begin{align*}
\mathfrak{q}_{\n_{1}}[\tau_{\cdot}\bphi_{\t}]\mathbf{Z}^{N}_{\t,\cdot}&=\mathfrak{n}_{\mathbf{Av}}^{-1}\sum_{\j=1,\ldots,\mathfrak{n}_{\mathbf{Av}}}\mathfrak{q}_{\n_{1}}[\tau_{\cdot+\j}\bphi_{\t}]\mathbf{Z}^{N}_{\t,\cdot+\j}-\mathfrak{n}_{\mathbf{Av}}^{-1}\sum_{\j=1,\ldots,\mathfrak{n}_{\mathbf{Av}}}\grad^{\mathbf{X}}_{\j}\mathfrak{q}_{\n_{1}}[\tau_{\cdot}\bphi_{\t}]\mathbf{Z}^{N}_{\t,\cdot},
\end{align*}
where the gradient acts on the implicit space-variable. When we convolve this against {\small$\mathscr{S}^{N}$}, we get
\begin{align*}
\mathscr{S}^{N}\star(\mathfrak{q}_{\n_{1}}[\tau_{\cdot}\bphi_{\t}]\mathbf{Z}^{N}_{\t,\cdot})&=\mathscr{S}^{N}\star\Big(\mathfrak{n}_{\mathbf{Av}}^{-1}\sum_{\j=1,\ldots,\mathfrak{n}_{\mathbf{Av}}}\mathfrak{q}_{\n_{1}}[\tau_{\cdot+\j}\bphi_{\t}]\mathbf{Z}^{N}_{\t,\cdot+\j}\Big)-\mathscr{S}^{N}\star\Big(\mathfrak{n}_{\mathbf{Av}}^{-1}\sum_{\j=1,\ldots,\mathfrak{n}_{\mathbf{Av}}}\grad^{\mathbf{X}}_{\j}\mathfrak{q}_{\n_{1}}[\tau_{\cdot}\bphi_{\t}]\mathbf{Z}^{N}_{\t,\cdot}\Big)\\
&=\mathscr{S}^{N}\star\Big(\mathfrak{n}_{\mathbf{Av}}^{-1}\sum_{\j=1,\ldots,\mathfrak{n}_{\mathbf{Av}}}\mathfrak{q}_{\n_{1}}[\tau_{\cdot+\j}\bphi_{\t}]\mathbf{Z}^{N}_{\t,\cdot+\j}\Big)-\mathfrak{n}_{\mathbf{Av}}^{-1}\sum_{\j=1,\ldots,\mathfrak{n}_{\mathbf{Av}}}\grad^{\mathbf{X}}_{\j}\mathscr{S}^{N}\star(\mathfrak{q}_{\n_{1}}[\tau_{\cdot}\bphi_{\t}]\mathbf{Z}^{N}_{\t,\cdot}),
\end{align*}
where the second identity follows because convolution commutes with discrete gradients. Note that the last term in the previous display is an average of discrete gradients acting on the far \abbr{LHS}. Thus, we can iterate this identity {\small$\mathrm{M}$}-many times and obtain the following (which we explain immediately afterwards):
\begin{align}
\mathscr{S}^{N}\star(\mathfrak{q}_{\n_{1}}[\tau_{\cdot}\bphi_{\t}]\mathbf{Z}^{N}_{\t,\cdot})&=\sum_{\substack{\m=0,\ldots,\mathrm{M}\\0<|\mathfrak{l}_{1}|,\ldots,|\mathfrak{l}_{\m}|\leq\mathfrak{n}_{\mathbf{Av}}}}\tfrac{1}{\mathfrak{n}_{\mathbf{Av}}^{\m}}\mathrm{a}_{\m,\mathfrak{l}_{1},\ldots,\mathfrak{l}_{\m}}(\grad^{\mathbf{X}}_{\mathfrak{l}_{1}}\ldots\grad^{\mathbf{X}}_{\mathfrak{l}_{\m}})\mathscr{S}^{N}\star\Big(\mathfrak{n}_{\mathbf{Av}}^{-1}\sum_{\j=1,\ldots,\mathfrak{n}_{\mathbf{Av}}}\mathfrak{q}_{\n_{1}}[\tau_{\cdot+\j}\bphi_{\t}]\mathbf{Z}^{N}_{\t,\cdot+\j}\Big)\label{eq:spacetimeaverage1a}\\
&+\sum_{0<|\mathfrak{l}_{1}|,\ldots,|\mathfrak{l}_{\mathrm{M}+1}|\leq\mathfrak{n}_{\mathbf{Av}}}\tfrac{1}{\mathfrak{n}_{\mathbf{Av}}^{\mathrm{M}+1}}\mathrm{a}_{\mathrm{M}+1,\mathfrak{l}_{1},\ldots,\mathfrak{l}_{\mathrm{M}+1}}(\grad^{\mathbf{X}}_{\mathfrak{l}_{1}}\ldots\grad^{\mathbf{X}}_{\mathfrak{l}_{\mathrm{M}+1}})\mathscr{S}^{N}\star(\mathfrak{q}_{\n_{1}}[\tau_{\cdot}\bphi_{\t}]\mathbf{Z}^{N}_{\t,\cdot}).\label{eq:spacetimeaverage1b}
\end{align}
Above, the coefficients {\small$\mathrm{a}_{\m,\mathfrak{l}_{1},\ldots,\mathfrak{l}_{\m}}$} are deterministic and {\small$\mathrm{O}(1)$} for {\small$\mathrm{m}=0,\ldots,\mathrm{M}+1$}. Also, the parameter {\small$\mathrm{M}$} is any positive integer; we will choose it to be large but {\small$\mathrm{O}(1)$} (to be specified shortly). We now control \eqref{eq:spacetimeaverage1b}. Recall the notation from \eqref{eq:zsmooth}; this lets write
\begin{align}
\eqref{eq:spacetimeaverage1b}&=\sum_{0<|\mathfrak{l}_{1}|,\ldots,|\mathfrak{l}_{\mathrm{M}+1}|\leq\mathfrak{n}_{\mathbf{Av}}}\tfrac{1}{\mathfrak{n}_{\mathbf{Av}}^{\mathrm{M}+1}}\mathrm{a}_{\mathrm{M}+1,\mathfrak{l}_{1},\ldots,\mathfrak{l}_{\mathrm{M}+1}}(\grad^{\mathbf{X}}_{\mathfrak{l}_{1}}\ldots\grad^{\mathbf{X}}_{\mathfrak{l}_{\mathrm{M}+1}})\mathscr{S}^{N}\star(\mathfrak{q}_{\n_{1}}[\tau_{\cdot}\bphi_{\t}]\mathbf{R}^{N}_{\t,\cdot}\mathbf{S}^{N}_{\t,\cdot}).\nonumber
\end{align}
We note that the \abbr{RHS} of the previous display is a linear operator in the input function {\small$\mathbf{S}^{N}_{\t,\cdot}$}. Now, note that each gradient has length {\small$\mathrm{O}(\mathfrak{n}_{\mathbf{Av}})$}. Moreover, {\small$\mathscr{S}^{N}$} is a smooth kernel at scale {\small$N^{-1+\delta_{\mathbf{S}}}$} (with support on this scale as well); see Definition \ref{definition:zsmooth}. In particular, each gradient introduces a factor of {\small$\lesssim|\mathfrak{n}_{\mathbf{Av}}|N^{-1+\delta_{\mathbf{S}}}$} when it acts on {\small$\mathscr{S}^{N}$}. Moreover, if we recall that {\small$|\mathfrak{n}_{\mathbf{Av}}|\lesssim N^{1-3\delta_{\mathbf{S}}/2}$} (see Definition \ref{definition:eq-operators}), then we have {\small$|\mathfrak{n}_{\mathbf{Av}}|N^{-1+\delta_{\mathbf{S}}}\lesssim N^{-\delta}$} for some {\small$\delta>0$}. Therefore, if we take {\small$\mathrm{M}$} large enough, then we have the pointwise gradient estimate 
\begin{align*}
|\grad^{\mathbf{X}}_{\mathfrak{l}_{1}}\ldots\grad^{\mathbf{X}}_{\mathfrak{l}_{\mathrm{M}+1}}\mathscr{S}^{N}_{\x-\w}|\lesssim |\mathfrak{n}_{\mathbf{Av}}|^{\mathrm{M}+1}(N^{-1+\delta_{\mathbf{S}}})^{\mathrm{M}+1}\cdot N^{-1+\delta_{\mathbf{S}}}\mathbf{1}_{|\x-\w|\lesssim N^{1-\delta_{\mathbf{S}}}}\lesssim N^{-\frac32}\cdot N^{-1+\delta_{\mathbf{S}}}\mathbf{1}_{|\x-\w|\lesssim N^{1-\delta_{\mathbf{S}}}}.
\end{align*}
This estimate is uniform in {\small$\mathfrak{l}_{1},\ldots,\mathfrak{l}_{\mathrm{M}+1}$} parameters in \eqref{eq:spacetimeaverage1b}. Thus, after we average over these parameters and multiply by {\small$\mathrm{a}_{\mathrm{M}+1,\mathfrak{l}_{1},\ldots,\mathfrak{l}_{\mathrm{M}+1}}=\mathrm{O}(1)$}, we deduce the following (similar to the proof of Lemma \ref{lemma:ferrestimate}), \emph{whose proof does not use anything about {\small$\mathbf{S}^{N}$}}:
\begin{align}
|\eqref{eq:spacetimeaverage1b}_{\x}|&\lesssim N^{-\frac32}\cdot N^{-1+\delta_{\mathbf{S}}}\sum_{|\w|\lesssim N^{1-\delta_{\mathbf{S}}}}|\mathfrak{q}_{\n_{1}}[\tau_{\x+\w}\bphi_{\t}]|\cdot|\mathbf{R}^{N}_{\t,\x+\w}|\cdot|\mathbf{S}^{N}_{\t,\x+\w}|\nonumber\\
&\lesssim N^{-\frac32}\cdot N^{-1+\delta_{\mathbf{S}}}\sum_{|\w|\lesssim N^{1-\delta_{\mathbf{S}}}}\Big(1+\sum_{|\z|\lesssim1}|\bphi_{\t,\x+\w+\z}|^{\mathrm{C}}\Big)\cdot|\mathbf{S}^{N}_{\t,\x+\w}|.\label{eq:spacetimeaverage1bestimate}
\end{align}
(The second estimate holds since {\small$\mathfrak{q}_{\n_{1}}[\bphi]$} has a local polynomial estimate; see Definition \ref{definition:admissible}.) We now return to the \abbr{RHS} of \eqref{eq:spacetimeaverage1a}. For each summand, we use the time-gradient \eqref{eq:timegrad} to get the following:
\begin{align*}
\mathscr{S}^{N}\star\Big(\mathfrak{n}_{\mathbf{Av}}^{-1}\sum_{\j=1,\ldots,\mathfrak{n}_{\mathbf{Av}}}\mathfrak{q}_{\n_{1}}[\tau_{\cdot+\j}\bphi_{\t}]\mathbf{Z}^{N}_{\t,\cdot+\j}\Big)&=\tfrac{1}{\mathfrak{t}_{\mathbf{Av}}}\int_{0}^{\mathfrak{t}_{\mathbf{Av}}}\mathscr{S}^{N}\star\Big(\mathfrak{n}_{\mathbf{Av}}^{-1}\sum_{\j=1,\ldots,\mathfrak{n}_{\mathbf{Av}}}\mathfrak{q}_{\n_{1}}[\tau_{\cdot+\j}\bphi_{\t-\r}]\mathbf{Z}^{N}_{\t-\r,\cdot+\j}\Big)\d\r\\
&+\grad^{\mathbf{T},\mathrm{av}}_{\mathfrak{t}_{\mathbf{Av}}}\mathscr{S}^{N}\star\Big(\mathfrak{n}_{\mathbf{Av}}^{-1}\sum_{\j=1,\ldots,\mathfrak{n}_{\mathbf{Av}}}\mathfrak{q}_{\n_{1}}[\tau_{\cdot+\j}\bphi_{\t}]\mathbf{Z}^{N}_{\t,\cdot+\j}\Big).
\end{align*}
Now, for the first term on the \abbr{RHS}, we may move the time-integration into the convolution operator (by linearity of integration). Then, we write {\small$\mathbf{Z}^{N}_{\t-\r,\cdot+\j}=\mathbf{Z}^{N}_{\t-\r,\cdot+\j}(\mathbf{Z}^{N}_{\t,\cdot})^{-1}\cdot\mathbf{Z}^{N}_{\t,\cdot}$} and {\small$\mathbf{Z}^{N}=\mathbf{R}^{N}\mathbf{S}^{N}$}. The result is
\begin{align*}
&\tfrac{1}{\mathfrak{t}_{\mathbf{Av}}}\int_{0}^{\mathfrak{t}_{\mathbf{Av}}}\Big[\mathscr{S}^{N}\star\Big(\mathfrak{n}_{\mathbf{Av}}^{-1}\sum_{\j=1,\ldots,\mathfrak{n}_{\mathbf{Av}}}\mathfrak{q}_{\n_{1}}[\tau_{\cdot+\j}\bphi_{\t-\r}]\mathbf{Z}^{N}_{\t-\r,\cdot+\j}\Big)\Big]_{\x}\d\r\\
&=[\mathscr{S}^{N}\star(\mathbf{Av}^{\mathbf{T},\mathbf{X},\mathfrak{q}_{\n_{1}}}_{\t,\cdot}\mathbf{Z}^{N}_{\t,\cdot})]_{\x}=[\mathscr{S}^{N}\star(\mathbf{Av}^{\mathbf{T},\mathbf{X},\mathfrak{q}_{\n_{1}}}_{\t,\cdot}\cdot\mathbf{R}^{N}_{\t,\cdot}\mathbf{S}^{N}_{\t,\cdot})]_{\x}.
\end{align*}
(See Definition \ref{definition:eq-operators} for the relevant notation in the above display.) By the same token, we also have 
\begin{align*}
\Big[\mathscr{S}^{N}\star\Big(\mathfrak{n}_{\mathbf{Av}}^{-1}\sum_{\j=1,\ldots,\mathfrak{n}_{\mathbf{Av}}}\mathfrak{q}_{\n_{1}}[\tau_{\cdot+\j}\bphi_{\t}]\mathbf{Z}^{N}_{\t,\cdot+\j}\Big)\Big]_{\x}&=\Big[\mathscr{S}^{N}\star\Big(\mathfrak{n}_{\mathbf{Av}}^{-1}\sum_{\j=1,\ldots,\mathfrak{n}_{\mathbf{Av}}}\mathfrak{q}_{\n_{1}}[\tau_{\cdot+\j}\bphi_{\t}]\mathbf{Z}^{N}_{\t,\cdot+\j}(\mathbf{Z}^{N}_{\t,\cdot})^{-1}\cdot\mathbf{Z}^{N}_{\t,\cdot}\Big)\Big]_{\x}\\
&=[\mathscr{S}^{N}\star(\mathbf{Av}^{\mathbf{X},\mathfrak{q}_{\n_{1}}}_{\t,\cdot}\mathbf{Z}^{N}_{\t,\cdot})]_{\x}=[\mathscr{S}^{N}\star(\mathbf{Av}^{\mathbf{X},\mathfrak{q}_{\n_{1}}}_{\t,\cdot}\mathbf{R}^{N}_{\t,\cdot}\mathbf{S}^{N}_{\t,\cdot})]_{\x}.
\end{align*}
By combining the previous three displays and plugging the result into the \abbr{RHS} of \eqref{eq:spacetimeaverage1a}, we deduce that 
\begin{align*}
\mathsf{RHS}\eqref{eq:spacetimeaverage1a}_{\x}&=\sum_{\substack{\m=0,\ldots,\mathrm{M}\\0<|\mathfrak{l}_{1}|,\ldots,|\mathfrak{l}_{\m}|\leq\mathfrak{n}_{\mathbf{Av}}}}\tfrac{1}{\mathfrak{n}_{\mathbf{Av}}^{\m}}\mathrm{a}_{\m,\mathfrak{l}_{1},\ldots,\mathfrak{l}_{\m}}\grad^{\mathbf{X}}_{\mathfrak{l}_{1}}\ldots\grad^{\mathbf{X}}_{\mathfrak{l}_{\m}}[\mathscr{S}^{N}\star(\mathbf{Av}^{\mathbf{T},\mathbf{X},\mathfrak{q}_{\n_{1}}}_{\t,\cdot}\cdot\mathbf{R}^{N}_{\t,\cdot}\mathbf{S}^{N}_{\t,\cdot})]_{\x}\\
&+\sum_{\substack{\m=0,\ldots,\mathrm{M}\\0<|\mathfrak{l}_{1}|,\ldots,|\mathfrak{l}_{\m}|\leq\mathfrak{n}_{\mathbf{Av}}}}\tfrac{1}{\mathfrak{n}_{\mathbf{Av}}^{\m}}\mathrm{a}_{\m,\mathfrak{l}_{1},\ldots,\mathfrak{l}_{\m}}\grad^{\mathbf{T},\mathrm{av}}_{\mathfrak{t}_{\mathbf{Av}}}\grad^{\mathbf{X}}_{\mathfrak{l}_{1}}\ldots\grad^{\mathbf{X}}_{\mathfrak{l}_{\m}}[\mathscr{S}^{N}\star(\mathbf{Av}^{\mathbf{X},\mathfrak{q}_{\n_{1}}}_{\t,\cdot}\cdot\mathbf{R}^{N}_{\t,\cdot}\mathbf{S}^{N}_{\t,\cdot})]_{\x}
\end{align*}
Since {\small$|\mathfrak{l}_{1}|,\ldots,|\mathfrak{l}_{\m}|\leq\mathfrak{n}_{\mathbf{Av}}$}, we can rewrite {\small$\mathfrak{n}_{\mathbf{Av}}^{-\m}\mathrm{a}_{\m,\mathfrak{l}_{1},\ldots,\mathfrak{l}_{\m}}=|\mathfrak{l}_{1}|^{-1}\ldots|\mathfrak{l}_{\m}|^{-1}\mathrm{a}'_{\m,\mathfrak{l}_{1},\ldots,\mathfrak{l}_{\m}}$} for deterministic and {\small$\mathrm{O}(1)$} coefficients {\small$\mathrm{a}'_{\m,\mathfrak{l}_{1},\ldots,\mathfrak{l}_{\m}}$}. Thus, the result follows by combining \eqref{eq:spacetimeaverage1a}-\eqref{eq:spacetimeaverage1b}, \eqref{eq:spacetimeaverage1bestimate}, and the previous display. 
\end{proof}
By the same argument (or applying {\small$N\grad^{\mathbf{X}}_{\mathfrak{l}_{0}}$} to \eqref{eq:spacetimeaverageIa}-\eqref{eq:spacetimeaverageIc}), we get a similar representation for the summands in \eqref{eq:qnformI2}. The only difference is that the summands in \eqref{eq:qnformI2} have an extra {\small$N\grad^{\mathbf{X}}_{\mathfrak{l}_{0}}$} for {\small$|\mathfrak{l}_{0}|\leq1$}. {\color{black}This} changes the structure of the coefficients appearing in \eqref{eq:spacetimeaverageagainIa}-\eqref{eq:spacetimeaverageagainIc} as explained below. (In particular, the coefficients are now {\small$\mathrm{O}(N)$} instead of {\small$\mathrm{O}(1)$}, but we are now allowed to restrict to terms where one gradient has length-scale {\small$|\mathfrak{l}|\leq1$}. Compare this to the condition on the coefficients {\small$\mathrm{c}_{N,\mathfrak{l}_{1},\ldots,\mathfrak{l}_{\m}}$} in Proposition \ref{prop:s-sde}.)
\begin{lemma}\label{lemma:spacetimeaverageagain}
\fsp Fix any {\small$\mathfrak{q}_{\n_{2}}$} in \eqref{eq:qnformI2}. We have the following (with notation to be explained afterwards):
\begin{align}
N\grad^{\mathbf{X}}_{\mathfrak{l}_{0}}[\mathscr{S}^{N}\star\mathfrak{q}_{\n_{2}}[\tau_{\cdot}\bphi_{\t}]\mathbf{Z}^{N}_{\t,\cdot}]_{\x}&=\sum_{\substack{\m=0,\ldots,\mathrm{M}\\0<|\mathfrak{l}_{1}|,\ldots,|\mathfrak{l}_{\m}|\leq\mathfrak{n}_{\mathbf{Av}}}}\tfrac{1}{|\mathfrak{l}_{1}|\ldots|\mathfrak{l}_{\m}|}\mathrm{b}'_{N,\m,\mathfrak{l}_{1},\ldots,\mathfrak{l}_{\m}}\grad^{\mathbf{X}}_{\mathfrak{l}_{1}}\ldots\grad^{\mathbf{X}}_{\mathfrak{l}_{\m}}[\mathscr{S}^{N}\star(\mathbf{Av}^{\mathbf{T},\mathbf{X},\mathfrak{q}_{\n_{2}}}_{\t,\cdot}\cdot\mathbf{R}^{N}_{\t,\cdot}\mathbf{S}^{N}_{\t,\cdot})]_{\x}\label{eq:spacetimeaverageagainIa}\\
&+\sum_{\substack{\m=0,\ldots,\mathrm{M}\\0<|\mathfrak{l}_{1}|,\ldots,|\mathfrak{l}_{\m}|\leq\mathfrak{n}_{\mathbf{Av}}}}\tfrac{1}{|\mathfrak{l}_{1}|\ldots|\mathfrak{l}_{\m}|}\mathrm{b}'_{N,\m,\mathfrak{l}_{1},\ldots,\mathfrak{l}_{\m}}\grad^{\mathbf{T},\mathrm{av}}_{\mathfrak{t}_{\mathbf{Av}}}\grad^{\mathbf{X}}_{\mathfrak{l}_{1}}\ldots\grad^{\mathbf{X}}_{\mathfrak{l}_{\m}}[\mathscr{S}^{N}\star(\mathbf{Av}^{\mathbf{T},\mathbf{X},\mathfrak{q}_{\n_{2}}}_{\t,\cdot}\cdot\mathbf{R}^{N}_{\t,\cdot}\mathbf{S}^{N}_{\t,\cdot})]_{\x}\label{eq:spacetimeaverageagainIb}\\
&+\mathrm{Err}_{2}[\mathbf{R}^{N}_{\t,\cdot}\mathbf{S}^{N}_{\t,\cdot}]_{\x}.\label{eq:spacetimeaverageagainIc}
\end{align}
%
\begin{itemize}
\item The parameter {\small$\mathrm{M}=\mathrm{O}(1)$} is a positive integer.
\item The {\small$\mathrm{b}'_{N,\m,\mathfrak{l}_{1},\ldots,\mathfrak{l}_{\m}}$} are deterministic and {\small$\mathrm{O}(N)$}. Moreover, {\small$\mathrm{b}'_{N,\m,\mathfrak{l}_{1},\ldots,\mathfrak{l}_{\m}}=0$} if {\small$\min\{|\mathfrak{l}_{1}|,\ldots,|\mathfrak{l}_{\m}|\}\geq2$}.
\item The term {\small$\mathrm{Err}_{2}[\mathbf{R}^{N}\mathbf{S}^{N}]_{\t,\x}$} is a linear operator in {\small$\mathbf{S}^{N}$} evaluated at {\small$(\t,\x)$}, and it satisfies the deterministic estimate below as long as {\small$\mathrm{C}=\mathrm{O}(1)$} is sufficiently large:
\begin{align}
|\mathrm{Err}_{2}[\mathbf{R}^{N}_{\t,\cdot}\mathbf{S}^{N}_{\t,\cdot}]_{\x}|\lesssim N^{-\frac32}\cdot N^{-1+\delta_{\mathbf{S}}}\sum_{|\w|\lesssim N^{1-\delta_{\mathbf{S}}}}\Big(1+\sum_{|\z|\lesssim1}|\bphi_{\t,\x+\w+\z}|^{\mathrm{C}}\Big)\cdot|\mathbf{R}^{N}_{\t,\x+\w}|\cdot|\mathbf{S}^{N}_{\t,\x+\w}|.\label{eq:spacetimeaverageagainII}
\end{align}
\end{itemize}
\end{lemma}
\subsection{Proof of Proposition \ref{prop:s-sde}}
Use Corollary \ref{corollary:mshe}, Lemma \ref{lemma:ferrestimate}, Corollary \ref{corollary:qnform}, Lemma \ref{lemma:spacetimeaverage}, and Lemma \ref{lemma:spacetimeaverageagain}. \qed
%
%
%
\section{Stochastic estimates}\label{section:technical}
The goal of this section is to gather the key estimates for stochastic objects (such as the {\small$\mathbf{Av}$} objects appearing in \eqref{eq:s-sde}-\eqref{eq:s-sdeIII}, the {\small$\t\mapsto\bphi_{\t}$} process, and the {\small$\mathbf{R}^{N}$} and {\small$\mathbf{R}^{N,\wedge}$} quantities from \eqref{eq:zsmooth} and \eqref{eq:rwedge}, respectively). We first give estimates for {\small$\t\mapsto\bphi_{\t}$} and {\small$\mathbf{R}^{N}$} and {\small$\mathbf{R}^{N,\wedge}$}. Then, we give bounds for {\color{black}the} {\small$\mathbf{Av}^{\mathbf{X}}$} terms. We conclude with estimates for {\color{black}the} {\small$\mathbf{Av}^{\mathbf{X},\mathbf{T}}$} terms, as the proof of these are longer (though these estimates are the most interesting and difficult).
\subsection{Estimates for \eqref{eq:phi} and \eqref{eq:rwedge}}
We start with a result which shows that {\small$\mathbf{R}^{N}=\mathbf{R}^{N,\wedge}$} with high probability on ``large" space-time sets. First, recall {\small$\mathbf{R}^{N}$} and {\small$\mathbf{R}^{N,\wedge}$} from \eqref{eq:zsmooth} and \eqref{eq:rwedge}, and recall {\small$\delta_{\mathbf{S}}$} from Definition \ref{definition:zsmooth}.
\begin{lemma}\label{lemma:rbound}
\fsp Fix any {\small$\mathrm{D}=\mathrm{O}(1)$} and {\small$\T>0$} independent of {\small$N$}. With high probability, we have 
\begin{align}
\sup_{\t\in[0,\T]}\sup_{|\x|\leq N^{\mathrm{D}}}|\mathbf{R}^{N}_{\t,\x}-\mathbf{R}^{N,\wedge}_{\t,\x}|=0.\label{eq:rboundI}
\end{align}
\end{lemma}
\begin{proof}
By \eqref{eq:rwedge} and \eqref{eq:zsmooth}, we have 
\begin{align}
|\mathbf{R}^{N}_{\t,\x}-\mathbf{R}^{N,\wedge}_{\t,\x}|&=|\mathbf{R}^{N}_{\t,\x}|\cdot\mathbf{1}\Big[|\tfrac{\mathbf{Z}^{N}_{\t,\x}}{\mathbf{S}^{N}_{\t,\x}}-1|> N^{-\frac13\delta_{\mathbf{S}}}\Big].\nonumber
\end{align}
In particular, it suffices to show that with high probability, we have 
\begin{align}
\sup_{\t\in[0,\T]}\sup_{|\x|\leq N^{\mathrm{D}}}|\tfrac{\mathbf{Z}^{N}_{\t,\x}}{\mathbf{S}^{N}_{\t,\x}}-1|\leq N^{-\frac13\delta_{\mathbf{S}}}.\label{eq:rboundI1}
\end{align}
Using \eqref{eq:zsmooth} and \eqref{eq:ch}, we can directly compute
\begin{align}
\mathbf{Z}^{N}_{\t,\x}-\mathbf{S}^{N}_{\t,\x}&=[\mathscr{S}^{N}\star(\mathbf{Z}^{N}_{\t,\x}-\mathbf{Z}^{N}_{\t,\cdot})]_{\x}=\sum_{\w\in\Z}\mathscr{S}^{N}_{\w}\cdot(\mathbf{Z}^{N}_{\t,\x}-\mathbf{Z}^{N}_{\t,\x-\w})\nonumber\\
&=\sum_{\w\in\Z}\mathscr{S}^{N}_{\w}\mathbf{Z}^{N}_{\t,\x-\w}\cdot\Big\{\exp(\lambda\mathbf{j}^{N}_{\t,\x}-\lambda\mathbf{j}^{N}_{\t,\x-\w})-1\Big\}\nonumber\\
&\lesssim\mathbf{S}^{N}_{\t,\x}\cdot\sup_{\t\in[0,\T]}\sup_{|\x|\leq N^{\mathrm{D}}}\sup_{|\w|\lesssim N^{1-\delta_{\mathbf{S}}}}\Big|\exp(\lambda\mathbf{j}^{N}_{\t,\x}-\lambda\mathbf{j}^{N}_{\t,\x-\w})-1\Big|.\label{eq:rboundI2}
\end{align}
We are left to control the triple supremum in \eqref{eq:rboundI2}. To this end, we recall that {\small$\mathbf{j}^{N}_{\t,\x}-\mathbf{j}^{N}_{\t,\x-1}=N^{-1/2}\bphi_{\t,\x}$} for any {\small$(\t,\x)\in[0,\infty)\times\Z$}. This gives the following deterministic identity:
\begin{align}
\exp(\lambda\mathbf{j}^{N}_{\t,\x}-\lambda\mathbf{j}^{N}_{\t,\x-\w})&=\exp\Big\{\lambda N^{-\frac12}(\bphi_{\t,\x-\w+1}+\ldots+\bphi_{\t,\x})\Big\}.\label{eq:rboundI2a}
\end{align}
Now, suppose that {\small$\bphi_{0,\cdot}$} is distributed according to the measure {\small$\mathbb{P}^{0}$} from \eqref{eq:gcmeasure}. Because {\small$\mathbb{P}^{0}$} is an invariant measure for \eqref{eq:phi}, we also know that {\small$\bphi_{\t,\cdot}\sim\mathbb{P}^{0}$} for any deterministic {\small$\t\geq0$} as well. In this case, the random variables {\small$\bphi_{\t,\z}$} are i.i.d., mean zero, and sub-Gaussian in the sense that {\small$\E^{0}\exp(\kappa\bphi_{\t,\z})\lesssim\exp(\mathrm{C}\kappa^{2})$} for some {\small$\mathrm{C}=\mathrm{O}(1)$} and any {\small$\kappa\in\R$}. By a standard moment generating function calculation, we then deduce the following estimate, in which {\small$\E^{\mathrm{path},0}$} is {\color{black}the} expectation with respect to the law of the process {\small$\t\mapsto\bphi_{\t,\cdot}$} with initial data sampled via {\small$\mathbb{P}^{0}$}: 
\begin{align}
\E^{\mathrm{path},0}\exp\Big\{\kappa\lambda N^{-\frac12}(\bphi_{\t,\x-\w+1}+\ldots+\bphi_{\t,\x})\Big\}\lesssim\exp(\mathrm{C}\kappa^{2}\lambda^{2}N^{-1}|\w|)\lesssim\exp(\mathrm{C}'\kappa^{2}\lambda^{2}N^{-\delta_{\mathbf{S}}}),\nonumber
\end{align}
where the last estimate above follows since we restrict to {\small$|\w|\lesssim N^{1-\delta_{\mathbf{S}}}$}. Now, we let {\small$\mathbb{P}^{\mathrm{path},0}$} denote the law of the process {\small$\t\mapsto\bphi_{\t,\cdot}$} with initial data sampled according to {\small$\mathbb{P}^{0}$}. It is a standard Chernoff inequality (see \cite{V18}) to get the following from the previous display; it states that a sub-Gaussian random variable with an effective variance {\small$\lesssim N^{-\delta_{\mathbf{S}}}$} exceeds the value of {\small$N^{-2\delta_{\mathbf{S}}/5}\gg(N^{-\delta_{\mathbf{S}}})^{1/2}$} with exponentially small probability (in {\small$N$}):
\begin{align}
\mathbb{P}^{\mathrm{path},0}\Big(N^{-\frac12}|\bphi_{\t,\x-\w+1}+\ldots+\bphi_{\t,\x}|\geq N^{-\frac25\delta_{\mathbf{S}}}\Big)\lesssim\exp(-\mathrm{c}N^{\frac15\delta_{\mathbf{S}}}).\nonumber
\end{align}
Now, fix any {\small$\mathrm{D}>0$}. A union bound turns the previous display into
\begin{align}
&\mathbb{P}^{\mathrm{path},0}\Big(\sup_{\t\in[0,\T]\cap N^{-\mathrm{D}}\Z}\sup_{|\x|\leq N^{\mathrm{D}}}N^{-\frac12}|\bphi_{\t,\x-\w+1}+\ldots+\bphi_{\t,\x}|\geq N^{-\frac25\delta_{\mathbf{S}}}\Big)\nonumber\\
&\lesssim N^{2\mathrm{D}}\exp(-\mathrm{c}N^{\frac15\delta_{\mathbf{S}}})\lesssim\exp(-\tfrac12\mathrm{c}N^{\frac15\delta_{\mathbf{S}}}).\nonumber
\end{align}
In the above estimate, the constant {\small$\mathrm{c}>0$} is independent of {\small$N$}. Now, let {\small$\mathbb{P}^{\mathrm{path}}$} be the law of the process {\small$\t\mapsto\bphi_{\t,\cdot}$}, and recall that {\small$\mathbb{P}^{\mathrm{path},0}$} is the law of the same but assuming its initial data has distribution {\small$\mathbb{P}^{0}$}. Since the dynamics behind the path-space measures {\small$\mathbb{P}^{\mathrm{path}}$} and {\small$\mathbb{P}^{\mathrm{path},0}$} are equal, the density {\small$\d\mathbb{P}^{\mathrm{path}}/\d\mathbb{P}^{\mathrm{path},0}$} is equal to the density of the distribution of the respective initial data. In particular, we have the uniform {\small$\mathrm{L}^{\infty}$}-bound of {\small$|\d\mathbb{P}^{\mathrm{path}}/\d\mathbb{P}^{\mathrm{path},0}|\lesssim N^{1/3}$} by \eqref{eq:noneq}. Thus, by a change-of-measure, we have
\begin{align}
&\mathbb{P}^{\mathrm{path}}\Big(\sup_{|\w|\lesssim N^{1-\delta_{\mathbf{S}}}}\sup_{\t\in[0,\T]\cap N^{-\mathrm{D}}\Z}\sup_{|\x|\leq N^{\mathrm{D}}}N^{-\frac12}|\bphi_{\t,\x-\w+1}+\ldots+\bphi_{\t,\x}|\geq N^{-\frac25\delta_{\mathbf{S}}}\Big)\nonumber\\
&\lesssim N^{\frac13}\exp(-\tfrac12\mathrm{c}N^{\frac15\delta_{\mathbf{S}}})\lesssim \exp(-\tfrac13\mathrm{c}N^{\frac15\delta_{\mathbf{S}}}).\nonumber
\end{align}
Because {\small$\mathrm{D}>0$} can be arbitrarily large in the previous display, so that {\small$N^{-\mathrm{D}}$} is much smaller than the microscopic time-scale of {\small$N^{-2}$}, it is a standard continuity argument to extend the previous display to the following estimate (which holds with high probability) over the entire continuum time-interval {\small$[0,1]$} as opposed to just a net:
\begin{align}
\sup_{|\w|\lesssim N^{1-\delta_{\mathbf{S}}}}\sup_{\t\in[0,\T]}\sup_{|\x|\leq N^{\mathrm{D}}}N^{-\frac12}|\bphi_{\t,\x-\w+1}+\ldots+\bphi_{\t,\x}|\lesssim N^{-\frac25\delta_{\mathbf{S}}}.\label{eq:rboundI3}
\end{align}
Let us restrict to the high probability event where \eqref{eq:rboundI3} holds. We plug this into \eqref{eq:rboundI2a} and use a Taylor expansion to deduce that the triple supremum in \eqref{eq:rboundI2} is {\small$\lesssim N^{-2\delta_{\mathbf{S}}/5}$}. In particular, we establish {\small$|\mathbf{Z}^{N}_{\t,\x}-\mathbf{S}^{N}_{\t,\x}|\lesssim N^{-2\delta_{\mathbf{S}}/5}\mathbf{S}^{N}_{\t,\x}$} with high probability simultaneously for all {\small$\t\in[0,\T]$} and {\small$|\x|\leq N^{\mathrm{D}}$}. This implies the estimate \eqref{eq:rboundI1}; as we noted right before \eqref{eq:rboundI1}, this completes the proof.
\end{proof}
We now present an a priori bound on {\small$\bphi_{\t,\x}$} from \eqref{eq:phi}, also on ``large" space-time sets.
\begin{lemma}\label{lemma:phibound}
\fsp Fix any {\small$\mathrm{D}=\mathrm{O}(1)$} and {\small$\delta>0$} and {\small$\T>0$}, all independent of {\small$N$}. With high probability, we have 
\begin{align}
\sup_{\t\in[0,\T]}\sup_{|\x|\leq N^{\mathrm{D}}}|\bphi_{\t,\x}|\leq N^{\delta}.\label{eq:phiboundI}
\end{align}
\end{lemma}
\begin{proof}
With respect to the measure {\small$\mathbb{P}^{0}$}, the random variable {\small$\bphi_{\x}$} is sub-Gaussian with effective variance of {\small$\mathrm{O}(1)$}; see \eqref{eq:gcmeasure} and Assumption \ref{assump:potential}. Thus, we have the following sub-Gaussian tail probability estimate for some {\small$\mathrm{c}>0$} independent of {\small$N$}:
\begin{align}
\mathbb{P}^{0}(|\bphi_{\x}|\geq N^{\delta})\lesssim\exp(-\mathrm{c}N^{2\delta}).\nonumber
\end{align}
Now, suppose that {\small$\bphi_{0,\cdot}\sim\mathbb{P}^{0}$}, so that {\small$\bphi_{\t,\cdot}\sim\mathbb{P}^{0}$} for any deterministic {\small$\t\geq0$} as well. Thus, for any {\small$\mathrm{D}>0$}, we can use the previous display and a union bound to deduce the following estimate, in which we again let {\small$\mathbb{P}^{\mathrm{path},0}$} be the law of the process {\small$\t\mapsto\bphi_{\t,\cdot}$} with initial data sampled via {\small$\mathbb{P}^{0}$}:
\begin{align}
\mathbb{P}^{\mathrm{path},0}(\sup_{\t\in[0,\T]\cap N^{-\mathrm{D}}\Z}\sup_{|\x|\leq N^{\mathrm{D}}}|\bphi_{\t,\x}|\geq N^{\delta})\lesssim N^{2\mathrm{D}}\exp(-\mathrm{c}N^{2\delta})\lesssim\exp(-\tfrac12\mathrm{c}N^{2\delta}).\nonumber
\end{align}
As in the proof of Lemma \ref{lemma:rbound}, we can change measures to obtain the following, in which {\small$\mathbb{P}^{\mathrm{path}}$} is the law of the process {\small$\t\mapsto\bphi_{\t,\cdot}$} with the initial data of interest in Theorems \ref{theorem:main} and \ref{theorem:mainwedge}:
\begin{align}
\mathbb{P}^{\mathrm{path}}(\sup_{\t\in[0,\T]\cap N^{-\mathrm{D}}\Z}\sup_{|\x|\leq N^{\mathrm{D}}}|\bphi_{\t,\x}|\geq N^{\delta})\lesssim N^{\frac13}\exp(-\tfrac12\mathrm{c}N^{2\delta})\lesssim\exp(-\tfrac13\mathrm{c}N^{2\delta}).\nonumber
\end{align}
Again, as in the proof of Lemma \ref{lemma:rbound}, because {\small$\mathrm{D}>0$} is arbitrarily large, a standard continuity argument extends the previous display to the desired estimate \eqref{eq:phiboundI}, so the proof is complete.
\end{proof}
\subsection{Estimates for {\color{black}the} {\small$\mathbf{Av}^{\mathbf{X},\mathfrak{q}_{\n}}$} objects}
We now focus on {\color{black}the} {\small$\mathbf{Av}^{\mathbf{X},\mathfrak{q}_{\n}}$} terms from Definition \ref{definition:eq-operators}, where the {\small$\mathfrak{q}_{\n}$} terms are admissible (see Definition \ref{definition:admissible}). The main estimate in this direction is a \abbr{CLT}-type bound for {\small$\mathbf{Av}^{\mathbf{X},\mathfrak{q}_{\n}}$} (which turns out to essentially behave like an average of {\small$\mathfrak{n}_{\mathbf{Av}}=N^{1-3\delta_{\mathbf{S}}/2}$}-many independent, fluctuating terms).
\begin{lemma}\label{lemma:avldp}
\fsp Fix any {\small$\mathrm{D}=\mathrm{O}(1)$} and {\small$\delta>0$}. With high probability, we have the following for any admissible {\small$\mathfrak{q}_{\n}$}:
\begin{align}
\sup_{\t\in[0,1]}\sup_{|\x|\leq N^{\mathrm{D}}}|\mathbf{Av}^{\mathbf{X},\mathfrak{q}_{\n}}_{\t,\x}|\lesssim N^{-\frac12+\frac34\delta_{\mathbf{S}}+\delta}.\label{eq:avldpI}
\end{align}
\end{lemma}
Before we proceed, we clarify what \eqref{eq:avldpI} really says, as the proof requires a few fairly technical points. It turns out that {\small$\mathbf{Av}^{\mathbf{X},\mathfrak{q}_{\n}}$} is essentially an average of quantities of the form {\small$\mathfrak{q}_{\n}[\tau_{\k}\bphi_{\t}]\mathbf{G}_{\k}[\bphi_{\t}]$} (see Definition \ref{definition:eq-operators}), where the shifts {\small$\mathfrak{q}_{\n}[\tau_{\k}\bphi_{\t}]$} are (after some technical modification) essentially independent and sufficiently bounded mean-zero quantities with respect to the measure {\small$\mathbb{P}^{0}$}. At this point, since {\small$\mathfrak{n}_{\mathbf{Av}}=N^{1-3\delta_{\mathbf{S}}/2}$}, the bound \eqref{eq:avldpI} would be a consequence of standard concentration bounds for averages of independent, mean-zero random variables, except that we must also deal with the {\small$\mathbf{G}_{\k}[\bphi_{\t}]$}-factors. But if {\small$\mathfrak{q}_{\n}[\tau_{\k}\bphi_{\t}]$} depends on {\small$\bphi_{\t,\w}$} for {\small$\w>\k$}, then by construction in Definition \ref{definition:eq-operators}, we know that {\small$\mathbf{G}_{\k}[\bphi_{\t}]$} depends only on {\small$\bphi_{\t,\w}$} for {\small$\w\leq\k$}. Thus in this case, the two factors {\small$\mathfrak{q}_{\n}[\tau_{\k}\bphi_{\t}]$} and {\small$\mathbf{G}_{\k}[\bphi_{\t}]$} are still independent, and we still have \abbr{CLT}-type bounds. It is also possible that {\small$\mathfrak{q}_{\n}[\tau_{\k}\bphi_{\t}]$} depends on {\small$\bphi_{\t,\w}$} for {\small$\w\leq\k$}, but a similar independence between {\small$\mathfrak{q}_{\n}[\tau_{\k}\bphi_{\t}]$} and {\small$\mathbf{G}_{\k}[\bphi_{\t}]$} turns out to hold in this case. Finally, to go beyond {\small$\mathbb{P}^{0}$} initial data, we use another change-of-measure argument as in the proofs of Lemmas \ref{lemma:rbound} and \ref{lemma:phibound}.
\begin{proof}[{\color{black}Proof of Lemma \ref{lemma:avldp}}]
In the argument below, we assume that {\small$\mathfrak{q}_{\n}[\bphi]$} depends on {\small$\bphi_{\w}$} for {\small$\w\in\{1,\ldots,\mathfrak{l}\}$} for some integer {\small$\mathfrak{l}\lesssim1$}. The case of {\small$\w\in\{-\mathfrak{l},\ldots,0\}$} follows from the same argument (after taking the mirror image in space about {\small$0.5$}, so that {\small$0\leftrightarrow1$}, for example). We first define the following cutoff of {\small$\mathfrak{q}_{\n}$} (as a function of {\small$\bphi\in\R^{\Z}$}):
\begin{align*}
\wt{\mathfrak{q}}_{\n}[\bphi]:=\mathfrak{q}_{\n}[\bphi]\cdot\mathbf{1}[|\mathfrak{q}_{\n}[\bphi]|\leq N^{\delta_{0}}] - \E^{0}\Big(\mathfrak{q}_{\n}[\boldsymbol{\psi}]\cdot\mathbf{1}[|\mathfrak{q}_{\n}[\boldsymbol{\psi}]|\leq N^{\delta_{0}}]\Big).
\end{align*}
We clarify that {\small$\boldsymbol{\psi}\in\R^{\Z}$} is the integration-variable in the expectation above. Also, the constant {\small$\delta_{0}>0$} is any small constant that is independent of {\small$N$}. We also define the following for any {\small$\k\geq1$}:
\begin{align}
\mathbf{G}_{\k}[\tau_{\x}\bphi_{\t}]:=\exp\Big\{\lambda N^{-\frac12}(\bphi_{\t,\x+1}+\ldots+\bphi_{\t,\x+\k})\Big\}\cdot\mathbf{1}\Big[|\exp\Big(\lambda N^{-\frac12}(\bphi_{\t,\x+1}+\ldots+\bphi_{\t,\x+\k})\Big)|\lesssim1\Big].\label{eq:gkterm}
\end{align}
We claim the above identity holds without the indicator on the \abbr{RHS} with high probability simultaneously for all {\color{black}{\small$\t\in[0,1]$}, {\small$|\x|\leq N^{\mathrm{D}}$}, and {\small$\k=1,\ldots,\mathfrak{n}_{\mathbf{Av}}$}}, where {\small$\mathfrak{n}_{\mathbf{Av}}=N^{1-3\delta_{\mathbf{S}}/2}$}. This follows by \eqref{eq:rboundI3}. We also claim that the following holds with high probability for any {\small$\mathrm{D}_{1},\mathrm{D}_{2}=\mathrm{O}(1)$}:
\begin{align}
\sup_{1\leq\k\leq\mathfrak{n}_{\mathbf{Av}}}\sup_{\t\in[0,1]}\sup_{|\x|\leq N^{\mathrm{D}_{1}}}|\wt{\mathfrak{q}}_{\n}[\tau_{\x+\k}\bphi_{\t}]-\mathfrak{q}_{\n}[\tau_{\x+\k}\bphi_{\t}]|\lesssim N^{-\mathrm{D}_{2}}.\label{eq:avldpI-1}
\end{align}
{\color{black}To see the previous claim}, we note that by Lemma \ref{lemma:phibound} and the local polynomial bound \eqref{eq:f-estimate} for {\small$\mathfrak{q}_{\n}$}, we have {\small$\mathfrak{q}_{\n}[\tau_{\x+\k}\bphi_{\t}]=\mathfrak{q}_{\n}[\tau_{\x+\k}\bphi_{\t}]\cdot\mathbf{1}[|\mathfrak{q}_{\n}[\tau_{\x+\k}\bphi_{\t}]|\leq N^{\delta_{0}}]$} for all {\small$\k,\t,\x$} on the \abbr{LHS} above with high probability. Moreover, we have {\small$\E^{0}(\mathfrak{q}_{\n}[\boldsymbol{\psi}]\cdot\mathbf{1}[|\mathfrak{q}_{\n}[\boldsymbol{\psi}]|\leq N^{\delta_{0}})=\E^{0}(\mathfrak{q}_{\n}[\boldsymbol{\psi}]\cdot\mathbf{1}[|\mathfrak{q}_{\n}[\boldsymbol{\psi}]|> N^{\delta_{0}})$} since {\small$\mathfrak{q}_{\n}\in\mathrm{Jet}_{\k}^{\perp}$} for {\small$\k\geq0$} and hence vanishes in expectation with respect to {\small$\mathbb{P}^{0}$}. By the same bound \eqref{eq:f-estimate} on {\small$\mathfrak{q}_{\n}$}, we have {\small$\mathbf{1}[|\mathfrak{q}_{\n}[\boldsymbol{\psi}]|>N^{\delta_{0}}\leq\mathbf{1}[\sup_{|\w|\lesssim1}|\boldsymbol{\psi}_{\w}|>N^{\delta_{1}}]\lesssim\sum_{|\w|\lesssim1}\mathbf{1}[|\boldsymbol{\psi}_{\w}|>N^{\delta_{1}}]$} with {\small$\delta_{1}>0$} depending only on {\small$\delta_{0}$}. Thus, by sub-Gaussianity of {\small$\boldsymbol{\psi}_{\w}$} (see \eqref{eq:gcmeasure} and Assumption \ref{assump:potential}), the quantity {\small$\E^{0}(\mathfrak{q}_{\n}[\boldsymbol{\psi}]\cdot\mathbf{1}[|\mathfrak{q}_{\n}[\boldsymbol{\psi}]|> N^{\delta_{0}})$} is exponentially small in {\small$N$}. This yields \eqref{eq:avldpI-1}.

Recall {\small$\mathbf{Av}^{\mathbf{X},\mathfrak{q}_{\n}}$} from Definition \ref{definition:eq-operators}. By the last two displays (and the fact that the indicator in {\small$\mathbf{G}_{\k}$} is redundant with high probability), we have the following with high probability simultaneously over {\small$\t\in[0,1]$} and {\small$|\x|\leq N^{\mathrm{D}}$}:
\begin{align}
\mathbf{Av}^{\mathbf{X},\mathfrak{q}_{\n}}_{\t,\x}&=\mathfrak{n}_{\mathbf{Av}}^{-1}\sum_{\k=1,\ldots,\mathfrak{n}_{\mathbf{Av}}}\wt{\mathfrak{q}}_{\n}[\tau_{\x+\k}\bphi_{\t}]\cdot\mathbf{G}_{\k}[\tau_{\x}\bphi_{\t}]+\mathrm{O}(N^{-\mathrm{D}_{2}}).\label{eq:avldpI0}
\end{align}
Now, we consider the set {\small$\{1,\ldots,\mathfrak{n}_{\mathbf{Av}}\}$} modulo {\small$\mathfrak{l}_{1}$} for some large but still {\small$\mathrm{O}(1)$} constant {\small$\mathfrak{l}_{1}$}; we clarify its utility shortly. We will label the different equivalence classes by {\small$[\k_{1}]$} for {\small$\k_{1}\in\{1,\ldots,\mathfrak{n}_{\mathbf{Av}}\}$}; note that there are {\small$\mathfrak{l}_{1}=\mathrm{O}(1)$}-many equivalence classes. We now write the following, where the first sum is over different equivalence classes, and the second is over members of a given equivalence class:
\begin{align}
\mathbf{Av}^{\mathbf{X},\mathfrak{q}_{\n}}_{\t,\x}&=\sum_{[\k_{1}]}\mathfrak{n}_{\mathbf{Av}}^{-1}\sum_{\k_{2}\in[\k_{1}]}\wt{\mathfrak{q}}_{\n}[\tau_{\x+\k_{2}}\bphi_{\t}]\cdot\mathbf{G}_{\k_{2}}[\tau_{\x}\bphi_{\t}]+\mathrm{O}(N^{-\mathrm{D}_{2}}).\label{eq:avldpI1}
\end{align}
For any equivalence class {\small$[\k_{1}]$}, we arrange the indices {\small$\k_{2}\in[\k_{1}]$} in this equivalence class in increasing order, and to be clear, we accordingly write {\small$[\k_{1}]=\{\k_{2,1},\ldots,\k_{2,\ell}\}$}. Now, we emphasize the following properties.
\begin{itemize}
\item For any {\small$\k_{2}$}, the term {\small$\wt{\mathfrak{q}}_{\n}[\tau_{\x+\k_{2}}\bphi_{\t}]$} depends only on {\small$\bphi_{\t,\w}$} for {\small$\w\in\mathbb{I}_{\k_{2},1}:=\x+\k_{2}+\{1,\ldots,\mathfrak{l}\}$}.
\item For any {\small$\k_{2}$}, the quantity {\small$\mathbf{G}_{\k_{2}}[\tau_{\x}\bphi_{\t}]$} depends only on {\small$\bphi_{\t,\w}$} for {\small$\w\in\mathbb{I}_{\k_{2},2}=\x+\{1,\ldots,\k_{2}\}$}. We emphasize that {\small$\mathbb{I}_{\k_{2,\i},1}\cap\mathbb{I}_{\k_{2,\j},2}=\emptyset$} for any {\small$\j\leq\i$}.
\end{itemize}
\emph{Now, assume {\small$\bphi_{0,\cdot}\sim\mathbb{P}^{0}$}}; we remove this assumption later. We also define the following filtration:
\begin{align*}
\i\mapsto\mathscr{F}_{\i}:=\boldsymbol{\sigma}\Big(\bphi_{\t,\w}, \ \w\in\bigcup_{\j\leq\i} \ (\mathbb{I}_{\k_{2,\j},1}\cup\mathbb{I}_{\k_{2,\j},2})\Big), \quad \i\in\{1,\ldots,\ell\}.
\end{align*}
In words, {\small$\mathscr{F}_{\i}$} is the {\small$\boldsymbol{\sigma}$}-algebra generated by {\small$\boldsymbol{\phi}_{\t,\w}$} for {\small$\w\in\mathbb{I}_{\k_{2,\j},1}\cup\mathbb{I}_{\k_{2,\j},2}$} over {\small$\j\leq\i$}. We now claim that the following discrete-time process is a martingale with respect to the above filtration:
\begin{align*}
\i\mapsto \mathsf{M}_{\i}:=\mathfrak{n}_{\mathbf{Av}}^{-1}\sum_{\k_{2}\in\{\k_{2,1},\ldots,\k_{2,\i}\}}\wt{\mathfrak{q}}_{\n}[\tau_{\x+\k_{2}}\bphi_{\t}]\cdot\mathbf{G}_{\k_{2}}[\tau_{\x}\bphi_{\t}].
\end{align*}
To see this, it follows by construction that {\small$\mathsf{M}_{\i}$} is {\small$\mathscr{F}_{\i}$}-measurable. Moreover, if we condition on {\small$\mathscr{F}_{\i}$} for some {\small$\i\in\{1,\ldots,\ell-1\}$}, then all summands in {\small$\mathsf{M}_{\i+1}$}, except {\small$\wt{\mathfrak{q}}_{\n}[\tau_{\x+\k_{2,\i+1}}\bphi_{\t}]\mathbf{G}_{\k_{2,\i+1}}[\tau_{\x}\bphi_{\t}]$}, become deterministic. Moreover, conditioning on {\small$\mathsf{F}_{\i}$} pushes {\small$\mathbb{P}^{0}$} forward to a product measure on {\small$\R^{\Z\setminus\mathbb{I}_{\i}}$} whose one-dimensional marginals agree with those of {\small$\mathbb{P}^{0}$} (see Definition \ref{definition:gcmeasure}), where {\small$\mathbb{I}_{\i}$} is disjoint from {\small$\mathbb{I}_{\k_{2,\i+1},1}$}, the support of {\small$\wt{\mathfrak{q}}_{\n}[\tau_{\x+\k_{2,\i+1}}\bphi]$}. Thus, after said conditioning, the increment {\small$\wt{\mathfrak{q}}_{\n}[\tau_{\x+\k_{2,\i+1}}\bphi_{\t}]\mathbf{G}_{\k_{2,\i+1}}[\tau_{\x}\bphi_{\t}]$} is mean-zero, and the martingale property follows.

Also, by construction, the increments of {\small$\mathsf{M}_{\i}$} are deterministically {\small$\mathrm{O}(N^{\delta_{0}})$}. Therefore, by the Azuma-Hoeffding martingale inequality, we deduce that 
\begin{align*}
\mathbb{P}^{\mathrm{path},0}\Big(\Big|\mathfrak{n}_{\mathbf{Av}}^{-1}\sum_{\k_{2}\in[\k_{1}]}\wt{\mathfrak{q}}_{\n}[\tau_{\x+\k_{2}}\bphi_{\t}]\cdot\mathbf{G}_{\k_{2}}[\tau_{\x}\bphi_{\t}]\Big|\geq \mathfrak{n}_{\mathbf{Av}}^{-\frac12}N^{2\delta_{0}}\Big)\lesssim\exp(-\mathrm{c}N^{\delta_{0}}),
\end{align*}
where {\small$\mathrm{c}>0$} is independent of {\small$N$} and {\small$\mathbb{P}^{\mathrm{path},0}$} is the probability with respect to the dynamics {\small$\t\mapsto\bphi_{\t,\cdot}$} with initial data sampled via {\small$\mathbb{P}^{0}$}. We clarify that the sum contains {\small$\lesssim\mathfrak{n}_{\mathbf{Av}}$}-many terms, and that the lower bound in the probability is a factor of {\small$N^{\delta_{0}}$}-larger than the typical size {\small$\mathfrak{n}_{\mathbf{Av}}^{-1/2}N^{\delta_{0}}$} predicted by the \abbr{CLT}, hence the exponentially small \abbr{RHS} above. By another union bound argument, we have the following for any {\small$\mathrm{D}_{3}>0$} large:
\begin{align*}
\mathbb{P}^{\mathrm{path},0}\Big(\sup_{\t\in[0,1]\cap N^{-\mathrm{D}_{3}}\Z}\sup_{|\x|\leq N^{\mathrm{D}}}\Big|\sum_{[\k_{1}]}\mathfrak{n}_{\mathbf{Av}}^{-1}\sum_{\k_{2}\in[\k_{1}]}\wt{\mathfrak{q}}_{\n}[\tau_{\x+\k_{2}}\bphi_{\t}]\cdot\mathbf{G}_{\k_{2}}[\tau_{\x}\bphi_{\t}]\Big|\gtrsim \mathfrak{n}_{\mathbf{Av}}^{-\frac12}N^{2\delta_{0}}\Big)\lesssim\exp(-\tfrac12\mathrm{c}N^{\delta_{0}}).
\end{align*}
(Recall that the number of equivalence classes is {\small$\mathrm{O}(1)$}.) We now change measure to {\small$\mathbb{P}^{\mathrm{path}}$} and then use a short-time continuity argument as in the proofs of {\color{black}Lemmas \ref{lemma:rbound} and \ref{lemma:phibound}} to deduce that with high probability, we have 
\begin{align*}
\sup_{\t\in[0,1]}\sup_{|\x|\leq N^{\mathrm{D}}}\Big|\sum_{[\k_{1}]}\mathfrak{n}_{\mathbf{Av}}^{-1}\sum_{\k_{2}\in[\k_{1}]}\wt{\mathfrak{q}}_{\n}[\tau_{\x+\k_{2}}\bphi_{\t}]\cdot\mathbf{G}_{\k_{2}}[\tau_{\x}\bphi_{\t}]\Big|\lesssim \mathfrak{n}_{\mathbf{Av}}^{-\frac12}N^{2\delta_{0}}= N^{-\frac12+\frac34\delta_{\mathbf{S}}+2\delta_{0}},
\end{align*}
where the last identity follows by construction of {\small$\mathfrak{n}_{\mathbf{Av}}$} in Definition \ref{definition:eq-operators}. If we combine the previous display with \eqref{eq:avldpI1}, then we deduce the desired estimate \eqref{eq:avldpI}, so the proof is complete.
\end{proof}
\subsection{Estimates for {\color{black}the} {\small$\mathbf{Av}^{\mathbf{T},\mathbf{X},\mathfrak{q}_{\n}}$} objects}
In this subsection, the main goal is to present two estimates for {\color{black}the} {\small$\mathbf{Av}^{\mathbf{T},\mathbf{X},\mathfrak{q}_{\n}}$} objects from Definition \ref{definition:eq-operators}, where the {\small$\mathfrak{q}_{\n}$} are admissible (see Definition \ref{definition:admissible}). We start with a technical preliminary that controls the process {\small$\t\mapsto\mathbf{j}^{N}_{\t,\cdot}$} on time-scales bounded {\color{black}from} above by the averaging-scale {\small$\mathfrak{t}_{\mathbf{Av}}=N^{-2/3-10\delta_{\mathbf{S}}}$} in {\small$\mathbf{Av}^{\mathbf{T},\mathbf{X},\mathfrak{q}_{\n}}$} (again, see Definition \ref{definition:eq-operators}).
\begin{lemma}\label{lemma:jnestimate}
\fsp Fix any {\small$\mathrm{D}=\mathrm{O}(1)$}, and recall {\small$\mathfrak{t}_{\mathbf{Av}}=N^{-2/3-10\delta_{\mathbf{S}}}$}. With high probability, we have 
\begin{align}
\sup_{\t\in[0,1]}\sup_{|\x|\leq N^{\mathrm{D}}}\sup_{\r\in[0,\mathfrak{t}_{\mathbf{Av}}]}|\mathbf{j}^{N}_{\t+\r,\x}-\mathbf{j}^{N}_{\t,\x}|\lesssim1.\label{eq:jnestimateI}
\end{align}
\end{lemma}
\begin{proof}
First, we claim that with high probability, we have the following, in which {\small$\mathscr{S}^{N}$} is from \eqref{eq:zsmooth}:
\begin{align}
\sup_{\t\in[0,1]}\sup_{|\x|\leq N^{\mathrm{D}}}\sup_{\r\in[0,\mathfrak{t}_{\mathbf{Av}}]}|\mathbf{j}^{N}_{\t+\r,\x}-\mathbf{j}^{N}_{\t,\x}|&\lesssim\sup_{\t\in[0,1]}\sup_{|\x|\leq N^{\mathrm{D}}}\sup_{\r\in[0,\mathfrak{t}_{\mathbf{Av}}]}|[\mathscr{S}^{N}\star\mathbf{j}^{N}_{\t+\r,\cdot}]_{\x}-[\mathscr{S}^{N}\star\mathbf{j}^{N}_{\t,\cdot}]_{\x}|\nonumber\\
&+\sup_{\t\in[0,1]}\sup_{|\x|\leq N^{\mathrm{D}}}\sup_{\r\in[0,\mathfrak{t}_{\mathbf{Av}}]}|[\mathscr{S}^{N}\star\mathbf{j}^{N}_{\t+\r,\cdot}]_{\x}-\mathbf{j}^{N}_{\t+\r,\x}|\nonumber\\
&+\sup_{\t\in[0,1]}\sup_{|\x|\leq N^{\mathrm{D}}}\sup_{\r\in[0,\mathfrak{t}_{\mathbf{Av}}]}|[\mathbf{j}^{N}_{\t,\x}-[\mathscr{S}^{N}\star\mathbf{j}^{N}_{\t,\cdot}]_{\x}|\nonumber\\
&\lesssim1+\sup_{\t\in[0,1]}\sup_{|\x|\leq N^{\mathrm{D}}}\sup_{\r\in[0,\mathfrak{t}_{\mathbf{Av}}]}|[\mathscr{S}^{N}\star\mathbf{j}^{N}_{\t+\r,\cdot}]_{\x}-[\mathscr{S}^{N}\star\mathbf{j}^{N}_{\t,\cdot}]_{\x}|.\label{eq:jnestimateI1}
\end{align}
The first bound follows by the triangle inequality. To prove the second, we recall that {\small$\mathscr{S}^{N}$} is a smoothing kernel supported on a neighborhood of {\small$0\in\Z$} of length {\small$\lesssim N^{1-\delta_{\mathbf{S}}}$}, and that {\small$\mathbf{j}^{N}_{\s,\x+\w}-\mathbf{j}^{N}_{\s,\x}=N^{-1/2}\sum_{\z\in\mathbb{I}_{\w}}\bphi_{\s,\x+\w}$} for some interval {\small$\mathbb{I}_{\w}$} of length {\small$|\w|$}. In particular, this and \eqref{eq:rboundI3} give the following estimate with high probability:
\begin{align*}
&\sup_{0\leq\s\lesssim1}\sup_{|\x|\leq N^{\mathrm{D}}}|[\mathscr{S}^{N}\star\mathbf{j}^{N}_{\s,\cdot}]_{\x}-\mathbf{j}^{N}_{\s,\x}|\lesssim \sup_{0\leq\s\lesssim1}\sup_{|\x|\leq N^{\mathrm{D}}} N^{-1+\delta_{\mathbf{S}}}\sum_{|\w|\lesssim N^{1-\delta_{\mathbf{S}}}}|\mathbf{j}^{N}_{\s,\x+\w}-\mathbf{j}^{N}_{\s,\x}|\lesssim N^{-\frac25\delta_{\mathbf{S}}}.
\end{align*}
Now, by convolving the \abbr{RHS} of \eqref{eq:curr} with {\small$\mathscr{S}^{N}$}, we have 
\begin{align}
\sup_{\r\in[0,\mathfrak{t}_{\mathbf{Av}}]}|[\mathscr{S}^{N}\star\mathbf{j}^{N}_{\t+\r,\cdot}]_{\x}-[\mathscr{S}^{N}\star\mathbf{j}^{N}_{\t,\cdot}]_{\x}|&\lesssim\int_{\t}^{\t+\mathfrak{t}_{\mathbf{Av}}}|[\mathscr{S}^{N}\star(N^{\frac32}\grad^{\mathbf{X}}_{1}\mathscr{U}'[\bphi_{\s,\cdot}])]_{\x}|\d\s\nonumber\\
&+\int_{\t}^{\t+\mathfrak{t}_{\mathbf{Av}}}|[\mathscr{S}^{N}\star (N\mathbf{F}[\tau_{\cdot}\bphi_{\s}])]_{\x}|\d\s\nonumber\\
&+\sqrt{2}\lambda N^{\frac12}\sup_{\r\in[0,\mathfrak{t}_{\mathbf{Av}}]}|[\mathscr{S}^{N}\star(\mathbf{b}_{\t+\r,\cdot}-\mathbf{b}_{\t,\cdot})]_{\x}|.\label{eq:jnestimateI2}
\end{align}
We {\color{black}now} control each term on the \abbr{RHS} of the previous display. For the first term on the \abbr{RHS}, first recall that discrete gradients commute with convolution. So, we have {\small$\mathscr{S}^{N}\star(N^{3/2}\grad^{\mathbf{X}}_{1}\mathscr{U}'[\bphi_{\s,\cdot}])=N^{3/2}\grad^{\mathbf{X}}_{1}\mathscr{S}^{N}\star\mathscr{U}'[\bphi_{\s,\cdot}]$}. Moreover, since {\small$\mathscr{S}^{N}$} is smooth at scale {\small$N^{1-\delta_{\mathbf{S}}}$}, the discrete gradient acting on {\small$\mathscr{S}^{N}$} gives an extra factor of {\small$N^{-1+\delta_{\mathbf{S}}}$}. Lastly, it is also supported on a neighborhood of length {\small$\lesssim N^{1-\delta_{\mathbf{S}}}$} around the origin. Thus, we have the pointwise estimate {\small$N^{3/2}|\grad^{\mathbf{X}}_{1}\mathscr{S}^{N}_{\w}|\lesssim N^{1/2+\delta_{\mathbf{S}}}\cdot N^{-1+\delta_{\mathbf{S}}}\mathbf{1}_{|\w|\lesssim N^{1-\delta_{\mathbf{S}}}}$}, so that for any {\small$\delta>0$}, we have
\begin{align}
&\sup_{\t\in[0,1]}\sup_{|\x|\leq N^{\mathrm{D}}}\int_{\t}^{\t+\mathfrak{t}_{\mathbf{Av}}}|[\mathscr{S}^{N}\star(N^{\frac32}\grad^{\mathbf{X}}_{1}\mathscr{U}'[\bphi_{\s,\cdot}])]_{\x}|\d\s\nonumber\\
&\lesssim\sup_{\t\in[0,1]}\sup_{|\x|\leq N^{\mathrm{D}}}\int_{\t}^{\t+\mathfrak{t}_{\mathbf{Av}}} N^{\frac12+\delta_{\mathbf{S}}}\cdot N^{-1+\delta_{\mathbf{S}}}\sum_{|\w|\lesssim N^{1-\delta_{\mathbf{S}}}}|\mathscr{U}'[\bphi_{\s,\x+\w}]|\d\s\nonumber\\
&\lesssim N^{\delta}N^{\frac12+\delta_{\mathbf{S}}}\mathfrak{t}_{\mathbf{Av}}\lesssim1 \label{eq:jnestimateI3}
\end{align}
with high probability, where the last line follows because {\small$\mathscr{U}'[\cdot]$} is uniformly Lipschitz (see Assumption \ref{assump:potential}) and by the high probability estimate in Lemma \ref{lemma:phibound} (along with {\small$\mathfrak{t}_{\mathbf{Av}}\lesssim N^{-2/3}$} and smallness of {\small$\delta,\delta_{\mathbf{S}}$}).

Now, we tackle the second term on the \abbr{RHS} of \eqref{eq:jnestimateI2}. For this, we recall {\small$\mathbf{F}$} from \eqref{eq:nonlinearity}. Since {\small$\bphi_{\w}$} are i.i.d. over {\small$\w\in\Z$} with respect to the measure {\small$\mathbb{P}^{0}$} from \eqref{eq:gcmeasure}, and since {\small$\E^{0}\mathscr{U}'[\bphi_{0}]=0$} (see Lemma \ref{lemma:classify}), we know {\small$\E^{0}\mathbf{F}[\bphi]=0$}. We also know that {\small$|\mathbf{F}[\bphi]|\lesssim1+\sum_{|\z|\lesssim1}|\bphi_{\z}|^{\deg}$} by \eqref{eq:nonlinearity}. Since {\small$\bphi_{\z}$} are i.i.d. sub-Gaussian random variables (see Assumption \ref{assump:noneq}), we have {\small$\E^{0}|\mathbf{F}[\bphi]|^{p}\lesssim_{p}1$} for any {\small$p\geq1$}. We finally remark that {\small$\mathbf{F}[\tau_{\z}\bphi]$} and {\small$\mathbf{F}[\tau_{\w}\bphi]$} are i.i.d. with respect to the {\small$\mathbb{P}^{0}$}-product measure if {\small$|\z-\w|\geq2\deg+1$}; see \eqref{eq:nonlinearity}. So, standard \abbr{CLT}-type moment estimates yield {\small$\E^{0}|[\mathscr{S}^{N}\star(\mathbf{F}[\tau_{\cdot}\bphi])]_{\x}|^{p}\lesssim_{p}N^{-p/2+p\delta_{\mathbf{S}}/2}$} for any {\small$p\geq1$}, since {\small$\mathscr{S}^{N}$} is averaging in space over length-scales of order {\small$N^{1-\delta_{\mathbf{S}}}$}. Therefore, by the Chebyshev inequality and this moment estimate for large enough {\small$p\geq1$}, we ultimately deduce the following for any fixed constants {\small$\delta,\mathrm{D}_{0}>0$}, in which {\small$\mathbb{P}^{\mathrm{path},0}$} is {\color{black}the} probability with respect to the law of the dynamics {\small$\t\mapsto\bphi_{\t,\cdot}$} with initial data sampled via {\small$\mathbb{P}^{0}$}:
\begin{align*}
\mathbb{P}^{\mathrm{path},0}\Big(|[\mathscr{S}^{N}\star(\mathbf{F}[\tau_{\cdot}\bphi_{\s}])]_{\x}|\geq N^{-\frac12+\frac12\delta_{\mathbf{S}}+\delta}\Big)\lesssim_{p}N^{-\frac{p}{2}+\frac{p}{2}\delta_{\mathbf{S}}+p\delta}\E^{0}|[\mathscr{S}^{N}\star(\mathbf{F}[\tau_{\cdot}\bphi])]_{\x}|^{p}\lesssim_{p,\delta,\mathrm{D}_{0}}N^{-\mathrm{D}_{0}}.
\end{align*}
We now take a union bound over {\small$|\x|\leq N^{\mathrm{D}}$} and {\small$\s\in[0,\T]\cap N^{-\mathrm{D}}\Z$} for any {\small$\T\lesssim1$}. Then, we change measure from {\small$\mathbb{P}^{\mathrm{path},0}$} to {\small$\mathbb{P}^{\mathrm{path}}$}, which is the probability measure for the law of the same dynamics {\small$\t\mapsto\bphi_{\t,\cdot}$} but with initial data that satisfies Assumption \ref{assump:noneq}. Lastly, we use short-time continuity. This is all in the same manner as in the proof of Lemma \ref{lemma:phibound}, for example, and it provides us the following bound with high probability:
\begin{align*}
\sup_{\s\in[0,\T]}\sup_{|\x|\leq N^{\mathrm{D}}}|[\mathscr{S}^{N}\star(\mathbf{F}[\tau_{\cdot}\bphi_{\s}])]_{\x}|\lesssim N^{-\frac12+\frac12\delta_{\mathbf{S}}+\delta}
\end{align*}
Plugging this into the second term on the \abbr{RHS} of \eqref{eq:jnestimateI2} and using {\small$\mathfrak{t}_{\mathbf{Av}}\lesssim N^{-2/3}$} (and smallness of {\small$\delta_{\mathbf{S}},\delta$}) gives
\begin{align}
\sup_{\t\in[0,\mathfrak{t}_{\mathbf{Av}}]}\sup_{|\x|\leq N^{\mathrm{D}}}\int_{\t}^{\t+\mathfrak{t}_{\mathbf{Av}}}|[\mathscr{S}^{N}\star (N\mathbf{F}[\tau_{\cdot}\bphi_{\s}])]_{\x}|\d\s\lesssim N^{\frac12+\frac12\delta_{\mathbf{S}}+\delta}\mathfrak{t}_{\mathbf{Av}}\lesssim1.\label{eq:jnestimateI4}
\end{align}
Finally, we move to the last term in \eqref{eq:jnestimateI2}. Because {\small$\mathbf{b}_{\cdot,\w}$} are jointly independent Brownian motions over {\small$\w\in\Z$}, and since {\small$\mathscr{S}^{N}$} is an averaging kernel of length-scale {\small$\gtrsim N^{1-\delta_{\mathbf{S}}}$}, we know that {\small$[\mathscr{S}^{N}\star(\mathbf{b}_{\t+\r,\cdot}-\mathbf{b}_{\t,\cdot})]_{\x}$} has the law of {\small$\mathrm{O}(N^{-1/2+\delta_{\mathbf{S}}/2})\wt{\mathbf{b}}_{\r,\x}$}, where {\small$\wt{\mathbf{b}}_{\cdot,\x}$} is a standard Brownian motion. Thus, by the reflection principle for Brownian {\color{black}motions}, we have the following sub-Gaussian estimate:
\begin{align*}
\mathbb{P}\Big(\sup_{\r\in[0,\mathfrak{t}_{\mathbf{Av}}]}|[\mathscr{S}^{N}\star(\mathbf{b}_{\t+\r,\cdot}-\mathbf{b}_{\t,\cdot})]_{\x}|\geq N^{-\frac12+\frac12\delta_{\mathbf{S}}+\delta}\mathfrak{t}_{\mathbf{Av}}^{\frac12}\Big)\lesssim\exp(-\mathrm{c}N^{2\delta}).
\end{align*}
We note that the convolution {\small$[\mathscr{S}^{N}\star(\mathbf{b}_{\t+\r,\cdot}-\mathbf{b}_{\t,\cdot})]_{\x}$} is H\"{o}lder-{\small$1/2$} continuous in the {\small$\t$}-variable as well, so another union bound and short-time continuity argument implies that with high probability, we have 
\begin{align}
N^{\frac12}\sup_{\t\in[0,1]}\sup_{|\x|\leq N^{\mathrm{D}}}\sup_{\r\in[0,\mathfrak{t}_{\mathbf{Av}}]}|[\mathscr{S}^{N}\star(\mathbf{b}_{\t+\r,\cdot}-\mathbf{b}_{\t,\cdot})]_{\x}|\lesssim N^{\frac12\delta_{\mathbf{S}}+\delta}\mathfrak{t}_{\mathbf{Av}}^{\frac12}\lesssim1,\label{eq:jnestimateI5}
\end{align}
where we again use {\small$\mathfrak{t}_{\mathbf{Av}}\lesssim N^{-2/3}$} and smallness of {\small$\delta_{\mathbf{S}},\delta$} to get the last estimate. We now combine \eqref{eq:jnestimateI1}, \eqref{eq:jnestimateI2}, \eqref{eq:jnestimateI3}, \eqref{eq:jnestimateI4}, and \eqref{eq:jnestimateI5} to deduce the desired estimate \eqref{eq:jnestimateI}, so the proof is complete.
\end{proof}
Using Lemma \ref{lemma:jnestimate}, we will now be able to obtain the following high probability estimate for {\small$\mathbf{Av}^{\mathbf{T},\mathbf{X},\mathfrak{q}_{\n}}$}. (This estimate is sub-optimal in that it does not take advantage of time-fluctuations, but it will still be important.)
\begin{lemma}\label{lemma:avldptx}
\fsp Fix any {\small$\mathrm{D}=\mathrm{O}(1)$} and {\small$\delta>0$}, both independent of {\small$N$}. With high probability, we have the following for any admissible {\small$\mathfrak{q}_{\n}$} in the sense of Definition \ref{definition:admissible}:
\begin{align}
\sup_{\t\in[0,1]}\sup_{|\x|\leq N^{\mathrm{D}}}|\mathbf{Av}^{\mathbf{T},\mathbf{X},\mathfrak{q}_{\n}}_{\t,\x}|\lesssim N^{-\frac12+\frac34\delta_{\mathbf{S}}+\delta}.\label{eq:avldptxI}
\end{align}
\end{lemma}
\begin{proof}
We assume {\color{black}that} {\small$\mathfrak{q}_{\n}[\bphi]$} depends only on {\small$\bphi_{\w}$} for {\small$\w\in\{1,\ldots,\mathfrak{l}\}$} for some {\small$\mathfrak{l}\lesssim1$}. For the case {\small$\w\in\{-\mathfrak{l},\ldots,0\}$}, the same argument works after a reflection that swaps {\small$0\leftrightarrow1$}. Recalling {\color{black}the} notation from Definition \ref{definition:eq-operators}, we write
\begin{align}
\mathbf{Av}^{\mathbf{T},\mathbf{X},\mathfrak{q}_{\n}}_{\t,\x}&=\mathfrak{t}_{\mathbf{Av}}^{-1}\int_{0}^{\mathfrak{t}_{\mathbf{Av}}}\mathfrak{n}_{\mathbf{Av}}^{-1}\sum_{\j=1,\ldots,\mathfrak{n}_{\mathbf{Av}}}\mathfrak{q}_{\n}[\tau_{\x+\j}\bphi_{\t-\r}]\cdot\mathbf{Z}^{N}_{\t-\r,\x+\j}(\mathbf{Z}^{N}_{\t,\x})^{-1}\d\r\nonumber\\
&=\mathfrak{t}_{\mathbf{Av}}^{-1}\int_{0}^{\mathfrak{t}_{\mathbf{Av}}}\mathbf{Z}^{N}_{\t-\r,\x}(\mathbf{Z}^{N}_{\t,\x})^{-1}\cdot \mathfrak{n}_{\mathbf{Av}}^{-1}\sum_{\j=1,\ldots,\mathfrak{n}_{\mathbf{Av}}}\mathfrak{q}_{\n}[\tau_{\x+\j}\bphi_{\t-\r}]\cdot\mathbf{Z}^{N}_{\t-\r,\x+\j}(\mathbf{Z}^{N}_{\t-\r,\x})^{-1}\d\r\nonumber\\
&=\mathfrak{t}_{\mathbf{Av}}^{-1}\int_{0}^{\mathfrak{t}_{\mathbf{Av}}}\mathbf{Z}^{N}_{\t-\r,\x}(\mathbf{Z}^{N}_{\t,\x})^{-1}\cdot \mathbf{Av}^{\mathbf{X},\mathfrak{q}_{\n}}_{\t-\r,\x}\d\r.\label{eq:avldptxI1}
\end{align}
By \eqref{eq:ch}, we know that {\small$\mathbf{Z}^{N}_{\t-\r,\x}(\mathbf{Z}^{N}_{\t,\x})^{-1}=\exp(\lambda\mathbf{j}^{N}_{\t-\r,\x}-\lambda\mathbf{j}^{N}_{\t,\x}-\lambda\mathscr{R}_{\lambda}\r)$}. Since {\small$\lambda,\mathscr{R}_{\lambda}$} are independent of {\small$N$}, with high probability, we know by Lemma \ref{lemma:jnestimate} that {\small$|\mathbf{Z}^{N}_{\t-\r,\x}(\mathbf{Z}^{N}_{\t,\x})^{-1}|\lesssim1$}. Now use Lemma \ref{lemma:avldp} to bound {\small$\mathbf{Av}^{\mathbf{X},\mathfrak{q}_{\n}}_{\t-\r,\x}$} above by the \abbr{RHS} of \eqref{eq:avldptxI} with high probability. Because all estimates so far hold on the same high probability event simultaneously over all {\small$\t\in[0,1]$} and {\small$|\x|\leq N^{\mathrm{D}}$} and {\small$\r\in[0,\mathfrak{t}_{\mathbf{Av}}]$}, the desired estimate \eqref{eq:avldptxI} follows.
\end{proof}
With a more sophisticated argument, we can show {\color{black}some} cancellations from time-fluctuations in {\small$\mathbf{Av}^{\mathbf{T},\mathbf{X},\mathfrak{q}_{\n}}$}.
\begin{prop}\label{prop:stestimate}
\fsp Fix any {\small$\mathrm{D}=\mathrm{O}(1)$}. There exists a constant {\small$\beta>0$} depending only on {\small$\gamma_{\mathrm{data}}>0$} from Assumption \ref{assump:noneq} and a path-space event {\small$\mathscr{E}\subseteq\mathscr{C}([0,1],\R^{\Z})$} of high probability such that the following second moment estimate holds for any admissible {\small$\mathfrak{q}_{\n}$} (see Definition \ref{definition:admissible}):
\begin{align}
\sup_{\t\in[0,\infty)}\sup_{|\x|\leq N^{\mathrm{D}}}\E|\mathbf{1}_{\mathscr{E}}\mathbf{Av}^{\mathbf{T},\mathbf{X},\mathfrak{q}_{\n}}_{\t,\x}|^{2}\lesssim N^{-2-\beta_{}}.\label{eq:stestimateI}
\end{align}
\end{prop}
Heuristically, {\small$\mathbf{Av}^{\mathbf{T},\mathbf{X},\mathfrak{q}_{\n}}$} is the space-time block average of a fluctuation {\small$\mathfrak{q}_{\n}$} on a time-scale of {\small$\mathfrak{t}_{\mathbf{Av}}$} and length-scale of {\small$\mathfrak{n}_{\mathbf{Av}}$}. Square-root \abbr{CLT}-cancellations suggest that the \abbr{LHS} of \eqref{eq:stestimateI} is {\small$\lesssim N^{-2}\mathfrak{t}_{\mathbf{Av}}^{-1}\mathfrak{n}_{\mathbf{Av}}^{-1}\lesssim N^{-7/3+23\delta_{\mathbf{S}}/2}$} (see Definition \ref{definition:eq-operators} for the last bound), where the extra {\small$N^{-2}$} comes from the fast {\small$N^{2}$}-speed of the {\small$\t\mapsto\bphi_{\t,\cdot}$} process in \eqref{eq:phi}. Multiplying by the change-of-measure factor {\small$\lesssim N^{1/3-\gamma_{\mathrm{data}}}$} from \eqref{eq:noneq} and choosing {\small$\delta_{\mathbf{S}}>0$} small thus explains where \eqref{eq:stestimateI} comes from. Making this heuristic rigorous is the goal for the remainder of this section; it follows Section 3 of \cite{DGP} (which itself follows Sections 3-4 in \cite{GJ15}) but with important technical refinements.

The proof uses a number of different lemmas. \emph{We will assume that {\small$\mathfrak{q}_{\n}[\bphi]$} depends only on {\small$\bphi_{\w}$} for {\small$\w\in\{1,\ldots,\mathfrak{l}\}$} for some positive integer {\small$\mathfrak{l}\lesssim1$}}. The case {\small$\w\in\{-\mathfrak{l},\ldots,0\}$} follows by the same argument (again, after a reflection).

The first step is a preliminary rewriting of {\small$\mathbf{Av}^{\mathbf{T},\mathbf{X},\mathfrak{q}_{\n}}$} on a high probability path-space event. First, let us give some notation. Let {\small$\Pi:\R\to[-{}(\log N)^{1/2}{},{}(\log N)^{1/2}{})$} be projection modulo {\small$2{}(\log N)^{1/2}{}$}. We now set
\begin{align}
\mathbf{h}^{N,\Pi}_{\t,\r,\x}:=\Pi(\lambda\mathbf{j}^{N}_{\t-\r,\x}-\lambda\mathbf{j}^{N}_{\t,\x}-\lambda\mathscr{R}_{\lambda}\r).\label{eq:hnpi}
\end{align}
The {\color{black}purpose} of the ``torus" {\small$[-{}(\log N)^{1/2}{},{}(\log N)^{1/2}{})$} is given as follows. We will want to exponentiate \eqref{eq:hnpi} (in the same spirit as in the Cole-Hopf map \eqref{eq:ch}), and if we exponentiate {\small$|\log N|^{1/2}$}, we get a factor which is {\small$\lesssim_{\e}N^{\e}$} for any {\small$\e>0$} (and thus effectively harmless). On the other hand, by Lemma \ref{lemma:jnestimate} and {\small$\lambda\mathscr{R}_{\lambda}=\mathrm{O}(1)$}, if {\small$\r\lesssim\mathfrak{t}_{\mathbf{Av}}\ll1$}, then this projection {\small$\Pi$} does nothing with high probability. The utility of {\small$\mathbf{h}^{N,\Pi}$} is that it is a diffusion on a compact set, and its invariant measure turns out to be the uniform measure on {\small$[-{}(\log N)^{1/2}{},{}(\log N)^{1/2}{})$}. These facts will all be important in our proof of Proposition \ref{prop:stestimate}, and we prove them when relevant.

In any case, the first step in proving Proposition \ref{prop:stestimate} is the following high probability bound on {\small$\mathbf{Av}^{\mathbf{T},\mathbf{X},\mathfrak{q}_{\n}}$}.
\begin{lemma}\label{lemma:stestimate-rep}
\fsp Fix any {\small$\mathrm{D}=\mathrm{O}(1)$}. There exists a path-space event {\small$\mathscr{E}\subseteq\mathscr{C}([0,1],\R^{\Z})$} of high probability such that the following, which uses notation to be explained after, holds {\color{black}for any {\small$\e>0$}, {\small$\t\in[0,1]$}, and {\small$|\x|\leq N^{\mathrm{D}}$}}:
\begin{align}
\mathbf{1}_{\mathscr{E}}|\mathbf{Av}^{\mathbf{T},\mathbf{X},\mathfrak{q}_{\n}}_{\t,\x}|^{2}&\lesssim_{\e} N^{\e}\Big(\mathfrak{t}_{\mathbf{Av}}^{-1}\int_{0}^{\mathfrak{t}_{\mathbf{Av}}}\exp(\mathrm{U}+\mathbf{h}^{N,\Pi}_{\t,\r,\x})\cdot\mathfrak{n}_{\mathbf{Av}}^{-1}\sum_{1\leq\j\leq\mathfrak{n}_{\mathbf{Av}}}\mathfrak{q}_{\n}[\tau_{\x+\j}\bphi_{\t-\r}]\cdot\mathbf{G}_{\j}[\tau_{\x}\bphi_{\t-\r}]\d\r\Big)^{2}.\label{eq:stestimate-repI}
\end{align}
The random variable {\small$\mathrm{U}$} is uniformly distributed on the torus {\small$[-{}(\log N)^{1/2}{},{}(\log N)^{1/2}{})$}. It is independent of all other random variables. The random variable {\small$\mathbf{G}_{\j}$} is defined in \eqref{eq:gkterm}.
\end{lemma}
We defer the proof to Section \ref{subsubsection:stestimate-rep}. (In fact, all preliminary ingredients to be listed below towards the proof of Proposition \ref{prop:stestimate} will be similarly deferred until we combine them to prove Proposition \ref{prop:stestimate}, since their proofs are quite technical.) However, let us briefly clarify what Lemma \ref{lemma:stestimate-rep} states. In a nutshell, the \abbr{RHS} of \eqref{eq:stestimate-repI} comes from an earlier representation \eqref{eq:avldptxI1} and then doing some modifications that have no effect with high probability. The first is rewriting the ratio of {\small$\mathbf{Z}^{N}$}-factors in \eqref{eq:avldptxI1} in terms of the \abbr{RHS} of \eqref{eq:hnpi} without the {\small$\Pi$}-projection. But, as noted after \eqref{eq:hnpi}, this projection does nothing with high probability, so we can include it. It is also {\color{black}stated} there that the factor {\small$\exp(-\mathrm{U})$}, which lets us put in {\small$\exp(\mathrm{U})$} on the \abbr{RHS} of \eqref{eq:stestimate-repI}, satisfies {\small$\exp(-\mathrm{U})\lesssim_{\e}N^{\e}$} for any {\small$\e>0$}, explaining the {\small$N^{\e}$}-factor in \eqref{eq:stestimate-repI}. Finally, the spatial-average in \eqref{eq:avldptxI1} can be manipulated in a way that has no effect with high probability and that is similar to what we did in the proof of Lemma \ref{lemma:avldp} {\color{black}in order} to have the form of the spatial-average on the \abbr{RHS} of \eqref{eq:stestimate-repI}.

The next ingredient in the proof of Proposition \ref{prop:stestimate} is often called a ``Kipnis-Varadhan estimate". It is an estimate for a specialization of the \abbr{RHS} of \eqref{eq:stestimate-repI} to ``fluctuating" {\small$\mathfrak{q}_{\n}$}.
\begin{lemma}\label{lemma:stestimate-kv}
\fsp Fix any {\small$\e>0$}, and recall the setting of Lemma \ref{lemma:stestimate-rep}. Suppose that {\small$\mathfrak{f}_{\j}$} satisfy the following constraints:
\begin{enumerate}
\item The quantity {\small$\mathfrak{f}_{\j}[\bphi]$} depends only on {\small$\bphi_{\w}$} for {\small$\w\in\{1,\ldots,\mathfrak{l}_{\star}\}$}, where {\small$\mathfrak{l}_{\star}>0$} may depend on {\small$N$}.
\item We have {\small$\E^{\sigma}\mathfrak{f}_{\j}=0$} for any {\small$\sigma\in\R$} (see Definition \ref{definition:gcmeasure}) and {\small$\E^{0}|\mathfrak{f}_{\j}|^{2}<\infty$}, all for each {\small$\j$}.
\end{enumerate}
Then, if we let {\small$\E^{\mathrm{path},0}$} denote {\color{black}the} expectation with respect to the law of {\small$\t\mapsto\bphi_{\t}$} with initial data {\small$\bphi_{0}\sim\mathbb{P}^{0}$}, we have
\begin{align}
&\E^{\mathrm{path},0}\Big(\mathfrak{t}_{\mathbf{Av}}^{-1}\int_{0}^{\mathfrak{t}_{\mathbf{Av}}}\exp(\mathrm{U}+\mathbf{h}^{N,\Pi}_{\t,\r,\x})\cdot\mathfrak{n}_{\mathbf{Av}}^{-1}\sum_{\j=1,\ldots,\mathfrak{n}_{\mathbf{Av}}}\mathfrak{f}_{\j}[\tau_{\x+\j}\bphi_{\t-\r}]\cdot\mathbf{G}_{\j}[\tau_{\x}\bphi_{\t-\r}]\d\r\Big)^{2}\nonumber\\
&\lesssim_{\e} N^{\e}\cdot N^{-2}\mathfrak{t}_{\mathbf{Av}}^{-1}\mathfrak{n}_{\mathbf{Av}}^{-1}\cdot\mathfrak{l}_{\star}^{3}\cdot\max_{\j=1,\ldots,\mathfrak{n}_{\mathbf{Av}}}\E^{0}|\mathfrak{f}_{\j}|^{2}. \label{eq:stestimate-kvI}
\end{align}
\end{lemma}
Assumption (2) in Lemma \ref{lemma:stestimate-kv} states that {\small$\mathfrak{f}_{\j}$} are fluctuating. Assumption (1) states that the different {\small$\mathfrak{f}_{\j}[\tau_{\j}\bphi]$} are essentially independent (up to some scale {\small$\mathfrak{l}_{\star}$}). Therefore, averaging in space and time yields square-root \abbr{CLT}-type cancellations as indicated by the \abbr{RHS} of \eqref{eq:stestimate-kvI}. Finally, we note that {\small$\r\mapsto\mathrm{U}+\mathbf{h}^{N,\Pi}_{\t,\r,\x}$} has initial data {\small$\mathrm{U}$} at time {\small$\r=0$}; in particular, it starts at its invariant measure, which will be an important technical input.

In our application of \eqref{eq:stestimate-kvI}, we \emph{will} always have {\small$\E^{0}|\mathfrak{f}_{\j}|^{2}\lesssim\mathfrak{l}_{\star}^{-3}$}, so the powers of {\small$\mathfrak{l}_{\star}$} are ultimately harmless.

Let us now explain our application of Lemma \ref{lemma:stestimate-kv} in more detail. In a nutshell, we will follow the argument in Section 3 of \cite{DGP} (see also Sections 3 and 4 in \cite{GJ15}), which uses Lemma \ref{lemma:stestimate-kv} to estimate the expectation of the \abbr{RHS} of \eqref{eq:stestimate-repI}. Our goal is to explain this; first, we require the following notation for local (ergodic) averages.
\begin{notation}\label{notation:localav}
\fsp Consider any {\small$\mathfrak{q}_{\n}$} from Lemma \ref{lemma:stestimate-rep}. Fix any {\small$\ell\geq2\mathfrak{l}$}, where {\small$\mathfrak{l}$} is the support length in Lemma \ref{lemma:stestimate-rep}. Following notation from \cite{DGP}, we define the following function on {\small$\R^{\Z}$}:
\begin{align}
\boldsymbol{\Psi}_{\mathfrak{q}_{\n}}[\ell,\bphi]:=\E^{0}\Big(\mathfrak{q}_{\n}[\boldsymbol{\psi}]\Big|\ell^{-1}(\boldsymbol{\psi}_{1}+\ldots+\boldsymbol{\psi}_{\ell})=\ell^{-1}(\bphi_{1}+\ldots+\bphi_{\ell})\Big). \label{eq:localavI}
\end{align}
We clarify that {\small$\boldsymbol{\psi}\in\R^{\Z}$} is the integration variable in {\small$\E^{0}$} above.
\end{notation}
We now record the following estimate on {\small$\boldsymbol{\Psi}_{\mathfrak{q}_{\n}}$}-terms. In a nutshell, it follows by an equivalence of ensembles that compares the conditional expectation in \eqref{eq:localavI} to an appropriate {\small$\E^{\sigma}$}-expectation with {\small$\sigma=\ell^{-1}(\bphi_{1}+\ldots+\bphi_{\ell})$}, and a Taylor expansion in {\small$\sigma$} (in which we can take advantage of the {\small$\mathrm{Jet}_{\k}^{\perp}$} spaces from Definition \ref{definition:jets}).
\begin{lemma}\label{lemma:stestimate-psi}
\fsp Take any {\small$\mathfrak{q}_{\n}$} from Lemma \ref{lemma:stestimate-rep}. Fix any {\small$N\gtrsim\ell\geq2\mathfrak{l}$}, where {\small$\mathfrak{l}$} is the support length from Lemma \ref{lemma:stestimate-rep}. We have the following second moment estimate:
\begin{align}
\E^{0}|\boldsymbol{\Psi}_{\mathfrak{q}_{\n}}[\ell,\cdot]|^{2}\lesssim\ell^{-\frac32}.\label{eq:stestimate-psiI}
\end{align}
\end{lemma}
We are now in position to prove Proposition \ref{prop:stestimate}, using the schematic from \cite{DGP,GJ15}.
\begin{proof}[Proof of Proposition \ref{prop:stestimate}]
Fix {\small$\ell\geq2\mathfrak{l}$}, where, again, {\small$\mathfrak{l}$} is the support length from Lemma \ref{lemma:stestimate-rep}. Consider the function {\small$\mathfrak{f}_{\j}[\bphi]:=\boldsymbol{\Psi}_{\mathfrak{q}_{\n}}[\ell,\bphi]-\boldsymbol{\Psi}_{\mathfrak{q}_{\n}}[2\ell,\bphi]$} from Notation \ref{notation:localav}. Note that the \abbr{RHS} of \eqref{eq:localavI} is unchanged if we replace {\small$\E^{0}$} by {\small$\E^{\sigma}$}; this follows by  (11) in \cite{DGP}. So, by the tower property of conditional expectation, we have {\small$\E^{\sigma}\mathfrak{f}_{\j}=0$} for any {\small$\sigma\in\R$}. Moreover, {\small$\mathfrak{f}_{\j}[\tau_{\j}\bphi]$} depends on {\small$\bphi_{\w}$} for {\small$\w\in\j+\{1,\ldots,2\ell\}$}, so we can use Lemma \ref{lemma:stestimate-kv} with {\small$\mathfrak{l}_{\star}\lesssim\ell$} to get
\begin{align}
&\E^{\mathrm{path},0}\Big(\mathfrak{t}_{\mathbf{Av}}^{-1}\int_{0}^{\mathfrak{t}_{\mathbf{Av}}}\exp(\mathrm{U}+\mathbf{h}^{N,\Pi}_{\t,\r,\x})\cdot\mathfrak{n}_{\mathbf{Av}}^{-1}\sum_{\j=1,\ldots,\mathfrak{n}_{\mathbf{Av}}}\Big(\boldsymbol{\Psi}_{\mathfrak{q}_{\n}}[\ell,\tau_{\x+\j}\bphi_{\t-\r}]-\boldsymbol{\Psi}_{\mathfrak{q}_{\n}}[2\ell,\tau_{\x+\j}\bphi_{\t-\r}]\Big)\cdot\mathbf{G}_{\j}[\tau_{\x}\bphi_{\t-\r}]\d\r\Big)^{2}\nonumber\\
&\lesssim_{\e} N^{\e}\cdot N^{-2}\mathfrak{t}_{\mathbf{Av}}^{-1}\mathfrak{n}_{\mathbf{Av}}^{-1}\cdot\ell^{3}\cdot\max_{\j=1,\ldots,\mathfrak{n}_{\mathbf{Av}}}\E^{0}|\mathfrak{f}_{\j}|^{2}\nonumber\\
&\lesssim N^{\e}\cdot N^{-2}\mathfrak{t}_{\mathbf{Av}}^{-1}\mathfrak{n}_{\mathbf{Av}}^{-1}\lesssim N^{-\frac73+\frac{23}{2}\delta_{\mathbf{S}}+\e},\label{eq:stestimateI1}
\end{align}
{\color{black}where the last line follows by \eqref{eq:stestimate-psiI}, by {\small$\mathfrak{t}_{\mathbf{Av}}=N^{-2/3-10\delta_{\mathbf{S}}}$}, and {\small$\mathfrak{n}_{\mathbf{Av}}=N^{1-\delta_{\mathbf{S}}}$}}. Now, take any large (but still {\small$\mathrm{O}(1)$}{\color{black})} constant {\small$\ell_{0}$}, and define {\small$\ell_{\i+1}:=2\ell_{\i}$} for any {\small$\i\geq0$}. We let {\small$\mathfrak{i}$} be the smallest positive integer such that {\small$\ell_{\mathfrak{i}}\geq N^{1-\nu}$} for any small but fixed {\small$\nu>0$}; note that {\small$\mathfrak{i}\lesssim \log N$}. By a telescoping sum, we have the following:
\begin{align*}
\mathfrak{q}_{\n}[\bphi]=\mathfrak{q}_{\n}[\bphi]-\boldsymbol{\Psi}_{\mathfrak{q}_{\n}}[\ell_{0},\bphi]+\Big\{\sum_{\i=0,\ldots,\mathfrak{i}-1}\Big(\boldsymbol{\Psi}_{\mathfrak{q}_{\n}}[\ell_{\i},\bphi]-\boldsymbol{\Psi}_{\mathfrak{q}_{\n}}[\ell_{\i+1},\bphi]\Big)\Big\}+\boldsymbol{\Psi}_{\mathfrak{q}_{\n}}[\ell_{\mathfrak{i}},\bphi].
\end{align*}
If we combine the previous display with \eqref{eq:stestimateI1}, the {\color{black}Cauchy-Schwarz} inequality, and {\small$\mathfrak{i}\lesssim\log N$}, then 
\begin{align}
&\E^{\mathrm{path},0}\Big(\mathfrak{t}_{\mathbf{Av}}^{-1}\int_{0}^{\mathfrak{t}_{\mathbf{Av}}}\exp(\mathrm{U}+\mathbf{h}^{N,\Pi}_{\t,\r,\x})\cdot\mathfrak{n}_{\mathbf{Av}}^{-1}\sum_{\j=1,\ldots,\mathfrak{n}_{\mathbf{Av}}}\mathfrak{q}_{\n}[\tau_{\x+\j}\bphi_{\t-\r}]\cdot\mathbf{G}_{\j}[\tau_{\x}\bphi_{\t-\r}]\d\r\Big)^{2}\nonumber\\
&\lesssim_{\e}N^{\e}\E^{\mathrm{path},0}\Big(\mathfrak{t}_{\mathbf{Av}}^{-1}\int_{0}^{\mathfrak{t}_{\mathbf{Av}}}\exp(\mathrm{U}+\mathbf{h}^{N,\Pi}_{\t,\r,\x})\cdot\mathfrak{n}_{\mathbf{Av}}^{-1}\sum_{\j=1,\ldots,\mathfrak{n}_{\mathbf{Av}}}\boldsymbol{\Psi}_{\mathfrak{q}_{\n}}[\ell_{\mathfrak{i}},\tau_{\x+\j}\bphi_{\t-\r}]\cdot\mathbf{G}_{\j}[\tau_{\x}\bphi_{\t-\r}]\d\r\Big)^{2}+N^{-\frac73+\frac{23}{2}\delta_{\mathbf{S}}+\e}.\nonumber
\end{align}
Let us again use Cauchy-Schwarz, the deterministic bounds {\small$|\mathbf{G}_{\j}|\lesssim1$} and {\small$|\mathrm{U}+\mathbf{h}^{N,\Pi}_{\t,\r,\x}|\lesssim(\log N)^{1/2}$} (see Lemma \ref{lemma:stestimate-rep}), and \eqref{eq:stestimate-psiI} to obtain the following estimate (again for any {\small$\e>0$}):
\begin{align*}
&\E^{\mathrm{path},0}\Big(\mathfrak{t}_{\mathbf{Av}}^{-1}\int_{0}^{\mathfrak{t}_{\mathbf{Av}}}\exp(\mathrm{U}+\mathbf{h}^{N,\Pi}_{\t,\r,\x})\cdot\mathfrak{n}_{\mathbf{Av}}^{-1}\sum_{\j=1,\ldots,\mathfrak{n}_{\mathbf{Av}}}\boldsymbol{\Psi}_{\mathfrak{q}_{\n}}[\ell_{\mathfrak{i}},\tau_{\x+\j}\bphi_{\t-\r}]\cdot\mathbf{G}_{\j}[\tau_{\x}\bphi_{\t-\r}]\d\r\Big)^{2}\lesssim_{\e}N^{-3+\e+3\nu}.
\end{align*}
(We clarify that since {\small$\mathbb{P}^{0}$} is an invariant measure for {\small$\t\mapsto\bphi_{\t,\cdot}$}, the law of {\small$\bphi_{\t-\r,\cdot}$} is given by {\small$\sim\mathbb{P}^{0}$} under the measure in {\small$\E^{\mathrm{path},0}$}.) Combining the previous two displays implies the following estimate for any {\small$\e>0$}:
\begin{align*}
&\E^{\mathrm{path},0}\Big(\mathfrak{t}_{\mathbf{Av}}^{-1}\int_{0}^{\mathfrak{t}_{\mathbf{Av}}}\exp(\mathrm{U}+\mathbf{h}^{N,\Pi}_{\t,\r,\x})\cdot\mathfrak{n}_{\mathbf{Av}}^{-1}\sum_{\j=1,\ldots,\mathfrak{n}_{\mathbf{Av}}}\mathfrak{q}_{\n}[\tau_{\x+\j}\bphi_{\t-\r}]\cdot\mathbf{G}_{\j}[\tau_{\x}\bphi_{\t-\r}]\d\r\Big)^{2}\lesssim_{\e}N^{-\frac73+\frac{23}{2}\delta_{\mathbf{S}}+\e}.
\end{align*}
We now plug the above display into the \abbr{RHS} \eqref{eq:stestimate-repI}. This gives the following estimate, in which {\small$\mathscr{E}\subseteq\mathscr{C}([0,1],\R^{\Z})$} is a high probability event in path-space:
\begin{align*}
\E^{\mathrm{path},0}\mathbf{1}_{\mathscr{E}}|\mathbf{Av}^{\mathbf{T},\mathbf{X},\mathfrak{q}_{\n}}_{\t,\x}|^{2}&\lesssim_{\e}N^{-\frac73+\frac{23}{2}\delta_{\mathbf{S}}+\e}.
\end{align*}
Finally, if {\small$\E^{\mathrm{path}}$} is {\color{black}the} expectation with respect to the law of the dynamics {\small$\t\mapsto\bphi_{\t}$} but with initial data satisfying Assumption \ref{assump:noneq}, then because the dynamics in {\small$\E^{\mathrm{path}}$} and {\small$\E^{\mathrm{path},0}$} are the same, the density of the law in {\small$\E^{\mathrm{path}}$} with respect to that in {\small$\E^{\mathrm{path},0}$} is equal to the density between the laws of the respective initial data. In particular, by \eqref{eq:noneq}, the change-of-measure factor we must pay to go from {\small$\E^{\mathrm{path},0}$} to {\small$\E^{\mathrm{path}}$} is {\small$\lesssim N^{1/3-\gamma_{\mathrm{data}}}$}. Thus, the previous display combines with \eqref{eq:noneq} to give
\begin{align*}
\E^{\mathrm{path}}\mathbf{1}_{\mathscr{E}}|\mathbf{Av}^{\mathbf{T},\mathbf{X},\mathfrak{q}_{\n}}_{\t,\x}|^{2}&\lesssim_{\e}N^{\frac13-\gamma_{\mathrm{data}}}N^{-\frac73+\frac{23}{2}\delta_{\mathbf{S}}+\e}=N^{-2-\gamma_{\mathrm{data}}+\frac{23}{2}\delta_{\mathbf{S}}+\e}.
\end{align*}
Choosing {\small$\delta_{\mathbf{S}},\e$} small enough now implies the desired estimate \eqref{eq:stestimateI}, so the proof is done.
\end{proof}
\subsection{Proof of Lemmas \ref{lemma:stestimate-rep}, \ref{lemma:stestimate-kv}, and \ref{lemma:stestimate-psi}}
We now show the ingredients used in the proof of Proposition \ref{prop:stestimate}.
\subsubsection{Proof of Lemma \ref{lemma:stestimate-rep}}\label{subsubsection:stestimate-rep}
Recall the following from the proof of Lemma \ref{lemma:avldp} for integer {\small$\k\geq1$}:
\begin{align*}
\mathbf{G}_{\k}[\tau_{\x}\bphi_{\t}]:=\exp\Big\{\lambda N^{-\frac12}(\bphi_{\t,\x+1}+\ldots+\bphi_{\t,\x+\k})\Big\}\cdot\mathbf{1}\Big[|\exp\Big(\lambda N^{-\frac12}(\bphi_{\t,\x+1}+\ldots+\bphi_{\t,\x+\k})\Big)|\lesssim1\Big].
\end{align*}
We also recall {\small$\mathbf{h}^{N,\Pi}$} from \eqref{eq:hnpi}. We claim that for some high-probability path-space event {\small$\mathscr{E}$}, we have
\begin{align}
\mathbf{1}_{\mathscr{E}}\mathbf{Av}^{\mathbf{T},\mathbf{X},\mathfrak{q}_{\n}}_{\t,\x}&=\mathbf{1}_{\mathscr{E}}\cdot\mathfrak{t}_{\mathbf{Av}}^{-1}\int_{0}^{\mathfrak{t}_{\mathbf{Av}}}\exp(\mathbf{h}^{N,\Pi}_{\t,\r,\x})\cdot\mathfrak{n}_{\mathbf{Av}}^{-1}\sum_{\j=1,\ldots,\mathfrak{n}_{\mathbf{Av}}}\mathfrak{q}_{\n}[\tau_{\x+\j}\bphi_{\t-\r}]\mathbf{G}_{\j}[\tau_{\x}\bphi_{\t-\r}]\d\r.\label{eq:stestimate-repI1}
\end{align}
To see this, recall \eqref{eq:avldptxI1}. Take the factor {\small$\mathbf{Z}^{N}_{\t-\r,\x}(\mathbf{Z}^{N}_{\t,\x})^{-1}=\exp(\lambda\mathbf{j}^{N}_{\t-\r,\x}-\lambda\mathbf{j}^{N}_{\t,\x}-\lambda\mathscr{R}_{\lambda}\r)$} therein. By Lemma \ref{lemma:jnestimate}, and because {\small$\lambda\mathscr{R}_{\lambda}=\mathrm{O}(1)$} and {\small$\r\leq\mathfrak{t}_{\mathbf{Av}}\lesssim N^{-2/3}$}, with high probability, we have {\small$\lambda\mathbf{j}^{N}_{\t-\r-,\x}-\lambda\mathbf{j}^{N}_{\t,\x}-\lambda\mathscr{R}_{\lambda}\r=\mathrm{O}(1)$}. So, with high probability, applying the {\small$\Pi$} projection does nothing to {\small$\lambda\mathbf{j}^{N}_{\t-\r-,\x}-\lambda\mathbf{j}^{N}_{\t,\x}-\lambda\mathscr{R}_{\lambda}\r$}, which implies the identity {\small$\mathbf{Z}^{N}_{\t-\r,\x}(\mathbf{Z}^{N}_{\t,\x})^{-1}=\exp(\mathbf{h}^{N,\Pi}_{\t,\r,\x})$} with high probability. Then, take the factor {\small$\mathbf{Av}^{\mathbf{X},\mathfrak{q}_{\n}}_{\t-\r,\x}$} in \eqref{eq:avldptxI1}; with high probability, it is given by \eqref{eq:avldpI0}, but with the following adjustments:
\begin{itemize}
\item Replace {\small$\t$} with {\small$\t-\r$}.
\item Replace {\small$\wt{\mathfrak{q}}_{\n}$} with {\small$\mathfrak{q}_{\n}$}, and drop the term {\small$\mathrm{O}(N^{-\mathrm{D}_{2}})$}. Indeed, this big-Oh error comes from replacing {\small$\mathfrak{q}_{\n}$} with {\small$\wt{\mathfrak{q}}_{\n}$} (see \eqref{eq:avldpI-1}); by dropping the {\small$\mathrm{O}(N^{-\mathrm{D}_{2}})$} term, we are just {\color{black}undoing} the replacement of  {\small$\mathfrak{q}_{\n}$} with {\small$\wt{\mathfrak{q}}_{\n}$} (see \eqref{eq:avldpI-1}).
\end{itemize}
Plugging this formula for {\small$\mathbf{Av}^{\mathbf{X},\mathfrak{q}_{\n}}_{\t-\r,\x}$} into \eqref{eq:avldptxI1} gives \eqref{eq:stestimateI1}. It now suffices to multiply by {\small$1=\exp(-\mathrm{U})\exp(\mathrm{U})$} on the \abbr{RHS} of \eqref{eq:stestimate-repI1}, move the factor of {\small$\exp(\mathrm{U})$} inside the {\small$\d\r$}-integration, and use the bound {\small$\exp(-\mathrm{U})\lesssim_{\e}N^{\e}$} for any {\small$\e>0$} (since {\small$|\mathrm{U}|\lesssim(\log N)^{1/2}$}). This shows \eqref{eq:stestimate-repI} and completes the proof. \qed
\subsubsection{Proof of Lemma \ref{lemma:stestimate-kv}}
Consider the process {\small$\r\mapsto(\mathrm{U}+\mathbf{h}^{N,\Pi}_{\t,\r,\x},\bphi_{\t-\r,\cdot})\in[-(\log N)^{1/2},(\log N)^{1/2})\times\R^{\Z}$}. It is a Markov process whose infinitesimal generator is a differential operator {\small$\mathscr{L}_{\mathbf{h},\bphi}$} acting on smooth functions of {\small$(\mathbf{h},\bphi)\in[-(\log N)^{1/2},(\log N)^{1/2})\times\R^{\Z}$}. This infinitesimal generator has the form
\begin{align}
\mathscr{L}_{\mathbf{h},\bphi}&=\mathscr{L}_{\mathbf{h}}+\mathscr{L}_{\bphi},\label{eq:stestimate-kvI0a}
\end{align}
where {\small$\mathscr{L}_{\mathbf{h}}$} comes from the {\small$\r\mapsto\mathrm{U}+\mathbf{h}^{N,\Pi}_{\t,\r,\x}$}-dynamics, and {\small$\mathscr{L}_{\bphi}$} comes from the {\small$\r\mapsto\bphi_{\t-\r}$} dynamics. By the formula \eqref{eq:hnpi} for {\small$\mathbf{h}^{N,\Pi}$} and the \abbr{SDE} \eqref{eq:curr}, we have
\begin{align}
\mathscr{L}_{\mathbf{h}}&=N\partial_{\mathbf{h}}^{2}+\mathsf{f}_{N}[\tau_{\x}\bphi]\partial_{\mathbf{h}},\label{eq:stestimate-kvI0b}
\end{align}
where {\small$\partial_{\mathbf{h}}$} denotes {\color{black}the} derivative with respect to {\small$\mathbf{h}$} in the torus {\small$[-(\log N)^{1/2},(\log N)^{1/2})$}, and {\small$\mathsf{f}:\R^{\Z}\to\R$} is a smooth local function depending on {\small$N$}, whose exact form is unimportant. On the other hand, the generator of the process {\small$\r\mapsto\bphi_{\t-\r,\cdot}$} is computed in Appendix \ref{subsection:generator} to be
\begin{align}
\mathscr{L}_{\bphi}&=\mathscr{L}_{\bphi,\mathrm{S}}-\mathscr{L}_{\bphi,\mathrm{A}},\label{eq:stestimate-kvI0c}\\
\mathscr{L}_{\bphi,\mathrm{S}}&=N^{2}\sum_{\w\in\Z}\Big\{(\partial_{\bphi_{\w}}-\partial_{\bphi_{\w+1}})^{2}-(\mathscr{U}'[\bphi_{\w}]-\mathscr{U}'[\bphi_{\w+1}])(\partial_{\bphi_{\w}}-\partial_{\bphi_{\w+1}})\Big\},\label{eq:stestimate-kvI0d}\\
\mathscr{L}_{\bphi,\mathrm{A}}&=N^{\frac32}\sum_{\w\in\Z}\wt{\mathbf{F}}[\tau_{\w}\bphi]\partial_{\bphi_{\w}}.\label{eq:stestimate-kvI0e}
\end{align}
The term {\small$\wt{\mathbf{F}}$} is defined in \eqref{eq:phi-nl}. Now, let {\small$\mathbb{P}^{\mathrm{Unif}}$} be {\color{black}the} uniform measure on the torus {\small$[-(\log N)^{1/2},(\log N)^{1/2})$}. We claim that the product measure {\small$\mathbb{P}^{\mathrm{Unif}}\otimes\mathbb{P}^{\sigma}$} is invariant for the ``total" generator {\small$\mathscr{L}_{\mathbf{h},\bphi}$}. To deduce that this measure is invariant for {\small$\mathscr{L}_{\mathbf{h}}$} is to prove that for any {\small$\bphi\in\R^{\Z}$}, the measure {\small$\mathbb{P}^{\mathrm{Unif}}$} is invariant for the operator in \eqref{eq:stestimate-kvI0b}, which follows by integration-by-parts on the torus. Thus, to prove {\color{black}the latter} invariance reduces to showing that {\small$\mathbb{P}^{\sigma}$} is invariant for {\small$\mathscr{L}_{\bphi}$}, since {\small$\mathscr{L}_{\bphi}$} is independent of the {\small$\mathbf{h}\in[-(\log N)^{1/2},(\log N)^{1/2})$} variable; we show this in Appendix \ref{subsection:generator}. So, {\small$\r\mapsto(\mathrm{U}+\mathbf{h}^{N,\Pi}_{\t,\r,\x},\bphi_{\t-\r,\cdot})\in[-(\log N)^{1/2},(\log N)^{1/2})\times\R^{\Z}$} is a continuous-time Markov process with invariant measure initial data. With this input, we get the following ``Kipnis-Varadhan inequality" of Lemma 2.4 in \cite{KLO}: 
\begin{align}
&\E^{\mathrm{path},0}\Big(\mathfrak{t}_{\mathbf{Av}}^{-1}\int_{0}^{\mathfrak{t}_{\mathbf{Av}}}\exp(\mathrm{U}+\mathbf{h}^{N,\Pi}_{\t,\r,\x})\cdot\mathfrak{n}_{\mathbf{Av}}^{-1}\sum_{\j=1,\ldots,\mathfrak{n}_{\mathbf{Av}}}\mathfrak{f}_{\j}[\tau_{\x+\j}\bphi_{\t-\r}]\cdot\mathbf{G}_{\j}[\tau_{\x}\bphi_{\t-\r}]\d\r\Big)^{2}\nonumber\\
&\lesssim\mathfrak{t}_{\mathbf{Av}}^{-2}\mathfrak{n}_{\mathbf{Av}}^{-2}\int_{0}^{\mathfrak{t}_{\mathbf{Av}}}\Big\|\exp(\mathbf{h})\sum_{\j=1,\ldots,\mathfrak{n}_{\mathbf{Av}}}\mathfrak{f}_{\j}[\tau_{\j}\bphi]\mathbf{G}_{\j}[\bphi]\Big\|_{-1}^{2}\d\r\lesssim\mathfrak{t}_{\mathbf{Av}}^{-1}\mathfrak{n}_{\mathbf{Av}}^{-2}\Big\|\exp(\mathbf{h})\sum_{\j=1,\ldots,\mathfrak{n}_{\mathbf{Av}}}\mathfrak{f}_{\j}[\tau_{\j}\bphi]\mathbf{G}_{\j}[\bphi]\Big\|_{-1}^{2}.\label{eq:stestimate-kvI1a}
\end{align}
Above, the {\small$\|\cdot\|_{-1}$}-norm is with respect to functions of {\small$(\mathbf{h},\bphi)\in[-(\log N)^{1/2},(\log N)^{1/2})\times\R^{\Z}$}; it is given by
\begin{align}
&\Big\|\exp(\mathbf{h})\sum_{\j=1,\ldots,\mathfrak{n}_{\mathbf{Av}}}\mathfrak{f}_{\j}[\tau_{\j}\bphi]\mathbf{G}_{\j}[\bphi]\Big\|_{-1}^{2}\nonumber\\
&=\sup_{\mathsf{F},\mathfrak{g}}\Big\{2\E^{\mathrm{stat}}\Big(\exp(\mathbf{h})\sum_{\j=1,\ldots,\mathfrak{n}_{\mathbf{Av}}}\mathfrak{f}_{\j}[\tau_{\j}\bphi]\mathbf{G}_{\j}[\bphi]\Big)\cdot\mathsf{F}[\mathbf{h}]\mathfrak{g}[\bphi] + \E^{\mathrm{stat}}(\mathsf{F}[\mathbf{h}]\mathfrak{g}[\bphi])\cdot\mathscr{L}_{\mathbf{h},\bphi}(\mathsf{F}[\mathbf{h}]\mathfrak{g}[\bphi])\Big\}.\label{eq:stestimate-kvI1b}
\end{align}
Above, the supremum is over smooth functions {\small$\mathsf{F}$} of {\small$\mathbf{h}$} and smooth, local functions {\small$\mathfrak{g}$} of {\small$\bphi$}. Also, {\small$\E^{\mathrm{stat}}=\E^{\mathrm{Unif}}\E^{0}$} denotes {\color{black}the} expectation with respect to the product stationary measure {\small$\mathrm{Unif}[-(\log N)^{1/2},(\log N)^{1/2})\otimes\mathbb{P}^{0}$}.

We will now estimate \eqref{eq:stestimate-kvI1b}, starting with the last term therein. We claim that 
\begin{align*}
&\E^{\mathrm{stat}}(\mathsf{F}[\mathbf{h}]\mathfrak{g}[\bphi])\cdot\mathscr{L}_{\mathbf{h},\bphi}(\mathsf{F}[\mathbf{h}]\mathfrak{g}[\bphi])\nonumber\\
&=\E^{\mathrm{stat}}|\mathfrak{g}[\bphi]|^{2}\cdot\mathsf{F}[\mathbf{h}]\mathscr{L}_{\mathbf{h}}\mathsf{F}[\mathbf{h}]+\E^{\mathrm{stat}}|\mathsf{F}[\mathbf{h}]|^{2}\cdot\mathfrak{g}[\bphi]\mathscr{L}_{\bphi}\mathfrak{g}[\bphi]\\
&=\E^{0}|\mathfrak{g}[\phi]|^{2}\E^{\mathrm{Unif}}\mathsf{F}[\mathbf{h}]\cdot N\partial_{\mathbf{h}}^{2}\mathsf{F}[\mathbf{h}]+\E^{\mathrm{Unif}}|\mathsf{F}[\mathbf{h}]|^{2}\E^{0}\mathfrak{g}[\bphi]\mathscr{L}_{\bphi,\mathrm{S}}\mathfrak{g}[\bphi]\\
&\leq\E^{\mathrm{Unif}}|\mathsf{F}[\mathbf{h}]|^{2}\E^{0}\mathfrak{g}[\bphi]\mathscr{L}_{\bphi,\mathrm{S}}\mathfrak{g}[\bphi].
\end{align*}
The first identity follows from \eqref{eq:stestimate-kvI0a}. The second estimate follows first by factorizing the product expectation as {\small$\E^{\mathrm{stat}}=\E^{\mathrm{Unif}}\E^{0}$}. Then, we note that the symmetric part of {\small$\mathscr{L}_{\mathbf{h}}$} is the second-derivative {\small$N\partial_{\mathbf{h}}^{2}$}, as the first derivative in \eqref{eq:stestimate-kvI0b} picks up a sign when we integrate by parts in {\small$\mathbf{h}$} (while keeping {\small$\bphi$} fixed!). Similarly, the symmetric part of {\small$\mathscr{L}_{\bphi}$} with respect to {\small$\mathbb{P}^{0}$} is {\small$\mathscr{L}_{\bphi,\mathrm{S}}$} as we show in {Appendix \ref{subsection:generator}}. The final inequality above follows because of the calculation {\small$\E^{\mathrm{Unif}}\mathsf{F}[\mathbf{h}]\partial_{\mathbf{h}}^{2}\mathsf{F}[\mathbf{h}]=-\E^{\mathrm{Unif}}|\partial_{\mathbf{h}}\mathsf{F}|^{2}\leq0$} due to integration-by-parts. 

Let {\small$\mathbb{I}_{\j}$} be the support of {\small$\bphi\mapsto\mathfrak{f}_{\j}[\tau_{\j}\bphi]$}, i.e. the smallest interval such that {\small$\mathfrak{f}_{\j}[\tau_{\j}\bphi]$} depends only on {\small$\bphi_{\w}$} for {\small$\w\in\mathbb{I}_{\j}$}. By assumption in Lemma \ref{lemma:stestimate-kv}, we have {\small$|\mathbb{I}_{\j}|\lesssim\mathfrak{l}_{\star}$} and {\small$\inf\mathbb{I}_{\j}\geq\j+1$}. We claim that for some {\small$\mathrm{c}_{1},\mathrm{c}_{2}>0$} independent of {\small$N$}, we have
\begin{align*}
\E^{0}\mathfrak{g}[\bphi]\mathscr{L}_{\bphi,\mathrm{S}}\mathfrak{g}[\bphi]&=-\mathrm{c}_{1}N^{2}\sum_{\w\in\Z}\E^{0}|(\partial_{\bphi_{\w}}-\partial_{\bphi_{\w+1}})\mathfrak{g}[\bphi]|^{2}\\
&\leq-\mathrm{c}_{2}\mathfrak{l}_{\star}^{-1}N^{2}\sum_{\j=1,\ldots,\mathfrak{n}_{\mathbf{Av}}}\sum_{\w,\w+1\in\mathbb{I}_{\j}}\E^{0}|(\partial_{\bphi_{\w}}-\partial_{\bphi_{\w+1}})\mathfrak{g}[\bphi]|^{2}\\
&=:-\mathfrak{c}_{2}\mathfrak{l}_{\star}^{-1}N^{2}\sum_{\j=1,\ldots,\mathfrak{n}_{\mathbf{Av}}}\mathfrak{D}_{\mathbb{I}_{\j}}^{0}[\mathfrak{g}].
\end{align*}
The first identity follows by a standard and direct integration-by-parts calculation using \eqref{eq:gcmeasure} and \eqref{eq:stestimate-kvI0d}, which relates the Dirichlet form and the generator of an It\^{o} diffusion. The second line follows because {\small$\mathbb{I}_{\j}$} overlaps with at most {\small$\lesssim\mathfrak{l}_{\star}$}-many other {\small$\mathbb{I}_{\k}$} sets, so for each {\small$\mathbb{I}_{\j}$} and each pair of neighboring points {\small$\w,\w+1\in\mathbb{I}_{\j}$}, the corresponding term {\small$\E^{0}|(\partial_{\bphi_{\w}}-\partial_{\bphi_{\w+1}})\mathfrak{g}[\bphi]|^{2}$} is counted at most {\small$\mathfrak{l}_{\star}$}-many times. The third line is standard notation for Dirichlet forms, and we introduce it to more cleanly present the sequel. By the previous three displays, we have
\begin{align}
&\Big\|\exp(\mathbf{h})\sum_{\j=1,\ldots,\mathfrak{n}_{\mathbf{Av}}}\mathfrak{f}_{\j}[\tau_{\j}\bphi]\mathbf{G}_{\j}[\bphi]\Big\|_{-1}^{2}\label{eq:stestimate-kvI1bb}\\
&\leq\sum_{\j=1,\ldots,\mathfrak{n}_{\mathbf{Av}}}\sup_{\mathsf{F},\mathfrak{g}}\Big\{2\E^{\mathrm{Unif}}\Big\{\exp(\mathbf{h})\mathsf{F}[\mathbf{h}]\Big\}\cdot\E^{0}\Big(\mathfrak{f}_{\j}[\tau_{\j}\bphi]\mathbf{G}_{\j}[\bphi]\mathfrak{g}[\bphi]\Big)\Big\}-\mathrm{c}_{2}\mathfrak{l}_{\star}^{-1}N^{2}\E^{\mathrm{Unif}}|\mathsf{F}[\mathbf{h}]|^{2}\mathfrak{D}_{\mathbb{I}_{\j}}^{0}[\mathfrak{g}]\Big\}.\nonumber
\end{align}
Now, in the supremum over {\small$\mathfrak{g}$}, we make the change of coordinates {\small$\mathfrak{g}\mapsto\mathfrak{g}\cdot(\E^{\mathrm{Unif}}\exp(\mathbf{h})\mathsf{F}[\mathbf{h}])\cdot(\E^{\mathrm{Unif}}|\mathsf{F}[\mathbf{h}]|^{2})^{-1}$}. This change of variables is invertible as long as restrict to nowhere-vanishing {\small$\mathsf{F}$} functions. (These, however, are dense, and since we take a supremum over {\small$\mathsf{F}$}, this restriction is allowed.) Thus, we deduce from \eqref{eq:stestimate-kvI1bb} that
\begin{align}
&\Big\|\exp(\mathbf{h})\sum_{\j=1,\ldots,\mathfrak{n}_{\mathbf{Av}}}\mathfrak{f}_{\j}[\tau_{\j}\bphi]\mathbf{G}_{\j}[\bphi]\Big\|_{-1}^{2}\label{eq:stestimate-kvI1c}\\
&\leq\sum_{\j=1,\ldots,\mathfrak{n}_{\mathbf{Av}}}\sup_{\mathsf{F}}\frac{(\E^{\mathrm{Unif}}\exp(\mathbf{h})\mathsf{F}[\mathbf{h}])^{2}}{\E^{\mathrm{Unif}}\mathsf{F}[\mathbf{h}]^{2}}\cdot\sup_{\mathfrak{g}}\Big\{2\E^{0}\Big(\mathfrak{f}_{\j}[\tau_{\j}\bphi]\mathbf{G}_{\j}[\bphi]\mathfrak{g}[\bphi]\Big)-\mathrm{c}_{2}\mathfrak{l}_{\star}^{-1}N^{2}\mathfrak{D}_{\mathbb{I}_{\j}}^{0}[\mathfrak{g}]\Big\}\nonumber\\
&\lesssim_{\e}N^{\e}\sum_{\j=1,\ldots,\mathfrak{n}_{\mathbf{Av}}}\sup_{\mathfrak{g}}\Big\{2\E^{0}\Big(\mathfrak{f}_{\j}[\tau_{\j}\bphi]\mathbf{G}_{\j}[\bphi]\mathfrak{g}[\bphi]\Big)-\mathrm{c}_{2}\mathfrak{l}_{\star}^{-1}N^{2}\mathfrak{D}_{\mathbb{I}_{\j}}^{0}[\mathfrak{g}]\Big\},\nonumber
\end{align}
where the last line follows since the sup over {\small$\mathsf{F}$} in the second line is the variational formula for {\small$\E^{\mathrm{Unif}}|\exp(\mathbf{h})|^{2}$}, and since {\small$|\mathbf{h}|\leq(\log N)^{1/2}$}, we have {\small$\exp(\mathbf{h})\lesssim_{\nu}N^{\nu}$} for any {\small$\nu>0$}. Now, for any {\small$\w,\w+1\in\mathbb{I}_{\j}$}, we claim that the following estimates hold in which {\small$\mathrm{c}>0$} is independent of {\small$N$}:
\begin{align*}
\E^{0}|(\partial_{\bphi_{\w}}-\partial_{\bphi_{\w+1}})\mathfrak{g}[\bphi]|^{2}&\geq\mathrm{c}^{2}\E^{0}\mathbf{G}_{\j}[\bphi]^{2}|(\partial_{\bphi_{\w}}-\partial_{\bphi_{\w+1}})\mathfrak{g}[\bphi]|^{2}=\mathrm{c}^{2}\E^{0}|(\partial_{\bphi_{\w}}-\partial_{\bphi_{\w+1}})(\mathbf{G}_{\j}[\bphi]\mathfrak{g}[\bphi])|^{2}.
\end{align*}
The lower bound follows because {\small$|\mathbf{G}_{\j}[\bphi]|\lesssim1$} deterministically. The identity follows since {\small$\mathbf{G}_{\j}[\bphi]$} does not depend on {\small$\bphi_{\w},\bphi_{\w+1}$}. (Indeed, by \eqref{eq:gkterm}, {\small$\mathbf{G}_{\j}[\bphi]$} is supported to the left of {\small$\j+1\leq\inf\mathbb{I}_{\j}$}.)  If we sum over {\small$\w,\w+1\in\mathbb{I}_{\j}$}, then we deduce {\small$\mathfrak{D}_{\mathbb{I}_{\j}}^{0}[\mathfrak{g}]\geq\mathrm{c}^{2}\mathfrak{D}_{\mathbb{I}_{\j}}^{0}[\mathbf{G}_{\j}\mathfrak{g}]$}, so the estimate \eqref{eq:stestimate-kvI1c} extends to the following:
\begin{align}
&\Big\|\exp(\mathbf{h})\sum_{\j=1,\ldots,\mathfrak{n}_{\mathbf{Av}}}\mathfrak{f}_{\j}[\tau_{\j}\bphi]\mathbf{G}_{\j}[\bphi]\Big\|_{-1}^{2}\nonumber\\
&\lesssim_{\e}N^{\e}\sum_{\j=1,\ldots,\mathfrak{n}_{\mathbf{Av}}}\sup_{\mathfrak{g}}\Big\{2\E^{0}\Big(\mathfrak{f}_{\j}[\tau_{\j}\bphi]\mathbf{G}_{\j}[\bphi]\mathfrak{g}[\bphi]\Big)-\mathrm{c}^{2}\mathrm{c}_{2}\mathfrak{l}_{\star}^{-1}N^{2}\mathfrak{D}_{\mathbb{I}_{\j}}^{0}[\mathbf{G}_{\j}\mathfrak{g}]\Big\}\nonumber\\
&\lesssim_{\e}N^{\e}\sum_{\j=1,\ldots,\mathfrak{n}_{\mathbf{Av}}}\sup_{\wt{\mathfrak{g}}}\Big\{2\E^{0}\Big(\mathfrak{f}_{\j}[\tau_{\j}\bphi]\wt{\mathfrak{g}}[\bphi]\Big)-\mathrm{c}^{2}\mathrm{c}_{2}\mathfrak{l}_{\star}^{-1}N^{2}\mathfrak{D}_{\mathbb{I}_{\j}}^{0}[\wt{\mathfrak{g}}]\Big\}.\nonumber
\end{align}
We clarify that the last line follows since in the second line, it is enough to replace {\small$\mathbf{G}_{\j}\mathfrak{g}$} by {\small$\wt{\mathfrak{g}}$} and take a supremum over functions that are in the image of {\color{black}the} multiplication by {\small$\mathbf{G}_{\j}$}. However, we can extend to all local, smooth functions in the supremum, {\color{black}obtaining an upper bound}. Next, we use the change-of-variables {\small$\wt{\mathfrak{g}}\mapsto N^{-2}\mathfrak{l}_{\star}\mathrm{c}^{-2}\mathrm{c}_{2}^{-1}\wt{\mathfrak{g}}$}. Then, we proceed as in the proof of Lemma 3.5 in \cite{DGP} to bound the resulting supremum and get the last line below:
\begin{align*}
&\Big\|\exp(\mathbf{h})\sum_{\j=1,\ldots,\mathfrak{n}_{\mathbf{Av}}}\mathfrak{f}_{\j}[\tau_{\j}\bphi]\mathbf{G}_{\j}[\bphi]\Big\|_{-1}^{2}\nonumber\\
&\lesssim N^{\e}N^{-2}\mathfrak{l}_{\star}\sum_{\j=1,\ldots,\mathfrak{n}_{\mathbf{Av}}}\sup_{\wt{\mathfrak{g}}}\Big\{2\E^{0}\mathfrak{f}_{\j}[\tau_{\j}\bphi]\wt{\mathfrak{g}}[\bphi]-\mathfrak{D}_{\mathbb{I}_{\j}}^{0}[\wt{\mathfrak{g}}]\Big\}\nonumber\\
&\lesssim N^{\e}N^{-2}\mathfrak{l}_{\star}\sum_{\j=1,\ldots,\mathfrak{n}_{\mathbf{Av}}}\mathfrak{l}_{\star}^{2}\cdot\E^{0}\mathfrak{f}_{\j}[\tau_{\j}\bphi]^{2}\lesssim N^{\e}N^{-2}\mathfrak{l}_{\star}^{3}\cdot\mathfrak{n}_{\mathbf{Av}}\cdot\max_{\j=1,\ldots,\mathfrak{n}_{\mathbf{Av}}}\E^{0}|\mathfrak{f}_{\j}|^{2}.
\end{align*}
If we plug the previous display into \eqref{eq:stestimate-kvI1a}, then the desired estimate \eqref{eq:stestimate-kvI} follows, so the proof is complete. \qed
\subsubsection{Proof of Lemma \ref{lemma:stestimate-psi}}
We first use the equivalence of ensembles; in particular, by Corollary B.3 in \cite{DGP}, we have the deterministic estimate below, which uses notation to be explained afterwards:
\begin{align}
|\boldsymbol{\Psi}_{\mathfrak{q}_{\n}}[\ell,\bphi]|&\lesssim|\E^{\rho[\bphi]}\mathfrak{q}_{\n}|+\ell^{-1}\sigma_{\rho[\bphi]}^{2}\partial_{\rho}^{2}\E^{\rho}\mathfrak{q}_{\n}|_{\rho=\rho[\bphi]}+\ell^{-\frac32}\E^{\rho[\bphi]}|\mathfrak{q}_{\n}|^{2}.\label{eq:stestimate-psiI1}
\end{align}
Above, we used the notation {\small$\rho[\bphi]=\ell^{-1}(\bphi_{1}+\ldots+\bphi_{\ell})$}. Also, {\small$\sigma_{\rho}^{2}:=\E^{\rho}|\bphi_{0}-\rho|^{2}$}. Before we proceed, we will first require the following technical estimates for all moments and expectations involved. Because {\small$\mathfrak{q}_{\n}$} admits a pointwise local polynomial bound (see Lemma \ref{lemma:stestimate-rep}), an elementary analysis shows that {\small$|\partial_{\rho}^{\k}\E^{\rho}\mathfrak{q}_{\n}|\lesssim1+|\rho|^{\mathrm{C}}$} for some {\small$\mathrm{C}=\mathrm{O}(1)$} and any {\small$\k=\mathrm{O}(1)$}. Similarly, we have {\small$\sigma_{\rho}^{2}\lesssim1+|\rho|^{\mathrm{C}}$}. 

Now, assume that {\small$\mathfrak{q}_{\n}\in\mathrm{Jet}_{2}^{\perp}$}; see Definition \ref{definition:jets}. In this case, we expand
\begin{align*}
\E^{\rho[\bphi]}\mathfrak{q}_{\n}&=\E^{0}\mathfrak{q}_{\n}+\partial_{\rho}\E^{\rho}\mathfrak{q}_{\n}|_{\rho=0}\cdot\rho[\bphi]+\tfrac12\partial_{\rho}^{2}\E^{\rho}\mathfrak{q}_{\n}|_{\rho=0}\cdot|\rho[\bphi]|^{2}+\mathrm{O}\Big(|\rho[\bphi]|^{3}(1+|\rho[\bphi]|^{\mathrm{C}})\Big),\\
\partial_{\rho}^{2}\E^{\rho}\mathfrak{q}_{\n}|_{\rho=\rho[\bphi]}&=\partial_{\rho}^{2}\E^{\rho}\mathfrak{q}_{\n}|_{\rho=0}+\mathrm{O}\Big(|\rho[\bphi]|(1+|\rho[\bphi]|^{\mathrm{C}})\Big).
\end{align*}
Now, we note that {\small$\rho[\bphi]=\ell^{-1}(\bphi_{1}+\ldots+\bphi_{\ell})$} is sub-Gaussian with variance {\small$\lesssim\ell^{-1}$} with respect to {\small$\mathbb{P}^{0}$}, since {\small$\bphi_{\w}$} are i.i.d. and sub-Gaussian (see Assumption \ref{assump:potential}). In particular, we have {\small$\E^{0}|\rho[\bphi]|^{p}\lesssim_{p}\ell^{-p/2}$} for any {\small$p\geq2$}. Now, we use this estimate after plugging the above display into the \abbr{RHS} of \eqref{eq:stestimate-psiI1} to deduce the desired bound \eqref{eq:stestimate-psiI}. 

If {\small$\mathfrak{q}_{\n}=N^{-1/2}\mathfrak{f}$} with {\small$\mathfrak{f}\in\mathrm{Jet}_{1}^{\perp}$}, the same argument works, except we Taylor expand {\small$\E^{\rho}\mathfrak{q}_{\n}$} to one less degree. This means we lose a factor of {\small$\rho[\bphi]$} (which means we gain a factor of {\small$\ell^{1/2}$} after taking expectation). However, we also multiply by {\small$N^{-1/2}$}, and since {\small$\ell\lesssim N$}, the total factor that we must include is {\small$N^{-1/2}\ell^{1/2}\lesssim1$}. If {\small$\mathfrak{q}_{\n}=N^{-1}\mathfrak{f}$} with {\small$\mathfrak{f}\in\mathrm{Jet}_{0}^{\perp}$}, then we gain another factor of {\small$N^{-1/2}\ell^{1/2}\lesssim1$}. Thus, the desired estimate \eqref{eq:stestimate-psiI} holds. \qed
%
%
%
\section{Proof of Proposition \ref{prop:stochheat}}\label{section:stochheat}
In this section, we will first list all the ingredients needed to prove Proposition \ref{prop:stochheat}. Then, before we give their proofs, we combine them to show Proposition \ref{prop:stochheat}. Next, we prove each ingredient. We remark that the proof of Proposition \ref{prop:stochheat} will only need Lemmas \ref{lemma:kzetamodify} and \ref{lemma:kzetaestimate-long}. However, their proofs use the other lemmas in the list of ingredients below, which is why we record them altogether.
\subsection{A technical modification}
The first step in the proof of Proposition \ref{prop:stochheat} is to consider the \abbr{SDE} \eqref{eq:kzeta-sde} for the heat kernel of interest and modify its coefficients (in a way that ultimately does nothing with high probability). To be precise, we let {\small$\mathbf{K}^{N,\zeta,\sim}_{\s,\t,\x,\y}$} be a function of {\small$(\s,\t,\x,\y)\in[0,\infty)^{2}\times\Z^{2}$} with {\small$\s\leq\t$} which satisfies the following modification of \eqref{eq:kzeta-sde} (where the change, marked in red, is explained afterwards):
\begin{align}
&\d\mathbf{K}^{N,\zeta,\sim}_{\s,\t,\x,\y}=\mathscr{T}_{N}\mathbf{K}^{N,\zeta,\sim}_{\s,\t,\x,\y}\d\t+{\boldsymbol{\chi}^{(\zeta_{\mathrm{large}})}_{\x}}[\mathscr{S}^{N}\star(\sqrt{2}\lambda N^{\frac12}\mathbf{R}^{N,\wedge}_{\t,\cdot}\mathbf{K}^{N,\zeta,\sim}_{\s,\t,\cdot,\y}\d\mathbf{b}_{\t,\cdot})]_{\x}\nonumber\\
&+N{\boldsymbol{\chi}^{(\zeta)}_{\x}}{\ocolor{red}\mathbf{1}_{\mathfrak{t}_{\mathrm{ap}}\geq\t}}\sum_{\substack{\n=1,\ldots,\mathrm{K}\\\m=0,\ldots,\mathrm{M}\\0<|\mathfrak{l}_{1}|,\ldots,|\mathfrak{l}_{\m}|\lesssim \mathfrak{n}_{\mathbf{Av}}}}\tfrac{1}{|\mathfrak{l}_{1}|\ldots|\mathfrak{l}_{\m}|}\mathrm{c}_{N,\n,\mathfrak{l}_{1},\ldots,\mathfrak{l}_{\m}}\grad^{\mathbf{X}}_{\mathfrak{l}_{1}}\ldots\grad^{\mathbf{X}}_{\mathfrak{l}_{\m}}[\mathscr{S}^{N}\star(\mathbf{Av}^{\mathbf{T},\mathbf{X},\mathfrak{q}_{\n}}_{\t,\cdot}\cdot\mathbf{R}^{N}_{\t,\cdot}\mathbf{K}^{N,\zeta,\sim}_{\s,\t,\cdot,\y})]_{\x}\d\t\nonumber\\
&+N\boldsymbol{\chi}^{(\zeta)}_{\x}\sum_{\substack{\n=1,\ldots,\mathrm{K}\\\m=0,\ldots,\mathrm{M}\\0<|\mathfrak{l}_{1}|,\ldots,|\mathfrak{l}_{\m}|\lesssim \mathfrak{n}_{\mathbf{Av}}}}\tfrac{1}{|\mathfrak{l}_{1}|\ldots|\mathfrak{l}_{\m}|}\mathrm{c}_{N,\n,\mathfrak{l}_{1},\ldots,\mathfrak{l}_{\m}}\grad^{\mathbf{T},\mathrm{av}}_{\mathfrak{t}_{\mathbf{Av}}}\grad^{\mathbf{X}}_{\mathfrak{l}_{1}}\ldots\grad^{\mathbf{X}}_{\mathfrak{l}_{\m}}\Big({\ocolor{red}\mathbf{1}_{\mathfrak{t}_{\mathrm{ap}}\geq\t}}[\mathscr{S}^{N}\star(\mathbf{Av}^{\mathbf{X},\mathfrak{q}_{\n}}_{\t,\cdot}\cdot\mathbf{R}^{N}_{\t,\cdot}\mathbf{K}^{N,\zeta,\sim}_{\s,\t,\cdot,\y})]_{\x}\Big)\d\t\nonumber\\
&+\boldsymbol{\chi}^{(\zeta)}_{\x}{\ocolor{red}\mathbf{1}_{\mathfrak{t}_{\mathrm{ap}}\geq\t}}\cdot\mathrm{Err}[\mathbf{R}^{N}_{\t,\cdot}\mathbf{K}^{N,\zeta,\sim}_{\s,\t,\cdot,\y}]_{\x}\d\t.\label{eq:kzetasim-sde}
\end{align}
Again, we will take the initial data {\small$\mathbf{K}^{N,\zeta,\sim}_{\s,\s,\x,\y}=\mathbf{1}_{\x=\y}$}. The object {\small$\mathfrak{t}_{\mathrm{ap}}$} is defined to be the following stopping time, which provides us important a priori estimates (hence the subscript):
\begin{align}
\mathfrak{t}_{\mathrm{ap}}:=&\inf\Big\{\s\in[0,\infty): \max_{|\x|\leq N^{2+2\zeta+\zeta_{\mathrm{large}}}}N^{-\delta_{\mathbf{S}}}|\bphi_{\s,\x}|+\max_{|\x|\leq N^{2+2\zeta+\zeta_{\mathrm{large}}}}N^{\frac13\delta_{\mathbf{S}}}|\mathbf{R}^{N}_{\s,\x}-1|\geq 1\Big\}\label{eq:tap1}\\
&\wedge\inf\Big\{\s\in[0,\infty): \max_{\n=1,\ldots,\mathrm{K}}\max_{|\x|\leq N^{2+2\zeta+\zeta_{\mathrm{large}}}}|\mathbf{Av}^{\mathbf{X},\mathfrak{q}_{\n}}_{\s,\x}|\geq N^{-\frac12+\delta_{\mathbf{S}}}\Big\}\label{eq:tap2}\\
&\wedge\inf\Big\{\s\in[0,\infty): \max_{\n=1,\ldots,\mathrm{K}}\max_{|\x|\leq N^{2+2\zeta+\zeta_{\mathrm{large}}}}|\mathbf{Av}^{\mathbf{T},\mathbf{X},\mathfrak{q}_{\n}}_{\s,\x}|\geq N^{-\frac12+\delta_{\mathbf{S}}}\Big\}.\label{eq:tap3}
\end{align}
(Recall the small parameter {\small$\delta_{\mathbf{S}}>0$} from Definition \ref{definition:zsmooth}.) We clarify that the above is a stopping time because the time-average {\small$\mathbf{Av}^{\mathbf{T},\mathbf{X},\mathfrak{q}_{\n}}$} is backwards in time (see Definition \ref{definition:eq-operators}).

Because of {Lemmas \ref{lemma:rbound}, \ref{lemma:phibound}, \ref{lemma:avldp}, and \ref{lemma:avldptx}}, we know {\small$\mathfrak{t}_{\mathrm{ap}}\geq1$} with high probability. Therefore, it is a standard It\^{o} calculus result that the modification in red in \eqref{eq:kzetasim-sde} is harmless with high probability (some care is needed to handle the infinite-dimensional nature of \eqref{eq:kzetasim-sde} and \eqref{eq:kzeta-sde}). We record this in the following lemma.
\begin{lemma}\label{lemma:kzetamodify}
\fsp With probability {\small$1$}, there exists a path-wise solution to \eqref{eq:kzetasim-sde} for {\small$\t\in[\s,1]$} and {\small$\x,\y\in\Z$}, such that the map {\small$(\s,\t)\mapsto\mathbf{K}^{N,\zeta,\sim}_{\s,\t,\x,\y}$} is jointly continuous on the domain {\small$\s\leq\t$}. Moreover,  {\small$\mathbf{K}^{N,\zeta}_{\s,\t,\x,\y}=\mathbf{K}^{N,\zeta,\sim}_{\s,\t,\x,\y}$} for all {\small$0\leq\s\leq\t\leq1$} and {\small$\x,\y\in\Z$} with high probability.
\end{lemma}
\subsection{Estimates for the heat kernel of a lattice \abbr{SHE}-type equation}
Because of the red indicators in \eqref{eq:kzetasim-sde}, we have \emph{deterministic} control for the last two lines in \eqref{eq:kzetasim-sde}. It will then be achievable to establish estimates for the heat kernel for the \abbr{SDE} corresponding to \eqref{eq:kzetasim-sde} but without the second line therein. In particular, let {\small$\mathbf{L}^{N,\zeta,\sim}_{\s,\t,\x,\y}$} be a function of {\small$(\s,\t,\x,\y)\in[0,\infty)^{2}\times\Z^{2}$} with {\small$\s\leq\t$} that solves the following \abbr{SDE}:
\begin{align}
&\d\mathbf{L}^{N,\zeta,\sim}_{\s,\t,\x,\y}=\mathscr{T}_{N}\mathbf{L}^{N,\zeta,\sim}_{\s,\t,\x,\y}\d\t+{\boldsymbol{\chi}^{(\zeta_{\mathrm{large}})}_{\x}}[\mathscr{S}^{N}\star(\sqrt{2}\lambda N^{\frac12}\mathbf{R}^{N,\wedge}_{\t,\cdot}\mathbf{L}^{N,\zeta,\sim}_{\s,\t,\cdot,\y}\d\mathbf{b}_{\t,\cdot})]_{\x}\nonumber\\
&+N\boldsymbol{\chi}^{(\zeta)}_{\x}\sum_{\substack{\n=1,\ldots,\mathrm{K}\\\m=0,\ldots,\mathrm{M}\\0<|\mathfrak{l}_{1}|,\ldots,|\mathfrak{l}_{\m}|\lesssim \mathfrak{n}_{\mathbf{Av}}}}\tfrac{1}{|\mathfrak{l}_{1}|\ldots|\mathfrak{l}_{\m}|}\mathrm{c}_{N,\n,\mathfrak{l}_{1},\ldots,\mathfrak{l}_{\m}}\grad^{\mathbf{T},\mathrm{av}}_{\mathfrak{t}_{\mathbf{Av}}}\grad^{\mathbf{X}}_{\mathfrak{l}_{1}}\ldots\grad^{\mathbf{X}}_{\mathfrak{l}_{\m}}\Big({\ocolor{red}\mathbf{1}_{\mathfrak{t}_{\mathrm{ap}}\geq\t}}[\mathscr{S}^{N}\star(\mathbf{Av}^{\mathbf{X},\mathfrak{q}_{\n}}_{\t,\cdot}\cdot\mathbf{R}^{N}_{\t,\cdot}\mathbf{L}^{N,\zeta,\sim}_{\s,\t,\cdot,\y})]_{\x}\Big)\d\t\nonumber\\
&+\boldsymbol{\chi}^{(\zeta)}_{\x}{\ocolor{red}\mathbf{1}_{\mathfrak{t}_{\mathrm{ap}}\geq\t}}\cdot\mathrm{Err}[\mathbf{R}^{N}_{\t,\cdot}\mathbf{L}^{N,\zeta,\sim}_{\s,\t,\cdot,\y}]_{\x}\d\t.\label{eq:lzetasim-sde}
\end{align}
Again, we assume the initial data {\small$\mathbf{L}^{N,\zeta,\sim}_{\s,\s,\x,\y}=\mathbf{1}_{\x=\y}$}. We will think of {\small$\mathbf{L}^{N,\zeta,\sim}$} as a perturbation of the deterministic heat kernel corresponding to the {\small$\mathscr{T}_{N}$} operator; recall this heat kernel from \eqref{eq:heatkernel}.
\begin{lemma}\label{lemma:lzetasimestimate}
\fsp Fix any {\small$\kappa=\mathrm{O}(1)$} and {\small$\delta>0$}. There exists a high probability event on which we have
\begin{align}
\sup_{0\leq\s\leq\t\leq1}\sup_{\x\in\Z}\sum_{\y\in\Z}\exp(\tfrac{2\kappa|\x-\y|}{N})|\mathbf{L}^{N,\zeta,\sim}_{\s,\t,\x,\y}-\mathbf{H}^{N}_{\s,\t,\x,\y}|^{2}&\lesssim N^{-1+\delta}.\label{eq:lzetasimestimateII}
\end{align}
\end{lemma}
Let us clarify the estimate \eqref{eq:lzetasimestimateII}. The squared {\small$\ell^{2}(\Z)$}-norm of the heat kernel {\small$\mathbf{H}^{N}_{\s,\t,\x,\y}$} is {\color{black}of order} {\small$\lesssim N^{-1}|\t-\s|^{-1/2}$}; see {Proposition \ref{prop:hke}}, which includes a similar estimate even after including the exponential weight from the \abbr{LHS} of \eqref{eq:lzetasimestimateII}. The estimate \eqref{eq:lzetasimestimateII} asserts that {\small$\mathbf{L}^{N,\zeta,\sim}$} is perturbative (around {\small$\mathbf{H}^{N}$}) in this {\small$\ell^{2}(\Z)$}-norm topology, since the singular factor {\small$|\t-\s|^{-1/2}$} does not appear in \eqref{eq:lzetasimestimateII}.
\subsection{The main estimates for {\small$\mathbf{K}^{N,\zeta,\sim}$}}
We now use the {\small$\mathbf{L}^{N,\zeta,\sim}$} kernel to present the following Duhamel formula. (Because it is not quite the classical Duhamel formula, we will justify it carefully. One goal of Lemma \ref{lemma:lkzeta} below is to establish convenient notation for forthcoming analysis.)
\begin{lemma}\label{lemma:lkzeta}
\fsp For any {\small$(\s,\t,\x,\y)\in[0,\infty)^{2}\times\Z^{2}$} with {\small$\s\leq\t$}, we have
\begin{align}
\mathbf{K}^{N,\zeta,\sim}_{\s,\t,\x,\y}&=\mathbf{L}^{N,\zeta,\sim}_{\s,\t,\x,\y}+\int_{\s}^{\t}\sum_{\w\in\Z}\mathbf{L}^{N,\zeta,\sim}_{\r,\t,\x,\w}\cdot\boldsymbol{\Phi}_{\r}[\mathbf{K}^{N,\zeta,\sim}_{\s,\r,\cdot,\y}]_{\w}\d\r,\label{eq:lkzetaIa}
\end{align}
where, for convenience, {\small$\boldsymbol{\Phi}_{\r}$} is the operator acting on {\small$\w\mapsto\mathbf{K}^{N,\zeta}_{\s,\r,\w,\y}$} arising from the second line in \eqref{eq:kzetasim-sde}:
\begin{align}
\boldsymbol{\Phi}_{\r}[\mathbf{K}^{N,\zeta,\sim}_{\s,\r,\cdot,\y}]_{\w}&:=N\boldsymbol{\chi}^{(\zeta)}_{\x}{\ocolor{red}\mathbf{1}_{\mathfrak{t}_{\mathrm{ap}}\geq\r}}\sum_{\substack{\n=1,\ldots,\mathrm{K}\\\m=0,\ldots,\mathrm{M}\\0<|\mathfrak{l}_{1}|,\ldots,|\mathfrak{l}_{\m}|\lesssim \mathfrak{n}_{\mathbf{Av}}}}\tfrac{1}{|\mathfrak{l}_{1}|\ldots|\mathfrak{l}_{\m}|}\boldsymbol{\Phi}_{\r}^{\n,\m,\mathfrak{l}_{1},\ldots,\mathfrak{l}_{\m}}[\mathbf{K}^{N,\zeta,\sim}_{\s,\r,\cdot,\y}]_{\w},\label{eq:lkzetaIb}\\
\boldsymbol{\Phi}_{\r}^{\n,\m,\mathfrak{l}_{1},\ldots,\mathfrak{l}_{\m}}[\mathbf{K}^{N,\zeta,\sim}_{\s,\r,\cdot,\y}]_{\w}&:=\mathrm{c}_{N,\n,\mathfrak{l}_{1},\ldots,\mathfrak{l}_{\m}}\grad^{\mathbf{X}}_{\mathfrak{l}_{1}}\ldots\grad^{\mathbf{X}}_{\mathfrak{l}_{\m}}[\mathscr{S}^{N}\star(\mathbf{Av}^{\mathbf{T},\mathbf{X},\mathfrak{q}_{\n}}_{\r,\cdot}\cdot\mathbf{R}^{N}_{\t,\cdot}\mathbf{K}^{N,\zeta,\sim}_{\s,\r,\cdot,\y})]_{\w}.\label{eq:lkzetaIc}
\end{align}
\end{lemma}
We will eventually use Lemma \ref{lemma:lkzeta} to establish the following estimate for {\small$\mathbf{K}^{N,\zeta,\sim}$} \emph{for short times}.
\begin{lemma}\label{lemma:kzetaestimate-short}
\fsp Fix any {\small$\delta>0$} (independent of {\small$N$}). There exist constants {\small$\beta,\kappa>0$} (also independent of {\small$N$}) and a high probability event on which the following holds for any {\small$\s,\t\in[0,1]$} with {\small$\s\leq\t$} and {\small$|\t-\s|\leq N^{\beta}N^{-\zeta}$}:
\begin{align}
\sup_{\x\in\Z}\sum_{\y\in\Z}\exp(\tfrac{\kappa|\x-\y|}{N})\cdot|\mathbf{K}^{N,\zeta,\sim}_{\s,\t,\x,\y}|&\lesssim N^{\delta}.\label{eq:kzetaestimate-shortI}
\end{align}
\end{lemma}
We also need estimates for {\small$\mathbf{K}^{N,\zeta,\sim}$} beyond the ``short-time" regime, i.e. {\small$|\t-\s|\leq N^{\beta}N^{-\zeta}$}. (Indeed, we think of {\small$\beta$} as small, and {\small$\zeta$} is not necessarily small.) We record such an estimate below, which ultimately follows from \eqref{eq:kzetaestimate-shortI} and a repeated application of the Chapman-Kolmogorov equations for {\small$\mathbf{K}^{N,\zeta,\sim}$}.
\begin{lemma}\label{lemma:kzetaestimate-long}
\fsp Retain the setting of Lemma \ref{lemma:kzetaestimate-short}, but forget the constraint that {\small$|\t-\s|\leq N^{\beta}N^{-\zeta}$}. On the same high probability event as in Lemma \ref{lemma:kzetaestimate-short}, we have 
\begin{align}
\sup_{0\leq\s\leq\t\leq1}\sup_{\x\in\Z}\sum_{\y\in\Z}\exp(\tfrac{\kappa|\x-\y|}{N})|\mathbf{K}^{N,\zeta,\sim}_{\s,\t,\x,\y}|\lesssim \exp(N^{-\frac12\beta}N^{\zeta})N^{\delta}.\label{eq:kzetaestimate-longI}
\end{align}
\end{lemma}
\subsection{Proof of Proposition \ref{prop:stochheat}}
Combine Lemmas \ref{lemma:kzetamodify} and \ref{lemma:kzetaestimate-long}. \qed
\subsection{Proofs of the lemmas}
Let us now prove Lemmas \ref{lemma:kzetamodify}, \ref{lemma:lzetasimestimate}, \ref{lemma:lkzeta}, \ref{lemma:kzetaestimate-short}, and \ref{lemma:kzetaestimate-long}. We defer the proof of Lemma \ref{lemma:lzetasimestimate} to the end, since it is the most technically involved and fairly standard. Otherwise, we proceed in order.
\subsubsection{Proof of Lemma \ref{lemma:kzetamodify}}
Existence of continuous path-wise solutions to \eqref{eq:kzetasim-sde} follows from a standard Picard iteration as discussed after \eqref{eq:szeta-sde} and \eqref{eq:kzeta-sde}. Thus, it remains to show that the identity {\small$\mathbf{K}^{N,\zeta,\sim}_{\s,\t,\x,\y}=\mathbf{K}^{N,\zeta}_{\s,\t,\x,\y}$} holds with high probability. Consider the following difference of kernels stopped at {\small$\mathfrak{t}_{\mathrm{ap}}$}:
\begin{align}
\mathbf{U}^{N,\zeta}_{\s,\t,\x,\y}:=\mathbf{K}^{N,\zeta}_{\s,\t\wedge(\mathfrak{t}_{\mathrm{ap}}\vee\s),\x,\y}-\mathbf{K}^{N,\zeta,\sim}_{\s,\t\wedge(\mathfrak{t}_{\mathrm{ap}}\vee\s),\x,\y}.\label{eq:kzetamodifyI1}
\end{align}
We clarify that we stop the kernels at time {\small$\mathfrak{t}_{\mathrm{ap}}$} (but we also restrict to times after {\small$\s$}, hence the additional {\small$\vee\s$} on the \abbr{RHS}). Linearity of the \abbr{SDE}s \eqref{eq:kzeta-sde} and \eqref{eq:kzetasim-sde} (and their agreement before {\small$\mathfrak{t}_{\mathrm{ap}}$}) shows that {\small$\mathbf{U}^{N,\zeta}$} solves the same equation but with zero initial data for times {\small$\t\in[\s,\mathfrak{t}_{\mathrm{ap}}\vee\s]$}. Namely, for such {\small$\t$}, we have 
\begin{align}
&\d\mathbf{U}^{N,\zeta}_{\s,\t,\x,\y}=\mathscr{T}_{N}\mathbf{U}^{N,\zeta}_{\s,\t,\x,\y}\d\t+{\boldsymbol{\chi}^{(\zeta_{\mathrm{large}})}_{\x}}[\mathscr{S}^{N}\star(\sqrt{2}\lambda N^{\frac12}\mathbf{R}^{N,\wedge}_{\t,\cdot}\mathbf{U}^{N,\zeta}_{\s,\t,\cdot,\y}\d\mathbf{b}_{\t,\cdot})]_{\x}\nonumber\\
&+N\boldsymbol{\chi}^{(\zeta)}_{\x}{\ocolor{red}\mathbf{1}_{\mathfrak{t}_{\mathrm{ap}}\geq\t}}\sum_{\substack{\n=1,\ldots,\mathrm{K}\\\m=0,\ldots,\mathrm{M}\\0<|\mathfrak{l}_{1}|,\ldots,|\mathfrak{l}_{\m}|\lesssim \mathfrak{n}_{\mathbf{Av}}}}\tfrac{1}{|\mathfrak{l}_{1}|\ldots|\mathfrak{l}_{\m}|}\mathrm{c}_{N,\n,\mathfrak{l}_{1},\ldots,\mathfrak{l}_{\m}}\grad^{\mathbf{X}}_{\mathfrak{l}_{1}}\ldots\grad^{\mathbf{X}}_{\mathfrak{l}_{\m}}[\mathscr{S}^{N}\star(\mathbf{Av}^{\mathbf{T},\mathbf{X},\mathfrak{q}_{\n}}_{\t,\cdot}\cdot\mathbf{R}^{N}_{\t,\cdot}\mathbf{U}^{N,\zeta}_{\s,\t,\cdot,\y})]_{\x}\d\t\nonumber\\
&+N\boldsymbol{\chi}^{(\zeta)}_{\x}\sum_{\substack{\n=1,\ldots,\mathrm{K}\\\m=0,\ldots,\mathrm{M}\\0<|\mathfrak{l}_{1}|,\ldots,|\mathfrak{l}_{\m}|\lesssim \mathfrak{n}_{\mathbf{Av}}}}\tfrac{1}{|\mathfrak{l}_{1}|\ldots|\mathfrak{l}_{\m}|}\mathrm{c}_{N,\n,\mathfrak{l}_{1},\ldots,\mathfrak{l}_{\m}}\grad^{\mathbf{T},\mathrm{av}}_{\mathfrak{t}_{\mathbf{Av}}}\grad^{\mathbf{X}}_{\mathfrak{l}_{1}}\ldots\grad^{\mathbf{X}}_{\mathfrak{l}_{\m}}\Big({\ocolor{red}\mathbf{1}_{\mathfrak{t}_{\mathrm{ap}}\geq\t}}[\mathscr{S}^{N}\star(\mathbf{Av}^{\mathbf{X},\mathfrak{q}_{\n}}_{\t,\cdot}\cdot\mathbf{R}^{N}_{\t,\cdot}\mathbf{U}^{N,\zeta}_{\s,\t,\cdot,\y})]_{\x}\Big)\d\t\nonumber\\
&+\boldsymbol{\chi}^{(\zeta)}_{\x}{\ocolor{red}\mathbf{1}_{\mathfrak{t}_{\mathrm{ap}}\geq\t}}\cdot\mathrm{Err}[\mathbf{R}^{N}_{\t,\cdot}\mathbf{U}^{N,\zeta}_{\s,\t,\cdot,\y}]_{\x}\d\t.\label{eq:uzeta-sde}
\end{align}
In particular, {\small$\mathbf{U}^{N,\zeta}$} solves a linear \abbr{SDE} with zero initial data; it is then standard to deduce that {\small$\mathbf{U}^{N,\zeta}_{\s,\t,\x,\y}=0$} for all {\small$\t\in[\s,\mathfrak{t}_{\mathrm{ap}}\vee\s]$}. Since {\small$\mathfrak{t}_{\mathrm{ap}}=1$} with high probability by {Lemmas \ref{lemma:rbound}, \ref{lemma:phibound}, \ref{lemma:avldp}, and \ref{lemma:avldptx}}, the proof is finished. \qed

\begin{remark}\label{remark:kzetasim-unique}
\fsp The same argument shows that solutions to \eqref{eq:kzetasim-sde} with the same initial data are equal.
\end{remark}
\subsubsection{Proof of Lemma \ref{lemma:lkzeta}}
Per Remark \ref{remark:kzetasim-unique}, it suffices to show that the \abbr{RHS} of \eqref{eq:lkzetaIa} satisfies the \abbr{SDE} \eqref{eq:kzetasim-sde}. (We note that its initial data at {\small$\t=\s$} agrees with {\small$\mathbf{K}^{N,\zeta,\sim}_{\s,\s,\x,\y}$}, since the integral on the \abbr{RHS} of \eqref{eq:lkzetaIa} vanishes at {\small$\t=\s$}.) To this end, let us take the {\small$\t$}-differential of the second term on the \abbr{RHS} of \eqref{eq:lkzetaIa} in {\small$\t$}. {\color{black}When we do this, we obtain two terms.} The first of these is from differentiating the upper limit of integration; this lets us remove the integral and set {\small$\r=\t$}. In particular, the first term that we get is equal to
\begin{align*}
\sum_{\w\in\Z}\mathbf{L}^{N,\zeta,\sim}_{\t,\t,\x,\w}\cdot\boldsymbol{\Phi}_{\t}[\mathbf{K}^{N,\zeta,\sim}_{\s,\t,\cdot,\y}]_{\w}\d\t&=\boldsymbol{\Phi}_{\t}[\mathbf{K}^{N,\zeta,\sim}_{\s,\t,\cdot,\y}]_{\x}\d\t,
\end{align*}
where the second identity follows by {\small$\mathbf{L}^{N,\zeta,\sim}_{\t,\t,\x,\w}=\mathbf{1}_{\x=\w}$}. The second term that we get when we take the {\small$\t$}-differential of the second term on the \abbr{RHS} of \eqref{eq:lkzetaIa} comes from {\small$\d\mathbf{L}^{N,\zeta,\sim}_{\r,\t,\x,\w}$}, which we can evaluate via \eqref{eq:lzetasim-sde}. Moreover, the \abbr{SDE} \eqref{eq:lzetasim-sde} is linear in its solution {\small$\mathbf{L}^{N,\zeta,\sim}_{\s,\t,\cdot,\w}$} with coefficients depending only on {\small$\t$}. Thus, if we denote the second term on the \abbr{RHS} of \eqref{eq:lkzetaIa} by {\small$\mathbf{N}^{N,\zeta,\sim}_{\s,\t,\x,\y}$} (only for this proof), then the other term that we get is given by the following (which is just the \abbr{RHS} of \eqref{eq:lzetasim-sde} but replacing {\small$\mathbf{L}$} by {\small$\mathbf{N}$} everywhere):
\begin{align}
&\mathscr{T}_{N}\mathbf{N}^{N,\zeta,\sim}_{\s,\t,\x,\y}\d\t+{\ocolor{red}\mathbf{1}_{\mathfrak{t}_{\mathrm{ap}}\geq\t}}{\boldsymbol{\chi}^{(\zeta_{\mathrm{large}})}_{\x}}[\mathscr{S}^{N}\star(\sqrt{2}\lambda N^{\frac12}\mathbf{R}^{N,\wedge}_{\t,\cdot}\mathbf{N}^{N,\zeta,\sim}_{\s,\t,\cdot,\y}\d\mathbf{b}_{\t,\cdot})]_{\x}\nonumber\\
&+N\boldsymbol{\chi}^{(\zeta)}_{\x}\sum_{\substack{\n=1,\ldots,\mathrm{K}\\\m=0,\ldots,\mathrm{M}\\0<|\mathfrak{l}_{1}|,\ldots,|\mathfrak{l}_{\m}|\lesssim \mathfrak{n}_{\mathbf{Av}}}}\tfrac{1}{|\mathfrak{l}_{1}|\ldots|\mathfrak{l}_{\m}|}\mathrm{c}_{N,\n,\mathfrak{l}_{1},\ldots,\mathfrak{l}_{\m}}\grad^{\mathbf{T},\mathrm{av}}_{\mathfrak{t}_{\mathbf{Av}}}\grad^{\mathbf{X}}_{\mathfrak{l}_{1}}\ldots\grad^{\mathbf{X}}_{\mathfrak{l}_{\m}}\Big({\ocolor{red}\mathbf{1}_{\mathfrak{t}_{\mathrm{ap}}\geq\t}}[\mathscr{S}^{N}\star(\mathbf{Av}^{\mathbf{X},\mathfrak{q}_{\n}}_{\t,\cdot}\cdot\mathbf{R}^{N}_{\t,\cdot}\mathbf{N}^{N,\zeta,\sim}_{\s,\t,\cdot,\y})]_{\x}\Big)\d\t\nonumber\\
&+\boldsymbol{\chi}^{(\zeta)}_{\x}{\ocolor{red}\mathbf{1}_{\mathfrak{t}_{\mathrm{ap}}\geq\t}}\cdot\mathrm{Err}[\mathbf{R}^{N}_{\t,\cdot}\mathbf{N}^{N,\zeta,\sim}_{\s,\t,\cdot,\y}]_{\x}\d\t.\nonumber
\end{align}
The previous two displays yield the {\small$\t$}-differential for the second on term on the \abbr{RHS} of \eqref{eq:lkzetaIa}. If we evaluate the {\small$\t$}-differential of the first term on the \abbr{RHS} of \eqref{eq:lkzetaIa} via \eqref{eq:lzetasim-sde}, then we ultimately obtain that the {\small$\t$}-differential of the entire \abbr{RHS} of \eqref{eq:lkzetaIa} is given by
\begin{align}
&\mathscr{T}_{N}\mathbf{N}^{N,\zeta,\sim}_{\s,\t,\x,\y}\d\t+{\ocolor{red}\mathbf{1}_{\mathfrak{t}_{\mathrm{ap}}\geq\t}}{\boldsymbol{\chi}^{(\zeta_{\mathrm{large}})}_{\x}}[\mathscr{S}^{N}\star(\sqrt{2}\lambda N^{\frac12}\mathbf{R}^{N,\wedge}_{\t,\cdot}\mathbf{N}^{N,\zeta,\sim}_{\s,\t,\cdot,\y}\d\mathbf{b}_{\t,\cdot})]_{\x}\nonumber\\
&+N\boldsymbol{\chi}^{(\zeta)}_{\x}\sum_{\substack{\n=1,\ldots,\mathrm{K}\\\m=0,\ldots,\mathrm{M}\\0<|\mathfrak{l}_{1}|,\ldots,|\mathfrak{l}_{\m}|\lesssim \mathfrak{n}_{\mathbf{Av}}}}\tfrac{1}{|\mathfrak{l}_{1}|\ldots|\mathfrak{l}_{\m}|}\mathrm{c}_{N,\n,\mathfrak{l}_{1},\ldots,\mathfrak{l}_{\m}}\grad^{\mathbf{T},\mathrm{av}}_{\mathfrak{t}_{\mathbf{Av}}}\grad^{\mathbf{X}}_{\mathfrak{l}_{1}}\ldots\grad^{\mathbf{X}}_{\mathfrak{l}_{\m}}\Big({\ocolor{red}\mathbf{1}_{\mathfrak{t}_{\mathrm{ap}}\geq\t}}[\mathscr{S}^{N}\star(\mathbf{Av}^{\mathbf{X},\mathfrak{q}_{\n}}_{\t,\cdot}\cdot\mathbf{R}^{N}_{\t,\cdot}\mathbf{N}^{N,\zeta,\sim}_{\s,\t,\cdot,\y})]_{\x}\Big)\d\t\nonumber\\
&+\boldsymbol{\chi}^{(\zeta)}_{\x}{\ocolor{red}\mathbf{1}_{\mathfrak{t}_{\mathrm{ap}}\geq\t}}\cdot\mathrm{Err}[\mathbf{R}^{N}_{\t,\cdot}\mathbf{N}^{N,\zeta,\sim}_{\s,\t,\cdot,\y}]_{\x}\d\t\nonumber\\
&+\mathscr{T}_{N}\mathbf{L}^{N,\zeta,\sim}_{\s,\t,\x,\y}\d\t+{\ocolor{red}\mathbf{1}_{\mathfrak{t}_{\mathrm{ap}}\geq\t}}{\boldsymbol{\chi}^{(\zeta_{\mathrm{large}})}_{\x}}[\mathscr{S}^{N}\star(\sqrt{2}\lambda N^{\frac12}\mathbf{R}^{N,\wedge}_{\t,\cdot}\mathbf{L}^{N,\zeta,\sim}_{\s,\t,\cdot,\y}\d\mathbf{b}_{\t,\cdot})]_{\x}\nonumber\\
&+N\boldsymbol{\chi}^{(\zeta)}_{\x}\sum_{\substack{\n=1,\ldots,\mathrm{K}\\\m=0,\ldots,\mathrm{M}\\0<|\mathfrak{l}_{1}|,\ldots,|\mathfrak{l}_{\m}|\lesssim \mathfrak{n}_{\mathbf{Av}}}}\tfrac{1}{|\mathfrak{l}_{1}|\ldots|\mathfrak{l}_{\m}|}\mathrm{c}_{N,\n,\mathfrak{l}_{1},\ldots,\mathfrak{l}_{\m}}\grad^{\mathbf{T},\mathrm{av}}_{\mathfrak{t}_{\mathbf{Av}}}\grad^{\mathbf{X}}_{\mathfrak{l}_{1}}\ldots\grad^{\mathbf{X}}_{\mathfrak{l}_{\m}}\Big({\ocolor{red}\mathbf{1}_{\mathfrak{t}_{\mathrm{ap}}\geq\t}}[\mathscr{S}^{N}\star(\mathbf{Av}^{\mathbf{X},\mathfrak{q}_{\n}}_{\t,\cdot}\cdot\mathbf{R}^{N}_{\t,\cdot}\mathbf{L}^{N,\zeta,\sim}_{\s,\t,\cdot,\y})]_{\x}\Big)\d\t\nonumber\\
&+\boldsymbol{\chi}^{(\zeta)}_{\x}{\ocolor{red}\mathbf{1}_{\mathfrak{t}_{\mathrm{ap}}\geq\t}}\cdot\mathrm{Err}[\mathbf{R}^{N}_{\t,\cdot}\mathbf{L}^{N,\zeta,\sim}_{\s,\t,\cdot,\y}]_{\x}\d\t+\boldsymbol{\Phi}_{\t}[\mathbf{K}^{N,\zeta,\sim}_{\s,\t,\cdot,\y}]_{\x}\d\t.\nonumber
\end{align}
The \abbr{RHS} of \eqref{eq:lkzetaIa} is equal to {\small$\mathbf{L}^{N,\zeta,\sim}_{\s,\t,\x,\y}+\mathbf{N}^{N,\zeta,\sim}_{\s,\t,\x,\y}$}, so if we combine each term in the first three lines in the previous display with their {\small$\mathbf{L}^{N,\zeta,\sim}$}-counterparts in the last four lines above, then we deduce that the \abbr{RHS} of \eqref{eq:lkzetaIa} satisfies the same \abbr{SDE} \eqref{eq:kzetasim-sde} for {\small$\s\leq\t$} as {\small$\mathbf{K}^{N,\zeta,\sim}_{\s,\t,\x,\y}$}. As noted at the beginning of this proof, this shows \eqref{eq:lkzetaIa}.\qed
\subsubsection{Proof of Lemma \ref{lemma:kzetaestimate-short}}
Throughout this argument, every statement holds with high probability, even if it is not explicitly mentioned. Since the intersection of {\small$\mathrm{O}(1)$}-many high probability events also has high probability, the final conclusion holds with high probability. To estimate the \abbr{LHS} of \eqref{eq:kzetaestimate-shortI}, we will control the \abbr{RHS} of \eqref{eq:lkzetaIa} (after multiplying by the appropriate exponential weight and summing over {\small$\y\in\Z$}). We will first estimate {\small$\mathbf{L}^{N,\zeta,\sim}$}, the first term on the \abbr{RHS} of \eqref{eq:lkzetaIa}. First, recall the heat kernel {\small$\mathbf{H}^{N}$} for the {\small$\mathscr{T}_{N}$}-operator in \eqref{eq:heatkernel}. We write
\begin{align}
\sum_{\y\in\Z}\exp(\tfrac{\kappa|\x-\y|}{N})|\mathbf{L}^{N,\zeta,\sim}_{\s,\t,\x,\y}|&\leq\sum_{\y\in\Z}\exp(\tfrac{\kappa|\x-\y|}{N})|\mathbf{H}^{N}_{\s,\t,\x,\y}|+\sum_{\y\in\Z}\exp(\tfrac{\kappa|\x-\y|}{N})|\mathbf{L}^{N,\zeta,\sim}_{\s,\t,\x,\y}-\mathbf{H}^{N}_{\s,\t,\x,\y}|\nonumber\\
&\lesssim_{\kappa}1+\sum_{\y\in\Z}\exp(\tfrac{\kappa|\x-\y|}{N})|\mathbf{L}^{N,\zeta,\sim}_{\s,\t,\x,\y}-\mathbf{H}^{N}_{\s,\t,\x,\y}|,\nonumber
\end{align}
where the second estimate follows from heat kernel estimates in {Proposition \ref{prop:hke}}. For the second term in the last line above, we apply Cauchy-Schwarz with respect to the space-summation, and then we use \eqref{eq:lzetasimestimateII}. This implies that for any {\small$\delta>0$} fixed, the following holds with high probability:
\begin{align}
\sum_{\y\in\Z}\exp(\tfrac{\kappa|\x-\y|}{N})|\mathbf{L}^{N,\zeta,\sim}_{\s,\t,\x,\y}-\mathbf{H}^{N}_{\s,\t,\x,\y}|&\lesssim\Big(\sum_{\y\in\Z}\exp(\tfrac{-\kappa|\x-\y|}{N})\Big)^{\frac12}\Big(\sum_{\y\in\Z}\exp(\tfrac{3\kappa|\x-\y|}{N})|\mathbf{L}^{N,\zeta,\sim}_{\s,\t,\x,\y}-\mathbf{H}^{N}_{\s,\t,\x,\y}|^{2}\Big)^{\frac12}\nonumber\\
&\lesssim_{\kappa}N^{\frac12}N^{-\frac12+\frac12\delta}=N^{\frac12\delta}.\nonumber
\end{align}
If we combine the previous two displays, then for any {\small$\delta>0$}, we obtain the estimate below with high probability:
\begin{align}
\sum_{\y\in\Z}\exp(\tfrac{\kappa|\x-\y|}{N})|\mathbf{L}^{N,\zeta,\sim}_{\s,\t,\x,\y}|\lesssim_{\kappa}N^{\frac12\delta}.\label{eq:kzetaestimate-shortI1}
\end{align}
We now move on to the second term on the \abbr{RHS} of \eqref{eq:lkzetaIa}. For this term, we use a similar decomposition of {\small$\mathbf{L}^{N,\zeta,\sim}$} into {\small$\mathbf{H}^{N}$} and {\small$\mathbf{L}^{N,\zeta,\sim}-\mathbf{H}^{N}$}. More precisely, we have 
\begin{align}
\int_{\s}^{\t}\sum_{\w\in\Z}\mathbf{L}^{N,\zeta,\sim}_{\r,\t,\x,\w}\boldsymbol{\Phi}_{\r}[\mathbf{K}^{N,\zeta,\sim}_{\s,\r,\cdot,\y}]_{\w}\d\r&=\int_{\s}^{\t}\sum_{\w\in\Z}\mathbf{H}^{N}_{\r,\t,\x,\w}\cdot\boldsymbol{\Phi}_{\r}[\mathbf{K}^{N,\zeta,\sim}_{\s,\r,\cdot,\y}]_{\w}\d\r\label{eq:kzetaestimate-shortI2a}\\
&+\int_{\s}^{\t}\sum_{\w\in\Z}\Big(\mathbf{L}^{N,\zeta,\sim}_{\r,\t,\x,\w}-\mathbf{H}^{N}_{\r,\t,\x,\w}\Big)\cdot\boldsymbol{\Phi}_{\r}[\mathbf{K}^{N,\zeta,\sim}_{\s,\r,\cdot,\y}]_{\w}\d\r.\label{eq:kzetaestimate-shortI2b}
\end{align}
We claim that the following holds with high probability, where {\small$\beta_{\star}>0$} depends only on {\small$\gamma_{\mathrm{data}}>0$} from \eqref{eq:noneq}:
\begin{align}
\int_{0}^{1}\sum_{\substack{|\w|\lesssim N^{1+\zeta}}}\sum_{\substack{\n=1,\ldots,\mathrm{K}\\\m=0,\ldots,\mathrm{M}\\0<|\mathfrak{l}_{1}|,\ldots,|\mathfrak{l}_{\m}|\lesssim \mathfrak{n}_{\mathbf{Av}}}}\tfrac{1}{|\mathfrak{l}_{1}|\ldots|\mathfrak{l}_{\m}|}\cdot |\mathbf{Av}^{\mathbf{T},\mathbf{X},\mathfrak{q}_{\n}}_{\r,\w}|^{2}\d\r\lesssim N^{-2-\beta_{\star}}\cdot N^{1+\zeta}.\label{eq:kzetaestimate-shortI-key}
\end{align}
We explain this as follows. Because {\small$|\mathfrak{l}_{1}|^{-1}\ldots|\mathfrak{l}_{\m}|^{-1}$} sums (over {\small$\n,\m,\mathfrak{l}_{1},\ldots,\mathfrak{l}_{\m}$}) to {\small$(\log N)^{\mathrm{O}(1)}\lesssim_{\e}N^{\e}$}, by {Proposition \ref{prop:stestimate}}, we know that the expectation of the summands on the \abbr{LHS} of \eqref{eq:kzetaestimate-shortI-key} is {\small$\lesssim_{\e} N^{\e}N^{-2-\beta_{\mathrm{univ},2}}$} for some {\small$\beta_{\mathrm{univ},2}>0$} depending only on {\small$\gamma_{\mathrm{data}}$} from \eqref{eq:noneq}. Now, sum this bound {\small$\mathrm{O}(N^{1+\zeta})$}-many times and use the Markov inequality to deduce that \eqref{eq:kzetaestimate-shortI-key} holds with high probability.

Let us start with the second term \eqref{eq:kzetaestimate-shortI2b}. Again, we want to multiply it by {\small$\exp(\kappa|\x-\y|/N)$}, sum over {\small$\y\in\Z$}, and control the resulting quantity. In the sequel, {\small$\beta>0$} is a small constant to be determined later in the argument. We claim that if {\small$|\t-\s|\leq N^{\beta}N^{-\zeta}$}, then
\begin{align}
&\sum_{\y\in\Z}\exp(\tfrac{\kappa|\x-\y|}{N})\cdot\int_{\s}^{\t}\sum_{\w\in\Z}|\mathbf{L}^{N,\zeta,\sim}_{\r,\t,\x,\w}-\mathbf{H}^{N}_{\r,\t,\x,\w}|\cdot|\boldsymbol{\Phi}_{\r}[\mathbf{K}^{N,\zeta,\sim}_{\s,\r,\cdot,\y}]_{\w}|\d\r\nonumber\\
&\lesssim\int_{\s}^{\t}\sum_{\w\in\Z}\exp(\tfrac{\kappa|\x-\w|}{N})|\mathbf{L}^{N,\zeta,\sim}_{\r,\t,\x,\w}-\mathbf{H}^{N}_{\r,\t,\x,\w}|\cdot\sum_{\y\in\Z}\exp(\tfrac{\kappa|\w-\y|}{N})|\boldsymbol{\Phi}_{\r}[\mathbf{K}^{N,\zeta,\sim}_{\s,\r,\cdot,\y}]_{\w}|\d\r\nonumber\\
&\lesssim\Big(\int_{\s}^{\t}\sum_{\w\in\Z}\exp(\tfrac{2\kappa|\x-\w|}{N})|\mathbf{L}^{N,\zeta,\sim}_{\r,\t,\x,\w}-\mathbf{H}^{N}_{\r,\t,\x,\w}|^{2}\d\r\Big)^{\frac12}\cdot\Big(\int_{\s}^{\t}\sum_{\w\in\Z}\Big\{\sum_{\y\in\Z}\exp(\tfrac{\kappa|\w-\y|}{N})|\boldsymbol{\Phi}_{\r}[\mathbf{K}^{N,\zeta,\sim}_{\s,\r,\cdot,\y}]_{\w}|\Big\}^{2}\d\r\Big)^{\frac12}\nonumber\\
&\lesssim_{\kappa}N^{-\frac12+\frac12\delta}|\t-\s|^{\frac12}\Big(\int_{\s}^{\t}\sum_{\w\in\Z}\Big\{\sum_{\y\in\Z}\exp(\tfrac{\kappa|\w-\y|}{N})|\boldsymbol{\Phi}_{\r}[\mathbf{K}^{N,\zeta,\sim}_{\s,\r,\cdot,\y}]_{\w}|\Big\}^{2}\d\r\Big)^{\frac12}\nonumber\\
&\lesssim N^{-\frac12+\frac12\delta}N^{\frac12\beta}N^{-\frac12\zeta}\Big(\int_{\s}^{\t}\sum_{\w\in\Z}\Big\{\sum_{\y\in\Z}\exp(\tfrac{\kappa|\w-\y|}{N})|\boldsymbol{\Phi}_{\r}[\mathbf{K}^{N,\zeta,\sim}_{\s,\r,\cdot,\y}]_{\w}|\Big\}^{2}\d\r\Big)^{\frac12}.\label{eq:kzetaestimate-shortI3a}
\end{align}
The first bound follows first by the triangle inequality, i.e. {\small$\exp(\kappa|\x-\y|/N)\leq\exp(\kappa|\x-\w|/N)\exp(\kappa|\w-\y|/N)$} for any {\small$\x,\y,\w\in\Z$}. The second estimate follows by Cauchy-Schwarz applied with respect to the {\small$\d\r$}-integration and the {\small$\w$}-summation. The third estimate follows by \eqref{eq:lzetasimestimateII}. The last estimate above follows because {\small$|\t-\s|\lesssim N^{\beta}N^{-\zeta}$}. Next, we claim that for any {\small$\e>0$}, we have the following (see \eqref{eq:lkzetaIb}-\eqref{eq:lkzetaIc} for the relevant notation):
\begin{align}
&\Big\{\sum_{\y\in\Z}\exp(\tfrac{\kappa|\w-\y|}{N})|\boldsymbol{\Phi}_{\r}[\mathbf{K}^{N,\zeta,\sim}_{\s,\r,\cdot,\y}]_{\w}]|\Big\}^{2}\nonumber\\
&\lesssim N^{2}\Big\{{\mathbf{1}_{|\w|\lesssim N^{1+\zeta}}}\sum_{\substack{\n=1,\ldots,\mathrm{K}\\\m=0,\ldots,\mathrm{M}\\0<|\mathfrak{l}_{1}|,\ldots,|\mathfrak{l}_{\m}|\lesssim \mathfrak{n}_{\mathbf{Av}}}}\tfrac{1}{|\mathfrak{l}_{1}|\ldots|\mathfrak{l}_{\m}|}\sum_{\y\in\Z}\exp(\tfrac{\kappa|\w-\y|}{N})|\boldsymbol{\Phi}_{\r}^{\n,\m,\mathfrak{l}_{1},\ldots,\mathfrak{l}_{\m}}[\mathbf{K}^{N,\zeta,\sim}_{\s,\r,\cdot,\y}]_{\w}|\Big\}^{2}\nonumber\\
&\lesssim_{\e} N^{2+\e}{\mathbf{1}_{|\w|\lesssim N^{1+\zeta}}}\sum_{\substack{\n=1,\ldots,\mathrm{K}\\\m=0,\ldots,\mathrm{M}\\0<|\mathfrak{l}_{1}|,\ldots,|\mathfrak{l}_{\m}|\lesssim \mathfrak{n}_{\mathbf{Av}}}}\tfrac{1}{|\mathfrak{l}_{1}|\ldots|\mathfrak{l}_{\m}|}\Big\{\sum_{\y\in\Z}\exp(\tfrac{\kappa|\w-\y|}{N})|\boldsymbol{\Phi}_{\r}^{\n,\m,\mathfrak{l}_{1},\ldots,\mathfrak{l}_{\m}}[\mathbf{K}^{N,\zeta,\sim}_{\s,\r,\cdot,\y}]_{\w}|\Big\}^{2}.\label{eq:kzetaestimate-shortI3b}
\end{align}
The first bound follows by \eqref{eq:lkzetaIb}-\eqref{eq:lkzetaIc}. The second follows by Cauchy-Schwarz with respect to the triple summation; we note that the number of {\small$\n,\m$} summands is {\small$\mathrm{K}\mathrm{M}\lesssim1$} (see Proposition \ref{prop:s-sde}), and the weight {\small$|\mathfrak{l}_{1}|^{-1}\ldots|\mathfrak{l}_{\m}|^{-1}$} has total sum of {\small$\lesssim (\log N)^{\mathrm{C}\m}\lesssim_{\e} N^{\e}$} for any {\small$\e>0$}, so the only cost to pay when moving the square into the {\small$\y$}-summation is a factor of {\small$\lesssim N^{\e}$}, and the final bound in \eqref{eq:kzetaestimate-shortI3b} follows. We now control each of the squared-{\small$\y$}-sums in \eqref{eq:kzetaestimate-shortI3b}; recall {\small$\boldsymbol{\Phi}^{\n,\m,\mathfrak{l}_{1}\ldots,\mathfrak{l}_{\m}}$} in \eqref{eq:lkzetaIc}. We first claim that 
\begin{align}
|\mathrm{c}_{N,\n,\mathfrak{l}_{1},\ldots,\mathfrak{l}_{\m}}\grad^{\mathbf{X}}_{\mathfrak{l}_{1}}\ldots\grad^{\mathbf{X}}_{\mathfrak{l}_{\m}}\mathscr{S}^{N}_{\w-\z}|\lesssim N^{-1+2\delta_{\mathbf{S}}}\mathbf{1}_{|\w-\z|\lesssim N^{1-\delta_{\mathbf{S}}}}.\label{eq:kzetaestimate-shortI3c}
\end{align}
To see this, we consider two separate cases.
\begin{itemize}
\item Assume that {\small$\mathrm{c}_{N,\n,\mathfrak{l}_{1},\ldots,\mathfrak{l}_{\m}}=\mathrm{O}(1)$}. We know {\small$|\mathscr{S}^{N}_{\w-\z}|\lesssim N^{-1+\delta_{\mathbf{S}}}\mathbf{1}_{|\w-\z|\lesssim N^{1-\delta_{\mathbf{S}}}}$}. We also restrict to {\small$|\mathfrak{l}_{1}|,\ldots,|\mathfrak{l}_{\m}|\lesssim\mathfrak{n}_{\mathbf{Av}}\ll N^{1-\delta_{\mathbf{S}}}$}. Because the number of gradients is {\small$\m\leq\mathrm{M}\lesssim1$}, we therefore deduce the claim \eqref{eq:kzetaestimate-shortI3c}.
\item Next, we assume that {\small$\mathrm{c}_{N,\n,\mathfrak{l}_{1},\ldots,\mathfrak{l}_{\m}}\neq\mathrm{O}(1)$}. As stated in Proposition \ref{prop:s-sde}, we still know that {\small$\mathrm{c}_{N,\n,\mathfrak{l}_{1},\ldots,\mathfrak{l}_{\m}}=\mathrm{O}(N)$}. Moreover, as also stated in Proposition \ref{prop:s-sde}, in this case, we know that {\small$|\mathfrak{l}_{\j}|\lesssim1$} for some {\small$\j=1,\ldots,\m$}. Because {\small$\mathscr{S}^{N}$} is smooth at scale {\small$N^{1-\delta_{\mathbf{S}}}$} (see Definition \ref{definition:zsmooth}), taking its {\small$\grad^{\mathbf{X}}_{\mathfrak{l}_{\j}}$}-gradient yields another factor of {\small$N^{-1+\delta_{\mathbf{S}}}$}. In particular, we deduce {\small$|\grad^{\mathbf{X}}_{\mathfrak{l}_{\j}}\mathscr{S}^{N}_{\w-\z}|\lesssim N^{-1+\delta_{\mathbf{S}}}\cdot N^{-1+\delta_{\mathbf{S}}}\mathbf{1}_{|\w-\z|\lesssim N^{1-\delta_{\mathbf{S}}}}$}. We now multiply by {\small$\mathrm{c}_{N,\n,\mathfrak{l}_{1},\ldots,\mathfrak{l}_{\m}}=\mathrm{O}(N)$} and take the other gradients as in the previous bullet point to deduce the claim \eqref{eq:kzetaestimate-shortI3c}.
\end{itemize}
Therefore, upon recalling {\small$\boldsymbol{\Phi}^{\n,\m,\mathfrak{l}_{1}\ldots,\mathfrak{l}_{\m}}$} in \eqref{eq:lkzetaIc}, we have 
\begin{align}
&\Big\{\sum_{\y\in\Z}\exp(\tfrac{\kappa|\w-\y|}{N})|\boldsymbol{\Phi}_{\r}^{\n,\m,\mathfrak{l}_{1},\ldots,\mathfrak{l}_{\m}}[\mathbf{K}^{N,\zeta,\sim}_{\s,\r,\cdot,\y}]_{\w}|\Big\}^{2}\nonumber\\
&\lesssim \Big\{\sum_{\y\in\Z} N^{-1+2\delta_{\mathbf{S}}}\sum_{|\z|\lesssim N^{1-\delta_{\mathbf{S}}}}|\mathbf{Av}^{\mathbf{T},\mathbf{X},\mathfrak{q}_{\n}}_{\r,\w+\z}|\cdot|\mathbf{R}^{N}_{\r,\w+\z}|\cdot\exp(\tfrac{\kappa|\w-\y|}{N})|\mathbf{K}^{N,\zeta,\sim}_{\s,\r,\w+\z,\y}|\Big\}^{2}\nonumber\\
&= \Big\{N^{-1+2\delta_{\mathbf{S}}}\sum_{|\z|\lesssim N^{1-\delta_{\mathbf{S}}}}|\mathbf{Av}^{\mathbf{T},\mathbf{X},\mathfrak{q}_{\n}}_{\r,\w+\z}|\cdot|\mathbf{R}^{N}_{\r,\w+\z}|\cdot\sum_{\y\in\Z}\exp(\tfrac{\kappa|\w-\y|}{N})|\mathbf{K}^{N,\zeta,\sim}_{\s,\r,\w+\z,\y}|\Big\}^{2}\nonumber\\
&\lesssim N^{-1+\delta_{\mathbf{S}}}\sum_{|\z|\lesssim N^{1-\delta_{\mathbf{S}}}}N^{2\delta_{\mathbf{S}}}|\mathbf{Av}^{\mathbf{T},\mathbf{X},\mathfrak{q}_{\n}}_{\r,\w+\z}|^{2}\cdot|\mathbf{R}^{N}_{\r,\w+\z}|^{2}\cdot\Big(\sum_{\y\in\Z}\exp(\tfrac{\kappa|\w-\y|}{N})|\mathbf{K}^{N,\zeta,\sim}_{\s,\r,\w+\z,\y}|\Big)^{2}.\label{eq:kzetaestimate-shortI3d}
\end{align}
The first bound follows from \eqref{eq:lkzetaIc} and \eqref{eq:kzetaestimate-shortI3c}. The line afterwards follows from rearranging summations. The last line follows from Cauchy-Schwarz with respect to the {\small$\z$}-summation. Let us now combine \eqref{eq:kzetaestimate-shortI3b}, and \eqref{eq:kzetaestimate-shortI3d}. We also note that since we restrict to {\small$|\w|\lesssim N^{1+\zeta}$} and {\small$|\z|\lesssim N^{1-\delta_{\mathbf{S}}}$}, with high probability, we can bound the {\small$\mathbf{R}^{N}$} factor in \eqref{eq:kzetaestimate-shortI3d} by {\small$1$}; this follows by {Lemma \ref{lemma:rbound}}. We claim that this ultimately produces the estimate
\begin{align}
&\sum_{\w\in\Z}\Big\{\sum_{\y\in\Z}\exp(\tfrac{\kappa|\w-\y|}{N})|\boldsymbol{\Phi}_{\r}[\mathbf{K}^{N,\zeta,\sim}_{\s,\r,\cdot,\y}]_{\w}]|\Big\}^{2}\nonumber\\
&\lesssim N^{2+\e+2\delta_{\mathbf{S}}}N^{-1+\delta_{\mathbf{S}}}\sum_{\substack{|\w|\lesssim N^{1+\zeta}\\|\z|\lesssim N^{1-\delta_{\mathbf{S}}}}}\sum_{\substack{\n=1,\ldots,\mathrm{K}\\\m=0,\ldots,\mathrm{M}\\0<|\mathfrak{l}_{1}|,\ldots,|\mathfrak{l}_{\m}|\lesssim \mathfrak{n}_{\mathbf{Av}}}}\tfrac{1}{|\mathfrak{l}_{1}|\ldots|\mathfrak{l}_{\m}|}\cdot |\mathbf{Av}^{\mathbf{T},\mathbf{X},\mathfrak{q}_{\n}}_{\r,\w+\z}|^{2}\cdot\Big(\sum_{\y\in\Z}\exp(\tfrac{\kappa|\w-\y|}{N})|\mathbf{K}^{N,\zeta,\sim}_{\s,\r,\w+\z,\y}|\Big)^{2}\nonumber\\
&\lesssim_{\kappa} \Big(\sup_{\substack{\r\in[\s,\t]\\\x\in\Z}}\sum_{\y\in\Z}\exp(\tfrac{\kappa|\x-\y|}{N})|\mathbf{K}^{N,\zeta,\sim}_{\s,\r,\x,\y}|\Big)^{2}\cdot N^{2+\e+2\delta_{\mathbf{S}}}N^{-1+\delta_{\mathbf{S}}}\sum_{\substack{|\w|\lesssim N^{1+\zeta}\\|\z|\lesssim N^{1-\delta_{\mathbf{S}}}}}\sum_{\substack{\n=1,\ldots,\mathrm{K}\\\m=0,\ldots,\mathrm{M}\\0<|\mathfrak{l}_{1}|,\ldots,|\mathfrak{l}_{\m}|\lesssim \mathfrak{n}_{\mathbf{Av}}}}\tfrac{1}{|\mathfrak{l}_{1}|\ldots|\mathfrak{l}_{\m}|}\cdot |\mathbf{Av}^{\mathbf{T},\mathbf{X},\mathfrak{q}_{\n}}_{\r,\w+\z}|^{2}\nonumber\\
&\lesssim\Big(\sup_{\substack{\r\in[\s,\t]\\\x\in\Z}}\sum_{\y\in\Z}\exp(\tfrac{\kappa|\x-\y|}{N})|\mathbf{K}^{N,\zeta,\sim}_{\s,\r,\x,\y}|\Big)^{2}\cdot N^{2+\e+2\delta_{\mathbf{S}}}\sum_{\substack{|\w|\lesssim N^{1+\zeta}}}\sum_{\substack{\n=1,\ldots,\mathrm{K}\\\m=0,\ldots,\mathrm{M}\\0<|\mathfrak{l}_{1}|,\ldots,|\mathfrak{l}_{\m}|\lesssim \mathfrak{n}_{\mathbf{Av}}}}\tfrac{1}{|\mathfrak{l}_{1}|\ldots|\mathfrak{l}_{\m}|}\cdot |\mathbf{Av}^{\mathbf{T},\mathbf{X},\mathfrak{q}_{\n}}_{\r,\w}|^{2}.\nonumber
\end{align}
The second bound above follows by first noting that {\small$\exp(\kappa|\w-\y|/N)\lesssim_{\kappa}\exp(\kappa|(\w+\z)-\y|/N)$} if {\small$|\z|\lesssim N^{1-\delta_{\mathbf{S}}}$}. Then, we bound the squared {\small$\y$}-summation in the second line by its supremum over {\small$\w+\z$}. The last line above follows because summing a non-negative function of {\small$\w+\z$} over {\small$|\w|\lesssim N^{1+\zeta}$} and averaging over {\small$|\z|\lesssim N^{1-\delta_{\mathbf{S}}}$} is bounded {\color{black}from} above by summing the same function but now {\color{black}over} {\small$|\w|\lesssim N^{1+\zeta}$}. Now, we combine the previous display with \eqref{eq:kzetaestimate-shortI3a}. After gathering powers of {\small$N$}, this provides the first estimate below, after which we use \eqref{eq:kzetaestimate-shortI-key}:
\begin{align}
&\sum_{\y\in\Z}\exp(\tfrac{\kappa|\x-\y|}{N})\cdot\int_{\s}^{\t}\sum_{\w\in\Z}|\mathbf{L}^{N,\zeta,\sim}_{\r,\t,\x,\w}-\mathbf{H}^{N}_{\r,\t,\x,\w}|\cdot|\boldsymbol{\Phi}_{\r}[\mathbf{K}^{N,\zeta,\sim}_{\s,\r,\cdot,\y}]_{\w}|\d\r\nonumber\\
&\lesssim N^{\frac12+\frac12\delta+\frac12\e+\delta_{\mathbf{S}}}N^{\frac12\beta}N^{-\frac12\zeta}\cdot\sup_{\substack{\r\in[\s,\t]\\\x\in\Z}}\sum_{\y\in\Z}\exp(\tfrac{\kappa|\x-\y|}{N})|\mathbf{K}^{N,\zeta,\sim}_{\s,\r,\x,\y}|\nonumber\\
&\times\Big(\int_{\s}^{\t}\sum_{\substack{|\w|\lesssim N^{1+\zeta}}}\sum_{\substack{\n=1,\ldots,\mathrm{K}\\\m=0,\ldots,\mathrm{M}\\0<|\mathfrak{l}_{1}|,\ldots,|\mathfrak{l}_{\m}|\lesssim \mathfrak{n}_{\mathbf{Av}}}}\tfrac{1}{|\mathfrak{l}_{1}|\ldots|\mathfrak{l}_{\m}|}\cdot |\mathbf{Av}^{\mathbf{T},\mathbf{X},\mathfrak{q}_{\n}}_{\r,\w}|^{2}\Big)^{\frac12}\nonumber\\
&\lesssim N^{-\frac12\beta_{\star}}N^{\frac12\delta+\frac12\e+\delta_{\mathbf{S}}}N^{\frac12\beta}\cdot\sup_{\substack{\r\in[\s,\t]\\\x\in\Z}}\sum_{\y\in\Z}\exp(\tfrac{\kappa|\x-\y|}{N})|\mathbf{K}^{N,\zeta,\sim}_{\s,\r,\x,\y}|.\label{eq:kzetaestimate-shortI3}
\end{align}
The {\small$\delta,\e,\delta_{\mathbf{S}},\beta$} are small constants of our choosing, and {\small$\beta_{\star}>0$} depends only on {\small$\gamma_{\mathrm{data}}$} from \eqref{eq:noneq}. This gives our estimate for \eqref{eq:kzetaestimate-shortI2b}. We now move to the \abbr{RHS} of \eqref{eq:kzetaestimate-shortI2a}. Recall \eqref{eq:lkzetaIb}-\eqref{eq:lkzetaIc}, and recall that with high probability (see {Lemma \ref{lemma:rbound}}), the {\small$\mathbf{R}^{N}$} factors therein are uniformly {\small$\mathrm{O}(1)$}; this is explained after \eqref{eq:kzetaestimate-shortI3d}. Ultimately, we get
\begin{align}
&\sum_{\y\in\Z}\exp(\tfrac{\kappa|\x-\y|}{N})\cdot \int_{\s}^{\t}\sum_{\w\in\Z}\mathbf{H}^{N}_{\r,\t,\x,\w}\cdot|\boldsymbol{\Phi}_{\r}[\mathbf{K}^{N,\zeta,\sim}_{\s,\r,\cdot,\y}]_{\w}|\d\r\nonumber\\
&\lesssim N\sum_{\substack{\n=1,\ldots,\mathrm{K}\\\m=0,\ldots,\mathrm{M}\\0<|\mathfrak{l}_{1}|,\ldots,|\mathfrak{l}_{\m}|\lesssim \mathfrak{n}_{\mathbf{Av}}}}\tfrac{1}{|\mathfrak{l}_{1}|\ldots|\mathfrak{l}_{\m}|}\sum_{\y\in\Z}\exp(\tfrac{\kappa|\x-\y|}{N})\cdot\mathbf{I}_{\s,\t,\x,\y}^{\n,\mathfrak{l}_{1},\ldots,\mathfrak{l}_{\m}},\label{eq:kzetaestimate-shortI4a}
\end{align}
where, in order to fit the display on the page, we use the notation
\begin{align}
\mathbf{I}^{\n,\mathfrak{l}_{1},\ldots,\mathfrak{l}_{\m}}_{\s,\t,\x,\y}&:=\int_{\s}^{\t}\sum_{|\w|\lesssim N^{1+\zeta}}\mathbf{H}^{N}_{\r,\t,\x,\w}\sum_{\z\in\Z}|\mathrm{c}_{N,\n,\mathfrak{l}_{1},\ldots,\mathfrak{l}_{\m}}\grad^{\mathbf{X}}_{1}\ldots\grad^{\mathbf{X}}_{\mathfrak{l}_{\m}}\mathscr{S}^{N}_{\w-\z}|\cdot|\mathbf{Av}^{\mathbf{T},\mathbf{X},\mathfrak{q}_{\n}}_{\r,\z}|\cdot|\mathbf{K}^{N,\zeta,\sim}_{\s,\r,\z,\y}|\d\r.\label{eq:kzetaestimate-shortI4b}
\end{align}
We now use the estimate \eqref{eq:kzetaestimate-shortI3c}; this implies that we can restrict the summation over {\small$\z\in\Z$} to {\small$|\w-\z|\lesssim N^{1-\delta_{\mathbf{S}}}$}. Moreover, because {\small$|\w|\lesssim N^{1+\zeta}$}, we can cut the summation over {\small$\z\in\Z$} ultimately to {\small$|\z|\lesssim N^{1+\zeta}$}. In particular, if we then swap the {\small$\w,\z$} summations, we obtain 
\begin{align}
\mathbf{I}^{\n,\mathfrak{l}_{1},\ldots,\mathfrak{l}_{\m}}_{\s,\t,\x,\y}&\lesssim\int_{\s}^{\t}\sum_{|\z|\lesssim N^{1+\zeta}}\Big(\sum_{\w\in\Z}\mathbf{H}^{N}_{\r,\t,\x,\w}|\mathrm{c}_{N,\n,\mathfrak{l}_{1},\ldots,\mathfrak{l}_{\m}}\grad^{\mathbf{X}}_{1}\ldots\grad^{\mathbf{X}}_{\mathfrak{l}_{\m}}\mathscr{S}^{N}_{\w-\z}|\Big)\cdot|\mathbf{Av}^{\mathbf{T},\mathbf{X},\mathfrak{q}_{\n}}_{\r,\z}|\cdot|\mathbf{K}^{N,\zeta,\sim}_{\s,\r,\z,\y}|\d\r.\nonumber
\end{align}
We now use the following notation for convenience:
\begin{align}
\mathbf{W}_{\r,\t,\x,\z}&:=\max_{\substack{\n=1,\ldots,\mathrm{K}\\\m=0,\ldots,\mathrm{M}\\0<|\mathfrak{l}_{1}|,\ldots,|\mathfrak{l}_{\m}|\lesssim \mathfrak{n}_{\mathbf{Av}}}}\sum_{\w\in\Z}\mathbf{H}^{N}_{\r,\t,\x,\w}|\mathrm{c}_{N,\n,\mathfrak{l}_{1},\ldots,\mathfrak{l}_{\m}}\grad^{\mathbf{X}}_{1}\ldots\grad^{\mathbf{X}}_{\mathfrak{l}_{\m}}\mathscr{S}^{N}_{\w-\z}|\nonumber\\
&\lesssim\sum_{\w\in\Z}\mathbf{H}^{N}_{\r,\t,\x,\w}\cdot\max_{\substack{\n=1,\ldots,\mathrm{K}\\\m=0,\ldots,\mathrm{M}\\0<|\mathfrak{l}_{1}|,\ldots,|\mathfrak{l}_{\m}|\lesssim \mathfrak{n}_{\mathbf{Av}}}}|\mathrm{c}_{N,\n,\mathfrak{l}_{1},\ldots,\mathfrak{l}_{\m}}\grad^{\mathbf{X}}_{1}\ldots\grad^{\mathbf{X}}_{\mathfrak{l}_{\m}}\mathscr{S}^{N}_{\w-\z}|.\label{eq:thetakerneldef}
\end{align}
If we combine the previous two displays with \eqref{eq:kzetaestimate-shortI4a}, then we obtain
\begin{align}
&\sum_{\y\in\Z}\exp(\tfrac{\kappa|\x-\y|}{N})\cdot \int_{\s}^{\t}\sum_{\w\in\Z}\mathbf{H}^{N}_{\r,\t,\x,\w}\cdot|\boldsymbol{\Phi}_{\r}[\mathbf{K}^{N,\zeta,\sim}_{\s,\r,\cdot,\y}]_{\w}|\d\r\nonumber\\
&\lesssim N\sum_{\y\in\Z}\exp(\tfrac{\kappa|\x-\y|}{N})\cdot \int_{\s}^{\t}\sum_{|\z|\lesssim N^{1+\zeta}}\mathbf{W}_{\r,\t,\x,\z}\cdot \sum_{\substack{\n=1,\ldots,\mathrm{K}\\\m=0,\ldots,\mathrm{M}\\0<|\mathfrak{l}_{1}|,\ldots,|\mathfrak{l}_{\m}|\lesssim \mathfrak{n}_{\mathbf{Av}}}}\tfrac{1}{|\mathfrak{l}_{1}|\ldots|\mathfrak{l}_{\m}|}|\mathbf{Av}^{\mathbf{T},\mathbf{X},\mathfrak{q}_{\n}}_{\r,\z}|\cdot|\mathbf{K}^{N,\zeta,\sim}_{\s,\r,\z,\y}|\d\r.\nonumber
\end{align}
Now, we use the estimate {\small$\exp(\kappa|\x-\y|/N)\leq\exp(\kappa|\x-\z|/N)\exp(\kappa|\z-\y|/N)$}, and we rearrange summations. This eventually turns the previous display into the following estimate:
\begin{align}
&\sum_{\y\in\Z}\exp(\tfrac{\kappa|\x-\y|}{N})\cdot \int_{\s}^{\t}\sum_{\w\in\Z}\mathbf{H}^{N}_{\r,\t,\x,\w}\cdot|\boldsymbol{\Phi}_{\r}[\mathbf{K}^{N,\zeta,\sim}_{\s,\r,\cdot,\y}]_{\w}|\d\r\nonumber\\
&\lesssim N\int_{\s}^{\t}\sum_{|\z|\lesssim N^{1+\zeta}}\exp(\tfrac{\kappa|\x-\z|}{N})\mathbf{W}_{\r,\t,\x,\z}\cdot\sum_{\substack{\n=1,\ldots,\mathrm{K}\\\m=0,\ldots,\mathrm{M}\\0<|\mathfrak{l}_{1}|,\ldots,|\mathfrak{l}_{\m}|\lesssim \mathfrak{n}_{\mathbf{Av}}}}\tfrac{1}{|\mathfrak{l}_{1}|\ldots|\mathfrak{l}_{\m}|}|\mathbf{Av}^{\mathbf{T},\mathbf{X},\mathfrak{q}_{\n}}_{\r,\z}|\cdot\sum_{\y\in\Z}\exp(\tfrac{\kappa|\z-\y|}{N})|\mathbf{K}^{N,\zeta,\sim}_{\s,\r,\z,\y}|\d\r\nonumber\\
&\lesssim N\int_{\s}^{\t}\sum_{|\z|\lesssim N^{1+\zeta}}\exp(\tfrac{\kappa|\x-\z|}{N})\mathbf{W}_{\r,\t,\x,\z}\cdot\sum_{\substack{\n=1,\ldots,\mathrm{K}\\\m=0,\ldots,\mathrm{M}\\0<|\mathfrak{l}_{1}|,\ldots,|\mathfrak{l}_{\m}|\lesssim \mathfrak{n}_{\mathbf{Av}}}}\tfrac{1}{|\mathfrak{l}_{1}|\ldots|\mathfrak{l}_{\m}|}|\mathbf{Av}^{\mathbf{T},\mathbf{X},\mathfrak{q}_{\n}}_{\r,\z}|\d\r\nonumber\\
&\quad\quad\times\sup_{\substack{\r\in[\s,\t]\\\x\in\Z}}\sum_{\y\in\Z}\exp(\tfrac{\kappa|\x-\y|}{N})|\mathbf{K}^{N,\zeta,\sim}_{\s,\r,\x,\y}|.\label{eq:kzetaestimate-shortI4c}
\end{align}
Now, we control the first factor on the far \abbr{RHS} of \eqref{eq:kzetaestimate-shortI4c}. Recall {\small$\mathbf{W}$} from \eqref{eq:thetakerneldef}. We claim that
\begin{align}
&\sum_{\z\in\Z}\exp(\tfrac{2\kappa|\x-\z|}{N})|\mathbf{W}_{\r,\t,\x,\z}|^{2}\nonumber\\
&\lesssim\sum_{\z\in\Z}\Big(\sum_{\w\in\Z}\exp(\tfrac{\kappa|\x-\w|}{N})\mathbf{H}^{N}_{\r,\t,\x,\w}\cdot\exp(\tfrac{\kappa|\w-\z|}{N})\cdot\max_{\substack{\n=1,\ldots,\mathrm{K}\\\m=0,\ldots,\mathrm{M}\\0<|\mathfrak{l}_{1}|,\ldots,|\mathfrak{l}_{\m}|\lesssim \mathfrak{n}_{\mathbf{Av}}}}|\mathrm{c}_{N,\n,\mathfrak{l}_{1},\ldots,\mathfrak{l}_{\m}}\grad^{\mathbf{X}}_{1}\ldots\grad^{\mathbf{X}}_{\mathfrak{l}_{\m}}\mathscr{S}^{N}_{\w-\z}|\Big)^{2}\nonumber\\
&\lesssim_{\kappa}\sum_{\w\in\Z}\exp(\tfrac{\kappa|\x-\w|}{N})\mathbf{H}^{N}_{\r,\t,\x,\w}\cdot\sum_{\z\in\Z}\exp(\tfrac{2\kappa|\w-\z|}{N})\cdot\max_{\substack{\n=1,\ldots,\mathrm{K}\\\m=0,\ldots,\mathrm{M}\\0<|\mathfrak{l}_{1}|,\ldots,|\mathfrak{l}_{\m}|\lesssim \mathfrak{n}_{\mathbf{Av}}}}|\mathrm{c}_{N,\n,\mathfrak{l}_{1},\ldots,\mathfrak{l}_{\m}}\grad^{\mathbf{X}}_{1}\ldots\grad^{\mathbf{X}}_{\mathfrak{l}_{\m}}\mathscr{S}^{N}_{\w-\z}|^{2}\nonumber\\
&\lesssim\sum_{\w\in\Z}\exp(\tfrac{\kappa|\x-\w|}{N})\mathbf{H}^{N}_{\r,\t,\x,\w}\cdot N^{-2+4\delta_{\mathbf{S}}}\sum_{|\z-\w|\lesssim N^{1-\delta_{\mathbf{S}}}}\exp(\tfrac{2\kappa|\w-\z|}{N})\nonumber\\
&\lesssim_{\kappa} N^{-1+3\delta_{\mathbf{S}}}\sum_{\w\in\Z}\exp(\tfrac{\kappa|\x-\w|}{N})\mathbf{H}^{N}_{\r,\t,\x,\w}\lesssim_{\kappa} N^{-1+3\delta_{\mathbf{S}}}.\label{eq:kzetaestimate-shortI4d}
\end{align}
The first estimate follows from \eqref{eq:thetakerneldef} moving the exponential factor inside of the square, and using {\small$\exp(\kappa|\x-\z|/N)\leq\exp(\kappa|\x-\w|/N)\exp(\kappa|\w-\z|/N)$}. The second estimate follows by Cauchy-Schwarz with respect to the {\small$\w$}-summation; in this, we note that {\small$\exp(\kappa|\x-\w|/N)\mathbf{H}^{N}_{\r,\t,\x,\w}$} is not a probability measure in {\small$\w\in\Z$}, but its total sum over {\small$\w\in\Z$} is {\small$\lesssim_{\kappa}1$}. (We also rearrange the {\small$\w,\z$}-summations to get the second estimate above.) The third inequality follows first by applying the estimate \eqref{eq:kzetaestimate-shortI3c} to bound the maximum in the third line above. The fourth line follows first from noting that {\small$\exp(2\kappa|\w-\z|/N)\lesssim_{\kappa}1$} if {\small$|\w-\z|\lesssim N^{1-\delta_{\mathbf{S}}}$}, and then by using the exponential heat kernel estimate in {Proposition \ref{prop:hke}}. Now, we claim that the following estimate holds:
\begin{align}
&\int_{\s}^{\t}\sum_{|\z|\lesssim N^{1+\zeta}}\exp(\tfrac{\kappa|\x-\z|}{N})\mathbf{W}_{\r,\t,\x,\z}\cdot\sum_{\substack{\n=1,\ldots,\mathrm{K}\\\m=0,\ldots,\mathrm{M}\\0<|\mathfrak{l}_{1}|,\ldots,|\mathfrak{l}_{\m}|\lesssim \mathfrak{n}_{\mathbf{Av}}}}\tfrac{1}{|\mathfrak{l}_{1}|\ldots|\mathfrak{l}_{\m}|}|\mathbf{Av}^{\mathbf{T},\mathbf{X},\mathfrak{q}_{\n}}_{\r,\z}|\d\r\nonumber\\
&\lesssim_{\e}N^{\e}\Big(\int_{\s}^{\t}\sum_{\z\in\Z}\exp(\tfrac{2\kappa|\x-\z|}{N})|\mathbf{W}_{\r,\t,\x,\z}|^{2}\d\r\Big)^{\frac12}\Big(\int_{\s}^{\t}\sum_{|\z|\lesssim N^{1+\zeta}}\sum_{\substack{\n=1,\ldots,\mathrm{K}\\\m=0,\ldots,\mathrm{M}\\0<|\mathfrak{l}_{1}|,\ldots,|\mathfrak{l}_{\m}|\lesssim \mathfrak{n}_{\mathbf{Av}}}}\tfrac{1}{|\mathfrak{l}_{1}|\ldots|\mathfrak{l}_{\m}|}|\mathbf{Av}^{\mathbf{T},\mathbf{X},\mathfrak{q}_{\n}}_{\r,\z}|^{2}\d\r\Big)^{\frac12}\nonumber\\
&\lesssim N^{-\frac12+\frac32\delta_{\mathbf{S}}+\e}|\t-\s|^{\frac12}N^{-1-\frac12\beta_{\star}}N^{\frac12+\frac12\zeta}\nonumber\\
&\lesssim N^{-1-\frac12\beta_{\star}}N^{\frac32\delta_{\mathbf{S}}+\e+\frac12\beta}.\nonumber
\end{align}
The first bound follows by Cauchy-Schwarz with respect to the {\small$\d\r$} integration, the {\small$\z$}-sum, and the {\small$\n,\m,\mathfrak{l}_{1},\ldots,\mathfrak{l}_{\m}$}-sum (using that the factor {\small$|\mathfrak{l}_{1}|^{-1}\ldots|\mathfrak{l}_{\m}|^{-1}$} sums to {\small$\lesssim_{\e}N^{\e}$} for any {\small$\e>0$}). The second estimate follows by \eqref{eq:kzetaestimate-shortI4d} and \eqref{eq:kzetaestimate-shortI-key}. The last estimate follows because {\small$|\t-\s|\leq N^{-\zeta}N^{\beta}$}. If we plug the previous display into the far \abbr{RHS} of \eqref{eq:kzetaestimate-shortI4c}, then we turn \eqref{eq:kzetaestimate-shortI4c} into the following:
\begin{align}
&\sum_{\y\in\Z}\exp(\tfrac{\kappa|\x-\y|}{N})\cdot \int_{\s}^{\t}\sum_{\w\in\Z}\mathbf{H}^{N}_{\r,\t,\x,\w}\cdot|\boldsymbol{\Phi}_{\r}[\mathbf{K}^{N,\zeta,\sim}_{\s,\r,\cdot,\y}]_{\w}|\d\r\nonumber\\
&\lesssim N^{-\frac12\beta_{\star}}N^{\frac32\delta_{\mathbf{S}}+\e+\frac12\beta}\cdot\sup_{\substack{\r\in[\s,\t]\\\x\in\Z}}\sum_{\y\in\Z}\exp(\tfrac{\kappa|\x-\y|}{N})|\mathbf{K}^{N,\zeta,\sim}_{\s,\r,\x,\y}|.\label{eq:kzetaestimate-shortI4}
\end{align}
Finally, we combine \eqref{eq:kzetaestimate-shortI2a}-\eqref{eq:kzetaestimate-shortI2b} with \eqref{eq:kzetaestimate-shortI3} and \eqref{eq:kzetaestimate-shortI4}. This implies (for any {\small$\e,\delta>0$})
\begin{align}
&\sum_{\y\in\Z}\exp(\tfrac{\kappa|\x-\y|}{N})\Big|\int_{\s}^{\t}\sum_{\w\in\Z}\mathbf{L}^{N,\zeta,\sim}_{\r,\t,\x,\w}\boldsymbol{\Phi}_{\r}[\mathbf{K}^{N,\zeta,\sim}_{\s,\r,\cdot,\w}]_{\x}\d\r\Big|\nonumber\\
&\lesssim_{\kappa,\e,\delta} N^{-\frac12\beta_{\star}}N^{\e+\delta+\frac12\beta+\frac32\delta_{\mathbf{S}}} \sup_{\substack{\r\in[\s,\t]\\\x\in\Z}}\sum_{\y\in\Z}\exp(\tfrac{\kappa|\x-\y|}{N})|\mathbf{K}^{N,\zeta,\sim}_{\s,\r,\x,\y}|.\label{eq:kzetaestimate-shortI5}
\end{align}
We now let {\small$\e,\delta,\delta_{\mathbf{S}}$} be small multiples of {\small$\beta_{\star}$} (recall from Definition \ref{definition:zsmooth} that {\small$\delta_{\mathbf{S}}$} is an arbitrary small constant that is independent of {\small$N$}). Then, we combine the previous display with \eqref{eq:kzetaestimate-shortI1} and \eqref{eq:lkzetaIa}. This implies
\begin{align}
\sum_{\y\in\Z}\exp(\tfrac{\kappa|\x-\y|}{N})|\mathbf{K}^{N,\zeta,\sim}_{\s,\t,\x,\y}|&\lesssim_{\kappa,\delta} N^{\frac12\delta}+N^{-\frac13\beta_{\star}}\sup_{\substack{\r\in[\s,\t]\\\x\in\Z}}\sum_{\y\in\Z}\exp(\tfrac{\kappa|\x-\y|}{N})|\mathbf{K}^{N,\zeta,\sim}_{\s,\r,\x,\y}|.\nonumber
\end{align}
Precisely, there is a high probability event on which the above holds {\color{black}for all {\small$\s\in[0,1]$}, {\small$\t\in[\s,\s+N^{-\zeta}N^{\beta}]$}, and {\small$\x\in\Z$}}. Thus, with high probability, we have
\begin{align}
&\sup_{\s\in[0,1]}\sup_{\t\in[\s,\s+N^{-\zeta}N^{\beta}]}\sup_{\x\in\Z}\sum_{\y\in\Z}\exp(\tfrac{\kappa|\x-\y|}{N})|\mathbf{K}^{N,\zeta,\sim}_{\s,\t,\x,\y}|\nonumber\\
&\lesssim_{\kappa,\delta}N^{\frac12\delta}+N^{-\frac13\beta_{\star}}\sup_{\s\in[0,1]}\sup_{\t\in[\s,\s+N^{-\zeta}N^{\beta}]}\sup_{\x\in\Z}\sum_{\y\in\Z}\exp(\tfrac{\kappa|\x-\y|}{N})|\mathbf{K}^{N,\zeta,\sim}_{\s,\t,\x,\y}|.\nonumber
\end{align}
To complete the proof of the desired estimate \eqref{eq:kzetaestimate-shortI}, we move the last term in the previous display to the \abbr{LHS} and divide out by {\small$1-N^{-\beta_{\star}/3}\gtrsim1$}. This finishes the proof. \qed
\subsubsection{Proof of Lemma \ref{lemma:kzetaestimate-long}}
Throughout this argument, we will condition on the estimate \eqref{eq:kzetaestimate-shortI}; we are allowed to do this because \eqref{eq:kzetaestimate-shortI} holds with high probability. 

{\color{black}Fix {\small$\s,\t\in[0,1]$} with {\small$\s\leq\t$} and any {\small$\x,\y\in\Z$}.} Next, take any {\small$\r\in[\s,\t]$}. We claim that
\begin{align}
\mathbf{K}^{N,\zeta,\sim}_{\s,\t,\x,\y}&=\sum_{\w\in\Z}\mathbf{K}^{N,\zeta,\sim}_{\r,\t,\x,\w}\mathbf{K}^{N,\zeta,\sim}_{\s,\r,\w,\y}.\label{eq:kzetaestimate-longI1}
\end{align}
To justify \eqref{eq:kzetaestimate-longI1}, we follow the argument in the proof of Lemma \ref{lemma:lkzeta}. In particular, we first note that both sides of \eqref{eq:kzetaestimate-longI1} solve the same \abbr{SDE} \eqref{eq:kzetasim-sde} for {\small$\t\in[\r,1]$} with the same data at {\small$\t=\r$}. (This uses linearity of the equation \eqref{eq:kzetasim-sde}.) Thus, by uniqueness of solutions to said \abbr{SDE} (see Remark \ref{remark:kzetasim-unique}), we obtain \eqref{eq:kzetaestimate-longI1}.

We now apply \eqref{eq:kzetaestimate-longI1} iteratively as follows. Let {\small$\t_{1},\ldots,\t_{\mathrm{F}}$} be a sequence of times such that:
\begin{enumerate}
\item We have {\small$\s\leq\t_{1}\leq\ldots\leq\t_{\mathrm{F}}\leq\t$}. We have {\small$|\t_{1}-\s|\leq N^{\beta}N^{-\zeta}$} and {\small$|\t_{\j+1}-\t_{\j}|\leq N^{\beta}N^{-\zeta}$} for all {\small$\j=1,\ldots,\mathrm{F}-1$} and {\small$|\t-\t_{\mathrm{F}}|\leq N^{\beta}N^{-\zeta}$}. Here, the parameter {\small$\beta>0$} is from Lemma \ref{lemma:kzetaestimate-short}.
\item We have {\small$\mathrm{F}\lesssim N^{-\beta}N^{\zeta}+1$}. This is achievable if {\small$\t_{1},\ldots,\t_{\mathrm{F}}$} are evenly spaced points in {\small$[\s,\t]$}, since {\small$\s,\t\in[0,1]$}.
\end{enumerate}
By iterating the Chapman-Kolmogorov equation in \eqref{eq:kzetaestimate-longI1}, we then have 
\begin{align}
\mathbf{K}^{N,\zeta,\sim}_{\s,\t,\x,\y}&=\sum_{\w_{1},\ldots,\w_{\mathrm{F}}\in\Z}\mathbf{K}^{N,\zeta,\sim}_{\t_{\mathrm{F}},\t,\x,\w_{\mathrm{F}}}\cdot\Big(\prod_{\j=1,\ldots,\mathrm{F}-1}\mathbf{K}^{N,\zeta,\sim}_{\t_{\j},\t_{\j+1},\w_{\j+1},\w_{\j}}\Big)\cdot\mathbf{K}^{N,\zeta,\sim}_{\s,\t_{1},\w_{1},\y}.\nonumber
\end{align}
Next, \emph{just for this proof}, for convenience, we will define {\small$\mathbf{K}^{N,\zeta,\sim,\kappa}_{\s,\t,\x,\y}:=\exp(\kappa|\x-\y|/N)\cdot\mathbf{K}^{N,\zeta,\sim}_{\s,\t,\x,\y}$}. By the previous display and the triangle inequality, we have the first estimate below, which is followed by a second estimate that we justify afterwards:
\begin{align}
\sum_{\y\in\Z}|\mathbf{K}^{N,\zeta,\sim,\kappa}_{\s,\t,\x,\y}|&\leq\sum_{\w_{1},\ldots,\w_{\mathrm{F}}\in\Z}|\mathbf{K}^{N,\zeta,\sim,\kappa}_{\t_{\mathrm{F}},\t,\x,\w_{\mathrm{F}}}|\cdot\Big(\prod_{\j=1,\ldots,\mathrm{F}-1}|\mathbf{K}^{N,\zeta,\sim,\kappa}_{\t_{\j},\t_{\j+1},\w_{\j+1},\w_{\j}}|\Big)\cdot\sum_{\y\in\Z}|\mathbf{K}^{N,\zeta,\sim,\kappa}_{\s,\t_{1},\w_{1},\y}|\nonumber\\
&\leq\sum_{\w_{\mathrm{F}}\in\Z}|\mathbf{K}^{N,\zeta,\sim,\kappa}_{\t_{\mathrm{F}},\t,\x,\w_{\mathrm{F}}}|\cdot\Big(\prod_{\j=1}^{\mathrm{F}-1}\sup_{\z\in\Z}\sum_{\w_{\j}\in\Z}|\mathbf{K}^{N,\zeta,\sim,\kappa}_{\t_{\j},\t_{\j+1},\z,\w_{\j}}|\Big)\cdot\sup_{\z\in\Z}\sum_{\y\in\Z}|\mathbf{K}^{N,\zeta,\sim,\kappa}_{\s,\t_{1},\z,\w_{1}}|.\nonumber
\end{align}
The second line is justified as follows. On the \abbr{RHS} of the first line, we bound the sum over {\small$\y\in\Z$} by its supremum over {\small$\w_{1}\in\Z$}. This leads us to the last factor in the second line. Then, the only term which depends on {\small$\w_{1}$} is the {\small$\j=1$} factor on the \abbr{RHS} of the first line, giving us a factor of {\small$\sum_{\w_{1}}|\mathbf{K}^{N,\zeta,\sim,\kappa}_{\t_{1},\t_{2},\w_{2},\w_{1}}|$}. We bound this by its supremum over {\small$\w_{2}\in\Z$}. {\color{black}Therefore}, the only term that depends on {\small$\w_{2}$} on the \abbr{RHS} of the first line is the factor {\small$\sum_{\w_{3}\in\Z}|\mathbf{K}^{N,\zeta,\sim,\kappa}_{\t_{2},\t_{3},\w_{3},\w_{2}}|$}. We then proceed inductively to obtain the second estimate above. Now, we use Lemma \ref{lemma:kzetaestimate-short}; because each time-increment is {\small$\leq N^{\beta}N^{-\zeta}$}, each sum in the second line above is {\small$\lesssim_{\kappa} N^{\delta}$} for any fixed (but small) {\small$\delta>0$}. Thus, we obtain the following estimate (which holds simultaneously over {\small$0\leq\s\leq\t\leq1$} and {\small$\x\in\Z$} with high probability):
\begin{align*}
\sum_{\y\in\Z}|\mathbf{K}^{N,\zeta,\sim,\kappa}_{\s,\t,\x,\y}|&\lesssim(\mathrm{C}_{\kappa}N^{\delta})^{\mathrm{F}}\lesssim_{\kappa}\exp(\mathrm{C}_{\kappa}'\log N\cdot N^{-\beta}N^{\zeta})N^{\delta}.
\end{align*}
Here, {\small$\mathrm{C}_{\kappa},\mathrm{C}_{\kappa}'=\mathrm{O}(1)$}. Next, we note that {\small$\exp(\mathrm{C}_{\kappa}'N^{-\beta}N^{\zeta}\log N)\lesssim\exp(N^{-\beta/2}N^{\zeta})$}. It then suffices to recall that {\small$\mathbf{K}^{N,\zeta,\sim,\kappa}_{\s,\t,\x,\y}:=\exp(\kappa|\x-\y|/N)\cdot\mathbf{K}^{N,\zeta,\sim}_{\s,\t,\x,\y}$} to complete the proof. \qed

\subsubsection{Proof of Lemma \ref{lemma:lzetasimestimate}}
By \eqref{eq:lzetasim-sde} and the Duhamel formula, we have the following expression:
\begin{align}
\mathbf{L}^{N,\zeta,\sim}_{\s,\t,\x,\y}&=\mathbf{H}^{N}_{\s,\t,\x,\y}+\int_{\s}^{\t}\sum_{\w\in\Z}\mathbf{H}^{N}_{\r,\t,\x,\w}\cdot{\boldsymbol{\chi}^{(\zeta_{\mathrm{large}})}_{\w}}[\mathscr{S}^{N}\star(\sqrt{2}\lambda N^{\frac12}\mathbf{R}^{N,\wedge}_{\r,\cdot}\mathbf{L}^{N,\zeta,\sim}_{\s,\r,\cdot,\y}\d\mathbf{b}_{\r,\cdot})]_{\w}\label{eq:lzetasimestimateII1a}\\
&+\int_{\s}^{\t}\sum_{\w\in\Z}\mathbf{H}^{N}_{\r,\t,\x,\w}\cdot\mathbf{1}_{\mathfrak{t}_{\mathrm{ap}}\geq\r}\boldsymbol{\chi}^{(\zeta)}_{\w}\mathrm{Err}[\mathbf{R}^{N}_{\r,\cdot}\mathbf{L}^{N,\zeta,\sim}_{\s,\r,\cdot,\y}]_{\w}\d\r\label{eq:lzetasimestimateII1b}\\
&+N\sum_{\substack{\n=1,\ldots,\mathrm{K}\\\m=0,\ldots,\mathrm{M}\\0<|\mathfrak{l}_{1}|,\ldots,|\mathfrak{l}_{\m}|\lesssim \mathfrak{n}_{\mathbf{Av}}}}\tfrac{1}{|\mathfrak{l}_{1}|\ldots|\mathfrak{l}_{\m}|}\int_{\s}^{\t}\sum_{\w\in\Z}\mathbf{H}^{N}_{\r,\t,\x,\w}\cdot\boldsymbol{\chi}^{(\zeta)}_{\w}\grad^{\mathbf{T},\mathrm{av}}_{\mathfrak{t}_{\mathbf{Av}}}\boldsymbol{\Omega}^{\n,\mathfrak{l}_{1},\ldots,\mathfrak{l}_{\m}}_{\r,\w,\y}\d\r,\label{eq:lzetasimestimateII1c}
\end{align}
where {\small$\mathbf{H}^{N}$} is the heat kernel from \eqref{eq:heatkernel}, and in the third line above, we used the following notation (to make the display fit), where the gradients act on the {\small$\w$}-variable below:
\begin{align}
\boldsymbol{\Omega}^{\n,\mathfrak{l}_{1},\ldots,\mathfrak{l}_{\m}}_{\r,\w,\y}&=\mathrm{c}_{N,\n,\mathfrak{l}_{1},\ldots,\mathfrak{l}_{\m}}\grad^{\mathbf{X}}_{\mathfrak{l}_{1}}\ldots\grad^{\mathbf{X}}_{\mathfrak{l}_{\m}}\Big({\mathbf{1}_{\mathfrak{t}_{\mathrm{ap}}\geq\r}}[\mathscr{S}^{N}\star(\mathbf{Av}^{\mathbf{X},\mathfrak{q}_{\n}}_{\r,\cdot}\cdot\mathbf{R}^{N}_{\r,\cdot}\mathbf{L}^{N,\zeta,\sim}_{\s,\r,\cdot,\y})]_{\w}\Big).\label{eq:lzetasimestimateII1d}
\end{align}
We now control each term on the \abbr{RHS} of \eqref{eq:lzetasimestimateII1a}-\eqref{eq:lzetasimestimateII1c}. First, by {Proposition \ref{prop:hke}}, we have 
\begin{align}
\sup_{\x\in\Z}\sum_{\y\in\Z}\exp(\tfrac{2\kappa|\x-\y|}{N})\cdot|\mathbf{H}^{N}_{\s,\t,\x,\y}|^{2}&\lesssim_{\kappa}N^{-1}|\t-\s|^{-\frac12}.\label{eq:lzetasimestimateII2}
\end{align}
Next, we control the last term in the first line \eqref{eq:lzetasimestimateII1a}. We write this stochastic integral as
\begin{align*}
\sqrt{2}\lambda N^{\frac12}\int_{\s}^{\t}\sum_{\z\in\Z}\sum_{\w\in\Z}\mathbf{H}^{N}_{\r,\t,\x,\w}\cdot\boldsymbol{\chi}^{(\zeta_{\mathrm{large}})}_{\w}\mathscr{S}^{N}_{\w-\z}\mathbf{R}^{N,\wedge}_{\r,\z}\mathbf{L}^{N,\zeta,\sim}_{\s,\r,\z,\y}\d\mathbf{b}_{\r,\z}.
\end{align*}
Because the Brownian motions are jointly independent, the quadratic variation of this stochastic integral is
\begin{align*}
&2\lambda^{2}N\int_{\s}^{\t}\sum_{\z\in\Z}\Big(\sum_{\w\in\Z}\mathbf{H}^{N}_{\r,\t,\x,\w}\cdot\boldsymbol{\chi}^{(\zeta_{\mathrm{large}})}_{\w}\mathscr{S}^{N}_{\w-\z}\Big)^{2}|\mathbf{R}^{N,\wedge}_{\r,\z}|^{2}|\mathbf{L}^{N,\zeta,\sim}_{\s,\r,\z,\y}|^{2}\d\r\\
&\lesssim N\int_{\s}^{\t}\sum_{\z\in\Z}|[\mathscr{S}^{N}\star(\mathbf{H}^{N}_{\r,\t,\x,\cdot}\cdot\boldsymbol{\chi}^{(\zeta_{\mathrm{large}})}_{\cdot})]_{\z}|^{2}|\mathbf{L}^{N,\zeta,\sim}_{\s,\r,\z,\y}|^{2}\d\r,
\end{align*}
where the inequality follows because {\small$|\mathbf{R}^{N,\wedge}_{\r,\z}|\lesssim1$} (see \eqref{eq:rwedge}). Thus, by the \abbr{BDG} inequality and uniform boundedness of {\small${\boldsymbol{\chi}^{(\zeta_{\mathrm{large}})}}$}, we have
\begin{align}
&\exp(\tfrac{2p\kappa|\x-\y|}{N})\E\Big(\int_{\s}^{\t}\sum_{\w\in\Z}\mathbf{H}^{N}_{\r,\t,\x,\w}\cdot{\boldsymbol{\chi}^{(\zeta_{\mathrm{large}})}_{\w}}[\mathscr{S}^{N}\star(\sqrt{2}\lambda N^{\frac12}\mathbf{R}^{N,\wedge}_{\r,\cdot}\mathbf{L}^{N,\zeta,\sim}_{\s,\r,\cdot,\y}\d\mathbf{b}_{\r,\cdot})]_{\w}\Big)^{2p}\nonumber\\
&\lesssim_{p}\exp(\tfrac{2p\kappa|\x-\y|}{N})\E\Big(\int_{\s}^{\t}\sum_{\z\in\Z}N|[\mathscr{S}^{N}\star\mathbf{H}^{N}_{\r,\t,\x,\cdot}]_{\z}|^{2}\cdot|\mathbf{L}^{N,\zeta,\sim}_{\s,\r,\z,\y}|^{2}\d\r\Big)^{p}\nonumber
\end{align}
for any {\small$p\geq1$}. Now, we move the exponential in the {\small$p$}-th moment. Then, we use 
\begin{align}
\exp(\tfrac{\kappa|\x-\y|}{N})\leq\exp(\tfrac{\kappa|\x-\z|}{N})\exp(\tfrac{\kappa|\z-\y|}{N}).\label{eq:expinequalitytool}
\end{align}
Ultimately, we get the first estimate below (the second estimate below follows by pointwise bounds on {\small$\mathbf{H}^{N}$} from {Proposition \ref{prop:hke}} and contractivity in {\small$\ell^{\infty}(\Z)$} of {\color{black}the} convolution with the probability density {\small$\mathscr{S}^{N}$}):
\begin{align}
&\exp(\tfrac{2p\kappa|\x-\y|}{N})\E\Big(\int_{\s}^{\t}\sum_{\w\in\Z}\mathbf{H}^{N}_{\r,\t,\x,\w}\cdot{\boldsymbol{\chi}^{(\zeta_{\mathrm{large}})}_{\w}}[\mathscr{S}^{N}\star(\sqrt{2}\lambda N^{\frac12}\mathbf{R}^{N,\wedge}_{\r,\cdot}\mathbf{L}^{N,\zeta,\sim}_{\s,\r,\cdot,\y}\d\mathbf{b}_{\r,\cdot})]_{\w}\Big)^{2p}\nonumber\\
&\lesssim_{p}\E\Big(\int_{\s}^{\t}\sum_{\z\in\Z}N\exp(\tfrac{2\kappa|\x-\z|}{N})|[\mathscr{S}^{N}\star\mathbf{H}^{N}_{\r,\t,\x,\cdot}]_{\z}|^{2}\cdot\exp(\tfrac{2\kappa|\z-\y|}{N})|\mathbf{L}^{N,\zeta,\sim}_{\s,\r,\z,\y}|^{2}\d\r\Big)^{p}\nonumber\\
&\lesssim_{\kappa,p}\E\Big(\int_{\s}^{\t}|\t-\r|^{-\frac12}\sum_{\z\in\Z}[\mathscr{S}^{N}\star\mathbf{H}^{N}_{\r,\t,\x,\cdot}]_{\z}\cdot\exp(\tfrac{2\kappa|\z-\y|}{N})|\mathbf{L}^{N,\zeta,\sim}_{\s,\r,\z,\y}|^{2}\d\r\Big)^{p}.\nonumber
\end{align}
Now, for any {\small$q\geq1$}, it will be convenient to use the norm notation {\small$\|\cdot\|_{\mathrm{L}^{q}(\mathbb{P})}:=(\E|\cdot|^{q})^{1/p}$}. Thus, if we take the {\small$1/p$}-th power of the previous display, use the triangle inequality for {\small$\|\cdot\|_{\mathrm{L}^{p}(\mathbb{P})}$}, and sum over {\small$\y\in\Z$}, then the above turns into the following estimate:
\begin{align}
&\sum_{\y\in\Z}\exp(\tfrac{2\kappa|\x-\y|}{N})\Big\|\Big(\int_{\s}^{\t}\sum_{\w\in\Z}\mathbf{H}^{N}_{\r,\t,\x,\w}\cdot{\boldsymbol{\chi}^{(\zeta_{\mathrm{large}})}_{\w}}[\mathscr{S}^{N}\star(\sqrt{2}\lambda N^{\frac12}\mathbf{R}^{N,\wedge}_{\r,\cdot}\mathbf{L}^{N,\zeta,\sim}_{\s,\r,\cdot,\y}\d\mathbf{b}_{\r,\cdot})]_{\w}\Big)^{2}\Big\|_{\mathrm{L}^{p}(\mathbb{P})}\nonumber\\
&\lesssim_{\kappa,p}\Big\|\int_{\s}^{\t}|\t-\r|^{-\frac12}\sum_{\z\in\Z}[\mathscr{S}^{N}\star\mathbf{H}^{N}_{\r,\t,\x,\cdot}]_{\z}\cdot\exp(\tfrac{2\kappa|\z-\y|}{N})|\mathbf{L}^{N,\zeta,\sim}_{\s,\r,\z,\y}|^{2}\d\r\Big\|_{\mathrm{L}^{p}(\mathbb{P})}\nonumber\\
&\lesssim\int_{\s}^{\t}|\t-\r|^{-\frac12}\sum_{\z\in\Z}[\mathscr{S}^{N}\star\mathbf{H}^{N}_{\r,\t,\x,\cdot}]_{\z}\cdot\sup_{\x\in\Z}\sum_{\y\in\Z}\exp(\tfrac{2\kappa|\x-\y|}{N})\||\mathbf{L}^{N,\zeta,\sim}_{\s,\r,\x,\y}|^{2}\|_{\mathrm{L}^{p}(\mathbb{P}}\d\r\nonumber\\
&\lesssim\int_{\s}^{\t}|\t-\r|^{-\frac12}\cdot\sup_{\x\in\Z}\sum_{\y\in\Z}\exp(\tfrac{2\kappa|\x-\y|}{N})\||\mathbf{L}^{N,\zeta,\sim}_{\s,\r,\x,\y}|^{2}\|_{\mathrm{L}^{p}(\mathbb{P})}\d\r,\label{eq:lzetasimestimateII3}
\end{align}
where the last estimate follows because both {\small$\mathscr{S}^{N}$} and {\small$\mathbf{H}^{N}_{\r,\t,\x,\cdot}$} are probability measures on {\small$\Z$}, and thus so is their convolution. We now estimate \eqref{eq:lzetasimestimateII1b}. For this, we use \eqref{eq:err-estimate} to get the following deterministic estimate:
\begin{align*}
&\mathbf{1}_{\mathfrak{t}_{\mathrm{ap}}\geq\r}\boldsymbol{\chi}^{(\zeta)}_{\w}|\mathrm{Err}[\mathbf{R}^{N}_{\r,\cdot}\mathbf{L}^{N,\zeta,\sim}_{\s,\r,\cdot,\y}]_{\w}|\\
&\lesssim N^{-\frac12+\mathrm{C}\delta_{\mathbf{S}}}\mathbf{1}_{\mathfrak{t}_{\mathrm{ap}}\geq\r}\cdot N^{-1+\delta_{\mathbf{S}}}\sum_{|\z|\lesssim N^{1-\delta_{\mathbf{S}}}}\Big(1+\sum_{|\u|\lesssim1}|\bphi_{\r,\w+\z+\u}|^{\mathrm{C}}\Big)\cdot|\mathbf{R}^{N}_{\w+\z}|\cdot|\mathbf{L}^{N,\zeta,\sim}_{\s,\r,\w+\z,\y}|\\
&\lesssim N^{-\frac12+2\mathrm{C}\delta_{\mathbf{S}}}\cdot N^{-1+\delta_{\mathbf{S}}}\sum_{|\z|\lesssim N^{1-\delta_{\mathbf{S}}}}|\mathbf{L}^{N,\zeta,\sim}_{\s,\r,\w+\z,\y}|.
\end{align*}
We clarify that the last bound follows by the a priori estimates on {\small$\bphi$} and {\small$\mathbf{R}^{N}$} provided by {\small$\mathfrak{t}_{\mathrm{ap}}$} (see \eqref{eq:tap1}), which hold uniformly before time {\small$\mathfrak{t}_{\mathrm{ap}}$} and for {\small$|\w|\lesssim N^{1+\zeta}$}, which we are allowed to restrict to since {\small$\boldsymbol{\chi}^{(\zeta)}_{\w}=0$} otherwise. We also recall that {\small$\mathrm{C}=\mathrm{O}(1)$} as noted in \eqref{eq:err-estimate}. By using the previous display and \eqref{eq:expinequalitytool}, we have 
\begin{align*}
&\exp(\tfrac{\kappa|\x-\y|}{N})\int_{\s}^{\t}\sum_{\w\in\Z}\mathbf{H}^{N}_{\r,\t,\x,\w}\cdot\mathbf{1}_{\mathfrak{t}_{\mathrm{ap}}\geq\r}\boldsymbol{\chi}^{(\zeta)}_{\w}\mathrm{Err}[\mathbf{R}^{N}_{\r,\cdot}\mathbf{L}^{N,\zeta,\sim}_{\s,\r,\cdot,\y}]_{\w}|\d\r\\
&\lesssim N^{-\frac12+2\mathrm{C}\delta_{\mathbf{S}}}\int_{\s}^{\t}\sum_{\w\in\Z}\exp(\tfrac{\kappa|\x-\w|}{N})\mathbf{H}^{N}_{\r,\t,\x,\w}\cdot N^{-1+\delta_{\mathbf{S}}}\sum_{|\z|\lesssim N^{1-\delta_{\mathbf{S}}}}\exp(\tfrac{\kappa|\w-\y|}{N})|\mathbf{L}^{N,\zeta,\sim}_{\s,\r,\w+\z,\y}|\d\r\\
&\lesssim_{\kappa} N^{-\frac12+2\mathrm{C}\delta_{\mathbf{S}}}\int_{\s}^{\t}\sum_{\w\in\Z}\exp(\tfrac{\kappa|\x-\w|}{N})\mathbf{H}^{N}_{\r,\t,\x,\w}\cdot N^{-1+\delta_{\mathbf{S}}}\sum_{|\z|\lesssim N^{1-\delta_{\mathbf{S}}}}\exp(\tfrac{\kappa|(\w+\z)-\y|}{N})|\mathbf{L}^{N,\zeta,\sim}_{\s,\r,\w+\z,\y}|\d\r,
\end{align*}
where the last line follows since {\small$|\z|\lesssim N^{1-\delta_{\mathbf{S}}}$}, implying {\small$\exp(\kappa|\w-\y|/N)\lesssim_{\kappa}\exp(\kappa|(\w+\z)-\y|/N)$}. Now, we square the previous estimate, take {\small$\mathrm{L}^{p}(\mathbb{P})$}-norms, and sum over {\small$\y\in\Z$}. If {\small$\delta_{\mathbf{S}}>0$} is small enough, then we get
\begin{align}
&\sum_{\y\in\Z}\exp(\tfrac{2\kappa|\x-\y|}{N})\Big\|\Big(\int_{\s}^{\t}\sum_{\w\in\Z}\mathbf{H}^{N}_{\r,\t,\x,\w}\cdot\mathbf{1}_{\mathfrak{t}_{\mathrm{ap}}\geq\r}\boldsymbol{\chi}^{(\zeta)}_{\w}\mathrm{Err}[\mathbf{R}^{N}_{\r,\cdot}\mathbf{L}^{N,\zeta,\sim}_{\s,\r,\cdot,\y}]_{\w}|\d\r\Big)^{2}\Big\|_{\mathrm{L}^{p}(\mathbb{P})}\nonumber\\
&\lesssim \sum_{\y\in\Z}N^{-\frac23}\Big\|\Big(\int_{\s}^{\t}\sum_{\w\in\Z}\exp(\tfrac{\kappa|\x-\w|}{N})\mathbf{H}^{N}_{\r,\t,\x,\w}\cdot N^{-1+\delta_{\mathbf{S}}}\sum_{|\z|\lesssim N^{1-\delta_{\mathbf{S}}}}\exp(\tfrac{\kappa|(\w+\z)-\y|}{N})|\mathbf{L}^{N,\zeta,\sim}_{\s,\r,\w+\z,\y}|\d\r\Big)^{2}\Big\|_{\mathrm{L}^{p}(\mathbb{P})}.\nonumber
\end{align}
Now, we use Cauchy-Schwarz to move the square inside the {\small$\d\r$}-integration and both summations; we can do this at a multiplicative cost of {\small$\mathrm{O}_{\kappa}(1)$} by summability of {\small$\exp(\kappa|\x-\w|/N)\mathbf{H}^{N}_{\r,\t,\x,\w}$} over {\small$\w\in\Z$} (see {Proposition \ref{prop:hke}}). Next, we move the summation over {\small$\y\in\Z$} past the integral and the other two summations, and we put the norm on only {\small$|\mathbf{L}^{N,\zeta,\sim}_{\s,\r,\w+\z,\y}|$} (using the triangle inequality for said norm). Ultimately, we obtain
\begin{align}
&\sum_{\y\in\Z}\exp(\tfrac{2\kappa|\x-\y|}{N})\Big\|\Big(\int_{\s}^{\t}\sum_{\w\in\Z}\mathbf{H}^{N}_{\r,\t,\x,\w}\cdot\mathbf{1}_{\mathfrak{t}_{\mathrm{ap}}\geq\r}\boldsymbol{\chi}^{(\zeta)}_{\w}\mathrm{Err}[\mathbf{R}^{N}_{\r,\cdot}\mathbf{L}^{N,\zeta,\sim}_{\s,\r,\cdot,\y}]_{\w}|\d\r\Big)^{2}\Big\|_{\mathrm{L}^{p}(\mathbb{P})}\nonumber\\
&\lesssim_{\kappa}N^{-\frac23}\int_{\s}^{\t}\sum_{\w\in\Z}\exp(\tfrac{\kappa|\x-\w|}{N})\mathbf{H}^{N}_{\r,\t,\x,\w}\cdot N^{-1+\delta_{\mathbf{S}}}\sum_{|\z|\lesssim N^{1-\delta_{\mathbf{S}}}}\sum_{\y\in\Z}\exp(\tfrac{2\kappa|(\w+\z)-\y|}{N})\||\mathbf{L}^{N,\zeta,\sim}_{\s,\r,\w+\z,\y}|^{2}\|_{\mathrm{L}^{p}(\mathbb{P})}\nonumber\\
&\lesssim_{\kappa} N^{-\frac23}\int_{\s}^{\t}\sup_{\x\in\Z}\sum_{\y\in\Z}\exp(\tfrac{2\kappa|\x-\y|}{N})\||\mathbf{L}^{N,\zeta,\sim}_{\s,\r,\x,\y}|^{2}\|_{\mathrm{L}^{p}(\mathbb{P})}\d\r\nonumber\\
&\lesssim N^{-\frac23}\sup_{\r\in[\s,\t]}\sup_{\x\in\Z}|\r-\s|^{\frac12}\sum_{\y\in\Z}\exp(\tfrac{\kappa|\x-\y|}{N})\||\mathbf{L}^{N,\zeta,\sim}_{\s,\r,\x,\y}|^{2}\|_{\mathrm{L}^{p}(\mathbb{P})}\cdot\int_{\s}^{\t}|\r-\s|^{-\frac12}\d\r\nonumber\\
&\lesssim N^{-\frac23}\sup_{\r\in[\s,\t]}\sup_{\x\in\Z}|\r-\s|^{\frac12}\sum_{\y\in\Z}\exp(\tfrac{\kappa|\x-\y|}{N})\||\mathbf{L}^{N,\zeta,\sim}_{\s,\r,\x,\y}|^{2}\|_{\mathrm{L}^{p}(\mathbb{P})}, \label{eq:lzetasimestimateII4}
\end{align}
where the second bound holds as the sum over {\small$\z$} is averaged out, and by {\color{black}the} summability of {\small$\exp(\kappa|\x-\w|/N)\mathbf{H}^{N}_{\r,\t,\x,\w}$} over {\small$\w\in\Z$} from {Proposition \ref{prop:hke}}. Let us now control each summand in \eqref{eq:lzetasimestimateII1c}. Recall the time-gradient formula in \eqref{eq:timegrad}. This and elementary integration manipulations let us write
\begin{align}
&\Big|\int_{\s}^{\t}\sum_{\w\in\Z}\mathbf{H}^{N}_{\r,\t,\x,\w}\cdot\boldsymbol{\chi}^{(\zeta)}_{\w}\grad^{\mathbf{T},\mathrm{av}}_{\mathfrak{t}_{\mathbf{Av}}}\boldsymbol{\Omega}^{\n,\mathfrak{l}_{1},\ldots,\mathfrak{l}_{\m}}_{\r,\w,\y}\d\r\Big|\nonumber\\
&=\Big|\mathfrak{t}_{\mathbf{Av}}^{-1}\int_{0}^{\mathfrak{t}_{\mathbf{Av}}}\d\u\int_{\s}^{\t}\sum_{\w\in\Z}\mathbf{H}^{N}_{\r,\t,\x,\w}\cdot\boldsymbol{\chi}^{(\zeta)}_{\w}\Big(\boldsymbol{\Omega}^{\n,\mathfrak{l}_{1},\ldots,\mathfrak{l}_{\m}}_{\r,\w}-\boldsymbol{\Omega}^{\n,\mathfrak{l}_{1},\ldots,\mathfrak{l}_{\m}}_{\r-\u,\w,\y}\Big)\d\r\Big|\nonumber\\
&=\Big|\mathfrak{t}_{\mathbf{Av}}^{-1}\int_{0}^{\mathfrak{t}_{\mathbf{Av}}}\d\u\Big(\int_{\s}^{\t}\sum_{\w\in\Z}\mathbf{H}^{N}_{\r,\t,\x,\w}\boldsymbol{\chi}^{(\zeta)}_{\w}\boldsymbol{\Omega}^{\n,\mathfrak{l}_{1},\ldots,\mathfrak{l}_{\m}}_{\r,\w,\y}\d\r-\int_{\s-\u}^{\t-\u}\sum_{\w\in\Z}\mathbf{H}^{N}_{\r+\u,\t,\x,\w}\boldsymbol{\chi}^{(\zeta)}_{\w}\boldsymbol{\Omega}^{\n,\mathfrak{l}_{1},\ldots,\mathfrak{l}_{\m}}_{\r,\w,\y}\d\r\Big)\Big|\nonumber\\
&\leq\Big|\mathfrak{t}_{\mathbf{Av}}^{-1}\int_{0}^{\mathfrak{t}_{\mathbf{Av}}}\d\u\int_{\s-\u}^{(\t-\u)\wedge\s}\sum_{\w\in\Z}\mathbf{H}^{N}_{\r+\u,\t,\x,\w}\boldsymbol{\chi}^{(\zeta)}_{\w}\boldsymbol{\Omega}^{\n,\mathfrak{l}_{1},\ldots,\mathfrak{l}_{\m}}_{\r,\w,\y}\d\r\Big|\label{eq:lzetasimestimateII5a}\\
&+\Big|\mathfrak{t}_{\mathbf{Av}}^{-1}\int_{0}^{\mathfrak{t}_{\mathbf{Av}}}\d\u\int_{(\t-\u)\wedge\s}^{\t-\u}\sum_{\w\in\Z}\Big(\mathbf{H}^{N}_{\r+\u,\t,\x,\w}-\mathbf{H}^{N}_{\r,\t,\x,\w}\Big)\boldsymbol{\chi}^{(\zeta)}_{\w}\boldsymbol{\Omega}^{\n,\mathfrak{l}_{1},\ldots,\mathfrak{l}_{\m}}_{\r,\w,\y}\d\r\Big|\label{eq:lzetasimestimateII5b}\\
&+\Big|\mathfrak{t}_{\mathbf{Av}}^{-1}\int_{0}^{\mathfrak{t}_{\mathbf{Av}}}\d\u\int_{\t-\u}^{\t}\sum_{\w\in\Z}\mathbf{H}^{N}_{\r,\t,\x,\w}\boldsymbol{\chi}^{(\zeta)}_{\w}\boldsymbol{\Omega}^{\n,\mathfrak{l}_{1},\ldots,\mathfrak{l}_{\m}}_{\r,\w,\y}\d\r\Big|\label{eq:lzetasimestimateII5c}\\
&+\Big|\mathfrak{t}_{\mathbf{Av}}^{-1}\int_{\s}^{(\t-\u)\wedge\s}\sum_{\w\in\Z}\mathbf{H}^{N}_{\r,\t,\x,\w}\boldsymbol{\chi}^{(\zeta)}_{\w}\boldsymbol{\Omega}^{\n,\mathfrak{l}_{1},\ldots,\mathfrak{l}_{\m}}_{\r,\w,\y}\d\r\Big|.\label{eq:lzetasimestimateII5d}
\end{align}
We now multiply each of \eqref{eq:lzetasimestimateII5a}-\eqref{eq:lzetasimestimateII5d} by {\small$\exp(\kappa|\x-\y|/N)$}, square it, take the {\small$\mathrm{L}^{p}(\mathbb{P})$}-norm, and sum over {\small$\y\in\Z$}. We first deal with \eqref{eq:lzetasimestimateII5a} to illustrate this. We claim that 
\begin{align}
&\sum_{\y\in\Z}\exp(\tfrac{2\kappa|\x-\y|}{N})\Big\|\Big(\mathfrak{t}_{\mathbf{Av}}^{-1}\int_{0}^{\mathfrak{t}_{\mathbf{Av}}}\d\u\int_{\s-\u}^{(\t-\u)\wedge\s}\sum_{\w\in\Z}\mathbf{H}^{N}_{\r+\u,\t,\x,\w}\boldsymbol{\chi}^{(\zeta)}_{\w}\boldsymbol{\Omega}^{\n,\mathfrak{l}_{1},\ldots,\mathfrak{l}_{\m}}_{\r,\w,\y}\d\r\Big)^{2}\Big\|_{\mathrm{L}^{p}(\mathbb{P})}\label{eq:lzetasimestimateII5e}\\
&\lesssim\mathfrak{t}_{\mathbf{Av}}^{-1}\int_{0}^{\mathfrak{t}_{\mathbf{Av}}}|\u|\d\u\int_{\s-\u}^{(\t-\u)\wedge\s}\sum_{|\w|\lesssim N^{1+\zeta}}\exp(\tfrac{\kappa|\x-\w|}{N})\mathbf{H}^{N}_{\r+\u,\t,\x,\w}\cdot\sum_{\y\in\Z}\exp(\tfrac{2\kappa|\w-\y|}{N})\||\boldsymbol{\Omega}^{\n,\mathfrak{l}_{1},\ldots,\mathfrak{l}_{\m}}_{\r,\w,\y}|^{2}\|_{\mathrm{L}^{p}(\mathbb{P})}\d\r.\nonumber
\end{align}
This follows by Cauchy-Schwarz with respect to all summations and integrations, noting that the time-integration domain has length {\small$\lesssim|\u|$}, and noting the summability of {\small$\exp(\kappa|\x-\w|/N)\mathbf{H}^{N}_{\r,\t,\x,\w}$} over {\small$\w\in\Z$} (see {Proposition \ref{prop:hke}}). Next, we refer to the formula \eqref{eq:lzetasimestimateII1d} for {\small$\boldsymbol{\Omega}^{\n,\mathfrak{l}_{1},\ldots,\mathfrak{l}_{\m}}$}. Moreover, we recall the estimate \eqref{eq:kzetaestimate-shortI3c}. Using these inputs, we claim that the following holds for any {\small$|\w|\lesssim N^{1+\zeta}$}:
\begin{align*}
&\||\boldsymbol{\Omega}^{\n,\mathfrak{l}_{1},\ldots,\mathfrak{l}_{\m}}_{\r,\w,\y}|^{2}\|_{\mathrm{L}^{p}(\mathbb{P})}\\
&\lesssim N^{2\delta_{\mathbf{S}}}\cdot N^{-1+\delta_{\mathbf{S}}}\sum_{\substack{|\z-\w|\lesssim N^{1-\delta_{\mathbf{S}}}}}\|\mathbf{1}_{\mathfrak{t}_{\mathrm{ap}}\geq\r}|\mathbf{Av}^{\mathbf{X},\mathfrak{q}_{\n}}_{\r,\z}|^{2}\cdot|\mathbf{R}^{N}_{\r,\z}|^{2}\cdot|\mathbf{L}^{N,\zeta,\sim}_{\s,\r,\z,\y}|^{2}\|_{\mathrm{L}^{p}(\mathbb{P})}\\
&\lesssim N^{-1+4\delta_{\mathbf{S}}}\cdot N^{-1+\delta_{\mathbf{S}}}\sum_{|\z-\w|\lesssim N^{1-\delta_{\mathbf{S}}}}\||\mathbf{L}^{N,\zeta,\sim}_{\s,\r,\z,\y}|^{2}\|_{\mathrm{L}^{p}(\mathbb{P})}
\end{align*}
The first estimate follows by Cauchy-Schwarz applied to the convolution in \eqref{eq:lzetasimestimateII1d}, noting that by \eqref{eq:kzetaestimate-shortI3c}, the kernel {\small$\mathrm{c}_{N,\n,\mathfrak{l}_{1},\ldots,\mathfrak{l}_{\m}}\grad^{\mathbf{X}}_{\mathfrak{l}_{1}}\grad^{\mathbf{X}}_{\mathfrak{l}_{\m}}\mathscr{S}^{N}_{\w-\z}$} is equal to {\small$\lesssim N^{\delta_{\mathbf{S}}}$} times a probability measure in {\small$\z\in\Z$} which is supported on the set {\small$|\w-\z|\lesssim N^{1-\delta_{\mathbf{S}}}$}. The second estimate follows by the a priori estimates in \eqref{eq:tap1} and \eqref{eq:tap2}. We now have
\begin{align*}
&\sum_{\y\in\Z}\exp(\tfrac{2\kappa|\w-\y|}{N})\||\boldsymbol{\Omega}^{\n,\mathfrak{l}_{1},\ldots,\mathfrak{l}_{\m}}_{\r,\w,\y}|^{2}\|_{\mathrm{L}^{p}(\mathbb{P})}\\
&\lesssim N^{-1+4\delta_{\mathbf{S}}}\cdot N^{-1+\delta_{\mathbf{S}}}\sum_{|\z-\w|\lesssim N^{1-\delta_{\mathbf{S}}}}\sum_{\y\in\Z}\exp(\tfrac{2\kappa|\w-\y|}{N})\||\mathbf{L}^{N,\zeta,\sim}_{\s,\r,\z,\y}|^{2}\|_{\mathrm{L}^{p}(\mathbb{P})}\\
&\lesssim_{\kappa} N^{-1+4\delta_{\mathbf{S}}}\sup_{\r\in[\s,\t]}\sup_{\x\in\Z}|\r-\s|^{\frac12}\sum_{\y\in\Z}\exp(\tfrac{2\kappa|\x-\y|}{N})\|\mathbf{L}^{N,\zeta,\sim}_{\s,\r,\x,\y}|^{2}\|_{\mathrm{L}^{p}(\mathbb{P})}\cdot|\r-\s|^{-\frac12},
\end{align*}
where the final inequality follows because {\small$\exp(2\kappa|\w-\y|/N)\lesssim_{\kappa}\exp(2\kappa|\z-\y|/N)$} if {\small$|\z-\y|\lesssim N^{1-\delta_{\mathbf{S}}}$}, after which we take the supremum over all {\small$\z\in\Z$}, and the sum over {\small$\z$} which remains is averaged out by the factor of {\small$N^{-1+\delta_{\mathbf{S}}}$}. If we now combine the previous display with \eqref{eq:lzetasimestimateII5e}, then we obtain
\begin{align}
&\sum_{\y\in\Z}\exp(\tfrac{2\kappa|\x-\y|}{N})\Big\|\Big(\mathfrak{t}_{\mathbf{Av}}^{-1}\int_{0}^{\mathfrak{t}_{\mathbf{Av}}}\d\u\int_{\s-\u}^{(\t-\u)\wedge\s}\sum_{\w\in\Z}\mathbf{H}^{N}_{\r+\u,\t,\x,\w}\boldsymbol{\chi}^{(\zeta)}_{\w}\boldsymbol{\Omega}^{\n,\mathfrak{l}_{1},\ldots,\mathfrak{l}_{\m}}_{\r,\w,\y}\d\r\Big)^{2}\Big\|_{\mathrm{L}^{p}(\mathbb{P})}\nonumber\\
&\lesssim_{\kappa}N^{-1+4\delta_{\mathbf{S}}}\sup_{\r\in[\s,\t]}\sup_{\x\in\Z}|\r-\s|^{\frac12}\sum_{\y\in\Z}\exp(\tfrac{2\kappa|\x-\y|}{N})\|\mathbf{L}^{N,\zeta,\sim}_{\s,\r,\x,\y}|^{2}\|_{\mathrm{L}^{p}(\mathbb{P})}\nonumber\\
&\quad\times\mathfrak{t}_{\mathbf{Av}}^{-1}\int_{0}^{\mathfrak{t}_{\mathbf{Av}}}|\u|\d\u\int_{\s-\u}^{(\t-\u)\wedge\s}\sum_{\w\in\Z}\exp(\tfrac{\kappa|\x-\w|}{N})\mathbf{H}^{N}_{\r+\u,\t,\x,\w}|\r-\s|^{-\frac12}\d\r\nonumber\\
&\lesssim_{\kappa}N^{-1+4\delta_{\mathbf{S}}}\sup_{\r\in[\s,\t]}\sup_{\x\in\Z}|\r-\s|^{\frac12}\sum_{\y\in\Z}\exp(\tfrac{2\kappa|\x-\y|}{N})\|\mathbf{L}^{N,\zeta,\sim}_{\s,\r,\x,\y}|^{2}\|_{\mathrm{L}^{p}(\mathbb{P})}\cdot\mathfrak{t}_{\mathbf{Av}}^{-1}\int_{0}^{\mathfrak{t}_{\mathbf{Av}}}|\u|^{\frac32}\d\u\nonumber\\
&\lesssim N^{-2-11\delta_{\mathbf{S}}}\sup_{\r\in[\s,\t]}\sup_{\x\in\Z}|\r-\s|^{\frac12}\sum_{\y\in\Z}\exp(\tfrac{2\kappa|\x-\y|}{N})\|\mathbf{L}^{N,\zeta,\sim}_{\s,\r,\x,\y}|^{2}\|_{\mathrm{L}^{p}(\mathbb{P})},\label{eq:lzetasimestimateII6}
\end{align}
where the third line follows due to {\color{black}the} summability of {\small$\exp(\kappa|\x-\w|/N)\mathbf{H}^{N}_{\r,\t,\x,\w}$} over {\small$\w\in\Z$} (see {Proposition \ref{prop:hke}}), and the last line follows since {\small$\mathfrak{t}_{\mathbf{Av}}=N^{-2/3-10\delta_{\mathbf{S}}}$} (see Definition \ref{definition:eq-operators}). Note that the only feature we used about the integration-domain {\small$[\s-\u,(\t-\u)\wedge\s]$} to obtain \eqref{eq:lzetasimestimateII6} is that its size is {\small$\lesssim\u$}. Thus, the same argument gives 
\begin{align}
&\sum_{\y\in\Z}\exp(\tfrac{2\kappa|\x-\y|}{N})\Big\|\Big(\mathfrak{t}_{\mathbf{Av}}^{-1}\int_{0}^{\mathfrak{t}_{\mathbf{Av}}}\d\u\int_{\t-\u}^{\t}\sum_{\w\in\Z}\mathbf{H}^{N}_{\r,\t,\x,\w}\boldsymbol{\chi}^{(\zeta)}_{\w}\boldsymbol{\Omega}^{\n,\mathfrak{l}_{1},\ldots,\mathfrak{l}_{\m}}_{\r,\w,\y}\d\r\Big)^{2}\Big\|_{\mathrm{L}^{p}(\mathbb{P})}\nonumber\\
&\lesssim_{\kappa}N^{-2-11\delta_{\mathbf{S}}}\sup_{\r\in[\s,\t]}\sup_{\x\in\Z}|\r-\s|^{\frac12}\sum_{\y\in\Z}\exp(\tfrac{2\kappa|\x-\y|}{N})\|\mathbf{L}^{N,\zeta,\sim}_{\s,\r,\x,\y}|^{2}\|_{\mathrm{L}^{p}(\mathbb{P})}\label{eq:lzetasimestimateII7}
\end{align}
and
\begin{align}
&\sum_{\y\in\Z}\exp(\tfrac{2\kappa|\x-\y|}{N})\Big\|\Big(\mathfrak{t}_{\mathbf{Av}}^{-1}\int_{\s}^{(\t-\u)\wedge\s}\sum_{\w\in\Z}\mathbf{H}^{N}_{\r,\t,\x,\w}\boldsymbol{\chi}^{(\zeta)}_{\w}\boldsymbol{\Omega}^{\n,\mathfrak{l}_{1},\ldots,\mathfrak{l}_{\m}}_{\r,\w,\y}\d\r.\Big)^{2}\Big\|_{\mathrm{L}^{p}(\mathbb{P})}\nonumber\\
&\lesssim_{\kappa}N^{-2-11\delta_{\mathbf{S}}}\sup_{\r\in[\s,\t]}\sup_{\x\in\Z}|\r-\s|^{\frac12}\sum_{\y\in\Z}\exp(\tfrac{2\kappa|\x-\y|}{N})\|\mathbf{L}^{N,\zeta,\sim}_{\s,\r,\x,\y}|^{2}\|_{\mathrm{L}^{p}(\mathbb{P})}.\label{eq:lzetasimestimateII8}
\end{align}
(Indeed, {\small$[\s,(\t-\u)\wedge\s]$} is nonempty if and only if {\small$\t-\u<\s$}. But {\small$\s\leq\t$}; thus, {\small$\t-\u<\s$} only happens if {\small$\t-\s\in[0,\u)$}, so that {\small$|(\t-\u)-\s|\leq\u$}.) We now move to \eqref{eq:lzetasimestimateII5b}, whose time-integral is not necessarily supported on a small interval. Again, we recall \eqref{eq:lzetasimestimateII1d}. This lets us compute
\begin{align*}
&\sum_{\w\in\Z}(\mathbf{H}^{N}_{\r+\u,\t,\x,\w}-\mathbf{H}^{N}_{\r,\t,\x,\w})\boldsymbol{\chi}^{(\zeta)}_{\w}\boldsymbol{\Omega}^{\n,\mathfrak{l}_{1},\ldots,\mathfrak{l}_{\m}}_{\r,\w,\y}\nonumber\\
&=\mathbf{1}_{\mathfrak{t}_{\mathrm{ap}}\geq\r}\sum_{\w\in\Z}\sum_{\z\in\Z}(\mathbf{H}^{N}_{\r+\u,\t,\x,\w}-\mathbf{H}^{N}_{\r,\t,\x,\w})\cdot\boldsymbol{\chi}^{(\zeta)}_{\w}\cdot\mathrm{c}_{\n,\mathfrak{l}_{1},\ldots,\mathfrak{l}_{\m}}\grad^{\mathbf{X}}_{\mathfrak{l}_{1}}\ldots\grad^{\mathbf{X}}_{\mathfrak{l}_{\m}}\Big(\mathscr{S}^{N}_{\w-\z}\mathbf{Av}^{\mathbf{X},\mathfrak{q}_{\n}}_{\r,\z}\mathbf{R}^{N}_{\r,\z}\mathbf{L}^{N,\zeta,\sim}_{\s,\r,\z,\y}\Big)\\
&\lesssim N^{-\frac12+\delta_{\mathbf{S}}}\sum_{\z\in\Z}|\mathbf{L}^{N,\zeta,\sim}_{\s,\r,\z,\y}|\cdot\Big|\sum_{\w\in\Z}(\mathbf{H}^{N}_{\r+\u,\t,\x,\w}-\mathbf{H}^{N}_{\r,\t,\x,\w})\cdot\boldsymbol{\chi}^{(\zeta)}_{\w}\cdot\mathrm{c}_{\n,\mathfrak{l}_{1},\ldots,\mathfrak{l}_{\m}}\grad^{\mathbf{X}}_{\mathfrak{l}_{1}}\ldots\grad^{\mathbf{X}}_{\mathfrak{l}_{\m}}\mathscr{S}^{N}_{\w-\z}\Big|,
\end{align*}
where the last line is justified as follows. First, interchange the {\small$\w,\z$}-sums in the second line. Then, after replacing everything by an absolute value, we can use \eqref{eq:kzetaestimate-shortI3c} and the support of {\small$\w\mapsto\boldsymbol{\chi}^{(\zeta)}_{\w}$} to restrict the {\small$\z$}-summation to {\small$|\z|\lesssim N^{1+\zeta}$}, at which point we use {\small$\mathfrak{t}_{\mathrm{ap}}\geq\r$} and \eqref{eq:tap1}-\eqref{eq:tap2} to get a priori estimates for the {\small$\mathbf{Av}^{\mathbf{X},\mathfrak{q}_{\n}}$} and {\small$\mathbf{R}^{N}$} terms. Next, we use the fundamental theorem of calculus and the \abbr{PDE} \eqref{eq:heatkernel} for {\small$\mathbf{H}^{N}$} to write
\begin{align*}
&\sum_{\w\in\Z}(\mathbf{H}^{N}_{\r+\u,\t,\x,\w}-\mathbf{H}^{N}_{\r,\t,\x,\w})\cdot\boldsymbol{\chi}^{(\zeta)}_{\w}\cdot\mathrm{c}_{\n,\mathfrak{l}_{1},\ldots,\mathfrak{l}_{\m}}\grad^{\mathbf{X}}_{\mathfrak{l}_{1}}\ldots\grad^{\mathbf{X}}_{\mathfrak{l}_{\m}}\mathscr{S}^{N}_{\w-\z}\\
&=\int_{0}^{\u}\d\upsilon\sum_{\w\in\Z}\partial_{\upsilon}\mathbf{H}^{N}_{\r+\upsilon,\t,\x,\w}\cdot\boldsymbol{\chi}^{(\zeta)}_{\w}\cdot\mathrm{c}_{\n,\mathfrak{l}_{1},\ldots,\mathfrak{l}_{\m}}\grad^{\mathbf{X}}_{\mathfrak{l}_{1}}\ldots\grad^{\mathbf{X}}_{\mathfrak{l}_{\m}}\mathscr{S}^{N}_{\w-\z}\\
&=\int_{0}^{\u}\d\upsilon\sum_{\w\in\Z}\mathscr{T}_{N}\mathbf{H}^{N}_{\r+\upsilon,\t,\x,\w}\cdot\Big(\boldsymbol{\chi}^{(\zeta)}_{\w}\cdot\mathrm{c}_{\n,\mathfrak{l}_{1},\ldots,\mathfrak{l}_{\m}}\grad^{\mathbf{X}}_{\mathfrak{l}_{1}}\ldots\grad^{\mathbf{X}}_{\mathfrak{l}_{\m}}\mathscr{S}^{N}_{\w-\z}\Big).
\end{align*}
Next, we recall that {\small$\mathscr{T}_{N}=\mathrm{O}(N^{2})\Delta$}, where {\small$\Delta=-\grad^{\mathbf{X}}_{-1}\grad^{\mathbf{X}}_{1}$} is the discrete Laplacian. Moreover, {\small$\mathbf{H}^{N}$} is the kernel for the {\small$\mathscr{T}_{N}$}-semigroup, which commutes with {\small$\mathscr{T}_{N}$}-itself. Thus, we can move {\small$\mathscr{T}_{N}$} onto the term in parentheses in the last line. However, {\small$\boldsymbol{\chi}^{(\zeta)}$} is smooth on length-scales {\small$\gg N^{1-\delta_{\mathbf{S}}}$}, and \eqref{eq:kzetaestimate-shortI3c} shows that {\small$\mathrm{c}_{\n,\mathfrak{l}_{1},\ldots,\mathfrak{l}_{\m}}\grad^{\mathbf{X}}_{\mathfrak{l}_{1}}\ldots\grad^{\mathbf{X}}_{\mathfrak{l}_{\m}}\mathscr{S}^{N}_{\w-\z}$} is {\small$\mathrm{O}(N^{\delta_{\mathbf{S}}})$} times a function that is also smooth on length-scales of {\small$N^{1-\delta_{\mathbf{S}}}$}. Thus, the proof of \eqref{eq:kzetaestimate-shortI3c} also gives
\begin{align*}
\mathscr{T}_{N}\Big(\boldsymbol{\chi}^{(\zeta)}_{\w}\cdot\mathrm{c}_{\n,\mathfrak{l}_{1},\ldots,\mathfrak{l}_{\m}}\grad^{\mathbf{X}}_{\mathfrak{l}_{1}}\ldots\grad^{\mathbf{X}}_{\mathfrak{l}_{\m}}\mathscr{S}^{N}_{\w-\z}\Big)\lesssim N^{2\delta_{\mathbf{S}}}N^{-1+2\delta_{\mathbf{S}}}\mathbf{1}_{|\w-\z|\lesssim N^{1-\delta_{\mathbf{S}}}},
\end{align*}
where the first factor of {\small$N^{2\delta_{\mathbf{S}}}$} is there since {\small$\mathscr{T}_{N}$} is a second-order gradient and introduces two factors of {\small$N^{\delta_{\mathbf{S}}}$} (as opposed to one factor of {\small$N^{\delta_{\mathbf{S}}}$} from a first-order gradient in \eqref{eq:kzetaestimate-shortI3c}). If we combine the last three displays, we get
\begin{align*}
\sum_{\w\in\Z}(\mathbf{H}^{N}_{\r+\u,\t,\x,\w}-\mathbf{H}^{N}_{\r,\t,\x,\w})\boldsymbol{\chi}^{(\zeta)}_{\w}\boldsymbol{\Omega}^{\n,\mathfrak{l}_{1},\ldots,\mathfrak{l}_{\m}}_{\r,\w,\y}&\lesssim N^{-\frac12+3\delta_{\mathbf{S}}}\sum_{\z\in\Z}|\mathbf{L}^{N,\zeta,\sim}_{\s,\r,\z,\y}|\cdot\int_{0}^{\u}\d\upsilon\sum_{\w\in\Z}\mathbf{H}^{N}_{\r+\upsilon,\t,\x,\w}\cdot N^{-1+2\delta_{\mathbf{S}}}\mathbf{1}_{|\w-\z|\lesssim N^{1-\delta_{\mathbf{S}}}}.
\end{align*}
Using the previous display as an input, we claim that the following holds:
\begin{align}
&\Big\|\Big(\mathfrak{t}_{\mathbf{Av}}^{-1}\int_{0}^{\mathfrak{t}_{\mathbf{Av}}}\d\u\int_{(\t-\u)\wedge\s}^{\t-\u}\sum_{\w\in\Z}\Big(\mathbf{H}^{N}_{\r+\u,\t,\x,\w}-\mathbf{H}^{N}_{\r,\t,\x,\w}\Big)\boldsymbol{\chi}^{(\zeta)}_{\w}\boldsymbol{\Omega}^{\n,\mathfrak{l}_{1},\ldots,\mathfrak{l}_{\m}}_{\r,\w,\y}\d\r\Big)^{2}\Big\|_{\mathrm{L}^{p}(\mathbb{P})}\nonumber\\
&\lesssim\mathfrak{t}_{\mathbf{Av}}^{-1}\int_{0}^{\mathfrak{t}_{\mathbf{Av}}}\d\u\int_{(\t-\u)\wedge\s}^{\t-\u}\Big\|\Big(\sum_{\w\in\Z}\Big(\mathbf{H}^{N}_{\r+\u,\t,\x,\w}-\mathbf{H}^{N}_{\r,\t,\x,\w}\Big)\boldsymbol{\chi}^{(\zeta)}_{\w}\boldsymbol{\Omega}^{\n,\mathfrak{l}_{1},\ldots,\mathfrak{l}_{\m}}_{\r,\w,\y}\Big)^{2}\Big\|_{\mathrm{L}^{p}(\mathbb{P})}\d\r\nonumber\\
&\lesssim N^{-1+6\delta_{\mathbf{S}}}\mathfrak{t}_{\mathbf{Av}}^{-1}\int_{0}^{\mathfrak{t}_{\mathbf{Av}}}\d\u\int_{(\t-\u)\wedge\s}^{\t-\u}\Big\|\Big(\int_{0}^{\u}\d\upsilon\sum_{\w\in\Z}\mathbf{H}^{N}_{\r+\upsilon,\t,\x,\w}\cdot N^{-1+2\delta_{\mathbf{S}}}\sum_{|\z-\w|\lesssim N^{1-\delta_{\mathbf{S}}}}|\mathbf{L}^{N,\zeta,\sim}_{\s,\r,\z,\y}|\Big)^{2}\Big\|_{\mathrm{L}^{p}(\mathbb{P})}\d\r\nonumber\\
&\lesssim N^{-1+8\delta_{\mathbf{S}}}\mathfrak{t}_{\mathbf{Av}}^{-1}\int_{0}^{\mathfrak{t}_{\mathbf{Av}}}|\u|\d\u\int_{(\t-\u)\wedge\s}^{\t-\u}\int_{0}^{\u}\d\upsilon\sum_{\w\in\Z}\mathbf{H}^{N}_{\r+\upsilon,\t,\x,\w}\cdot N^{-1+\delta_{\mathbf{S}}}\sum_{|\z-\w|\lesssim N^{1-\delta_{\mathbf{S}}}}\||\mathbf{L}^{N,\zeta,\sim}_{\s,\r,\z,\y}|^{2}\|_{\mathrm{L}^{p}(\mathbb{P})}\d\r.\nonumber
\end{align}
The first estimate follows by Cauchy-Schwarz and triangle inequality for the {\small$\mathrm{L}^{p}(\mathbb{P})$}-norm through the two time-integrals. The last estimate follows by Cauchy-Schwarz through the {\small$\d\upsilon$}-integral and the two summations. Next, we record the following estimate, which again uses {\small$\exp(2\kappa|\x-\y|/N)\lesssim_{\kappa}\exp(2\kappa|\x-\w|/N)\exp(2\kappa|\z-\y|/N)$} for {\small$|\z-\w|\lesssim N^{1-\delta_{\mathbf{S}}}$}, summability of {\small$\exp(\kappa|\x-\w|/N)\mathbf{H}^{N}_{\r,\t,\x,\w}$} over {\small$\w\in\Z$} (see {Proposition \ref{prop:hke}}), and elementary manipulations:
\begin{align}
&\sum_{\y\in\Z}\exp(\tfrac{2\kappa|\x-\y|}{N})\sum_{\w\in\Z}\mathbf{H}^{N}_{\r+\upsilon,\t,\x,\w}\cdot N^{-1+\delta_{\mathbf{S}}}\sum_{|\z-\w|\lesssim N^{1-\delta_{\mathbf{S}}}}\||\mathbf{L}^{N,\zeta,\sim}_{\s,\r,\z,\y}|^{2}\|_{\mathrm{L}^{p}(\mathbb{P})}\nonumber\\
&\lesssim_{\kappa}\sum_{\w\in\Z}\exp(\tfrac{2\kappa|\x-\w|}{N})\mathbf{H}^{N}_{\r+\upsilon,\t,\x,\w}\cdot N^{-1+\delta_{\mathbf{S}}}\sum_{|\z-\w|\lesssim N^{1-\delta_{\mathbf{S}}}}\sum_{\y\in\Z}\exp(\tfrac{2\kappa|\z-\y|}{N})\||\mathbf{L}^{N,\zeta,\sim}_{\s,\r,\z,\y}|^{2}\|_{\mathrm{L}^{p}(\mathbb{P})}\nonumber\\
&\lesssim_{\kappa}|\r-\s|^{-\frac12}\cdot\sup_{\r\in[\s,\t]}\sup_{\x\in\Z}|\r-\s|^{\frac12}\sum_{\y\in\Z}\exp(\tfrac{2\kappa|\x-\y|}{N})\||\mathbf{L}^{N,\zeta,\sim}_{\s,\r,\x,\y}|^{2}\|_{\mathrm{L}^{p}(\mathbb{P})}.\nonumber
\end{align}
If we combine the previous two displays, then we have 
\begin{align}
&\sum_{\y\in\Z}\exp(\tfrac{2\kappa|\x-\y|}{N})\Big\|\Big(\mathfrak{t}_{\mathbf{Av}}^{-1}\int_{0}^{\mathfrak{t}_{\mathbf{Av}}}\d\u\int_{(\t-\u)\wedge\s}^{\t-\u}\sum_{\w\in\Z}\Big(\mathbf{H}^{N}_{\r+\u,\t,\x,\w}-\mathbf{H}^{N}_{\r,\t,\x,\w}\Big)\boldsymbol{\chi}^{(\zeta)}_{\w}\boldsymbol{\Omega}^{\n,\mathfrak{l}_{1},\ldots,\mathfrak{l}_{\m}}_{\r,\w,\y}\d\r\Big)^{2}\Big\|_{\mathrm{L}^{p}(\mathbb{P})}\nonumber\\
&\lesssim_{\kappa}N^{-1+8\delta_{\mathbf{S}}}\sup_{\r\in[\s,\t]}\sup_{\x\in\Z}|\r-\s|^{\frac12}\sum_{\y\in\Z}\exp(\tfrac{2\kappa|\x-\y|}{N})\||\mathbf{L}^{N,\zeta,\sim}_{\s,\r,\x,\y}|^{2}\|_{\mathrm{L}^{p}(\mathbb{P})}\times\mathfrak{t}_{\mathbf{Av}}^{-1}\int_{0}^{\mathfrak{t}_{\mathbf{Av}}}|\u|\d\u\int_{(\t-\u)\wedge\s}^{\t-\u}|\r-\s|^{-\frac12}\int_{0}^{\u}\d\upsilon\d\r\nonumber\\
&\lesssim N^{-1+8\delta_{\mathbf{S}}}\mathfrak{t}_{\mathbf{Av}}^{2}\sup_{\r\in[\s,\t]}\sup_{\x\in\Z}|\r-\s|^{\frac12}\sum_{\y\in\Z}\exp(\tfrac{2\kappa|\x-\y|}{N})\||\mathbf{L}^{N,\zeta,\sim}_{\s,\r,\x,\y}|^{2}\|_{\mathrm{L}^{p}(\mathbb{P})}\nonumber\\
&\lesssim N^{-\frac73+8\delta_{\mathbf{S}}}\sup_{\r\in[\s,\t]}\sup_{\x\in\Z}|\r-\s|^{\frac12}\sum_{\y\in\Z}\exp(\tfrac{2\kappa|\x-\y|}{N})\||\mathbf{L}^{N,\zeta,\sim}_{\s,\r,\x,\y}|^{2}\|_{\mathrm{L}^{p}(\mathbb{P})}.\label{eq:lzetasimestimateII9}
\end{align}
We clarify that the last line follows because {\small$\mathfrak{t}_{\mathbf{Av}}\lesssim N^{-2/3}$}; see Definition \ref{definition:eq-operators}. We now combine \eqref{eq:lzetasimestimateII5a}-\eqref{eq:lzetasimestimateII5d} with \eqref{eq:lzetasimestimateII6}, \eqref{eq:lzetasimestimateII7}, \eqref{eq:lzetasimestimateII8}, and \eqref{eq:lzetasimestimateII9}. This implies that 
\begin{align}
&\sum_{\y\in\Z}\exp(\tfrac{2\kappa|\x-\y|}{N})\Big\|\Big(\int_{\s}^{\t}\sum_{\w\in\Z}\mathbf{H}^{N}_{\r,\t,\x,\w}\cdot\boldsymbol{\chi}^{(\zeta)}_{\w}\grad^{\mathbf{T},\mathrm{av}}_{\mathfrak{t}_{\mathbf{Av}}}\boldsymbol{\Omega}^{\n,\mathfrak{l}_{1},\ldots,\mathfrak{l}_{\m}}_{\r,\w,\y}\d\r\Big)^{2}\Big\|_{\mathrm{L}^{p}(\mathbb{P})}\nonumber\\
&\lesssim N^{-2-11\delta_{\mathbf{S}}}\sup_{\r\in[\s,\t]}\sup_{\x\in\Z}|\r-\s|^{\frac12}\sum_{\y\in\Z}\exp(\tfrac{2\kappa|\x-\y|}{N})\||\mathbf{L}^{N,\zeta,\sim}_{\s,\r,\x,\y}|^{2}\|_{\mathrm{L}^{p}(\mathbb{P})}.\label{eq:lzetasimestimateII10}
\end{align}
We combine the previous display with \eqref{eq:lzetasimestimateII1a}-\eqref{eq:lzetasimestimateII1c}, \eqref{eq:lzetasimestimateII2}, \eqref{eq:lzetasimestimateII3}, and \eqref{eq:lzetasimestimateII4}. This gives
\begin{align}
&|\t-\s|^{\frac12}\sum_{\y\in\Z}\exp(\tfrac{2\kappa|\x-\y|}{N})\||\mathbf{L}^{N,\zeta,\sim}_{\s,\t,\x,\y}|^{2}\|_{\mathrm{L}^{p}(\mathbb{P})}\nonumber\\
&\lesssim_{\kappa} N^{-1}+|\t-\s|^{\frac12}\int_{\s}^{\t}|\t-\r|^{-\frac12}|\r-\s|^{-\frac12}\cdot\sup_{\x\in\Z}|\r-\s|^{\frac12}\sum_{\y\in\Z}\exp(\tfrac{2\kappa|\x-\y|}{N})\|\mathbf{L}^{N,\zeta,\sim}_{\s,\r,\x,\y}\|_{\mathrm{L}^{p}(\mathbb{P})}^{2}\d\r\nonumber\\
&+N^{-11\delta_{\mathbf{S}}}|\t-\s|^{\frac12}\sup_{\r\in[\s,\t]}\sup_{\x\in\Z}|\r-\s|^{\frac12}\sum_{\y\in\Z}\exp(\tfrac{2\kappa|\x-\y|}{N})\||\mathbf{L}^{N,\zeta,\sim}_{\s,\r,\x,\y}|^{2}\|_{\mathrm{L}^{p}(\mathbb{P})}.\nonumber
\end{align}
Let us now define {\small$\boldsymbol{\Lambda}_{\s,\t}:=\sup_{\x\in\Z}|\t-\s|^{\frac12}\sum_{\y\in\Z}\exp(\tfrac{2\kappa|\x-\y|}{N})\||\mathbf{L}^{N,\zeta,\sim}_{\s,\r,\x,\y}|^{2}\|_{\mathrm{L}^{p}(\mathbb{P})}$}. The previous display now reads
\begin{align}
\boldsymbol{\Lambda}_{\s,\t}&\lesssim_{\kappa} N^{-1}+|\t-\s|^{\frac12}\int_{\s}^{\t}|\t-\r|^{-\frac12}|\r-\s|^{-\frac12}\boldsymbol{\Lambda}_{\s,\r}\d\r+N^{-11\delta_{\mathbf{S}}}|\t-\s|^{\frac12}\sup_{\r\in[\s,\t]}\boldsymbol{\Lambda}_{\s,\r}.\label{eq:lzetasimestimateII11}
\end{align}
Now, let {\small$\t_{0}\in[\s,1]$} be the first time that {\small$\boldsymbol{\Lambda}_{\s,\t}\gtrsim_{\kappa} N^{-1+\e}$}, where we have fixed a (small) {\small$\e>0$} satisfying {\small$\e\leq\delta_{\mathbf{S}}$}. (If no such time exists, we set {\small$\t_{0}=1$}.) In this case, we have the following for any {\small$\t\in[\s,\t_{0}]$}: 
\begin{align*}
\boldsymbol{\Lambda}_{\s,\t}&\lesssim_{\kappa}N^{-1}+\int_{\s}^{\t}|\t-\r|^{-\frac12}|\r-\s|^{-\frac12}\boldsymbol{\Lambda}_{\s,\r}\d\r,
\end{align*}
which, by the Gronwall inequality, implies
\begin{align*}
\sup_{\t\in[\s,\t_{0}]}\boldsymbol{\Lambda}_{\s,\t}\lesssim_{\kappa}N^{-1}\cdot\sup_{\t\in[\s,\t_{0}]}\exp(\int_{\s}^{\t}|\t-\r|^{-\frac12}|\r-\s|^{-\frac12}\d\r\Big)\lesssim N^{-1}.
\end{align*}
Since the {\color{black}rightmost term} is {\small$\ll N^{1+\e}$}, we must have {\small$\t_{0}=1$} and thus {\small$\boldsymbol{\Lambda}_{\s,\t}\lesssim N^{1+\e}$} for all {\small$\t\in[\s,1]$}. Now, we again use  \eqref{eq:lzetasimestimateII1a}-\eqref{eq:lzetasimestimateII1c}, but now we subtract {\small$\mathbf{H}^{N}_{\s,\t,\x,\y}$} from both sides. Then, we use \eqref{eq:lzetasimestimateII3}, \eqref{eq:lzetasimestimateII4}, and \eqref{eq:lzetasimestimateII10} to get the following estimate, which is essentially \eqref{eq:lzetasimestimateII11}, but we have subtracted off the first term on the \abbr{RHS}: 
\begin{align}
&\sup_{\x\in\Z}|\t-\s|^{\frac12}\sum_{\y\in\Z}\exp(\tfrac{2\kappa|\x-\y|}{N})\||\mathbf{L}^{N,\zeta,\sim}_{\s,\t,\x,\y}-\mathbf{H}^{N}_{\s,\t,\x,\y}|^{2}\|_{\mathrm{L}^{p}(\mathbb{P})}\nonumber\\
&\lesssim |\t-\s|^{\frac12}\int_{\s}^{\t}|\t-\r|^{-\frac12}|\r-\s|^{-\frac12}\boldsymbol{\Lambda}_{\s,\r}\d\r+N^{-13\delta_{\mathbf{S}}}|\t-\s|^{\frac12}\sup_{\r\in[\s,\t]}\boldsymbol{\Lambda}_{\s,\r}\lesssim N^{-1+\e}|\t-\s|^{\frac12}.\label{eq:lzetasimestimateII12a}
\end{align}
It is now a standard Chebyshev inequality argument to use the above display for large enough {\small$p$} and deduce {\color{black}that}
\begin{align*}
\sum_{\y\in\Z}\exp(\tfrac{2\kappa|\x-\y|}{N})|\mathbf{L}^{N,\zeta,\sim}_{\s,\t,\x,\y}-\mathbf{H}^{N}_{\s,\t,\x,\y}|^{2}\lesssim_{\kappa}N^{-1+2\e}
\end{align*}
holds with probability {\small$1-\mathrm{O}_{\mathrm{D}}(N^{-\mathrm{D}})$} for any {\small$\mathrm{D}>0$}. Now, fix any {\small$\mathrm{D}_{1},\mathrm{D}_{2}>0$}. A union bound and the previous display {\color{black}imply} now the following with high probability:
\begin{align*}
\sup_{\s\in[0,1]\cap N^{-\mathrm{D}_{1}}\Z}\sup_{\t\in[\s,1]\cap N^{-\mathrm{D}_{1}}\Z}\sup_{|\x|\lesssim N^{\mathrm{D}_{2}}}\sum_{\y\in\Z}\exp(\tfrac{2\kappa|\x-\y|}{N})|\mathbf{L}^{N,\zeta,\sim}_{\s,\t,\x,\y}-\mathbf{H}^{N}_{\s,\t,\x,\y}|^{2}\lesssim_{\kappa}N^{-1+\delta}.
\end{align*}
Because {\small$\mathrm{D}_{1}>0$} is arbitrarily large (in particular, {\small$N^{-\mathrm{D}_{1}}$} is much smaller than the microscopic time-scale {\small$N^{-2}$}), a standard continuity argument lets us remove the intersections with {\small$N^{-\mathrm{D}_{1}}\Z$}, so that 
\begin{align}
\sup_{\t\in[\s,1]}\sup_{|\x|\lesssim N^{\mathrm{D}_{2}}}\sum_{\y\in\Z}\exp(\tfrac{2\kappa|\x-\y|}{N})|\mathbf{L}^{N,\zeta,\sim}_{\s,\t,\x,\y}-\mathbf{H}^{N}_{\s,\t,\x,\y}|^{2}\lesssim_{\kappa}N^{-1+\delta}.\label{eq:lzetasimestimateII12}
\end{align}
It now suffices to handle {\small$|\x|\gtrsim N^{\mathrm{D}_{2}}$}. By \eqref{eq:lzetasim-sde}, if {\small$|\x|\gtrsim N^{1+\zeta_{\mathrm{large}}+\zeta}$}, then {\small$\d\mathbf{L}^{N,\zeta,\sim}_{\s,\t,\x,\y}=\mathscr{T}_{N}\mathbf{L}^{N,\zeta,\sim}_{\s,\t,\x,\y}$}, which is the same evolution equation as the one \eqref{eq:heatkernel} for {\small$\mathbf{H}^{N}$}. Thus, since {\small$\mathscr{T}_{N}$} is {\small$\mathrm{O}(N^{2})$} times a discrete Laplacian on {\small$\Z$}, for any such {\small$\x$}, we have the deterministic estimate
\begin{align*}
\sum_{\y\in\Z}\exp(\tfrac{2\kappa|\x-\y|}{N})|\mathbf{L}^{N,\zeta,\sim}_{\s,\t,\x,\y}-\mathbf{H}^{N}_{\s,\t,\x,\y}|^{2}&\lesssim N^{4}\sum_{|\z|\leq1}\int_{\s}^{\t}\sum_{\y\in\Z}\exp(\tfrac{2\kappa|\x-\y|}{N})|\mathbf{L}^{N,\zeta,\sim}_{\s,\t,\x+\z,\y}-\mathbf{H}^{N}_{\s,\r,\x+\z,\y}|^{2}\d\r\\
&\lesssim_{\kappa} N^{4}\sum_{|\z|\leq1}\int_{\s}^{\t}\sum_{\y\in\Z}\exp(\tfrac{2\kappa|(\x+\z)-\y|}{N})|\mathbf{L}^{N,\zeta,\sim}_{\s,\t,\x+\z,\y}-\mathbf{H}^{N}_{\s,\r,\x+\z,\y}|^{2}\d\r.
\end{align*}
In particular, for {\small$|\x|\gtrsim N^{\mathrm{D}_{2}}$} with {\small$\mathrm{D}_{2}>0$} large, we can iterate this bound a total of {\small$\mathrm{c}\cdot\mathrm{dist}(\x,\mathrm{supp}(\boldsymbol{\chi}^{\zeta_{\mathrm{large}}}))$}-many times (for some {\small$\mathrm{c}>0$}). We obtain the following (in which {\small$\mathrm{J}:=\mathrm{c}\cdot\mathrm{dist}(\x,\mathrm{supp}(\boldsymbol{\chi}^{\zeta_{\mathrm{large}}}))$} just for convenience):
\begin{align*}
&\sum_{\y\in\Z}\exp(\tfrac{2\kappa|\x-\y|}{N})|\mathbf{L}^{N,\zeta,\sim}_{\s,\t,\x,\y}-\mathbf{H}^{N}_{\s,\t,\x,\y}|^{2}\\
&\lesssim \mathrm{C}_{\kappa}^{\mathrm{J}}N^{4\mathrm{J}}\sum_{\substack{\z=\z_{1}+\ldots\z_{\mathrm{J}+1}\\|\z_{1}|,\ldots,|\z_{\mathrm{J}+1}|\leq1}}\int_{\s}^{\t}\int_{\s}^{\r_{1}}\ldots\int_{\s}^{\r_{\mathrm{J}}}\sum_{\y\in\Z}\exp(\tfrac{2\kappa|(\x+\z)-\y|}{N})|\mathbf{L}^{N,\zeta,\sim}_{\s,\r,\x+\z,\y}-\mathbf{H}^{N}_{\s,\r,\x+\z,\y}|^{2}\d\r\d\r_{\mathrm{J}}\ldots\d\r_{1}.
\end{align*}
Above, {\small$\mathrm{C}_{\kappa}\lesssim_{\kappa}1$}. If we now take {\small$\mathrm{L}^{p}(\mathbb{P})$}-norms, use the triangle inequality, and then use \eqref{eq:lzetasimestimateII12a}, we obtain
\begin{align*}
\sum_{\y\in\Z}\exp(\tfrac{2\kappa|\x-\y|}{N})\||\mathbf{L}^{N,\zeta,\sim}_{\s,\t,\x,\y}-\mathbf{H}^{N}_{\s,\t,\x,\y}|^{2}\|_{\mathrm{L}^{p}(\mathbb{P})}&\lesssim \wt{\mathrm{C}}_{\kappa}^{\mathrm{J}}N^{4\mathrm{J}}\sum_{\substack{\z=\z_{1}+\ldots\z_{\mathrm{J}}\\|\z_{1}|,\ldots,|\z_{\mathrm{J}}|\leq1}}\int_{\s}^{\t}\int_{\s}^{\r_{1}}\ldots\int_{\s}^{\r_{\mathrm{J}}}\d\r\d\r_{\mathrm{J}}\ldots\d\r_{1}\\
&\lesssim (\mathrm{J}!)^{-1}\cdot N^{5\mathrm{J}}\lesssim\exp(-\mathrm{c}_{\kappa}\mathrm{J}),
\end{align*}
where the second line follows by integration and Stirling's formula (noting that {\small$\mathrm{J}\gtrsim N^{\mathrm{D}_{2}}$} for some {\small$\mathrm{D}_{2}>0$} large, so that the factorial dominates the power of {\small$N$}). Again, since {\small$\mathrm{J}\gtrsim N^{\mathrm{D}_{2}}$} and {\small$\mathrm{J}\gtrsim\mathrm{dist}(\x,\mathrm{supp}(\boldsymbol{\chi}^{\zeta_{\mathrm{large}}}))$}, we get
\begin{align*}
\sum_{\y\in\Z}\exp(\tfrac{2\kappa|\x-\y|}{N})\||\mathbf{L}^{N,\zeta,\sim}_{\s,\t,\x,\y}-\mathbf{H}^{N}_{\s,\t,\x,\y}|^{2}\|_{\mathrm{L}^{p}(\mathbb{P})}\lesssim \exp(-N)\cdot\exp(-\mathrm{c}'_{\kappa}\cdot\mathrm{dist}(\x,\mathrm{supp}(\boldsymbol{\chi}^{\zeta_{\mathrm{large}}}))).
\end{align*}
Therefore, by a union bound, Markov inequality, and the previous display for {\small$p=1$}, we have 
\begin{align*}
&\mathbb{P}\Big(\sup_{|\x|\gtrsim N^{\mathrm{D}_{2}}}\sum_{\y\in\Z}\exp(\tfrac{2\kappa|\x-\y|}{N})|\mathbf{L}^{N,\zeta,\sim}_{\s,\t,\x,\y}-\mathbf{H}^{N}_{\s,\t,\x,\y}|^{2}\gtrsim \exp(-\tfrac{N}{2})\Big)\\
&\leq\sum_{|\x|\gtrsim N^{\mathrm{D}_{2}}}\mathbb{P}\Big(\sum_{\y\in\Z}\exp(\tfrac{2\kappa|\x-\y|}{N})|\mathbf{L}^{N,\zeta,\sim}_{\s,\t,\x,\y}-\mathbf{H}^{N}_{\s,\t,\x,\y}|^{2}\gtrsim \exp(-\tfrac{N}{2})\Big)\\
&\lesssim\exp(-\tfrac{N}{2})\sum_{|\x|\gtrsim N^{\mathrm{D}_{2}}}\exp(-\mathrm{c}'_{\kappa}\cdot\mathrm{dist}(\x,\mathrm{supp}(\boldsymbol{\chi}^{\zeta_{\mathrm{large}}})))\lesssim_{\kappa}\exp(-\tfrac{N}{2}),
\end{align*}
where the last bound is an elementary geometric series estimate. In particular, we can use another union bound and continuity argument (like the one used to get \eqref{eq:lzetasimestimateII12}) to turn the previous display into 
\begin{align}
\sup_{\t\in[\s,1]}\sup_{|\x|\gtrsim N^{\mathrm{D}_{2}}}\sum_{\y\in\Z}\exp(\tfrac{2\kappa|\x-\y|}{N})|\mathbf{L}^{N,\zeta,\sim}_{\s,\t,\x,\y}-\mathbf{H}^{N}_{\s,\t,\x,\y}|^{2}\lesssim_{\kappa}\exp(-\tfrac{N}{2})\label{eq:lzetasimestimateII13}
\end{align}
if {\small$\mathrm{D}_{2}>0$} is large enough. Combining \eqref{eq:lzetasimestimateII12} and \eqref{eq:lzetasimestimateII13} now yields the desired inequality \eqref{eq:lzetasimestimateII}. \qed
%
%
%

\section{Proof of {Lemma \ref{lemma:stochheat}}}\label{section:compact}
For convenience, we define the following object for any index {\small$\i=1,\ldots,\i_{\e_{\mathrm{cut}}}$}:
\begin{align}
\mathbf{D}^{N,\zeta_{\i},\zeta_{\i+1}}_{\t,\x}:=\mathbf{S}^{N,\zeta_{\i}}_{\t,\x}-\mathbf{S}^{N,\zeta_{\i+1}}_{\t,\x}.\label{eq:dzetaformula}
\end{align}
Now recall \eqref{eq:szeta-sde}. The two terms that we subtract in {\small$\mathbf{D}^{N,\zeta_{\i},\zeta_{\i+1}}$} solve the same linear equation, except {\color{black}that} the term {\small$\mathbf{S}^{N,\zeta_{\i+1}}$} has indicators imposing a stricter cutoff of {\small$|\x|\leq N^{1+\zeta_{\i+1}}$} instead of {\small$|\x|\leq N^{1+\zeta_{\i}}$} as in {\small$\mathbf{S}^{N,\zeta_{\i}}$} (see Notation \ref{notation:zeta} for how different {\small$\zeta_{\i}$} parameters relate to each other). In particular, we have the following formula, which we explain afterwards:
\begin{align}
&\d\mathbf{D}^{N,\zeta_{\i},\zeta_{\i+1}}_{\t,\x}=\mathscr{T}_{N}\mathbf{D}^{N,\zeta_{\i},\zeta_{\i+1}}_{\t,\x}\d\t+{\boldsymbol{\chi}^{(\zeta_{\mathrm{large}})}_{\x}}[\mathscr{S}^{N}\star(\sqrt{2}\lambda N^{\frac12}\mathbf{R}^{N,\wedge}\mathbf{D}^{N,\zeta_{\i},\zeta_{\i+1}}\d\mathbf{b}_{\t,\cdot})]_{\x}\nonumber\\
&+N\boldsymbol{\chi}^{(\zeta_{\i})}_{\x}\sum_{\substack{\n=1,\ldots,\mathrm{K}\\\m=0,\ldots,\mathrm{M}\\0<|\mathfrak{l}_{1}|,\ldots,|\mathfrak{l}_{\m}|\lesssim \mathfrak{n}_{\mathbf{Av}}}}\tfrac{1}{|\mathfrak{l}_{1}|\ldots|\mathfrak{l}_{\m}|}\mathrm{c}_{N,\n,\mathfrak{l}_{1},\ldots,\mathfrak{l}_{\m}}\grad^{\mathbf{X}}_{\mathfrak{l}_{1}}\ldots\grad^{\mathbf{X}}_{\mathfrak{l}_{\m}}[\mathscr{S}^{N}\star(\mathbf{Av}^{\mathbf{T},\mathbf{X},\mathfrak{q}_{\n}}_{\t,\cdot}\cdot\mathbf{R}^{N}_{\t,\cdot}\mathbf{D}^{N,\zeta_{\i},\zeta_{\i+1}}_{\t,\cdot})]_{\x}\d\t\nonumber\\
&+N\boldsymbol{\chi}^{(\zeta_{\i})}_{\x}\sum_{\substack{\n=1,\ldots,\mathrm{K}\\\m=0,\ldots,\mathrm{M}\\0<|\mathfrak{l}_{1}|,\ldots,|\mathfrak{l}_{\m}|\lesssim \mathfrak{n}_{\mathbf{Av}}}}\tfrac{1}{|\mathfrak{l}_{1}|\ldots|\mathfrak{l}_{\m}|}\mathrm{c}_{N,\n,\mathfrak{l}_{1},\ldots,\mathfrak{l}_{\m}}\grad^{\mathbf{T},\mathrm{av}}_{\mathfrak{t}_{\mathbf{Av}}}\grad^{\mathbf{X}}_{\mathfrak{l}_{1}}\ldots\grad^{\mathbf{X}}_{\mathfrak{l}_{\m}}[\mathscr{S}^{N}\star(\mathbf{Av}^{\mathbf{X},\mathfrak{q}_{\n}}_{\t,\cdot}\cdot\mathbf{R}^{N}_{\t,\cdot}\mathbf{D}^{N,\zeta_{\i},\zeta_{\i+1}}_{\t,\cdot})]_{\x}\d\t\nonumber\\
&+N{\ocolor{purple}\mathrm{O}\Big(\mathbf{1}_{\substack{|\x|\lesssim N^{1+\zeta_{\i}}\\|\x|\gtrsim N^{1+\zeta_{\i+1}}}}\Big)}\sum_{\substack{\n=1,\ldots,\mathrm{K}\\\m=0,\ldots,\mathrm{M}\\0<|\mathfrak{l}_{1}|,\ldots,|\mathfrak{l}_{\m}|\lesssim \mathfrak{n}_{\mathbf{Av}}}}\tfrac{1}{|\mathfrak{l}_{1}|\ldots|\mathfrak{l}_{\m}|}\mathrm{c}_{N,\n,\mathfrak{l}_{1},\ldots,\mathfrak{l}_{\m}}\grad^{\mathbf{X}}_{\mathfrak{l}_{1}}\ldots\grad^{\mathbf{X}}_{\mathfrak{l}_{\m}}[\mathscr{S}^{N}\star(\mathbf{Av}^{\mathbf{T},\mathbf{X},\mathfrak{q}_{\n}}_{\t,\cdot}\cdot\mathbf{R}^{N}_{\t,\cdot}\mathbf{S}^{N,\zeta_{\i+1}}_{\t,\cdot})]_{\x}\d\t\nonumber\\
&+N{\ocolor{purple}\mathrm{O}\Big(\mathbf{1}_{\substack{|\x|\lesssim N^{1+\zeta_{\i}}\\|\x|\gtrsim N^{1+\zeta_{\i+1}}}}\Big)}\sum_{\substack{\n=1,\ldots,\mathrm{K}\\\m=0,\ldots,\mathrm{M}\\0<|\mathfrak{l}_{1}|,\ldots,|\mathfrak{l}_{\m}|\lesssim \mathfrak{n}_{\mathbf{Av}}}}\tfrac{1}{|\mathfrak{l}_{1}|\ldots|\mathfrak{l}_{\m}|}\mathrm{c}_{N,\n,\mathfrak{l}_{1},\ldots,\mathfrak{l}_{\m}}\grad^{\mathbf{T},\mathrm{av}}_{\mathfrak{t}_{\mathbf{Av}}}\grad^{\mathbf{X}}_{\mathfrak{l}_{1}}\ldots\grad^{\mathbf{X}}_{\mathfrak{l}_{\m}}[\mathscr{S}^{N}\star(\mathbf{Av}^{\mathbf{X},\mathfrak{q}_{\n}}_{\t,\cdot}\cdot\mathbf{R}^{N}_{\t,\cdot}\mathbf{S}^{N,\zeta_{\i+1}}_{\t,\cdot})]_{\x}\d\t\nonumber\\
&+N\boldsymbol{\chi}^{(\zeta_{\i})}_{\x}\cdot\mathrm{Err}[\mathbf{R}^{N}_{\t,\cdot}\mathbf{D}^{N,\zeta_{\i},\zeta_{\i+1}}_{\t,\cdot}]_{\x}\d\t+N{\ocolor{purple}\mathrm{O}\Big(\mathbf{1}_{\substack{|\x|\lesssim N^{1+\zeta_{\i}}\\|\x|\gtrsim N^{1+\zeta_{\i+1}}}}\Big)}\cdot\mathrm{Err}[\mathbf{R}^{N}_{\t,\cdot}\mathbf{S}^{N,\zeta_{\i+1}}_{\t,\cdot}]_{\x}\d\t.\nonumber
\end{align}
As we noted earlier, the two terms {\small$\mathbf{S}^{N,\zeta_{\i}}$} and {\small$\mathbf{S}^{N,\zeta_{\i+1}}$} satisfy the same \abbr{SDE}, except {\small$\mathbf{S}^{N,\zeta_{\i+1}}$} is missing terms that are compensated for by the quantities in the above display that have a purple indicator function; again, see \eqref{eq:szeta-sde}. These are, equivalently, the terms that do not have a factor of {\small$\mathbf{D}^{N,\zeta_{\i},\zeta_{\i+1}}$}. Observe that the terms without purple indicator functions, i.e. the terms that have a factor of {\small$\mathbf{D}^{N,\zeta_{\i+1},\zeta_{\i}}$}, are exactly the linear operators appearing in \eqref{eq:kzeta-sde} for {\small$\zeta=\zeta_{\i}$} but with {\small$\mathbf{D}^{N,\zeta_{\i},\zeta_{\i+1}}$} instead of {\small$\mathbf{K}^{N,\zeta_{\i}}$}. In particular, by following the reasoning from the proof of Lemma \ref{lemma:lkzeta}, we get the following Duhamel-type identity:
\begin{align}
\mathbf{D}^{N,\zeta_{\i},\zeta_{\i+1}}_{\t,\x}&=\int_{0}^{\t}\sum_{\y\in\Z}\mathbf{K}^{N,\zeta_{\i}}_{\s,\t,\x,\y}\Upsilon^{(1)}_{\s,\y}\d\s+\int_{0}^{\t}\sum_{\y\in\Z}\mathbf{K}^{N,\zeta_{\i}}_{\s,\t,\x,\y}\Upsilon^{(2)}_{\s,\y}\d\s+\int_{0}^{\t}\sum_{\y\in\Z}\mathbf{K}^{N,\zeta_{\i}}_{\s,\t,\x,\y}\Upsilon^{(3)}_{\s,\y}\d\s.\label{eq:stochheatII1}
\end{align}
We emphasize that {\small$\mathbf{D}^{N,\zeta_{\i},\zeta_{\i+1}}$} has zero initial data by construction (see \eqref{eq:dzetaformula}), and
\begin{align*}
\Upsilon^{(1)}_{\t,\x}&:=N{\ocolor{purple}\mathrm{O}\Big(\mathbf{1}_{\substack{|\x|\lesssim N^{1+\zeta_{\i}}\\|\x|\gtrsim N^{1+\zeta_{\i+1}}}}\Big)}\sum_{\substack{\n=1,\ldots,\mathrm{K}\\\m=0,\ldots,\mathrm{M}\\0<|\mathfrak{l}_{1}|,\ldots,|\mathfrak{l}_{\m}|\lesssim \mathfrak{n}_{\mathbf{Av}}}}\tfrac{1}{|\mathfrak{l}_{1}|\ldots|\mathfrak{l}_{\m}|}\mathrm{c}_{N,\n,\mathfrak{l}_{1},\ldots,\mathfrak{l}_{\m}}\grad^{\mathbf{X}}_{\mathfrak{l}_{1}}\ldots\grad^{\mathbf{X}}_{\mathfrak{l}_{\m}}[\mathscr{S}^{N}\star(\mathbf{Av}^{\mathbf{T},\mathbf{X},\mathfrak{q}_{\n}}_{\t,\cdot}\cdot\mathbf{R}^{N}_{\t,\cdot}\mathbf{S}^{N,\zeta_{\i+1}}_{\t,\cdot})]_{\x}\d\t,\\
\Upsilon^{(2)}_{\t,\x}&:=N{\ocolor{purple}\mathrm{O}\Big(\mathbf{1}_{\substack{|\x|\lesssim N^{1+\zeta_{\i}}\\|\x|\gtrsim N^{1+\zeta_{\i+1}}}}\Big)}\sum_{\substack{\n=1,\ldots,\mathrm{K}\\\m=0,\ldots,\mathrm{M}\\0<|\mathfrak{l}_{1}|,\ldots,|\mathfrak{l}_{\m}|\lesssim \mathfrak{n}_{\mathbf{Av}}}}\tfrac{1}{|\mathfrak{l}_{1}|\ldots|\mathfrak{l}_{\m}|}\mathrm{c}_{N,\n,\mathfrak{l}_{1},\ldots,\mathfrak{l}_{\m}}\grad^{\mathbf{T},\mathrm{av}}_{\mathfrak{t}_{\mathbf{Av}}}\grad^{\mathbf{X}}_{\mathfrak{l}_{1}}\ldots\grad^{\mathbf{X}}_{\mathfrak{l}_{\m}}[\mathscr{S}^{N}\star(\mathbf{Av}^{\mathbf{X},\mathfrak{q}_{\n}}_{\t,\cdot}\cdot\mathbf{R}^{N}_{\t,\cdot}\mathbf{S}^{N,\zeta_{\i+1}}_{\t,\cdot})]_{\x}\d\t,\\
\Upsilon^{(3)}_{\t,\x}&:=N{\ocolor{purple}\mathrm{O}\Big(\mathbf{1}_{\substack{|\x|\lesssim N^{1+\zeta_{\i}}\\|\x|\gtrsim N^{1+\zeta_{\i+1}}}}\Big)}\cdot\mathrm{Err}[\mathbf{R}^{N}_{\t,\cdot}\mathbf{S}^{N,\zeta_{\i+1}}_{\t,\cdot}]_{\x}\d\t.
\end{align*}
Let us make a few observations. First, we claim that with high probability, we have the following simultaneously over {\small$\t\in[0,1]$} {\color{black}and all} {\small$\x\in\Z$}, in which {\small$\kappa>0$} is a large but {\small$\mathrm{O}(1)$} constant and {\small$\beta,\delta>0$} are independent of {\small$N$}:
\begin{align}
&\exp(-\tfrac{\kappa|\x|}{N})|\mathbf{S}^{N,\zeta_{\i+1}}_{\t,\x}|\lesssim\exp(-\tfrac{\kappa|\x|}{N})\sum_{\y\in\Z}|\mathbf{K}^{N,\zeta_{\i+1}}_{0,\t,\x,\y}|\cdot|\mathbf{S}^{N}_{0,\y}|\nonumber\\
&\lesssim \exp(N^{-\beta}N^{\zeta_{\i+1}})N^{\delta}\sum_{\y\in\Z}\exp(-\tfrac{\kappa|\x|}{N}-\tfrac{\kappa|\x-\y|}{N})\cdot|\mathbf{S}^{N}_{0,\y}|\nonumber\\
&\lesssim_{\kappa}\exp(N^{-\beta}N^{\zeta_{\i+1}})N^{\delta}\sum_{\y\in\Z}\exp(-\tfrac{\kappa|\y|}{N})\cdot|\mathbf{S}^{N}_{0,\y}|\lesssim N^{1+\delta}\cdot\exp(N^{-\beta}N^{\zeta_{\i+1}}).\label{eq:stochheatII2}
\end{align}
The first estimate follows by \eqref{eq:duhamelszeta}. The second estimate follows by Proposition \ref{prop:stochheat} for {\small$\zeta=\zeta_{\i+1}$}. The third line follows by first using {\small$\exp(\upsilon|\y|)\leq\exp(\upsilon|\x|)\exp(\upsilon|\x-\y|)$} for any {\small$\upsilon\geq0$}. Then, we use {Lemma \ref{lemma:databound}}, which gives an sub-exponential estimate for the initial data {\small$\mathbf{S}^{N}$} (with high probability), and an elementary exponential-sum bound. Next, we claim that with high probability, we have 
\begin{align}
&\mathbf{1}_{|\x|\lesssim N^{1+\zeta_{\i}}}\exp(-\tfrac{\kappa|\x|}{N})|\mathrm{Err}[\mathbf{R}^{N}\mathbf{S}^{N,\zeta_{\i+1}}]_{\t,\x}|\nonumber\\
&\lesssim N^{-\mathrm{D}_{0}}N^{-1+\delta_{\mathbf{S}}}\sum_{|\w|\lesssim N^{1-\delta_{\mathbf{S}}}}\Big\{1+\sum_{|\z|\lesssim1}\mathbf{1}_{|\x|\lesssim N^{1+\zeta_{1}}}|\bphi_{\t,\x+\w+\z}|^{\mathrm{C}}\cdot|\mathbf{R}^{N}_{\t,\x+\w}|\cdot\exp(-\tfrac{\kappa|\x|}{N})|\mathbf{S}^{N,\zeta_{\i+1}}_{\t,\x+\w}|\Big\}\nonumber\\
&\lesssim\exp (N^{-\beta}N^{\zeta_{\i+1}})N^{\delta}.\label{eq:stochheatII3}
\end{align}
The first estimate follows from \eqref{eq:err-estimate}. The second bound follows by \eqref{eq:stochheatII2} and {Lemma \ref{lemma:phibound}} (to give polynomial estimates on {\small$\bphi_{\t,\cdot}$}) and {Lemma \ref{lemma:rbound}} (which gives a priori {\small$\mathrm{O}(1)$} estimates for {\small$\mathbf{R}^{N}$}). Next, we claim that with high probability we have 
\begin{align}
&\sup_{\t\in[0,1]}\sup_{|\x|\lesssim N^{1+\zeta_{\i}}}\exp(-\tfrac{\kappa|\x|}{N})|[\mathscr{S}^{N}\star(\mathbf{Av}^{\mathbf{T},\mathbf{X},\mathfrak{q}_{\n}}_{\t,\cdot}\cdot\mathbf{R}^{N}_{\t,\cdot}\mathbf{S}^{N,\zeta_{\i+1}}_{\t,\cdot})]_{\x}|\nonumber\\
&\lesssim\sup_{\t\in[0,1]}\sup_{|\x|\lesssim N^{1+\zeta_{\i}}}\sup_{|\w|\lesssim N}|\mathbf{Av}^{\mathbf{T},\mathbf{X},\mathfrak{q}_{\n}}_{\t,\x+\w}|\cdot|\mathbf{R}^{N}_{\t,\x+\w}|\cdot\exp(-\tfrac{\kappa|\x|}{N})|\mathbf{S}^{N,\zeta_{\i+1}}_{\t,\x+\w}|\lesssim \exp(N^{-\beta}N^{\zeta_{\i+1}})N^{\delta}.\label{eq:stochheatII3b}
\end{align}
To get the first estimate, we recall from Definition \ref{definition:zsmooth} that the convolution kernel {\small$\mathscr{S}^{N}$} is supported on a block of size {\small$\lesssim N$}. To get the second, we again use {Lemma \ref{lemma:rbound}} to control {\small$\mathbf{R}^{N}$}, but now we also use {Lemma \ref{lemma:avldptx}} to control the {\small$\mathbf{Av}^{\mathbf{T},\mathbf{X},\mathfrak{q}_{\n}}$} factor by {\small$\lesssim1$} uniformly over the desired {\small$\t,\x,\w$}. We then use \eqref{eq:stochheatII2} to conclude. We now claim that with high probability, we have a similar {\color{black}lower bound}, in which {\small$\mathfrak{l}_{1},\ldots,\mathfrak{l}_{\m}$} are all {\small$\mathrm{O}(N)$} and {\small$\mathfrak{m}=\mathrm{O}(1)$}:
\begin{align}
&\sup_{\t\in[0,1]}\sup_{|\x|\lesssim N^{1+\zeta_{\i}}}\exp(-\tfrac{\kappa|\x|}{N})|\grad^{\mathbf{X}}_{\mathfrak{l}_{1}}\ldots\grad^{\mathbf{X}}_{\mathfrak{l}_{\m}}[\mathscr{S}^{N}\star(\mathbf{Av}^{\mathbf{T},\mathbf{X},\mathfrak{q}_{\n}}_{\t,\cdot}\cdot\mathbf{R}^{N}_{\t,\cdot}\mathbf{S}^{N,\zeta_{\i+1}}_{\t,\cdot})]_{\x}|\lesssim N^{\delta}\cdot \exp(N^{-\beta}N^{\zeta_{\i+1}}),\nonumber\\
&\sup_{\t\in[0,1]}\sup_{|\x|\lesssim N^{1+\zeta_{\i}}}\exp(-\tfrac{\kappa|\x|}{N})|\grad^{\mathbf{X}}_{\mathfrak{l}_{1}}\ldots\grad^{\mathbf{X}}_{\mathfrak{l}_{\m}}[\mathscr{S}^{N}\star(\mathbf{Av}^{\mathbf{X},\mathfrak{q}_{\n}}_{\t,\cdot}\cdot\mathbf{R}^{N}_{\t,\cdot}\mathbf{S}^{N,\zeta_{\i+1}}_{\t,\cdot})]_{\x}|\lesssim N^{\delta}\cdot \exp(N^{-\beta}N^{\zeta_{\i+1}}).\label{eq:stochheatII4}
\end{align}
This follows by the same argument; we just replace the smoothing kernel {\small$\mathscr{S}^{N}$} in Definition \ref{definition:eq-operators} with its (iterated) gradient, and we use the space-average estimate in {Lemma \ref{lemma:avldp}}. Lastly, since the above estimate holds for all {\small$\t\in[0,1]$}, it holds if we additionally apply the time-gradient operator \eqref{eq:timegrad}. Ultimately, by \eqref{eq:stochheatII3}, \eqref{eq:stochheatII3b}, and \eqref{eq:stochheatII4}, as well as the estimate {\small$\mathrm{c}_{N,\m,\mathfrak{l}_{1},\ldots,\mathfrak{l}_{\m}}\lesssim N$} (see Proposition \ref{prop:s-sde}), the following holds with high probability for any {\small$\delta>0$} (if {\small$\kappa>0$} is large but still {\small$\mathrm{O}(1)$}):
\begin{align}
\sup_{\t\in[0,1]}\sup_{\x\in\Z}\exp(-\tfrac{\kappa|\x|}{N})\Big\{|\Upsilon^{(1)}_{\t,\x}|+|\Upsilon^{(2)}_{\t,\x}|+|\Upsilon^{(3)}_{\t,\x}|\Big\}\lesssim_{\delta} N^{2+\delta}\cdot\exp(N^{-\beta}N^{\zeta_{\i+1}})\cdot\mathbf{1}_{\substack{|\x|\lesssim N^{1+\zeta_{\i}}\\|\x|\gtrsim N^{1+\zeta_{\i+1}}}}.\label{eq:stochheatII5}
\end{align}
Now, let us restrict to the (high probability) event on which \eqref{eq:stochheatII5} holds. We claim that with high probability, for any {\small$\k=1,2,3$}, we have the following simultaneously over all {\small$\s,\t\in[0,1]$} with {\small$\s\leq\t$} and all {\small$\x$} with {\small$|\x|\lesssim N$}:
\begin{align*}
&\exp(-\tfrac{\kappa|\x|}{N})\cdot|\mathbf{K}^{N,\zeta_{\i}}_{\s,\t,\x,\y}|\cdot|\Upsilon^{(\k)}_{\s,\y}|\\
&\lesssim\exp(\tfrac{\kappa|\x-\y|}{N})\cdot|\mathbf{K}^{N,\zeta_{\i}}_{\s,\t,\x,\y}|\cdot\exp(-\tfrac{\kappa|\y|}{N})\cdot|\Upsilon^{(\k)}_{\s,\y}|\\
&\lesssim\exp(\tfrac{\kappa|\x-\y|}{N})|\mathbf{K}^{N,\zeta_{\i}}_{\s,\t,\x,\y}|\cdot N^{2+\delta}\mathbf{1}_{\substack{|\y|\lesssim N^{1+\zeta_{\i}}\\|\y|\gtrsim N^{1+\zeta_{\i+1}}}}\cdot\exp(N^{-\beta}N^{\zeta_{\i+1}})\\
&\lesssim_{\delta} N^{2+\delta}\cdot\exp(N^{-\beta}N^{\zeta_{\i+1}})\exp(N^{\frac12-\beta}N^{\zeta_{\i}})\cdot\mathbf{1}_{\substack{|\y|\lesssim N^{1+\zeta_{\i}}\\|\y|\gtrsim N^{1+\zeta_{\i+1}}}}\cdot\exp(-\tfrac{\kappa|\x-\y|}{N}).
\end{align*}
The first estimate follows by the triangle inequality inside the exponential. The second estimate follows by \eqref{eq:stochheatII5}. The third estimate follows by \eqref{eq:stochheatI} for {\small$\zeta=\zeta_{\i}$}. Because {\small$|\x|\lesssim N$}, we can drop the factor {\small$\exp(-\kappa|\x|/N)$} in the first line, and the resulting estimate still holds if we give up a factor of {\small$\mathrm{O}_{\kappa}(1)$}. So, when we use \eqref{eq:stochheatII1}, integrate the previous display over {\small$\s\in[0,\t]$}, and sum over {\small$\y\in\Z$}, we obtain the following with high probability, in which {\small$\mathrm{c},\mathrm{C}>0$} are fixed (depending only on {\small$\kappa$}):
\begin{align*}
\sup_{|\x|\lesssim N}|\mathbf{D}^{N,\zeta_{\i},\zeta_{\i+1}}_{\t,\x}|&\lesssim N^{\mathrm{C}}\exp(N^{-\beta}N^{\zeta_{\i+1}})\exp(N^{-\beta}N^{\zeta_{\i}})\sup_{\substack{|\x|\lesssim N}}\sup_{{|\y|\gtrsim N^{1+\zeta_{\i+1}}}}\exp(-\tfrac{\kappa|\x-\y|}{N})\\
&\lesssim N^{\mathrm{C}}\exp(2N^{-\beta}N^{\zeta_{\i}})\exp(-\mathrm{c} N^{\zeta_{\i+1}}).
\end{align*}
The second bound follows since the distance between {\small$\x,\y$} in the previous display is {\small$\gtrsim N^{1+\zeta_{\i+1}}$}. Now, by Notation \ref{notation:zeta}, we have {\small$\zeta_{\i+1}=\zeta_{\i}-\delta_{\mathrm{cut}}$} with {\small$\delta_{\mathrm{cut}}$} as small as we want. If we choose {\small$\delta_{\mathrm{cut}}\leq\beta/2$}, then the last line of the previous display is exponentially small as {\small$N\to\infty$}. This implies the desired estimate \eqref{eq:stochheatII}. \qed
\section{Proof of Proposition \ref{prop:comparison}}\label{section:comparison}
We proceed in a few steps. The first step essentially compares {\small$\mathbf{Z}^{N}$} to {\small$\mathbf{S}^{N}$} via \eqref{eq:zsmooth}, compares {\small$\mathbf{S}^{N}$} to {\small$\mathbf{S}^{N,\zeta_{\mathrm{final}}}$} via Lemmas \ref{lemma:aprioricompact} and \ref{lemma:stochheat}, takes \eqref{eq:duhamelszeta} with {\small$\zeta=\zeta_{\mathrm{final}}$} from Notation \ref{notation:zeta}, and shows that {\small$\mathbf{K}^{N,\zeta}$} therein can be replaced by the kernel {\small$\mathbf{L}^{N,\zeta,\sim}$} from \eqref{eq:lzetasim-sde} (whose evolution is much closer to a lattice \abbr{SHE}).
\begin{lemma}\label{lemma:comparisonstep1}
\fsp Fix any {\small$\mathrm{L}>0$}. There exists {\small$\gamma_{\mathrm{L},1}>0$} depending on {\small$\mathrm{L}$} such that with high probability, we have 
\begin{align}
\sup_{\t\in[0,1]}\sup_{|\x|\leq \mathrm{L}N}|\mathbf{Z}^{N}_{\t,\x}-\mathbf{Y}^{N,1}_{\t,\x}|\lesssim_{\mathrm{L}} N^{-\gamma_{\mathrm{L},1}},\label{eq:comparisonstep1I}
\end{align}
where, upon recalling {\small$\mathbf{L}^{N,\zeta_{\mathrm{final}},\sim}$} from \eqref{eq:lzetasim-sde}, we set
\begin{align}
\mathbf{Y}^{N,1}_{\t,\x}&:=\sum_{\y\in\Z}\mathbf{L}^{N,\zeta_{\mathrm{final}},\sim}_{0,\t,\x,\y}\mathbf{S}^{N}_{0,\y}.\label{eq:comparisonstep1II}
\end{align}
\end{lemma}
The second step amounts to comparing {\small$\mathbf{Y}^{N,1}$} from \eqref{eq:comparisonstep1II} to something that is essentially a lattice \abbr{SHE} (with minor cosmetic adjustments). Before we state the lemma, we need to introduce the required notation. Let {\small$\mathbf{B}^{N,\sim}_{\s,\t,\x,\y}$} be a function of {\small$(\s,\t,\x,\y)\in[0,1]^{2}\times\Z^{2}$} with {\small$\s\leq\t$}, such that {\small$\mathbf{B}^{N,\sim}_{\s,\s,\x,\y}=\mathbf{1}_{\x=\y}$}, and such that the following version of \eqref{eq:lzetasim-sde} (without any of the error terms therein) is satisfied for all {\small$\s\leq\t$} and {\small$\x,\y\in\Z$}:
\begin{align}
\d\mathbf{B}^{N,\sim}_{\s,\t,\x,\y}&=\mathscr{T}_{N}\mathbf{B}^{N,\sim}_{\s,\t,\x,\y}\d\t+{\boldsymbol{\chi}}^{(\zeta_{\mathrm{large}})}_{\x}\cdot[\mathscr{S}^{N}\star(\sqrt{2}\lambda N^{\frac12}\mathbf{R}^{N,\wedge}_{\t,\cdot}\mathbf{B}^{N,\sim}_{\s,\t,\cdot,\y}\d\mathbf{b}_{\t,\cdot})]_{\x}.\label{eq:bzetasim-sde}
\end{align}
We clarify that {\small$\mathscr{T}_{N}$} acts on the {\small$\x$}-variable above.
\begin{lemma}\label{lemma:comparisonstep2}
\fsp Fix any {\small$\mathrm{L}>0$}. There exists {\small$\gamma_{\mathrm{L},2}>0$} universal such that with high probability, we have 
\begin{align}
\sup_{\t\in[0,1]}\sup_{|\x|\leq\mathrm{L}N}|\mathbf{Y}^{N,1}_{\t,\x}-\mathbf{Y}^{N,2}_{\t,\x}|\lesssim_{\mathrm{L}}N^{-\gamma_{\mathrm{L},2}},\label{eq:comparisonstep2I}
\end{align}
where we have used {\small$\mathbf{Y}^{N,1}$} from \eqref{eq:comparisonstep1II}, and where we have used
\begin{align}
\mathbf{Y}^{N,2}_{\t,\x}&:=\sum_{\y\in\Z}\mathbf{B}^{N,\sim}_{0,\t,\x,\y}\mathbf{S}^{N}_{0,\y}.\label{eq:comparisonstep2II}
\end{align}
\end{lemma}
The third and final step is to compare {\small$\mathbf{Y}^{N,2}$} to the object {\small$\mathbf{Q}^{N}$} of interest (see \eqref{eq:qsde}).
\begin{lemma}\label{lemma:comparisonstep3}
\fsp Fix any {\small$\mathrm{L}>0$}. There exists {\small$\gamma_{\mathrm{L},3}>0$} universal such that with high probability, we have 
\begin{align}
\sup_{\t\in[0,1]}\sup_{|\x|\leq \mathrm{L}N}|\mathbf{Y}^{N,2}_{\t,\x}-\mathbf{Q}^{N}_{\t,\x}|\lesssim_{\mathrm{L}} N^{-\gamma_{\mathrm{L},3}}.\label{eq:comparisonstep3I}
\end{align}
\end{lemma}
\begin{proof}[Proof of Proposition \ref{prop:comparison}]
Use Lemmas \ref{lemma:comparisonstep1}, \ref{lemma:comparisonstep2}, and \ref{lemma:comparisonstep3} and the triangle inequality.
\end{proof}
%
\subsection{Proofs of Lemmas \ref{lemma:comparisonstep1}, \ref{lemma:comparisonstep2}, and \ref{lemma:comparisonstep3}}
We will now prove the previous ingredients in order.
\subsubsection{Proof of Lemma \ref{lemma:comparisonstep1}}
We iteratively use Lemmas \ref{lemma:aprioricompact} and \ref{lemma:stochheat} to deduce {\color{black}that}
\begin{align}
\sup_{\t\in[0,1]}\sup_{|\x|\leq\mathrm{L}N}|\mathbf{S}^{N}_{\t,\x}-\mathbf{S}^{N,\zeta_{\mathrm{final}}}_{\t,\x}|\lesssim_{\mathrm{L}}N^{-\gamma_{\mathrm{L}}}\label{eq:comparisonstep1I1}
\end{align}
for some {\small$\gamma_{\mathrm{L}}>0$} with high probability. We now apply {Lemma \ref{lemma:rbound}} and \eqref{eq:comparisonstep1I1} to get
\begin{align}
\sup_{\t\in[0,1]}\sup_{|\x|\leq\mathrm{L}N}|\mathbf{S}^{N}_{\t,\x}-\mathbf{Z}^{N}_{\t,\x}|&\lesssim N^{-\frac13\delta_{\mathbf{S}}}\sup_{\t\in[0,1]}\sup_{|\x|\leq\mathrm{L}N}|\mathbf{S}^{N}_{\t,\x}|\nonumber\\
&\lesssim_{\mathrm{L}} N^{-\frac13\delta_{\mathbf{S}}}\sup_{\t\in[0,1]}\sup_{|\x|\leq\mathrm{L}N}|\mathbf{S}^{N,\mathrm{final}}_{\t,\x}|+N^{-\frac13\delta_{\mathbf{S}}-\gamma_{\mathrm{L}}}.\label{eq:comparisonstep1I2}
\end{align}
We now claim that the following estimate holds on a high probability event if {\small$\kappa>0$} is large enough:
\begin{align*}
\exp(-\tfrac{\kappa|\x|}{N})|\mathbf{S}^{N,\zeta_{\mathrm{final}}}_{\t,\x}|&\lesssim\sum_{\y\in\Z}\exp(\tfrac{\kappa|\x-\y|}{N})|\mathbf{K}^{N,\mathrm{final}}_{0,\t,\x,\y}|\cdot\exp(-\tfrac{\kappa|\y|}{N})|\mathbf{S}^{N}_{0,\y}|\\
&\lesssim_{\kappa,\delta}\exp(N^{-\beta}N^{\zeta_{\mathrm{final}}})N^{\delta}\lesssim N^{\delta}.
\end{align*}
The first bound follows from \eqref{eq:duhamelszeta} and the triangle inequality. To show the second estimate, we use Proposition \ref{prop:stochheat}, noting that {\small$\zeta_{\mathrm{final}}$} is arbitrarily small (see Notation \ref{notation:zeta}), so the \abbr{RHS} of \eqref{eq:stochheatI} is {\small$\lesssim_{\delta}N^{\delta}$}. We also use {Lemma \ref{lemma:databound}} to get {\small$\sup_{\y\in\Z}\exp(-\kappa|\y|/N)|\mathbf{S}^{N}_{0,\y}|\lesssim_{\delta} N^{\delta}$} with high probability. If we restrict the previous display to {\small$|\x|\leq\mathrm{L}N$}, then the factor of {\small$\exp(-\kappa|\x|/N)$} is bounded uniformly away from {\small$0$}. So, if we take {\small$\delta>0$} small enough in the above display and combine it with \eqref{eq:comparisonstep1I2}, then we have
\begin{align}
\sup_{\t\in[0,1]}\sup_{|\x|\leq\mathrm{L}N}|\mathbf{S}^{N}_{\t,\x}-\mathbf{Z}^{N}_{\t,\x}|\lesssim_{\mathrm{L},\delta}N^{-\frac16\delta_{\mathbf{S}}}.\label{eq:comparisonstep1I3}
\end{align}
Now, recall {\small$\mathbf{Y}^{N,1}$} from \eqref{eq:comparisonstep1II}. Moreover, recall {\small$\mathbf{K}^{N,\zeta_{\mathrm{final}},\sim}$} from \eqref{eq:kzetasim-sde}, and recall that by Lemma \ref{lemma:kzetamodify}, we have {\small$\mathbf{K}^{N,\zeta_{\mathrm{final}}}=\mathbf{K}^{N,\zeta_{\mathrm{final}},\sim}$} with high probability. In particular, with high probability, we have 
\begin{align*}
\exp(-\tfrac{\kappa|\x|}{N})|\mathbf{Y}^{N,1}_{\t,\x}-\mathbf{S}^{N,\zeta_{\mathrm{final}}}_{\t,\x}|&\leq\sum_{\y\in\Z}\exp(\tfrac{\kappa|\x-\y|}{N})|\mathbf{L}^{N,\zeta_{\mathrm{final}},\sim}_{0,\t,\x,\y}-\mathbf{K}^{N,\zeta_{\mathrm{final}},\sim}_{0,\t,\x,\y}|\cdot\exp(-\tfrac{\kappa|\y|}{N})|\mathbf{S}^{N}_{0,\y}|\\
&\lesssim_{\kappa,\delta} N^{\delta}\sum_{\y\in\Z}\exp(\tfrac{\kappa|\x-\y|}{N})|\mathbf{L}^{N,\zeta_{\mathrm{final}},\sim}_{0,\t,\x,\y}-\mathbf{K}^{N,\zeta_{\mathrm{final}},\sim}_{0,\t,\x,\y}|,
\end{align*}
where we again used {Lemma \ref{lemma:databound}} to get the second line above. Now, we note that the difference of kernels in the second line can be computed using \eqref{eq:lkzetaIa}. After plugging in this expression, we use \eqref{eq:kzetaestimate-shortI5} for small {\small$\e,\delta,\delta_{\mathbf{S}}>0$}, and then we use \eqref{eq:kzetaestimate-longI} for {\small$\zeta=\zeta_{\mathrm{final}}$} and {\small$\delta>0$} small. This implies that
\begin{align*}
\sum_{\y\in\Z}\exp(\tfrac{\kappa|\x-\y|}{N})|\mathbf{L}^{N,\zeta_{\mathrm{final}},\sim}_{0,\t,\x,\y}-\mathbf{K}^{N,\zeta_{\mathrm{final}},\sim}_{0,\t,\x,\y}|\lesssim_{\kappa} N^{-\omega_{0}}\sup_{\substack{\r\in[0,\t]\\\x\in\Z}}\sup_{\y\in\Z}\exp(\tfrac{\kappa|\x-\y|}{N})|\mathbf{K}^{N,\zeta_{\mathrm{final}},\sim}_{\r,\t,\x,\y}|\lesssim N^{-\frac12\omega_{0}}
\end{align*}
for some universal constant {\small$\omega_{0}>0$}. Again, this holds on some high probability event simultaneously over all {\small$\t\in[0,1]$} and {\small$\x\in\Z$}. If we combine the previous two displays (with {\small$\delta>0$} small enough), then we obtain
\begin{align*}
\sup_{\t\in[0,1]}\sup_{\x\in\Z}\exp(-\tfrac{\kappa|\x|}{N})|\mathbf{Y}^{N,1}_{\t,\x}-\mathbf{S}^{N,\zeta_{\mathrm{final}}}_{\t,\x}|\lesssim_{\kappa}N^{-\frac13\omega_{0}}.
\end{align*}
If we restrict to {\small$|\x|\leq\mathrm{L}N$} in the previous display, then the exponential factor is bounded uniformly away from {\small$0$}. Combining the previous display with \eqref{eq:comparisonstep1I1} and \eqref{eq:comparisonstep1I3} thus yields the desired estimate \eqref{eq:comparisonstep1I}. \qed
\subsubsection{Proof of Lemma \ref{lemma:comparisonstep2}}
The key is controlling the following difference of kernels propagating {\small$\mathbf{Y}^{N,1},\mathbf{Y}^{N,2}$}:
\begin{align}
\mathbf{V}^{N}_{\s,\t,\x,\y}:=\mathbf{L}^{N,\zeta,\sim}_{\s,\t,\x,\y}-\mathbf{B}^{N,\sim}_{\s,\t,\x,\y}.\label{eq:comparisonstep2I1}
\end{align}
(We have introduced this notation because it will be convenient for analyzing this difference of kernels through its dynamical equation.) By \eqref{eq:lzetasim-sde} and \eqref{eq:bzetasim-sde}, this difference kernel satisfies the following for {\small$\t\geq\s$} and {\small$\x,\y\in\Z$}:
\begin{align}
&\d\mathbf{V}^{N}_{\s,\t,\x,\y}=\mathscr{T}_{N}\mathbf{V}^{N}_{\s,\t,\x,\y}\d\t+{\boldsymbol{\chi}}^{(\zeta_{\mathrm{large}})}_{\x}\cdot[\mathscr{S}^{N}\star(\sqrt{2}\lambda N^{\frac12}\mathbf{R}^{N,\wedge}_{\t,\cdot}\mathbf{V}^{N}_{\s,\t,\cdot,\y}\d\mathbf{b}_{\t,\cdot})]_{\x}\nonumber\\
&+N\boldsymbol{\chi}^{(\zeta)}_{\x}\sum_{\substack{\n=1,\ldots,\mathrm{K}\\\m=0,\ldots,\mathrm{M}\\0<|\mathfrak{l}_{1}|,\ldots,|\mathfrak{l}_{\m}|\lesssim \mathfrak{n}_{\mathbf{Av}}}}\tfrac{1}{|\mathfrak{l}_{1}|\ldots|\mathfrak{l}_{\m}|}\mathrm{c}_{N,\n,\mathfrak{l}_{1},\ldots,\mathfrak{l}_{\m}}\grad^{\mathbf{T},\mathrm{av}}_{\mathfrak{t}_{\mathbf{Av}}}\grad^{\mathbf{X}}_{\mathfrak{l}_{1}}\ldots\grad^{\mathbf{X}}_{\mathfrak{l}_{\m}}\Big({\mathbf{1}_{\mathfrak{t}_{\mathrm{ap}}\geq\t}}[\mathscr{S}^{N}\star(\mathbf{Av}^{\mathbf{X},\mathfrak{q}_{\n}}_{\t,\cdot}\cdot\mathbf{R}^{N}_{\t,\cdot}\mathbf{L}^{N,\zeta,\sim}_{\s,\t,\cdot,\y})]_{\x}\Big)\d\t\nonumber\\
&+\boldsymbol{\chi}^{(\zeta)}_{\x}{\mathbf{1}_{\mathfrak{t}_{\mathrm{ap}}\geq\t}}\cdot\mathrm{Err}[\mathbf{R}^{N}_{\t,\cdot}\mathbf{L}^{N,\zeta,\sim}_{\s,\t,\cdot,\y}]_{\x}\d\t\nonumber\\
&=:\mathscr{T}_{N}\mathbf{V}^{N}_{\s,\t,\x,\y}\d\t+\mathbf{1}_{\mathfrak{t}_{\mathrm{ap}}\geq\t}{\boldsymbol{\chi}}^{(\zeta_{\mathrm{large}})}_{\x}\cdot[\mathscr{S}^{N}\star(\sqrt{2}\lambda N^{\frac12}\mathbf{R}^{N,\wedge}_{\t,\cdot}\mathbf{V}^{N}_{\s,\t,\cdot,\y}\d\mathbf{b}_{\t,\cdot})]_{\x}+\boldsymbol{\Upsilon}_{\t,\x,\y}\d\t.\label{eq:vn-sde}
\end{align}
We clarify that {\small$\mathscr{T}_{N}$} acts in {\small$\x$}, and the differential is in {\small$\t$}. By the Duhamel formula, we have 
\begin{align}
\mathbf{V}^{N}_{\s,\t,\x,\y}&=\int_{\s}^{\t}\sum_{\w\in\Z}\mathbf{H}^{N}_{\r,\t,\x,\w}{\boldsymbol{\chi}}^{(\zeta_{\mathrm{large}})}_{\w}\cdot[\mathscr{S}^{N}\star(\sqrt{2}\lambda N^{\frac12}\mathbf{R}^{N,\wedge}_{\r,\cdot}\mathbf{V}^{N}_{\s,\r,\cdot,\y}\d\mathbf{b}_{\r,\cdot})]_{\w}\nonumber\\
&+\int_{\s}^{\t}\sum_{\w\in\Z}\mathbf{H}^{N}_{\r,\t,\x,\w}\cdot\boldsymbol{\Upsilon}_{\r,\w,\y}\d\r.\label{eq:comparisonstep2I2}
\end{align}
We will again recall the notation {\small$\|\cdot\|_{\mathrm{L}^p(\mathbb{P})}:=(\E|\cdot|^{p})^{1/p}$}. For any {\small$\kappa>0$}, we get the following moment estimate by combining \eqref{eq:lzetasimestimateII1d}, \eqref{eq:lzetasimestimateII4}, and \eqref{eq:lzetasimestimateII10}:
\begin{align}
&\sum_{\y\in\Z}\exp(\tfrac{2\kappa|\x-\y|}{N})\Big\|\Big(\int_{\s}^{\t}\sum_{\w\in\Z}\mathbf{H}^{N}_{\r,\t,\x,\w}\cdot\boldsymbol{\Upsilon}_{\r,\w,\y}\d\r\Big)^{2}\Big\|_{\mathrm{L}^{p}(\mathbb{P})}\nonumber\\
&\lesssim_{\kappa}N^{-\omega_{0}}\sup_{\r\in[\s,\t]}\sup_{\x\in\Z}|\r-\s|^{\frac12}\sum_{\y\in\Z}\exp(\tfrac{2\kappa|\x-\y|}{N})\||\mathbf{L}^{N,\zeta,\sim}_{\s,\r,\x,\y}|^{2}\|_{\mathrm{L}^{p}(\mathbb{P})},\nonumber
\end{align}
where {\small$\omega_{0}>0$} is a fixed constant. We now claim that the following holds (with explanation given afterwards):
\begin{align*}
&\sup_{\r\in[\s,\t]}\sup_{\x\in\Z}|\r-\s|^{\frac12}\sum_{\y\in\Z}\exp(\tfrac{2\kappa|\x-\y|}{N})\||\mathbf{L}^{N,\zeta,\sim}_{\s,\r,\x,\y}|^{2}\|_{\mathrm{L}^{p}(\mathbb{P})}\\
&\lesssim\sup_{\r\in[\s,\t]}\sup_{\x\in\Z}|\r-\s|^{\frac12}\sum_{\y\in\Z}\exp(\tfrac{2\kappa|\x-\y|}{N})\||\mathbf{L}^{N,\zeta,\sim}_{\s,\r,\x,\y}-\mathbf{H}^{N}_{\s,\r,\x,\y}|^{2}\|_{\mathrm{L}^{p}(\mathbb{P})}\\
&+\sup_{\r\in[\s,\t]}\sup_{\x\in\Z}|\r-\s|^{\frac12}\sum_{\y\in\Z}\exp(\tfrac{2\kappa|\x-\y|}{N})|\mathbf{H}^{N}_{\s,\r,x,\y}|^{2}|\lesssim_{\delta,\kappa}N^{-1+\delta},
\end{align*}
where {\small$\delta>0$} is any small (but fixed) constant. The first bound is elementary. The second bound follows by \eqref{eq:lzetasimestimateII12a} and {Proposition \ref{prop:hke}}. Thus, the previous two displays yield
\begin{align}
\sum_{\y\in\Z}\exp(\tfrac{2\kappa|\x-\y|}{N})\Big\|\Big(\int_{\s}^{\t}\sum_{\w\in\Z}\mathbf{H}^{N}_{\r,\t,\x,\w}\cdot\boldsymbol{\Upsilon}_{\r,\w,\y}\d\r\Big)^{2}\Big\|_{\mathrm{L}^{p}(\mathbb{P})}\lesssim_{\delta,\kappa}N^{-1-\omega_{0}+\delta}.\label{eq:comparisonstep2I3}
\end{align}
Next, by the same \abbr{BDG}-martingale argument that gave us \eqref{eq:lzetasimestimateII3}, we have 
\begin{align}
&\sum_{\y\in\Z}\exp(\tfrac{2\kappa|\x-\y|}{N})\Big\|\Big(\int_{\s}^{\t}\sum_{\w\in\Z}\mathbf{H}^{N}_{\r,\t,\x,\w}{\boldsymbol{\chi}}^{(\zeta_{\mathrm{large}})}_{\w}\cdot[\mathscr{S}^{N}\star(\sqrt{2}\lambda N^{\frac12}\mathbf{R}^{N,\wedge}_{\r,\cdot}\mathbf{V}^{N}_{\s,\r,\cdot,\y}\d\mathbf{b}_{\r,\cdot})]_{\w}\Big)^{2}\Big\|_{\mathrm{L}^{p}(\mathbb{P})}\nonumber\\
&\lesssim_{\kappa} \int_{\s}^{\t}|\t-\s|^{-\frac12}\cdot\sup_{\x\in\Z}\sum_{\y\in\Z}\exp(\tfrac{2\kappa|\z-\y|}{N})\||\mathbf{V}^{N}_{\s,\r,\x,\y}|^{2}\|_{\mathrm{L}^{p}(\mathbb{P})}\d\r.\label{eq:comparisonstep2I4}
\end{align}
If we now combine \eqref{eq:comparisonstep2I2} with \eqref{eq:comparisonstep2I3}-\eqref{eq:comparisonstep2I4}, and choose {\small$\delta>0$} small enough, then we obtain
\begin{align*}
&\sup_{\x\in\Z}\sum_{\y\in\Z}\exp(\tfrac{2\kappa|\x-\y|}{N})\||\mathbf{V}^{N}_{\s,\r,\x,\y}|^{2}\|_{\mathrm{L}^{p}(\mathbb{P})}\lesssim_{\kappa}N^{-1-\frac12\omega_{0}}+\int_{\s}^{\t}|\t-\s|^{-\frac12}\cdot\sup_{\x\in\Z}\sum_{\y\in\Z}\exp(\tfrac{2\kappa|\z-\y|}{N})\||\mathbf{V}^{N}_{\s,\r,\x,\y}|^{2}\|_{\mathrm{L}^{p}(\mathbb{P})}\d\r.
\end{align*}
Now, we can use the Gronwall inequality so that the \abbr{LHS} of the previous estimate is {\small$\lesssim_{\kappa}N^{-1-\omega_{0}/2}$}. Now, by the same Chebyshev inequality, union bound, and short-time continuity argument that gave us \eqref{eq:lzetasimestimateII12}, from this estimate on the \abbr{LHS} of the previous display, we get the following with high probability:
\begin{align}
\sup_{0\leq\s\leq\t\leq1}\sup_{|\x|\lesssim N}\sum_{\y\in\Z}\exp(\tfrac{2\kappa|\x-\y|}{N})|\mathbf{V}^{N}_{\s,\r,\x,\y}|^{2}\lesssim N^{-1-\frac13\omega_{0}}.\label{eq:comparisonstep2I5}
\end{align}
Finally, we claim that the following estimate holds for large enough (but still {\small$\mathrm{O}(1)$}) constant {\small$\wt{\kappa}>0$}:
\begin{align*}
\exp(-\tfrac{\wt{\kappa}|\x|}{N})|\mathbf{Y}^{N,1}_{\t,\x}-\mathbf{Y}^{N,2}_{\t,\x}|&\leq\sum_{\y\in\Z}\exp(\tfrac{\wt{\kappa}|\x-\y|}{N})|\mathbf{V}^{N}_{0,\t,\x,\y}|\cdot\exp(-\tfrac{\wt{\kappa}|\y|}{N})|\mathbf{S}^{N}_{0,\y}|\\
&\lesssim_{\wt{\kappa},\delta}N^{\delta}\sum_{\y\in\Z}\exp(\tfrac{\wt{\kappa}|\x-\y|}{N})|\mathbf{V}^{N}_{0,\t,\x,\y}|\\
&\lesssim N^{\delta}\Big(\sum_{\y\in\Z}\exp(-\tfrac{\wt{\kappa}|\x-\y|}{N})\Big)^{\frac12}\Big(\sum_{\y\in\Z}\exp(\tfrac{3\wt{\kappa}|\x-\y|}{N})|\mathbf{V}^{N}_{0,\t,\x,\y}|^{2}\Big)^{\frac12}\lesssim_{\wt{\kappa}} N^{\delta}N^{-\frac16\omega_{0}}.
\end{align*}
The first estimate follows by \eqref{eq:comparisonstep1II}, \eqref{eq:comparisonstep2II}, and the triangle inequality. The second bound follows by {Lemma \ref{lemma:databound}}, and it holds for any {\small$\delta>0$}. The third estimate follows by Cauchy-Schwarz. The last estimate follows by \eqref{eq:comparisonstep2I5} and an elementary geometric series bound. We emphasize that the previous estimate holds on a high probability event simultaneously for {\small$\t\in[0,1]$} and {\small$|\x|\leq\mathrm{L}N$} if {\small$\mathrm{L}\lesssim1$}. But for such {\small$\x$}, the exponential factor on the far \abbr{LHS} of the previous display is bounded uniformly away from {\small$0$}. Thus, by changing the implied constant, we can drop this exponential factor. We then deduce the desired estimate \eqref{eq:comparisonstep2I}, so the proof is complete. \qed
\subsubsection{Proof of Lemma \ref{lemma:comparisonstep3}}
Recalling {\small$\mathbf{Y}^{N,2}$} from \eqref{eq:comparisonstep2II} and taking a time-differential gives
\begin{align}
\d\mathbf{Y}^{N,2}_{\t,\x}&=\mathscr{T}_{N}\mathbf{Y}^{N,2}_{\t,\x}\d\t+\boldsymbol{\chi}^{(\zeta_{\mathrm{large}})}_{\x}[\mathscr{S}^{N}\star(\sqrt{2}\lambda N^{\frac12}\mathbf{R}^{N,\wedge}_{\t,\cdot}\mathbf{Y}^{N,2}_{\t,\cdot}\d\mathbf{b}_{\t,\cdot})]_{\x}.\label{eq:comparisonstep3I1}
\end{align}
Let us now consider the following modification of \eqref{eq:comparisonstep3I1}, which removes the {\small$\boldsymbol{\chi}$} function therein:
\begin{align}
\d\mathbf{Y}^{N,3}_{\t,\x}&=\mathscr{T}_{N}\mathbf{Y}^{N,3}_{\t,\x}\d\t+[\mathscr{S}^{N}\star(\sqrt{2}\lambda N^{\frac12}\mathbf{R}^{N,\wedge}_{\t,\cdot}\mathbf{Y}^{N,3}_{\t,\cdot}\d\mathbf{b}_{\t,\cdot})]_{\x}.\label{eq:comparisonstep3I2}
\end{align}
The initial data {\color{black}for} these two equations are given by {\small$\mathbf{Y}^{N,2}_{0,\cdot},\mathbf{Y}^{N,3}_{0,\cdot}=\mathbf{S}^{N}_{0,\cdot}$}. By {{Lemmas \ref{lemma:rbound}, \ref{lemma:phibound}, \ref{lemma:avldp}, and \ref{lemma:avldptx}}, we have {\small$\mathfrak{t}_{\mathrm{ap}}=1$} with high probability. Moreover, since {\small$\zeta_{\mathrm{large}}>0$} is large, we have the following analogue of Lemma \ref{lemma:aprioricompact} with high probability for any {\small$\mathrm{L}>0$} independent of {\small$N$}:
\begin{align}
\sup_{\t\in[0,1]}\sup_{|\x|\leq\mathrm{L}N}|\mathbf{Y}^{N,2}_{\t,\x}-\mathbf{Y}^{N,3}_{\t,\x}|\lesssim N^{-\gamma_{\mathrm{L}}}.\label{eq:comparisonstep3I3}   
\end{align}
Above, {\small$\gamma_{\mathrm{L}}>0$} is deterministic and independent of {\small$N$}. Now, we define {\small$\mathbf{Y}^{N,4}$} via the relation {\small$\mathbf{Y}^{N,3}=\mathscr{S}^{N}\star\mathbf{Y}^{N,4}$}. In particular, it is the solution to the following \abbr{SDE} with initial data {\small$\mathbf{Y}^{N,4}_{0,\cdot}=\mathbf{Z}^{N}_{0,\cdot}$} (see \eqref{eq:zsmooth}):
\begin{align}
\d\mathbf{Y}^{N,4}_{\t,\x}&=\mathscr{T}_{N}\mathbf{Y}^{N,4}_{\t,\x}\d\t+\sqrt{2}\lambda N^{\frac12}\mathbf{R}^{N,\wedge}_{\t,\x}[\mathscr{S}^{N}\star\mathbf{Y}^{N,4}_{\t,\cdot}]_{\x}\d\mathbf{b}_{\t,\x}.\label{eq:comparisonstep3I2}
\end{align}
This is the same \abbr{SDE} as the equation \eqref{eq:qsde} for {\small$\mathbf{Q}^{N}$} (with the same initial data), except for the factor of {\small$\mathbf{R}^{N,\wedge}$} and the convolution {\small$\mathscr{S}^{N}$} acting on {\small$\mathbf{Y}^{N,4}$}. However, we know from \eqref{eq:rwedge} that {\small$|\mathbf{R}^{N,\wedge}_{\t,\x}-1|\lesssim N^{-\delta_{\mathbf{S}}/3}$}. Moreover, standard moment estimates (see Lemma 4.2 in \cite{BG}) show that {\small$\mathbf{Y}^{N,4}$} has regularity on length-scales of order {\small$N$}. That is, let us fix any {\small$\ell\in\Z$} with {\small$|\ell|\lesssim N$}, and fix any {\small$p,\kappa\geq1$} large (but {\small$\mathrm{O}(1)$}). For any {\small$\upsilon\in(0,1/2)$}, we have
\begin{align*}
\sup_{\t\in[0,1]}\sup_{\x\in\Z}\exp(\tfrac{-2p\kappa|\x|}{N})\E|\mathbf{Y}^{N,4}_{\t,\x}+\mathbf{Y}^{N,4}_{\t,\x+\ell}|^{2p}\lesssim_{p} N^{-2p\upsilon}|\ell|^{2p\upsilon}.
\end{align*}
Thus, because {\small$\mathscr{S}^{N}$} is {\color{black}a} convolution on {\color{black}a} scale {\small$\lesssim N^{1-\delta_{\mathbf{S}}}\ll N$}, we have {\small$\mathbf{Y}^{N,3}:=\mathscr{S}^{N}\star\mathbf{Y}^{N,4}\sim\mathbf{Y}^{N,4}$} (in the sense of {\small$2p$}-th moments as in the previous display). At this point, it is standard stochastic analysis to compare {\small$\mathbf{Y}^{N,4},\mathbf{Q}^{N}$} and show that the following holds with high probability, in which {\small$|\mathrm{L}|\lesssim1$}, where {\small$\gamma_{\mathrm{L}}$} again depends only on {\small$\mathrm{L}$}:
\begin{align}
\sup_{\t\in[0,1]}\sup_{|\x|\leq\mathrm{L}N}|\mathbf{Y}^{N,3}_{\t,\x}-\mathbf{Y}^{N,4}_{\t,\x}|+\sup_{\t\in[0,1]}\sup_{|\x|\leq\mathrm{L}N}|\mathbf{Y}^{N,4}_{\t,\x}-\mathbf{Q}^{N}_{\t,\x}|\lesssim_{\mathrm{L}} N^{-\gamma_{\mathrm{L}}}.\nonumber
\end{align}
Combining the previous display with \eqref{eq:comparisonstep1I3} gives the desired estimate \eqref{eq:comparisonstep3I}, so we are done. \qed
\appendix
\section{General facts}
\subsection{Deterministic heat kernel estimates}
We record below estimates for the heat kernel {\small$\mathbf{H}^{N}$} from \eqref{eq:heatkernel}.
\begin{prop}\label{prop:hke}
\fsp Fix any {\small$(\s,\t,\x,\y)\in[0,\infty)^{2}\times\Z^{2}$} with {\small$\s\leq\t\lesssim1$}. Fix any {\small$\kappa>0$}. We have
\begin{align}
\exp(\tfrac{\kappa|\x-\y|}{N})\mathbf{H}^{N}_{\s,\t,\x,\y}&\lesssim_{\kappa}N^{-1}|\t-\s|^{-\frac12},\label{eq:hkeIa}\\
\exp(\tfrac{\kappa|\x-\y|}{N})|[\mathscr{S}^{N}\star\mathbf{H}^{N}_{\s,\t,\cdot,\y}]_{\x}|&\lesssim_{\kappa}N^{-1}|\t-\s|^{-\frac12},\label{eq:hkeIb}\\
\sum_{\w\in\Z}\exp(\tfrac{2\kappa|\x-\w|}{N})\mathbf{H}^{N}_{\s,\t,\x,\w}&\lesssim_{\kappa}1,\label{eq:hkeIc}\\
\sum_{\w\in\Z}\exp(\tfrac{2\kappa|\x-\w|}{N})|\mathbf{H}^{N}_{\s,\t,\x,\w}|^{2}&\lesssim_{\kappa}N^{-1}|\t-\s|^{-\frac12}.\label{eq:hkeId}
\end{align}
\end{prop}
\begin{proof}
\eqref{eq:hkeIa} follows from (A.12) in \cite{DT}, noting that in \cite{DT}, the {\small$t$}-variable therein should be replaced by {\small$N^{2}t$} given the scaling in \eqref{eq:heatkernel}. \eqref{eq:hkeIb} follows from
\begin{align*}
\exp(\tfrac{\kappa|\x-\y|}{N})|[\mathscr{S}^{N}\star\mathbf{H}^{N}_{\s,\t,\cdot,\y}]_{\x}|&\lesssim N^{-1+\delta_{\mathbf{S}}}\sum_{|\w-\x|\lesssim N^{1-\delta_{\mathbf{S}}}}\exp(\tfrac{\kappa|\x-\w|}{N})\cdot\exp(\tfrac{\kappa|\w-\y|}{N})\mathbf{H}^{N}_{\s,\t,\w,\y}
\end{align*}
and then using \eqref{eq:hkeIa} and noting that {\small$\exp(\kappa|\x-\w|/N)\lesssim_{\kappa}1$} for all {\small$|\x-\w|\ll N$}. \eqref{eq:hkeIc} follows from (A.24) in \cite{DT} (noting that {\small$\e$} therein is equal to {\small$N^{-1}$} here). \eqref{eq:hkeId} follows by \eqref{eq:hkeIa} and \eqref{eq:hkeIc}.
\end{proof}
\subsection{Positivity of $\alpha$ and non-zero $\beta$}\label{subsection:alpha-beta}
Recall from the line before \eqref{eq:beta} that {\small$\alpha=\partial_{\sigma}\E^{\sigma}\mathscr{U}'[\bphi_{0}]|_{\sigma=0}$}. By \eqref{eq:classify1} for {\small$\mathsf{F}\equiv1$}, we deduce that {\small$\E^{\sigma}\mathscr{U}'[\bphi_{0}]=\upsilon_{\sigma}$} and thus {\small$\alpha=\partial_{\sigma}\upsilon_{\sigma}|_{\sigma=0}$}. By Remark 1.3 in \cite{DGP}, this is strictly positive. 

Now, recall that {\small$\beta=\frac12\partial_{\sigma}^{2}\E^{\sigma}(\mathbf{F}_{2}[\bphi])|_{\sigma=0}$} from \eqref{eq:beta}. By \eqref{eq:phi-nl}, we have {\small$\beta=\frac12\beta_{2}\partial_{\sigma}^{2}\E^{\sigma}(\mathscr{U}'[\bphi_{0}]\mathscr{U}'[\bphi_{1}])|_{\sigma=0}$}. By independence of {\small$\bphi_{0},\bphi_{1}$} under {\small$\E^{\sigma}$} and by \eqref{eq:classify1} with {\small$\mathsf{F}\equiv1$}, we deduce {\small$\beta=\frac12\beta_{2}\partial_{\sigma}^{2}(\upsilon_{\sigma}^{2})|_{\sigma=0}=\beta_{2}\upsilon''_{0}\upsilon_{0}+\beta_{2}(\upsilon'_{0})^{2}$}. By Assumption \ref{assump:potential}, we know {\small$\upsilon_{0}=0$}. Again, by Remark 1.3 in \cite{DGP}, we have {\small$\upsilon'_{0}>0$}. So, {\small$\beta\neq0$} if {\small$\beta_{2}\neq0$}.
\subsection{Infinitesimal generators}\label{subsection:generator}
We first note that the following discussion holds if we restrict \eqref{eq:phi} to a discrete torus {\small$\Z/\ell\Z$} and replace {\small$\Z$} with said torus. The generator of {\small$\t\mapsto\bphi_{\t,\cdot}$} from \eqref{eq:phi} is {\small$\mathscr{L}_{\bphi}:=\mathscr{L}_{\bphi,\mathrm{S}}+\mathscr{L}_{\bphi,\mathrm{A}}$}, where
\begin{align}
\mathscr{L}_{\bphi,\mathrm{S}}&:=N^{2}\sum_{\x\in\Z}\Big\{(\partial_{\bphi_{\x}}-\partial_{\bphi_{\x+1}})^{2}-(\mathscr{U}'[\bphi_{\x}]-\mathscr{U}'[\bphi_{\x+1}])(\partial_{\bphi_{\x}}-\partial_{\bphi_{\x+1}})\Big\},\label{eq:gen-sym}\\
\mathscr{L}_{\bphi,\mathrm{A}}&:=N^{\frac32}\sum_{\x\in\Z}\wt{\mathbf{F}}[\tau_{\x}\bphi]\partial_{\bphi_{\x}}.\label{eq:gen-asymm}
\end{align}
Fix any smooth local functions {\small$\mathsf{f},\mathsf{g}:\R^{\Z}\to\R$}, i.e. {\small$\mathsf{f}[\bphi],\mathsf{g}[\bphi]$} depend only on {\small$\bphi_{\x}$} for {\small$\x$} in some finite subset of {\small$\Z$}. It is shown in Lemma 2.1 of \cite{DGP} that {\small$\E^{\sigma}\mathsf{f}\mathscr{L}_{\bphi,\mathrm{S}}\mathsf{g}=\E^{\sigma}\mathsf{g}\mathscr{L}_{\bphi,\mathrm{S}}\mathsf{f}$}. Next, by recalling the formula \eqref{eq:gcmeasure} and integrating by parts, we have the following calculation:
\begin{align*}
&\E^{\sigma}\mathsf{f}\mathscr{L}_{\bphi,\mathrm{A}}\mathsf{g}=N^{\frac32}\sum_{\x\in\Z}\E^{\sigma}\mathsf{f}[\bphi]\cdot\wt{\mathbf{F}}[\tau_{\x}\bphi]\cdot\partial_{\bphi_{\x}}\mathsf{g}[\bphi]\\
&=-N^{\frac32}\sum_{\x\in\Z}\E^{\sigma}\mathsf{g}[\bphi]\cdot\wt{\mathbf{F}}[\tau_{\x}\bphi]\cdot\partial_{\bphi_{\x}}\mathsf{f}[\bphi]+N^{\frac32}\E^{\sigma}\mathsf{g}[\bphi]\mathsf{f}[\bphi]\cdot\sum_{\x\in\Z}\wt{\mathbf{F}}[\tau_{\x}\bphi]\cdot\mathscr{U}'[\bphi_{\x}]-\upsilon_{\sigma}N^{\frac32}\E^{\sigma}\mathsf{g}[\bphi]\mathsf{f}[\bphi]\cdot\sum_{\x\in\Z}\wt{\mathbf{F}}[\tau_{\x}\bphi]
\end{align*}
The last two terms vanish as they are {\color{black}telescopic} sums (see \eqref{eq:phi-nl}), so {\small$\E^{\sigma}\mathsf{f}\mathscr{L}_{\bphi,\mathrm{A}}\mathsf{g}=-\E^{\sigma}\mathsf{g}\mathscr{L}_{\bphi,\mathrm{A}}\mathsf{f}$}. As in Corollary 2.2 in \cite{DGP}, this and {\small$\E^{\sigma}\mathsf{f}\mathscr{L}_{\bphi,\mathrm{S}}\mathsf{g}=\E^{\sigma}\mathsf{g}\mathscr{L}_{\bphi,\mathrm{S}}\mathsf{f}$} imply invariance of {\small$\mathbb{P}^{\sigma}$} under \eqref{eq:phi}. From these calculations, we also have that {\small$\mathscr{L}_{\bphi,\mathrm{S}}$} is symmetric with respect to {\small$\mathbb{P}^{\sigma}$}, and {\small$\mathscr{L}_{\bphi,\mathrm{A}}$} is anti-symmetric with respect to {\small$\mathbb{P}^{\sigma}$}. So, if {\small$\bphi_{0,\cdot}\sim\mathbb{P}^{\sigma}$}, the time-reversed process {\small$\r\mapsto\bphi_{\t-\r}$} for any deterministic {\small$\t\geq0$} and {\small$\r\in[0,\t]$} is a Markov process with generator {\small$\mathscr{L}_{\bphi}^{\ast}=\mathscr{L}_{\bphi,\mathrm{S}}-\mathscr{L}_{\bphi,\mathrm{A}}$}, where the {\small$\ast$}-superscript denotes adjoint with respect to {\small$\mathbb{P}^{\sigma}$}. This follows by standard Markov process theory; see the proof of Lemma 2.4 in \cite{KLO} (in particular, (2.21) therein).
\subsection{An initial data estimate}
We now record a technical estimate for the initial data {\small$\mathbf{Z}^{N}_{0,\cdot}$}.
\begin{lemma}\label{lemma:databound}
\fsp Fix {\small$\delta>0$}. If {\small$\kappa$} is large enough, then with high probability, we have (for all {\small$\x\in\Z$}) that
\begin{align}
\exp(-\tfrac{\kappa|\x|}{N})\mathbf{Z}^{N}_{0,\x}+\exp(-\tfrac{\kappa|\x|}{N})\mathbf{S}^{N}_{0,\x}\lesssim N^{\delta}.\label{eq:databoundI}
\end{align}
\end{lemma}
\begin{proof}
The estimate for {\small$\mathbf{S}^{N}$} follows by the claim for {\small$\mathbf{Z}^{N}$} because {\small$\mathbf{S}^{N}$} is a convolution of {\small$\mathbf{Z}^{N}$} against a probability measure on {\small$\Z$} that is supported on a neighborhood of the origin of length {\small$\ll N$}; see Definition \ref{definition:zsmooth}. So, it suffices to show the estimate for {\small$\mathbf{Z}^{N}$}. For this, we use a union bound, the Chebyshev inequality, and assumptions on the initial data in Theorem \ref{theorem:main} to obtain the following estimate in which {\small$p\geq1$} is large but {\small$\mathrm{O}(1)$}:
\begin{align*}
\mathbb{P}\Big(\sup_{\x\in\Z}\exp(-\tfrac{\kappa|\x|}{N})\mathbf{Z}^{N}_{0,\x}\geq N^{\delta}\Big)&\lesssim N^{-2p\delta}\sum_{\x\in\Z}\exp(-\tfrac{2p\kappa|\x|}{N})\E|\mathbf{Z}^{N}_{0,\x}|^{2p}\lesssim_{p}N^{-2p\delta}\sum_{\x\in\Z}\exp(-\tfrac{2p\kappa|\x|}{N})\exp(\tfrac{\kappa_{p}|\x|}{N}).
\end{align*}
Choose {\small$p$} large so that {\small$2p\delta\geq2$}. If we take {\small$\kappa$} large enough depending only on {\small$p$}, then the last summation is {\small$\lesssim_{p}N$}. Thus, \eqref{eq:databoundI} follows, and the proof is complete.
\end{proof}
\subsection{Stability of \abbr{SHE} with Dirac initial data}
We now record the following result that is used in the proof of Theorem \ref{theorem:mainwedge} to show convergence to \abbr{SHE} with Dirac initial data.
\begin{lemma}\label{lemma:qstabilitywedge}
\fsp Suppose {\small$\wt{\mathbf{Q}}^{N}:[0,\infty)\times\Z\to\R$} satisfies the following \abbr{SDE}, where {\small$\mathscr{T}_{N}$} is from Proposition \ref{prop:s-sde}:
\begin{align}
\d\wt{\mathbf{Q}}^{N}_{\t,\x}=\mathscr{T}_{N}\wt{\mathbf{Q}}^{N}_{\t,\x}\d\t+\sqrt{2}\lambda N^{\frac12}\wt{\mathbf{Q}}^{N}_{\t,\x}\d\mathbf{b}_{\t,\x}.\label{eq:qtildesde}
\end{align}
Suppose also that the initial data is non-negative, possibly random, and satisfies the deterministic bounds
\begin{align}
\sup_{\x\in\Z}\mathrm{e}^{\t\mathscr{T}_{N}}[\wt{\mathbf{Q}}^{N}_{0,\cdot}]_{\x}\lesssim\t^{-1/2}\quad\mathrm{and}\quad N^{-1}\sum_{\x\in\Z}\wt{\mathbf{Q}}^{N}_{0,\x}\lesssim1.\label{eq:qtildesdeII}
\end{align}
Next, suppose that for any {\small$\t>0$} fixed, the function {\small$\X\mapsto\mathrm{e}^{\t\mathscr{T}_{N}}[\wt{\mathbf{Q}}^{N}_{0,\cdot}]_{N\X}$} converges uniformly in {\small$\X\in\R$} to {\small$\mathbf{H}_{0,\t,\X,0}$}. Then, there exists a coupling between {\small$\wt{\mathbf{Q}}^{N}$} and the narrow-wedge solution {\small$\mathbf{Z}^{\mathrm{nw}}$} to \eqref{eq:she} so that for any fixed and small {\small$\tau>0$}, we have {\small$\wt{\mathbf{Q}}^{N}_{\t,N\X}-\mathbf{Z}^{\mathrm{nw}}_{\t,N\X}\to0$} locally uniformly on {\small$[\tau,1]\times\R$} in probability.
\end{lemma}
\begin{proof}
The argument is the same as {\color{black}in} the proofs of Lemmas 5.1 and 5.2 in \cite{DT}. Indeed, Lemma \ref{lemma:qstabilitywedge} and its proof are about the \abbr{SDE} for {\small$\wt{\mathbf{Q}}^{N}$} for a choice of initial data converging to a Dirac point mass. Thus, we will only sketch the proof. Below, {\small$\wt{\mathbf{Q}}^{N}$} plays the same role as {\small$\mathscr{Z}^{\e}$} in Lemmas 5.1 and 5.2 of \cite{DT}, and {\small$\e$} there is {\small$N^{-1}$} here.

This argument proceeds in two steps. First, we show that for any $\eta>0$ independent of $N$, the data {\small$\x\mapsto\wt{\mathbf{Q}}^{N}_{\eta,\x}$} satisfies the initial data assumptions of Theorem \ref{theorem:main}, i.e. that 
\begin{align}
\sup_{\x\in\Z}\E|\wt{\mathbf{Q}}^{N}_{\eta,\x}|^{2p}\lesssim_{p,\eta}1\quad\mathrm{and}\quad\sup_{\x\in\Z}\sup_{|\mathfrak{l}|\lesssim N}\E|\grad^{\mathbf{X}}_{\mathfrak{l}}\wt{\mathbf{Q}}^{N}_{\eta,\x}|^{2p}\lesssim_{p,\eta}N^{-p+p\zeta}|\mathfrak{l}|^{p-p\zeta}\label{eq:qstabilitywedge1}
\end{align}
given any $p\geq1$ and any $\zeta>0$. This follows from standard moment estimates for the Duhamel version of \eqref{eq:qtildesde}, using the a priori inputs \eqref{eq:qtildesdeII}; the details are in the proof of Lemma 5.1 in \cite{DT}. Next, by Kolmogorov continuity (and Skorokhod representation), there is a subsequence along which {\small$\X\mapsto\wt{\mathbf{Q}}^{N}_{\eta,N\X}$} converges locally uniformly in {\small$\X\in\R$} in probability. Call this limit {\small$\wt{\mathbf{Q}}_{\eta,\cdot}$}. Now, using \eqref{eq:qstabilitywedge1} and the stability of \abbr{SHE} with continuous initial data (see Theorem 2.1 in \cite{BG}), we obtain that along this subsequence, {\small$(\t,\X)\mapsto\wt{\mathbf{Q}}^{N}_{\t,N\X}$} converges locally uniformly on {\small$[\eta,1]\times\R$} to the solution $\wt{\mathbf{Q}}$ of
\begin{align}
\wt{\mathbf{Q}}_{\t,\X}&={\int_{\R}}\mathbf{H}_{\eta,\t,\X,\Y}\wt{\mathbf{Q}}_{\eta,\Y}\d\Y+\lambda{\int_{\eta}^{\t}\int_{\R}}\mathbf{H}_{\s,\t,\X,\Y}\wt{\mathbf{Q}}_{\s,\Y}\xi_{\s,\Y}\d\Y\d\s.\label{eq:qstabilitywedgelimit}
\end{align}
To conclude the proof of Lemma \ref{lemma:qstabilitywedge}, it suffices to use the assumed convergence of {\small$\X\mapsto\mathrm{e}^{\t\mathscr{T}_{N}}[\wt{\mathbf{Q}}^{N}_{0,\cdot}]_{N\X}$} uniformly in {\small$\X\in\R$} to {\small$\mathbf{H}_{0,\t,\X,0}$} and show that pointwise on {\small$(0,1]\times\R$}, we have
\begin{align}
\mathbf{H}_{0,\t,\X,\y}-{\int_{\R}}\mathbf{H}_{\eta,\t,\X,\Y}\wt{\mathbf{Q}}_{\eta,\Y}\d\Y&\to_{\eta\to0}0\label{eq:qstabilitywedge2}\\
{\int_{0}^{\eta}\int_{\R}}\mathbf{H}_{\s,\t,\X,\Y}\wt{\mathbf{Q}}_{\s,\Y}\xi_{\s,\Y}\d\Y\d\s&\to_{\eta\to0}0\label{eq:qstabilitywedge3}
\end{align}
in probability. Indeed, this whole argument thus far would imply that any subsequence of {\small$(\t,\X)\mapsto\wt{\mathbf{Q}}^{N}_{\t,N\X}$} has a subsequential limit in probability on {\small$[\tau,1]\times\R$} (in the local-uniform-convergence topology) which coincides on {\small$[\tau,1]\times\R$} with the narrow-wedge solution to \eqref{eq:she} for any fixed $\tau>0$, thereby completing the proof. 
\end{proof}
%
%
%


\end{document}